%% file: main.tex
\documentclass{memo-l}

\usepackage{amsmath}
\usepackage{amsthm}
\usepackage{verbatim}
\usepackage{enumitem}
\usepackage[table]{xcolor}
\usepackage{amssymb, latexsym, mathrsfs, graphicx, fancyhdr,bm}
\usepackage{multirow}
\usepackage[final]{listings}
\usepackage{longtable}
\usepackage[leftmargin=3em]{quoting}

\usepackage[hidelinks]{hyperref}

\newtheorem{mainthA}{Theorem}

\newtheorem{mainthB}{Theorem}

\newtheorem{mainthC}{Theorem}

\newtheorem*{mmainth}{Main Theorem}

\newtheorem{Th}{Theorem}[chapter]
\newtheorem{Lem}[Th]{Lemma}
\newtheorem{Prob}[Th]{Problem}
\newtheorem{Cor}[Th]{Corollary}
\newtheorem{Que}[Th]{Question}
\newtheorem{Prop}[Th]{Proposition}
\newtheorem*{T1}{Theorem~\ref{theorem}}
\newtheorem*{T2}{Theorem~\ref{theoremGU}}
\newtheorem*{T3}{Theorem~\ref{theoremSp}}
\newtheorem*{T4}{Theorem~\ref{theoremGR}}
\newtheorem*{T5}{Theorem~\ref{theoremSpGR}}

\theoremstyle{definition}
\newtheorem{Def}[Th]{Definition}
\newtheorem*{Def*}{Definition}
\newtheorem{example}[Th]{Example}

\theoremstyle{remark}
\newtheorem{Rem}[Th]{Remark}

\numberwithin{section}{chapter}
\numberwithin{equation}{chapter}

\newcommand{\fpr}{{\mathrm{fpr}}}

\newcommand{\Sym}{\mathrm{Sym}}
\newcommand{\Fix}{\mathrm{Fix}}
\newcommand{\Gal}{\mathrm{Gal}}
\newcommand{\diag}{\mathrm{diag}}

\newcommand{\Stab}{\mathrm{Stab}}
\newcommand{\sgn}{\mathrm{sgn}}
\newcommand{\per}{\mathrm{perm}}
\newcommand{\Det}{\mathrm{Det}}
\newcommand{\Reg}{\mathrm{Reg}}
\newcommand{\Aut}{\mathrm{Aut}}

\newcommand{\GL}{{{\Gamma L}}}
\newcommand{\GU}{{{\Gamma U}}}
\newcommand{\GS}{{{\Gamma Sp}}}
\newcommand*{\Scale}[2][4]{\scalebox{#1}{$#2$}}

\makeindex

\begin{document}

\frontmatter

\title{Intersection of conjugate solvable subgroups in finite classical groups}


\author{Anton A. Baykalov}
\address{The University of Auckland, Auckland, New Zealand}
\curraddr{}
\email{a.baykalov@auckland.ac.nz}
\thanks{}

\date{\today}




\begin{abstract}
 We consider the following problem stated by Vdovin (2010) in the ``Kourovka notebook''  (Problem 17.41):

\medskip

Let $H$ be a solvable subgroup of a finite group $G$ that has no nontrivial solvable
normal subgroups.
 Do there always exist five conjugates of $H$ whose intersection is trivial?

\medskip

This problem is closely related to a conjecture by Babai, Goodman and Pyber (1997) about an upper bound for the index of a normal solvable subgroup in a finite group. In particular, a positive answer to Vdovin's problem yields that  if $G$ has a solvable subgroup of index $n$, then it has a solvable normal subgroup of index at most $n^5$.

The problem was reduced by Vdovin (2012) to the case when $G$ is an almost simple group.  Let $G$ be an almost simple group with  socle isomorphic to a simple linear, unitary or symplectic group. For all such groups $G$ we provide a positive answer to Vdovin's  problem.
\end{abstract}

\maketitle

\tableofcontents

\include{akn}

\mainmatter
\include{ch1}

\include{ch2}

\include{ch3}

\include{ch4}


\backmatter
\bibliographystyle{abbrv}
\bibliography{main.bib}
\printindex

\end{document}

%% file: akn.tex
\chapter*{Acknowledgments}
The work presented in this volume was done while I was a PhD student at the University of Auckland. I thank my supervisors Eamonn O'Brien and Jianbei An for all their guidance and inspiration. I thank the University of Auckland for a PhD scholarship and for support while this volume was prepared. I also thank Professor Timothy  Burness and Professor Peter Cameron, who were  examiners of my PhD thesis, for their constructive comments and corrections.

%% file: ch1.tex
\chapter{Introduction}

\section{Statement of the problem}

Consider some property $\Psi$ of a finite group inherited by all its subgroups.  Important examples of  such a property are  the following:
\begin{itemize}
\item cyclicity; 
\item commutativity; 
\item nilpotence; 
\item solvability.
\end{itemize}
A natural question arises: how large is a normal $\Psi$-subgroup in an arbitrary finite group $G$? A more precise formulation of this question is the following:

\begin{Que}\label{que1}
Given a finite group $G$ with $\Psi$-subgroup $H$ of index $n$, is it true that $G$ has a normal $\Psi$-subgroup whose index is bounded by some function $f(n)?$
\end{Que}

Since the kernel of the action of $G$ on the set of right cosets of $H$ by right multiplication is a subgroup of $H$ and such an action provides a homomorphism to the symmetric group $\Sym(n),$   it always suffices to take $f(n)=n!$ for every such $\Psi$. We are interested in  stronger bounds, in particular those of shape $f(n)=n^c$ for some constant $c.$

Babai, Goodman and Pyber \cite{bab} prove some related results and state several conjectures.
In particular, they prove that
if a finite group $G$ has a cyclic subgroup ${C}$ of index $n$, then $ \cap_{g \in G} {{C}}^g$ has index at most $n^7.$
They also conjectured that the bound $n^2-n$ holds  and showed that it is best possible. Lucchini \cite{luccini} and, independently,  Kazarin and  Strunkov \cite{kaz} proved this, so for cyclicity the question is resolved. 

\begin{Th}
If a finite group $G$ has a cyclic subgroup ${C}$ of index $n$, then $ \cap_{g \in G} {{C}}^g$ has index at most $n^2-n.$
\end{Th}

The following theorem about commutativity follows from results  by Chermak and  Delgado \cite{cher}:
\begin{Th}
Let $G$ be a finite group. If $G$ has an abelian subgroup of index $n$, then 
 it has a normal abelian subgroup of index at most $n^2.$
\end{Th}
\noindent While this bound is not best possible, it is the best of shape $n^c.$

Zenkov \cite{zen21} proved the following  when $\Psi$ is nilpotence.

\begin{Th}
Let $G$ be a finite group and let ${\bf F}(G)$ be its maximal normal nilpotent subgroup.  If $G$ has a nilpotent subgroup of index $n$, then 
$|G:{\bf F}(G)| \le n^3.$
\end{Th}
Babai, Goodman and Pyber \cite{bab}   proved the following statement.

\begin{Th}\label{indexBab}
There is an absolute constant $c$ such that, if a finite group $G$ has a
solvable subgroup of index $n$, then $G$ has a solvable normal subgroup of index
at most $n^c$.
\end{Th}

\noindent Although their proof does not yield an explicit value, they  conjectured  that $c \le 7$.

This conjecture is closely related to  \cite[Problem 17.41 b)]{kt}:

\begin{Prob}\label{prob}
Let $H$ be a solvable subgroup of a finite group $G$ that has no nontrivial solvable
normal subgroups.
 Do there always exist five conjugates of $H$ whose intersection is trivial?
\end{Prob}

Before we explain how Problem \ref{prob} is related to Question \ref{que1}, we need to introduce some notation. Problem \ref{prob} can be reformulated using the notion of {\bf base size}.

\begin{Def}\label{def1}
Assume that a finite group $G$ acts on a set $\Omega.$ A point $\alpha \in \Omega$ is  $G${\bf-regular}   
 if its stabiliser in $G$ is trivial.
Define the action of  $G$ on $\Omega^k$ by 
$$(\alpha_1, \ldots,\alpha_k)g = (\alpha_1g,\ldots,\alpha_kg).$$
If $G$ acts faithfully and transitively on $\Omega$, then the minimal number $k$ such that the set $\Omega^k$ contains a 
$G$-regular point is the
{\bf base size} \index{base size} of $G$ and is denoted by  $b(G).$ For a positive integer $m$, a regular point in $\Omega^m$  is a {\bf base} \index{base} for the action of $G$ on $\Omega.$
Denote the number of $G$-regular orbits on $\Omega^m$   by $\Reg(G,m)$ (this number is 0 if $m < b(G)$).
If  $G$ acts by  right multiplication on the set $\Omega$ of right cosets of a subgroup $H$, then $G/H_G$ acts faithfully and transitively on $\Omega.$ (Here $H_G=\cap_{g \in G} H^g.$) In this case, we denote 
$$b_H(G):=b(G/H_G) \text{ and } \Reg_H(G,m):=\Reg(G/H_G,m).$$
\end{Def}

Therefore, for $G$ and $H$ as in Problem \ref{prob}, the existence of five conjugates of $H$ whose intersection is trivial is equivalent to the statement that $b_H(G)\le 5.$ Notice that $5$ is the best possible bound for $b_H(G)$ since $b_H(G)=5$ if $G=\Sym(8)$ and $H=\Sym(4) \wr \Sym(2).$ This can be easily verified. In fact, there are infinitely many examples with $b_H(G)=5$, for example see \cite[Remark 8.3]{burPS}.

 Let  $G$ act transitively on $\Omega$ and let $H$ be a point stabiliser, so $|\Omega|=|G:H|.$ If $(\beta_1 , \ldots , \beta_n )$ is a base for the natural action of $G/H_G$ on $\Omega$, then $$|(\beta_1 , \ldots , \beta_n )^G| \le |\Omega| \cdot (|\Omega|-1) \ldots (|\Omega|-n+1)<|\Omega|^n=|G:H|^n.$$
Therefore,
$$|G:H_G|<|G:H|^n,$$
and if Problem \ref{prob} has a positive answer, then $c \le 5$ in Theorem \ref{indexBab}.

\medskip

A finite group $G$ is {\bf almost simple} if $$G_0 \le G \le \Aut(G_0)$$
for some non-abelian simple group $G_0.$

Problem \ref{prob} is essentially reduced to the case when $G$ is almost simple  by Vdovin  \cite{vd}. We introduce some notation before stating the reduction theorem.

Let $A$ and $B$ be subgroups of $G$ such that $B \trianglelefteq A.$ Then $N_G(A/B) := N_G(A) \cap N_G(B)$ is the {\bf normaliser} of $A/B$ in $G$. If $x\in N_G(A/B)$, then $x$ induces an
automorphism of $A/B$ by $Ba \mapsto Bx^{-1}ax.$ Thus, there exists a homomorphism $N_G(A/B) \to \Aut(A/B).$ The image of $N_G(A/B)$ under this homomorphism is
denoted by $\Aut_G(A/B)$ and is the {\bf group of $G$-induced automorphisms of $A/B$}.

\begin{Th}[\cite{vd}]\label{sved}
Let $G$ be a finite group and let
$$\{1\}=G_0<G_1< G_2< \ldots < G_n=G $$
be a composition series of $G$ which is a refinement of a chief series. We identify non-abelian $G_i/G_{i-1}$ with the isomorphic normal subgroup of $\Aut_G(G_i/G_{i-1})$.   Assume that for
some $k$ the following condition holds: If $G_i/G_{i-1}$ is non-abelian, then for
every solvable subgroup $T$ of $\Aut_G(G_i/G_{i-1})$  
$$b_T(T \cdot (G_i/G_{i-1})) \le k \mbox{ and } \Reg_T(T \cdot (G_i/G_{i-1}),k)\ge5.$$
Then $b_H(G)\le k$ for every maximal solvable subgroup $H$ of $G$.  
\end{Th}

\begin{Rem}
The formulation of Theorem \ref{sved} in \cite{vd}  differs from ours. Specifically, the condition there  is the following: 
\medskip
\begin{quoting}[leftmargin=\parindent]
If $G_i/G_{i-1}$ is non-abelian, then for
every solvable subgroup $T$ of $\Aut_G(G_i/G_{i-1})$  
$$b_T(\Aut_G(G_i/G_{i-1})) \le k \mbox{ and } \Reg_T(\Aut_G(G_i/G_{i-1}),k)\ge5.$$
\end{quoting}
But the proof uses our formulation of the condition. An updated version of \cite{vd} is available on the arXiv; see the link in the Bibliography. 
\end{Rem}

In particular, Theorem \ref{sved} implies that, in order to solve Problem \ref{prob}, it is sufficient to prove $$\Reg_H(G,5) \ge 5$$ for every almost simple group $G$ and each of its maximal solvable subgroups $H$. 

Our main goal is to study Problem \ref{prob} for almost simple groups. In particular, we  focus on the almost simple classical groups.
 

\section{Review of existing literature}

The intersection of various subgroups in finite groups has been studied since the middle of the 20th century,  and associated results have proved useful in the study of group structure. For example, intersections of Sylow subgroups of a finite group are closely connected to representations of the group \cite{michler, strunSP}. Let us mention some important results on intersections of Sylow, nilpotent and abelian subgroups of finite groups.  While not directly applicable to  Problem \ref{prob}, they help to establish background and context. 
 
Let $\pi$ be a set of primes and let $p$ be a prime. A finite group $G$ is $\pi$-solvable if 
none of its non-abelian composition factors has order divisible by a prime from  $\pi$. If $\pi=\{p\}$ and $G$ satisfies this property, then $G$ is $p$-solvable.  Passman \cite{passman} proved that if  a finite group $G$ is $p$-solvable and $P$ is a Sylow $p$-subgroup of $G$, then there exist $x,y \in G$ such that $P \cap P^x \cap P^y$ is  the unique largest normal $p$-subgroup of $G$. Zenkov \cite{ZenP} generalised this statement to an arbitrary finite group.  Vdovin \cite{vdovinReg} and Dolfi \cite{dolfi} independently proved that if $G$ is $\pi$-solvable and $H$ is a solvable Hall $\pi$-subgroup (a $\pi$-subgroup of index  coprime to all primes in $\pi$), then there exist $x,y \in G$ such that $H \cap H^x \cap H^y \le {\bf F}(G)$. Recently, Zenkov \cite{zen21}  proved that if $N$ is a nilpotent subgroup of a finite group $G$, then there exists $x,y \in G$ such that $N \cap N^x \cap N^y \le {\bf F}(G)$.  We use the following related result of Zenkov  \cite{zen}.

\begin{Th}
\label{zenab}
If $A$ and $B$ are abelian subgroups of a finite group $G$, then there exists $x \in G$ such that $A \cap B^x \le {\bf F}(G).$
\end{Th}

\bigskip

Let us now discuss the  progress on Problem \ref{prob} for almost simple groups. In particular, we are interested in bounds for $b_S(G)$ and $\Reg_S(G,5)$ for an almost simple $G$ and its maximal solvable subgroup $S$. The following lemma is useful here.

\begin{Lem}[{\cite[Lemma 3]{bay}}] \label{base4}
Let $G$ be a finite group. If $H \le G$ and $b_H(G)\le 4$, then $\Reg_H(G,5)\ge 5.$
\end{Lem}

  If $G$ is almost simple, $S$ is a maximal solvable subgroup of $G$, and  $S \le H \le G$, then $b_S(G) \le b_H(G).$ Indeed, if $H^{a_1} \cap \ldots \cap H^{a_c}=1$ for $a_i \in G,$ then $$S^{a_1} \cap \ldots \cap S^{a_c}=1.$$ 


\begin{Def}
\label{nonstdef}
Let $G$ be a finite almost simple classical group over $\mathbb{F}_q,$ where $q=p^f$ and  $p$ is prime, with socle $G_0$ and natural module $V$. A maximal subgroup $H$ of $G$ not containing $G_0$ is a {\bf subspace subgroup} \index{subspace subgroup} if every
maximal subgroup $M$ of $G_0$ containing $H \cap G_0$ either acts reducibly on $V$ or $(G_0,M, p) = (Sp_{2m}(q)',O^{\pm}_{2m}(q), 2)$.   
A faithful transitive action of $G$ on a set $\Omega$ is a {\bf subspace action} \index{action!subspace} if the $G$-stabiliser of a point in
$\Omega$ is a subspace subgroup of $G$.  Non-subspace subgroups and actions are defined accordingly.
\end{Def}

\begin{Def}
Let $G$ be a finite almost simple group with socle $G_0$. A primitive  
action of $G$ on a set $\Omega$ (so the $G$-stabiliser of a point is a maximal subgroup of $G$) is {\bf standard} \index{action!standard} if one of the following holds:
\begin{enumerate}
\item $G_0 = A_n$ and $\Omega$ is an orbit of subsets or partitions of $\{1, \ldots,n\}$.
\item $G$ is a classical group in a subspace action.
\end{enumerate}
\end{Def}

  Liebeck and Shalev \cite{lieb} proved the following conjecture of Cameron and Kantor \cite{Camer}:  if $G$ is an almost simple finite group and $H\le G$ is maximal, then there exists an absolute constant $c$ such that $b_H(G)\le  c$ unless $(G,H)$ lies in a prescribed list of exceptions. The exceptions arise when the action of $G$ on the set of right cosets of $H$  is {\bf standard}. Below we discuss  results specifying bounds for $b_H(G)$ relevant to our study.

\bigskip

\subsection*{Symmetric groups}

\begin{Th}[{\cite{bay}}]
Let $G$ be a finite almost simple group with socle isomorphic to an alternating group $\mathrm{Alt}(n)$ for $n\ge 5.$ If $H$ is a maximal solvable subgroup of $G$, then $\Reg_H(G,5)\ge 5$.
\end{Th}

 The proof uses a constructive and inductive approach and exploits  the following result of Burness, Guralnick and Saxl \cite{prim11}.

\begin{Th}
Let $G$ be  $\Sym(n)$ or $\mathrm{Alt}(n)$ and let $H<G$ be maximal.  Assume that $H$ acts primitively on $\{1, \ldots ,n\}$ and does not contain $\mathrm{Alt}(n)$. Then $b_H(G) \le 3$ for all $n \ge 11.$
\end{Th}

\medskip

\subsection*{Classical groups}

  Burness \cite{fpr} obtains information on fixed point ratios of elements of prime order in classical groups in a non-standard action.
The fixed point ratio data underpins the {\bf probabilistic method} used in \cite{burness} to obtain the following result. We describe the probabilistic method in Chapter 2 since we use it in our proofs. 

\begin{Th}
\label{bernclass}
If $G$ is a finite almost simple classical group in a faithful primitive non-standard action with  point stabiliser $H$, then either $b_H(G) \le 4$, or $G = U_6 (2) \cdot 2$, $H = U_4 (3) \cdot 2^2$ and $b(G) = 5$.
\end{Th}

Roughly speaking, Theorem \ref{bernclass} is true for maximal subgroups $H \notin \mathcal{C}_1$ (with some exceptions). Here $\mathcal{C}_i$ for $i=1, \ldots, 8$ are Aschbacher's classes introduced in \cite{asch}   and described in \cite[\S 2.1]{maxlow}  and \cite[Chapter 4]{kleidlieb}.  If $H \in \mathcal{C}_1,$ then it stabilises a subspace (or a pair of subspaces) of the natural module of $G$. Tables 2 and 3 in \cite{burness} contain detailed information on $b_H(G)$ for $n\le 5$ and $H$ from distinct Aschbacher's classes.

\medskip

\subsection*{Exceptional groups of Lie type}

\begin{Th}[{\cite[Theorem 1]{bls}}]
\label{intexc}
Let $G$ be a finite almost simple group of exceptional Lie type, and let $\Omega$ be
a primitive faithful $G$-set. Then $b(G) \le 6$.
\end{Th}

 The proof is based on the probabilistic method.

\medskip

\subsection*{Sporadic groups}

We summarise the results of \cite{spor} and \cite{mon}.

\begin{Th}
\label{intspor}
 Let $G$ be a finite almost simple sporadic group and let $\Omega$ be a
faithful primitive $G$-set with  point stabiliser $H$. One of the following
holds:
\begin{enumerate}[font=\normalfont]
\item $b(G)=2$;
\item $(G, H, b(G))$ is listed in  {\rm \cite[Table 1 and 2]{spor}}; in most cases $b(G)\le 4$, $b(G)=5$ in $12$ cases, $b(G)=6$ in four cases, $b(G)=7$ in one case;
\item $G$ is the Baby Monster, $H = 2^{2+10+20}.(M_{22} : 2 \times S_3)$, $b(G)=3$. 
\end{enumerate}
\end{Th} 

 The proof uses  probabilistic, character-theoretic and computational methods. 
 
 Recently Burness \cite{burspor} proved the following.

\begin{Th}
 If $G$ is a finite almost simple group with sporadic socle and $H$ is a solvable subgroup, then $b_H(G) \le 3.$ 
\end{Th} 

The proof uses computational methods, unless the socle is isomorphic to the Monster or Baby Monster groups where the probabilistic method is used. 

\medskip

\subsection*{Primitive non-standard actions of $G$ with $b_H(G)>5$}

By Theorems \ref{intexc} and \ref{intspor}, if $G_0$ is exceptional or sporadic, then $b_H(G)\le 7$ for all maximal subgroups $H <G$ with equality only in one case. The following theorem lists all cases with $b_H(G)=6.$  

\begin{Th}[{\cite[Theorem 5.15]{sfa}}]
\label{b6except}
If $G$ is a finite almost simple  group in a faithful primitive non-standard action with point stabiliser $H$, then $b(G) = 6$ if and only if one of
the following holds:
\begin{enumerate}[font=\normalfont]
\item $(G, H) = (M_{23}, M_{22}), (Co_3, McL.2), (Co_2,PSU_6(2).2)$,\\
or $(Fi_{22}.2, 2.PSU_6(2).2);$
\item $G_0 = E_7(q)$ and $H = P_7$;
\item $G_0 = E_6(q)$ and $H = P_1$ or $P_6$.
\end{enumerate}
\end{Th} 
  Each $P_i$ is a maximal {\bf parabolic} subgroup; for details see the discussion before \cite[Theorem 3]{bls}. Therefore,  if $G_0$ is exceptional or sporadic, then either $b_H(G) \le 5$ or $(G,H)$ is listed in Theorems \ref{intspor}  and \ref{b6except}.

\subsection*{Maximal subgroups that are solvable}

Sometimes a maximal subgroup of an almost simple group is solvable.  An explicit list is given by Li and Zhang \cite{maxsolprim}. Recently, Burness \cite{burPS} proved the following.
\begin{Th}
\label{bur2020}
Let $G$ be a finite almost simple  group with socle $G_0$. If a maximal subgroup $H <G $ is solvable, then
  $b_H(G) \le 5$, with equality if and only if one of the following holds:
\begin{itemize}
\item[(a)] $G = \Sym (8)$ and $H = \Sym(4) \wr \Sym (2)$;
\item[(b)] $G_0 = PSL_4(3)$ and $H=P_2$;
\item[(c)] $G_0 = PSU_5(2)$ and $H = P_1$.
\end{itemize}
\end{Th}
The proof exploits both the probabilistic method and computation. Although Theorem \ref{bur2020} does not establish $\Reg_H(G,5)\ge 5$  when $b_H(G)=5$, it can be done routinely by computation.

\section{Main results}

As is clear from the above results, if $G$ is sporadic or exceptional of Lie type, then $b_H(G) \le 5$ for every maximal subgroup $H$ of $G$ apart from a short list of exceptions where  $b_H(G)$ is 6 or 7. If $G$ is classical of Lie type and $H \in \mathcal{C}_1$, then $b_H(G)$ can be arbitrarily large since the order
of $G$ is not always bounded by a fixed polynomial function of the degree of the action.  In particular, as the following lemma shows, $b_H(G)$ is not bounded by a constant.
\begin{Lem}
If $G$ acts faithfully on $\Omega$ and $d=|\Omega|,$ then $b(G) \ge \log_d|G|$.
\end{Lem}
\begin{proof}
Let $B \in \Omega^{b(G)}$ be a base. Every element of $G$ is uniquely determined by its action on $B$. Indeed, if $Bx=By$ for $x,y \in G$, then $Bxy^{-1}=B$ and $x=y$ since $B$ is a regular point. Hence $|G|\le d^{b(G)}.$ 
\end{proof}
    Therefore,  if a maximal solvable subgroup lies only in a $\mathcal{C}_1$-subgroup of $G,$ then one cannot solve Problem \ref{prob} simply by studying the corresponding problem for maximal subgroups.

We study the situation when $G_0$ is a simple classical group of Lie type isomorphic to $PSL_n(q),$ $PSU_n(q)$ or $PSp_n(q)'$ for some $(n,q)$ and $G$ is an almost simple classical group with  socle isomorphic to $G_0.$ In particular, we identify $G_0$ with its group of  inner automorphisms, so 
$$G_0 \le G \le \Aut (G_0).$$ 
Here $PSp_n(q)'$ is the commutator subgroup of $PSp_n(q)'.$ If $n \ge 4$ and $q \ge 3$, then $PSp_n(q)$ is simple, but $PSp_4(2)=Sp_4(2) \cong \Sym(6)$, so  $PSp_4(2)' \cong \mathrm{Alt}(6)$ is simple. We write $PSp_n(q)'$ to include this group.

Our main result is the following.
\begin{mmainth}
\label{fulltheorem}
Let $G$ be a finite almost simple group with socle isomorphic to $PSL_n(q),$ $PSU_n(q)$ or $PSp_n(q)'.$ If $S \le G$ is solvable, then $\Reg_S(G,5) \ge 5$. In particular, $b_S(G) \le 5.$   
\end{mmainth}


\begin{Rem}
Classical groups of Lie type are naturally divided into four classes: linear, unitary, symplectic and orthogonal groups. Although we believe that our approach could be successfully applied to orthogonal groups, we expect that their consideration will require much more technical work than needed for the other classes because 
of the greater complexity of their structure.     
\end{Rem}

If $X$ is $\GL_n(q),$ $\GU_n(q)$ or $\GS_n(q)$ (see Section \ref{secnot} for definitions) and $N$ is the subgroup of all scalar matrices in $X$, then
$X/N$ is isomorphic to a subgroup of $\Aut(G_0)$ of index at most 2 where $G_0$ is equal to $PSL_n(q),$ $PSU_n(q)$ and $PSp_n(q)'$ respectively. Precisely, the corresponding index is 2  if $G_0=PSL_n(q)$ with $n \ge 3$ or $G_0=PSp_4(q)'$  with $q$ even.  If $G_0=PSL_n(q)$ and $n\ge 3,$ then $\Aut(G_0)$ is isomorphic to $A(n,q)/N$ where $A(n,q)=\GL_n(q) \rtimes \langle \iota \rangle$ and $\iota$ is the inverse-transpose map on $GL_n(q).$

We obtain the Main Theorem as a corollary of the following   theorems. Each of the theorems provide additional details depending on $G_0$.
\begin{mainthA}
\label{theorem}
Let $X=\GL_n(q)$, $n \ge 2$ and $(n,q)$ is neither $(2,2)$ nor $(2,3).$ If $S$ is a maximal solvable subgroup of $X$, 
 then $\Reg_S(S \cdot SL_n(q),5)\ge 5$, in particular $b_S(S \cdot SL_n(q)) \le 5.$
\end{mainthA}

\begin{mainthA}
\label{theoremGR}
Let $n\ge 3.$ If $S$ is a maximal solvable subgroup of $A(n,q)$ not contained in $\Gamma L_n(q),$ then one of the following holds:
\begin{enumerate}
\item[$(1)$] $b_S(S \cdot SL_n(q))\le 4$;
\item[$(2)$] $(n,q)=(4,3)$, $S$ is the normaliser in $A(n,q)$ of the stabiliser in $\GL_n(q)$ of a $2$-dimensional subspace of $V$, $b_S(S \cdot SL_n(q))=5$ and $\Reg_S(S \cdot SL_n(q),5)\ge 5.$
\end{enumerate}
\end{mainthA}

\begin{mainthB}\label{theoremGU}
Let $X=\GU_n(q)$, $n \ge 3$ and $(n,q)$ is not  $(3,2).$ If $S$ is a maximal solvable subgroup of $X$, 
 then one of the following holds:
\begin{enumerate}
\item[$(1)$]  $b_S(S \cdot SU_n(q)) \le 4,$ so $\Reg_S(S \cdot SU_n(q),5)\ge 5$;
\item[$(2)$]   $(n,q)=(5,2)$ and $S$ is the stabiliser in $X$ of a totally isotropic subspace of dimension $1$, $b_S(S \cdot SU_n(q)) =5$ and $\Reg_S(S \cdot SU_n(q),5)\ge 5$. 
\end{enumerate}
\end{mainthB}

\begin{mainthC}\label{theoremSp}
Let $X=\GS_n(q)$ and  $n \ge 4$. If $S$ is a maximal solvable subgroup of $X$, 
 then  $b_S(S \cdot Sp_n(q)) \le 4,$ so $\Reg_S(S \cdot Sp_n(q),5)\ge 5$. 
\end{mainthC}

\begin{mainthC}\label{theoremSpGR}
Let $q$ be even and let ${A}=\Aut(PSp_4(q)')$. If ${S} \le {A}$ is a maximal solvable subgroup, then $b_{{S}}({S} \cdot Sp_4(q)') \le 4,$ so $\Reg_S(S \cdot Sp_n(q)',5)\ge 5$.
\end{mainthC}

\begin{proof}[Proof of Main Theorem]
Let $G_0 \le G \le \Aut (G_0)$ and let $S\le G$ be solvable. Let $H$ be a maximal solvable subgroup of $  \Aut(G_0)$ containing $S$. By Theorems \ref{theorem}, \ref{theoremGR}, \ref{theoremGU},  \ref{theoremSp} and \ref{theoremSpGR},
$$\Reg_H(H \cdot G_0,5)\ge 5,$$
so there exist $x_{(i,1)},x_{(i,2)},x_{(i,3)},x_{(i,4)},x_{(i,5)} \in G_0$ for $i \in\{1, \ldots, 5\}$ such that
$$\omega_i=(Hx_{(i,1)},Hx_{(i,2)},Hx_{(i,3)},Hx_{(i,4)},Hx_{(i,5)})$$ 
are $H \cdot G_0$-regular points   in $\Omega^5=\{Hx \mid, x \in H \cdot G_0 \}^5$, and the $\omega_i$ lie in distinct orbits. 
We claim that $$w_i'=(Sx_{(i,1)},Sx_{(i,2)},Sx_{(i,3)},Sx_{(i,4)},Sx_{(i,4)}) \text{ for } i \in \{1, \ldots, 5\}$$ lie in distinct $G$-regular orbits in $(\Omega')^5=\{Sx \mid x \in G\}^5.$ Indeed, $$S^{x_{(i,1)}} \cap S^{x_{(i,2)}} \cap S^{x_{(i,3)}} \cap S^{x_{(i,4)}} \cap S^{x_{(i,5)}} \le H^{x_{(i,1)}} \cap H^{x_{(i,2)}} \cap H^{x_{(i,3)}} \cap H^{x_{(i,4)}} \cap H^{x_{(i,5)}}=1,$$ so $\omega_i'$ are regular. Assume that $\omega_1' g = \omega_2'$ for some $g \in G.$ Therefore,
$$(Sx_{(1,i)})g=Sx_{(2,i)} \text{ for } i \in \{1, \ldots, 5\}$$
and 
$$g \in \cap_{i=1}^5 (x_{(1,i)}^{-1}S x_{(2,i)}) \subseteq \cap_{i=1}^5 (x_{(1,i)}^{-1}H x_{(2,i)})= \emptyset$$
where the last equality holds since $\omega_1$ and $\omega_2$ lie in distinct $H \cdot G_0$-orbits. Hence $\omega_1'$ and $\omega_2'$ lie in distinct $G$-orbits. The same argument shows that all of the $\omega_i'$ lie in distinct $G$-orbits, so $\Reg_S(G,5)\ge 5$.
\end{proof}

\section{Summary of contents }

 Chapter 2 is devoted to notation, definitions and preliminary results.  We present notation and definitions for classical groups and forms in Sections \ref{secnot} and \ref{clsec},  and  briefly introduce algebraic groups in Section \ref{algsec}. In Section \ref{missec} we collect technical results that play significant roles in the proof of Theorems \ref{theorem} -- \ref{theoremSpGR}. These include lemmas on the structure of maximal solvable subgroups of classical groups, and the subgroups stabilising certain structures, such as a subspace of the natural module or a decomposition of the natural module into direct sum of subspaces. Section \ref{sinsec} is devoted to  Singer cycles -- cyclic subgroups of $GL_n(q)$ of order $q^n-1$ -- and their normalisers. Such subgroups play an important role in the structure of irreducible solvable linear groups. In Section \ref{fprsec} we describe the probabilistic method we mentioned earlier and give the necessary information on fixed point ratios of elements of prime order (modulo scalars) of classical groups. Finally, in Section \ref{gapsec} we describe the computational methods and software we used.

 In Chapter \ref{ch2} we obtain upper bounds for $b_S(L)$ where $L$ is $GL_n(q)$, $GU_n(q)$ or $GSp_n(q)$ and $S$ is a solvable irreducible subgroup of $L.$ Our results  are refinements of Theorem \ref{bernclass} in the sense that they provide better estimates for $b_H(G)$  for solvable $H$ not lying in a $\mathcal{C}_1$-subgroup of $G \le L/Z(L)$ in the cases described above. 
 In particular, with an explicit list of exceptions, we obtain $b_S(L) =2$ for $L=GL_n(q)$ and $b_S(L) \le 3$ for $L=GU_n(q)$ or $GSp_n(q).$  
 These estimates form an important part of our proof of the main results and are necessary since the bound $b_S(L) \le 4$  from Theorem \ref{bernclass} is not sufficient for the proof. As a ``basic'' case we take the situation when $S$ is a primitive (for $L=GL_n(q)$) or quasi-primitive (for $L=GU_n(q)$ or $GSp_n(q)$) maximal solvable subgroup, so we first study such subgroups. We use the probabilistic method based on fixed point ratios for elements of prime orders to obtain the bounds for $b_S(L)$  for primitive and quasi-primitive $S.$ We do not explicitly  construct  $x,y \in L$ such that $S \cap S^x \cap S^y \le Z(L).$ Nevertheless, the reduction of the remaining cases to this case is constructive in most situations. We illustrate this point for linear groups, so $L=GL_n(q).$ If an irreducible subgroup of $L$ is not (quasi-)primitive, then it must stabilise a nontrivial decomposition of the natural module into a direct sum of subspaces having specified shapes.   In particular, if $S$ is an imprimitive maximal solvable group of $GL_n(q)$, then it is a wreath product of a linear primitive maximal solvable group of smaller degree $S_1 \le GL_m(q)$ and a group of permutations $\Gamma \le \Sym(k)$ where $n=mk$ (see Lemma \ref{supirr}).  If we know $x_1 \in SL_m(q)$ such that $S_1 \cap S_1^{x_1} \le Z(GL_m(q))$, then the proof of Theorem   \ref{irred} can be used to construct explicitly   $x \in SL_n(q)$ such that $S \cap S^x \le Z(GL_n(q)).$

  In Chapter \ref{ch3} we consider the general case  where $S$ is a maximal solvable subgroup of $X$ or $A$. Since for subgroups $S$ stabilising no subspace of the natural module Theorems \ref{theorem} -- \ref{theoremSpGR} follow (with some exceptions) by Theorem \ref{bernclass}, the main obstacle is the situation when $S$ lies in a maximal $\mathcal{C}_1$.   Our strategy is to combine effectively the results of Chapter \ref{ch2} and the structure of $S$. In particular, we use the fact that $S$ stabilises a non-zero proper subspace $U$ of the natural module, so $S^x$ must stabilise $(U)x.$ 
  Our proof  is mostly constructive, we again illustrate it in the case of Theorem \ref{theorem} for simplicity.  If $S\le \GL_n(q)$ is reducible, then, in some basis, matrices of $S\cap GL_n(q)$ are upper-block-diagonal (see Lemma \ref{supreduce}) with blocks forming irreducible solvable subgroups $S_i$ of smaller degree $n_i$ where $i=1, \ldots, k$ for some $k$ and $n=\sum_{i=1}^k n_i.$ If we know  $x_i \in SL_{n_i}(q)$ such that $S_i \cap S_i^{x_i} \le Z(GL_{n_i}(q))$, then the proof of Theorem \ref{theorem} can be used to construct $5$ distinct regular orbits of the action of $G/S_G$ on $\Omega^5.$

%% file: ch2.tex
\chapter{Definitions and preliminaries}

\section{Notation and basic definitions}
\label{secnot}

All group actions we use are right actions. For example, the action of a linear transformation $g$ of a vector space $V$ on $v \in V$ is $(v)g \in V.$


We write $GL(V)=GL(V, \mathbb{F})$ for the {\bf general linear group}, which is the group of all invertible linear transformations of  a vector space $V$ over a field $\mathbb{F}.$ 

 Let $p$ be a prime and $q=p^f$,  $f \in \mathbb{N}$. Denote a finite field of size $q$ by $\mathbb{F}_q$, its algebraic closure by $\overline{\mathbb{F}_q}$ and the multiplicative group of $\mathbb{F}_q$ by $\mathbb{F}_q^*.$ 
Throughout, unless stated otherwise, $V=\mathbb{F}_{q^{\bf u}}^n$ denotes a vector space of dimension $n$ over $\mathbb{F}_{q^{\bf u}}$ with ${\bf u} \in \{1,2\}.$

We reserve the letter $\beta$ for a {\bf basis} of $V$. A basis is an ordered set. Let $\beta =\{v_1, \ldots, v_n\}$ be such a basis and let ${\bf u}=1$. If $g \in GL(V)$, then $g_{\beta}$ denotes the $n \times n$ matrix such that 
$$(v_i)g=\sum_{j=1}^{n} (g_{\beta})_{i,j} \cdot v_j.$$ Here $A_{i,j}$ is the $(i,j)$ entry of a matrix $A$.   We denote the group $\{g_{\beta} \mid g \in GL(V)\}$ by $GL_n(q, \beta)$ or simply $GL_n(q)$ when $\beta$ is understood. If $X \subseteq GL(V)$, then $X_{\beta}$ is $\{g_{\beta} \mid g \in X \}.$  It is easy to see that  $GL(V)$ and $GL_n(q)$ are isomorphic, and the map $g \mapsto g_{\beta}$ is an isomorphism. Since matrices from $GL_n(q)$ act on $V=\mathbb{F}_q^n$ by right multiplication, we  refer to them as linear transformations. In what
follows, we make no essential distinction between the groups $GL(V)$ and $GL_n(q)$ and use $GL(V)$ or $GL_n(q)$ depending on which one is more suitable for our purpose.


We fix the following notation.
\medskip\\
\begin{longtable}{ p{7em} p{27em} }
${\bf F}(G)$ & {\bf Fitting subgroup} \index{Fitting subgroup} of a finite group $G$ (unique maximal \\ & normal  nilpotent subgroup);\\
$O_{\pi}(G)$ & unique maximal normal $\pi$-subgroup for a set of primes $\pi$;\\
$Z(G)$ & {\bf center} of a group $G$;\\
$g^G$ & conjugacy class of $g \in G$;\\
$A \rtimes B$ & semidirect product of groups $A$ and $B$ with $A$ normal;\\
 $\Sym(n)$ &  symmetric group of degree $n$; \\
$\sgn(\pi)$ &  {\bf sign} of a permutation $\pi$;\\
$M_n(\mathbb{F})$& algebra of all $n \times n$ matrices over $\mathbb{F}$;\\
 $\diag(\alpha_1, \ldots, \alpha_n)$ &  {\bf diagonal} matrix with entries $\alpha_1, \ldots, \alpha_n$ on its diagonal;  \\ 
 $\diag(\alpha, \ldots, \alpha)$ &  {\bf scalar} matrix, or simply a {\bf scalar};  \\  
 $\diag[g_1, \ldots, g_k]$ &  {\bf block-diagonal} matrix with blocks $g_1, \ldots, g_k$ on its diagonal;  \\
$\per(\sigma)$ & {\bf permutation matrix} corresponding to $\sigma \in \Sym(n);$\\
$g^{\top}$ & {\bf transpose} of a matrix $g$;\\
$\det(g)$ &  {\bf determinant} of a matrix $g$;\\
$\Det(H)$ & $\{\det(h) \mid h \in H \}$ for $H \le GL(V)$;\\
$g \otimes h$ &  {\bf Kronecker product} \index{Kronecker product}   
$\begin{pmatrix}
g \cdot h_{1,1}     &  \ldots & g \cdot h_{1,m}  \\
\ldots           &   \ldots   & \ldots   \\
g \cdot h_{m,1}     & \ldots     & g \cdot h_{m,m}      
\end{pmatrix} \in GL_{nm}(q)$ \\ & for $g \in GL_n(q)$ and $h \in GL_m(q)$; \\
$SL(V)$ & {\bf special linear group} $\{g \in GL(V) \mid \det(g)=1\}$;\\
$D(G)$ & subgroup of all diagonal matrices of a matrix group $G$;\\
$RT(G)$ & subgroup of all upper-triangular 
 matrices of a matrix group $G$;\\
$p'$ & set of all primes except $p$;\\
$(a,b)$ & greatest common divisor of integers $a$ and $b$;\\
$\delta_{ij}$ & {\bf Kronecker delta}, $\delta_{ij}=1$ if $i=j$ and $\delta_{ij}=0$ otherwise.
\end{longtable}

\medskip



A map $g: V \to V$ is an $\mathbb{F}$-{\bf semilinear} \index{semilinear} transformation of $V$ if there exists $\sigma(g) \in \Aut(\mathbb{F})$ such that for all $u,v \in V$ and $\lambda \in \mathbb{F}$,
$$(u+v)g=ug +vg \text{ and } (\lambda v)g=\lambda^{\sigma(g)}(vg).$$
We write $\GL(V)=\GL(V, \mathbb{F})$ for the {\bf general semilinear group}, which is the group of all invertible $\mathbb{F}$-semilinear transformations of $V$. It is easy to see that $\sigma(gh)=\sigma(g)\sigma(h)$ for $g,h \in \GL(V, \mathbb{F}).$ Let $\beta$ be a basis of $V$. As each element of $GL(V, \mathbb{F})$ is determined by its action on $\beta$,  each  $g \in \GL(V, \mathbb{F})$ is determined by its action on $\beta$ and $\sigma(g).$ If $\alpha \in \Aut(\mathbb{F})$, then $\phi_{\beta}(\alpha)$ denotes the unique  $g \in \GL(V, \mathbb{F})$ such that $\sigma(g)=\alpha$ and $(v_i)g=v_i$ for all $v_i \in \beta.$ So
\begin{equation}
\label{defphibet}
\left(\sum_{i=1}^n \lambda_i v_i \right) \phi_{\beta}(\alpha)= \sum_{i=1}^n\lambda_i^{\alpha} v_i.
\end{equation}
 If $\mathbb{F}= \mathbb{F}_{q^{\bf u}}$ and $\alpha \in \Aut(\mathbb{F})$ is such that $\lambda^{\alpha}=\lambda^p$ for all $\lambda \in \mathbb{F}$,  then we denote $\phi_{\beta}(\alpha)$ by $\phi_{\beta}$ \index{$\phi_{\beta}$} or simply $\phi$ when $\beta$ is understood. It is routine to check (see \cite[\S 2.2]{kleidlieb}) that  $$\GL(V, \mathbb{F}_q)= GL(V, \mathbb{F}_q) \rtimes \langle \phi \rangle \cong GL_n(q, \beta) \rtimes \langle \phi \rangle.$$ We denote $GL_n(q, \beta) \rtimes \langle \phi \rangle$ by $\GL_n(q, \beta)$ or simply $\GL_n(q)$ when $\beta$ is understood. In what follows,
we make no essential distinction between the groups $\GL(V )$ and $\GL_n(q)$ and use $\GL(V )$ or
$\GL_n (q)$ depending on which one is more suitable for our purpose.

For a basis $\beta$ of $V$ let $\iota_{\beta}: GL_n(q, \beta) \to GL_n(q, \beta)$ be the inverse-transpose map 
$$\iota_{\beta}: g \mapsto (g^{-1})^{\top}.$$ Therefore, $\langle \iota_{\beta} \rangle$ acts on $GL_n(q, \beta)$ and we define 
$$A(n,q):=\GL_n(q) \rtimes \langle \iota_{\beta} \rangle$$ by letting $\iota_{\beta}$  commute with $\phi_{\beta}.$

It is convenient to view  the symmetric group 
as a group of permutation matrices. We define the {\bf wreath product} \index{wreath product}  of  $X \le GL_n(q)$ and a  group of permutation matrices $Y \le GL_m(q)$ as the matrix group $X \wr Y \le GL_{nm}(q)$ obtained
by replacing the entries 1 and 0  in every matrix in $Y$ by
arbitrary matrices in $X$ and by zero $(n \times n)$ matrices respectively. 

Let $A$ be an $(nm\times nm)$ matrix. We can view $A$ as the matrix
$$\begin{pmatrix}
A_{11}      &  \ldots & A_{1m}  \\
\ldots           &   \ldots   & \ldots   \\
A_{m1}     & \ldots     & A_{mm}      
\end{pmatrix}$$
where the $A_{ij}$ are $(n \times n)$ matrices. The vector $(A_{i1}, \ldots, A_{im})$ is the $i$-th $(n \times n)$-{\bf row} \index{$(n \times n)$-row} of $A.$


\medskip

 Let $\mathbb{F}$ be a field and let $G$ be a group. An $\mathbb{F}$-{\bf representation} \index{representation} of $G$ is a homomorphism $$\mathfrak{X}: G \to GL(V,\mathbb{F})$$ with $V=\mathbb{F}^n$ for some $n \in  \mathbb{N}$. By linear extension, $\mathfrak{X}$ determines an $\mathbb{F}$-representation of the {\bf group algebra} $\mathbb{F}[G]$, \index{group algebra} which is an algebra homomorphism from $\mathbb{F}[G]$ to $M_n(\mathbb{F})$ denoted by the same letter $\mathfrak{X}.$ Therefore, the action via $\mathfrak{X}$ makes $V$  an $\mathbb{F}[G]$-module.

 If $W \le V$ is an  $\mathbb{F}[G]$-submodule, then $W$ is a $G$-{\bf invariant} subspace of $V$, sometimes we state this fact as $(W)G=W.$ If there exists  a non-zero $G$-invariant subspace of $V$, then   $V$ is  a {\bf reducible} $\mathbb{F}[G]$-module, $\mathfrak{X}$ is a {\bf reducible} representation, and $G$ is a {\bf reducible} group. Otherwise $V$, $\mathfrak{X}$ and $G$ are {\bf irreducible}. \index{irreducible!group} \index{irreducible!module} \index{irreducible!representation} A subgroup of $\GL_n(q)$ is irreducible if it   stabilises (as a group of semilinear transformations) no non-zero proper subspace of $V$. 

 A representation $\mathfrak{X}$ (respectively, a module $V$ and a group $G$) is {\bf completely reducible} if $V$ is a direct sum of $\mathbb{F}[G]$-irreducible submodules. If $V$ is a completely reducible $\mathbb{F}[G]$-module and $M$ is an irreducible $\mathbb{F}[G]$-module, then the sum of those $\mathbb{F}[G]$-submodules of $V$ which are isomorphic to $M$ is the  {$M$-{\bf homogeneous}} {\bf component} $M(V)$ of $G$ on $V$.   If $V=M(V)$ for some irreducible $M$, then $V$ is a {\bf homogeneous} $\mathbb{F}[G]$-module. \index{homogeneous!module} \index{homogeneous!component}

Let  $\mathbb{E}$ be a field extension of $\mathbb{F}$. Then  $(G)\mathfrak{X} \le GL_n(\mathbb{F}) \le GL_n(\mathbb{E}),$ so $\mathfrak{X}$ can be viewed as an $\mathbb{E}$-representation of $G$ which we denote by $\mathfrak{X}^\mathbb{E}.$ A representation $\mathfrak{X}$ (and a group $G$) is {\bf absolutely irreducible} if $\mathfrak{X}^\mathbb{E}$ is irreducible for every field $\mathbb{E} \supseteq \mathbb{F}.$

Let $V$ be an irreducible $\mathbb{F}[G]$-module. If $V$ has a direct sum decomposition 
\begin{equation}\label{imprdecomp}
V=V_1 \oplus \ldots \oplus V_k \text{ for } k>1
\end{equation}
such that for each $i=1, \ldots, k$ and $g \in G$ there exists $j \in \{1, \ldots, k\}$ (unique, since $g$ is invertible) with 
$$(V_i)g=V_j,$$
then $G$ is {\bf imprimitive} and $\{V_1, \ldots, V_k\}$ is a system of imprimitivity of $G$. If $G$ has no system of imprimitivity, then it is {\bf primitive}. \index{primitive}

If $G \le GL(V,\mathbb{F})$, then we assume $(g)\mathfrak{X}=g$ for $g \in G$, unless stated otherwise. Abusing our notation, we denote the subalgebra $A$ of $M_n(\mathbb{F})$ generated by $X \subseteq M_n(\mathbb{F})$ by $\mathbb{F}[X].$ It should not be confusing since for $G \le GL_n(\mathbb{F})$ (complete) reducibility, (absolute) irreducibility, primitivity and other properties of representations we use do not depend on the choice of definition of $\mathbb{F}[G]$  (since $A=(\mathbb{F}[G])\mathfrak{X}$ for the group algebra  $\mathbb{F}[G]$).

\section{Classical forms and groups}
\label{clsec}

Let ${\bf f}$ be a map from $V \times V$ to $\mathbb{F}_{q^{\bf u}}.$ The map ${\bf f}$ is {\bf non-degenerate} \index{form!non-degenerate} if, for every $v \in V \backslash \{0\},$ the  maps $V \to \mathbb{F}$ given by $x \mapsto {\bf f}(x,v)$ and $x \mapsto {\bf f}(v,x)$ are non-zero. If ${\bf f}$ is fixed, then we write $(v,w)$ instead of ${\bf f}(v,w)$ for convenience. The vectors $v$, $w$ are mutually {\bf orthogonal} if $(v,w)=(w,v)=0$. A set of vectors $\{v_1, \ldots, v_n\}$ is {\bf orthonormal} if $(v_i,v_j)=0$  and $(v_i,v_i)=1$ for $i,j \in \{1, \ldots, n\}$ such that  $i\ne j$. 

Let ${\bf u}=2$, so $V=(\mathbb{F}_{q^2})^n$. A {\bf unitary form} \index{form!unitary} is a map ${\bf f}$ from $V \times V$ to $\mathbb{F}_{q^{2}}$ such that for all $u,v,w \in V$ and $\lambda \in \mathbb{F}_{q^{2}}$ the following hold:
\begin{itemize}
\item 
$(u+v,w)=(u,w)+(v,w) \text{ and } (\lambda u, v)=\lambda(u,v);$
\item $(u,v)=(v,u)^q$. 
\end{itemize}

Let ${\bf u}=1$, so $V=(\mathbb{F}_{q})^n$.  A {\bf symplectic form} \index{form!symplectic} is a map ${\bf f}$ from $V \times V$ to $\mathbb{F}_{q}$ such that for all $u,v,w \in V$ and $\lambda \in \mathbb{F}_{q}$ the following hold:
\begin{itemize}
\item 
$(u+v,w)=(u,w)+(v,w) \text{ and } (\lambda u, v)=\lambda(u,v);$
\item $(u,v)=-(v,u)$;
\item $(u,u)=0.$ 
\end{itemize}

Let ${\bf f}$ be a non-degenerate unitary (symplectic) form. The pair $(V,{\bf f})$ is a {\bf unitary (symplectic) space}. 
 Two unitary (symplectic) spaces $(V_1, {\bf f}_1)$ and  $(V_2, {\bf f}_2)$ are {\bf isometric} if there exists an isomorphism of vector spaces $\varphi: V_1 \to V_2$ such that $${\bf f}_1(v,u)={\bf f}_2((v)\varphi,(u)\varphi)$$
for every $v$ and $u$ from $V_1$. Such $\varphi$ is an {\bf isometry.} \index{isometry} A {\bf similarity} \index{similarity} of unitary (symplectic) spaces $(V_1, {\bf f}_1)$ and  $(V_2, {\bf f}_2)$ is an isomorphism of vector spaces $\varphi: V_1 \to V_2$ such that there exists $\lambda \in \mathbb{F}_{q^{\bf u}}$ with 
\begin{equation}
\label{simillambda}
{\bf f}_1(v,u)=\lambda{\bf f}_2((v)\varphi,(u)\varphi)
\end{equation}
for every $v$ and $u$ from $V_1$. 

Let us fix ${\bf f}$ to be either  identically zero, or a non-degenerate unitary or symplectic form on $V$ for the rest of the section. Let $W$ be a subspace of $V$. If the restriction ${\bf f}_{W}$ of ${\bf f}$ to $W$ is non-degenerate, then $W$ is a {\bf non-degenerate subspace} \index{subspace!non-degenerate} of $V$. If ${\bf f}_{W}=0,$ then $W$ is a {\bf totally isotropic subspace} \index{subspace!totally isotropic} of $V.$ 

Two subspaces $U$ and $W$ of $V$ are {\bf orthogonal} if $(u,w)=0$ for all $u \in U$ and all $w \in W.$ We write $U \bot W$ for the direct sum of orthogonal subspaces. The {\bf orthogonal complement } $W^{\bot}$ of $W$ in $V$ is 
$$\{v \in V \mid (v,u)=0 \text{ for all } u \in W \}.$$   More details about spaces with forms can be found in \cite[\S 2.1]{kleidlieb}.
 
 Let $I(V, {\bf f})$ and $\Delta(V, {\bf f})$ be the group of all ${\bf f}$-isometries and all ${\bf f}$-similarities from $V$ to itself respectively. By definition, $I(V, {\bf f})$ and $\Delta(V, {\bf f})$ are subgroups of $GL(V)$, so $\Sigma(V, {\bf f}):=SL(V) \cap I(V, {\bf f})$ is well-defined. It is easy to see that if ${\bf f}$ is identically zero, then $\Sigma(V, {\bf f})=SL(V)$ and $I(V, {\bf f})=\Delta(V, {\bf f})=GL(V).$ 


All non-degenerate unitary (respectively symplectic) spaces of the same dimension over $\mathbb{F}_{q^{\bf u}}$ are isometric by the following lemmas.

\begin{Lem}[{\cite[Propositions 2.3.1 and 2.3.2]{kleidlieb}}]\label{unibasisl}
Let ${\bf f}$ be unitary.
\begin{enumerate}[font=\normalfont]
\item The space $(V,{\bf f})$ has an orthonormal basis.
\item The space $(V,{\bf f})$ has a basis 
\begin{equation}\label{unibasis}
\begin{cases}
\{f_1, \ldots, f_m, e_1, \ldots, e_m\}, & \text{ if $n=2m$} \\
\{f_1, \ldots, f_m, x, e_1, \ldots, e_m\}, & \text{ if $n=2m+1$ }
\end{cases}
\end{equation}
where $(e_i,e_j)=(f_i,f_j)=0,$ $(e_i,f_j)=\delta_{ij}$ and $(e_i,x)=(f_i,x)=0$ for all $i,j,$ 
and $(x,x)=1.$
\end{enumerate}
\end{Lem}

\begin{Lem}[{\cite[Proposition 2.4.1]{kleidlieb}}]\label{sympbasisl}
Let ${\bf f}$ be symplectic.
The dimension $n$ of $V$ is even and the space $(V,{\bf f})$ has a basis 
\begin{equation}\label{sympbasis}
\{f_1, \ldots, f_m, e_1, \ldots, e_m\}, 
\end{equation}
where $2m=n$,  $(e_i,e_j)=(f_i,f_j)=0$ and $(e_i,f_j)=\delta_{ij}$ for all $i,j.$
\end{Lem}

Hence, for a non-degenerate unitary or symplectic space $(V, {\bf f}),$ the groups $\Sigma(V, {\bf f}),$ $I(V, {\bf f})$ and $\Delta(V, {\bf f})$ are also defined uniquely (up to conjugation in $GL(V)$) by $\dim V$ and $q.$

An ${\bf f}$-{\bf semisimilarity} \index{semisimilarity} is $g \in \GL_n(q^{\bf u})$ such that there exist $\lambda \in \mathbb{F}_{q^{\bf u}}^*$ and $\alpha \in \Aut (\mathbb{F}_{q^{\bf u}})$ satisfying
\begin{equation}
\label{GAMlambda}{\bf f}(vg,ug)=\lambda{\bf f}(v,u)^{\alpha} \text{ for all } v,u \in V.
\end{equation}
By \cite[Lemma 2.1.2]{kleidlieb}, $\alpha$ is determined uniquely by $g$, and $\alpha=\sigma(g).$ We denote the group of ${\bf f}$-semisimilarities of $V$ by $\Gamma(V, {\bf f}).$ It is easy to see that 
$$\Delta(V, {\bf f}) \le \Gamma(V, {\bf f}).$$

\begin{Def}
\label{taudef}
By \cite[Lemma 2.1.2]{kleidlieb}, if ${\bf f}$ is non-degenerate, then the $\lambda$ in \eqref{simillambda} and \eqref{GAMlambda} are uniquely determined by $g$. Moreover, there exists a homomorphism $\tau:\Delta(V, {\bf f}) \to \mathbb{F}_{q^{\bf u}}^*$ satisfying
$\tau(g)=\lambda.$
\end{Def}

We say that we work on the case {\bf L}, {\bf U} or {\bf S} \index{case {\bf L}, {\bf U}, {\bf S}} when ${\bf f}$ is identically zero, unitary or symplectic respectively. We summarise  notation for the groups $\Sigma,$ $I$, $\Delta$ and $\Gamma$ in Table \ref{classnot}. For more details on classical groups and the  equalities claimed in the table  see \cite[\S 2.1]{kleidlieb}. 

\begin{table}[h] 
\centering
\caption{Notation for classical groups}
\begin{tabular}{|c|c|c|c|} 
\hline
case                                      &            & notation & terminology                        \\ \hline
\multirow{3}{*}{{\bf L}} & $\Sigma$   & $SL(V)$  & \multirow{3}{*}{linear groups}     \\ \cline{2-3}
                                          & $I=\Delta$ & $GL(V)$  &                                    \\ \cline{2-3}
           & $\Gamma$ & $\GL(V)$  &                                    \\ \hline
\multirow{3}{*}{\bf U} & $\Sigma$   & $SU(V)$  & \multirow{3}{*}{unitary groups}    \\ \cline{2-3}
                                          & $I$        & $GU(V)$  &                                    \\ \cline{2-3}
& $\Gamma$        & $\GU(V)$  &                                    \\ \hline
\multirow{3}{*}{\bf S} & $\Sigma=I$ & $Sp(V)$    & \multirow{3}{*}{symplectic groups} \\ \cline{2-3}
                                          & $\Delta$   & $GSp(V)$   &                                    \\ \cline{2-3}
& $\Gamma$        & $\GS(V)$  &                                    \\ \hline
\end{tabular}
\label{classnot}
\end{table} 

Denote the identity $(n\times n)$ matrix by $I_n$ \index{$I_n$} and let $J_{2k}$ \index{$J_{2k}$} be the matrix
\begin{equation*}
\begin{pmatrix}
    & I_k\\
 -I_k &     
\end{pmatrix}.
\end{equation*} 
For $g \in GL_n(q^{\bf u})$ let $\overline{g}$ be  the matrix obtained from $g$ by taking every entry to the $q$-th power (so if ${\bf u}=1$, then $\overline{g}=g$). 
We write $g^{\dagger}$ for $(\overline{g}^{\top})^{-1}$ and $X^{\dagger}$ for $\{g^{\dagger} \mid g \in X\}$, where $X \subseteq GL_n(q^{\bf u}).$

Fix a basis $\beta=\{v_1, \ldots , v_n\}$ of $V$ and denote by ${\bf f}_{\beta}$ the matrix whose $(i,j)$ entry is ${\bf f}(v_i,v_j).$ By fixing the basis, we identify $I(V, {\bf f})$ and $\Delta(V, {\bf f})$ with the matrix groups
\begin{equation}\label{unimatr}
\{g \in GL_n(q^{\bf u}) \mid g {\bf f}_{\beta}\overline{g}^{\top}={\bf f}_{\beta}\} \text{ and } \{g \in GL_n(q^{\bf u}) \mid g {\bf f}_{\beta}\overline{g}^{\top}=\lambda{\bf f}_{\beta}, \lambda \in \mathbb{F}_{q^{\bf u}}^*\}
\end{equation}
 respectively; we identify $\Gamma(V, {\bf f})$ with the subgroup  $\Gamma(V, {\bf f})_{\beta} \le \GL(V, \beta)$ of ${\bf f}$-semisimilarities.

 Denote the group of matrices representing the isometries from $I(V, {\bf f})$ with respect to a basis $\beta$ such that ${\bf f}_{\beta}=\Phi$ by $GU_n(q,\Phi)$ (respectively $Sp_n(q,\Phi)$)  or $GU_n(q, \beta)$ (respectively $Sp_n(q, \beta)$).  We write  $GU_n(q)$ (respectively $Sp_n(q)$) instead of $GU_n(q, I_n)$ (respectively $Sp_n(q, J_n)$) for simplicity; we use similar notation for $\Sigma(V, {\bf f})$, $\Delta(V, {\bf f})$ and $\Gamma(V, {\bf f})$ in cases {\bf U} and {\bf S}. We also use $GL_n^{\varepsilon}(q)$ with $\varepsilon \in \{+, -\}$ where $GL_n^+(q)=GL_n(q)$ and $GL_n^-(q)=GU_n(q).$ 

\medskip

Note the following observations and notation:
\begin{itemize}
\item In some literature $\Delta(V, {\bf f})$ for case {\bf S} is denoted by $CSp(V)$ and called the ``conformal symplectic group''.
\item If $\beta$ is as in Lemmas \ref{unibasisl} and \ref{sympbasisl} for cases {\bf U} and {\bf S} respectively, then $\phi_{\beta} \in \Gamma(V, {\bf f})$ and $\Gamma(V, {\bf f})_{\beta}= \Delta(V, \bf{f})_{\beta} \rtimes \langle \phi_{\beta} \rangle$. 
\item The group   $\Delta(V, {\bf f})$ for case {\bf U} is omitted in Table \ref{classnot} since here $$\Delta(V, {\bf f})=I(V, {\bf f}) \cdot \mathbb{F}_q^*.$$ Therefore, 
  $\Delta(V, \bf{f})_{\beta} \rtimes \langle \phi_{\beta} \rangle$ and $I(V, \bf{f})_{\beta} \rtimes \langle \phi_{\beta} \rangle$ (and their maximal solvable subgroups) coincide modulo scalars. It is more convenient for us to work with $I(V, \bf{f})_{\beta} \rtimes \langle \phi_{\beta} \rangle$, so in what follows we abuse notation by letting   $$\Gamma(V, {\bf f})= I(V, {\bf f}) \rtimes \langle \phi_{\beta}\rangle$$
for an orthonormal basis $\beta$ in case {\bf U}.
\item If $\Sigma(V, {\bf f}) \le G \le \Gamma(V, {\bf f}),$ then $G$ is solvable if and only if $\Sigma(V, {\bf f})$ is solvable since $\Delta(V, {\bf f})/\Sigma(V, {\bf f})$ and $\Gamma(V, {\bf f})/\Delta(V, {\bf f})$ are abelian. Therefore, such $G$ is solvable if and only if either $n=1$ or  $\Sigma(V, {\bf f})$ is one of the following groups: $SL_2(q)=Sp_2(q)\cong SU_2(q) $ for $q \in \{2,3\}$,   $SU_3(2).$ 
We often write ``$G$ is not solvable'' where we ignore these groups.
\end{itemize}

We state  a particular case of Witt's Lemma, which we use later. For a proof see \cite[\S 20]{asch}.

\begin{Lem}
\label{witt}
 Assume that $(V_1,{\bf f}_1)$, $(V_2,{\bf f}_2)$ are isometric unitary (symplectic) spaces and  $W_i$ is a subspace of $V_i$ for $i=1,2.$ If there is an isometry $g$ from $(W_1,{\bf f}_1)$ to $(W_2,{\bf f}_2),$ then $g$ extends to an isometry from $(V_1,{\bf f}_1)$ to $(V_2,{\bf f}_2).$  
\end{Lem}

\section{Algebraic groups}
\label{algsec}
 In this section we state necessary notation and results on algebraic groups.  
Informally speaking, an algebraic group is a group that is an algebraic variety such that  the multiplication and inversion operations are morphisms (polynomial maps) of the variety. To avoid a long series of definitions on varieties we use the fact that an affine algebraic group over an algebraically closed field of positive characteristic is isomorphic (as an algebraic group, which means that there exists a group isomorphism $\varphi$ such that $\varphi$ and $\varphi^{-1}$ are also  morphisms of the corresponding varieties) to a linear group \cite[p. 63]{hump}. Hence we state definitions and results in terms of linear groups.  Our standard references are \cite[Chapter 1]{carter}, \cite[Chapter 1]{gorlyo} and \cite{hump}.

\begin{Def}
Let $\overline{\mathbb{F}}$ be the algebraic closure of the field of size $p.$
\begin{itemize}
\item     The {\bf Zariski topology} on $GL_n(\overline{\mathbb{F}})$ is the topology defined by the condition that the {\bf closed sets} are the solution sets of systems of polynomial equations in matrix entries and the function $g \mapsto (\det g)^{-1}$ for $g \in GL_n(\overline{\mathbb{F}}).$ An $\overline{\mathbb{F}}$-{\bf linear algebraic group} (which we abbreviate to $\overline{\mathbb{F}}${\bf -algebraic group} or just { \bf algebraic group}) is a closed subgroup $\overline{K}$ of  $GL_n(\overline{\mathbb{F}})$ for some $n.$ The Zariski topology on $\overline{K}$ is the topology inherited from that of $GL_n(\overline{\mathbb{F}})$.
\item The {\bf connected component containing the identity element} \index{connected component} (in Zariski topology) of $\overline{K}$ is denoted by $\overline{K}^0.$
\item A {\bf torus} \index{torus} is an algebraic group \index{algebraic group} isomorphic to the direct product of finitely many copies of $GL_1(\overline{\mathbb{F}}).$ A {\bf subtorus} of an algebraic group $\overline{K}$ is a closed subgroup of $\overline{K}$ which
is a torus. A {\bf maximal torus} \index{torus!maximal} of $\overline{K}$ is a subtorus of $\overline{K}$ not contained in any other
subtorus of $\overline{K}$.
\item A {\bf Frobenius endomorphism} \index{Frobenius endomorphism} of  $\overline{K}$ is a surjective endomorphism $\sigma$ of $\overline{K}$
whose {\bf fixed point subgroup} $\overline{K}_{\sigma}$ is finite.
\item If $\overline{K}$ is nontrivial and connected but has no proper closed connected normal subgroup, then $\overline{K}$ is a {\bf simple} algebraic group. \index{algebraic group!simple}
\end{itemize}
\end{Def}

We are interested in simple algebraic groups since most  finite classical groups appear as fixed point subgroups (or their normal subgroups) for suitable simple algebraic groups and Frobenius endomorphisms. The classification of simple algebraic groups is based on the classification of their root systems. A simple algebraic group has an irreducible reduced root system. We do not define root systems here, but use their labels, see \cite[Chapter 1]{carter} for details.


A simple algebraic group is not  uniquely determined by its root system.
\begin{Th}[{\cite[Theorem 1.10.4]{gorlyo}}]
Let $\overline{\mathbb{F}}$ be the algebraic closure of the field of size $p.$ Let $\Sigma$ be an irreducible reduced root system. There exist simple algebraic groups $\overline{K}_u=\overline{K}_u(\Sigma)$ and $\overline{K}_a=\overline{K}_a(\Sigma)$ over $\overline{\mathbb{F}},$ unique up to isomorphism of algebraic groups, with the following properties:
\begin{itemize}
\item $\Sigma$ is the root system of both $\overline{K}_u$ and $\overline{K}_a;$
\item for every simple algebraic group $\overline{K}$ over $\overline{\mathbb{F}}$ with root system isomorphic to $\Sigma$ there exist  surjective homomorphisms of algebraic groups   $\overline{K}_u \to \overline{K}\to \overline{K}_a;$
\item $Z(\overline{K}_u)$ is finite and $Z(\overline{K}_a)=1.$
\end{itemize} 
\end{Th}
We call $\overline{K}_u$ and $\overline{K}_a$ the {\bf universal} \index{algebraic group!universal} and {\bf adjoint} \index{algebraic group!adjoint} simple algebraic group of type $\Sigma$ respectively.

Let ${\bf q}:GL_n(\overline{\mathbb{F}}) \to GL_n(\overline{\mathbb{F}})$ for $q=p^f$ be the map taking each entry of a  matrix to its $q$-th power and let ${\bf g}:GL_n(\overline{\mathbb{F}}) \to GL_n(\overline{\mathbb{F}})$ be the inverse-transpose map. The maps {\bf q} and {\bf qg} are Frobenius morphisms. We collect information about certain classical simple algebraic groups and their fixed point subgroups in Tables \ref{algAC} and \ref{algSIG}. For this information on all classical simple algebraic groups see \cite[Theorem 1.10.7]{gorlyo} and \cite[\S 1.19]{carter}. The labels  $A_l$ and $C_l$, for a positive integer $l$, are {\bf types} of irreducible reduced root systems. 

\begin{table}[] 
\centering
\caption{Simple algebraic groups with root systems $A_l$ and $C_l$}
\begin{tabular}{|l|l|l|} 
\hline
$\Sigma$ & $\overline{K}_u$  \phantom{ \LARGE G}     & $\overline{K}_a$        \\ \hline
$A_l$    & $SL_{l+1}(\overline{\mathbb{F}})$ & $PGL_{l+1}(\overline{\mathbb{F}})$ \\ \hline
$C_l$    & $Sp_{2l}(\overline{\mathbb{F}})$  & $PGSp_{2l}(\overline{\mathbb{F}})$ \\ \hline
\end{tabular}
\label{algAC}
\end{table}
\begin{table}[] 
\centering
\caption{Fixed point subgroups}
\begin{tabular}{|l|l|l|l|}  
\hline
$\Sigma$ & $\sigma$ & $(\overline{K}_u)_{\sigma}$ \phantom{ \LARGE G} & $(\overline{K}_a)_{\sigma}$ \\ \hline
$A_l$    & {\bf q}  & $SL_{l+1}(q)$    & $PGL_{l+1}(q)$   \\ \hline
$A_l$    & {\bf qg} & $SU_{l+1}(q)$    & $PGU_{l+1}(q)$   \\ \hline
$C_l$    & {\bf q}  & $Sp_{2l}(q)$     & $PGSp_{2l}(q)$   \\ \hline
\end{tabular}
\label{algSIG}
\end{table}


\section{Miscellaneous results}
\label{missec}

 We begin by stating two classical results on the structure of solvable linear groups.
\begin{Lem}[{\cite[\S 18, Theorem 5]{sup}}] \label{supirr}
An irreducible solvable subgroup of $GL_n(q)$ is either primitive, or conjugate in $GL_n(q)$ to a subgroup of the wreath product
$S \wr \Gamma$
where $S$ is a primitive solvable subgroup of $GL_m(q)$ and $\Gamma$ is a transitive solvable
subgroup of the symmetric group $\Sym(k)$ and $km = n$. In particular, an irreducible maximal solvable  subgroup of 
$GL_n(q)$ is either primitive, or conjugate in $GL_n(q)$
to  $S \wr \Gamma$, where $S$ is a  primitive maximal solvable subgroup of $GL_m(q)$
and $\Gamma$ is a  transitive maximal solvable subgroup of $\Sym(k)$.
\end{Lem}

\begin{Lem}[{\cite[\S 18, Theorem 3]{sup}}]
\label{supreduce}
Let  $H$ be a
subgroup of $GL_n(q)$. In a suitable basis $\beta$ of $V$, the matrices $g\in H$ have the
shape
\begin{gather}\label{stup}
\begin{pmatrix}
g_k     & g_{k,(k-1)} & \ldots & g_{k,1}  \\
0          &    g_{k-1}  & \ldots & g_{(k-1),1}  \\
\ldots     & \ldots     & \ldots & \ldots     \\
      0    &   0        & \ldots & g_1 
\end{pmatrix}
\end{gather}
where the mapping $\gamma_i: H \to GL_{n_i}(q)$, $g \mapsto g_i$ is an irreducible 
representation of $H$ of degree $n_i$, and $g_{i,j}$ is an $(n_i \times n_j)$ matrix over $\mathbb{F}_q$, and $n _1 + \ldots + n _k = n$.
 The group $H$  is solvable if and only if all the
groups
$$H_i ={\mathrm{Im} } (\gamma_i) \text{ for } i = 1, \ldots, k,$$
are solvable.
\end{Lem}

 Now we state three technical lemmas about solvable linear groups. We use them and ideas from their proofs many times throughout  this work. 

\begin{Lem}\label{supSL}
Let $H \le G \le GL_n(q).$ Assume $\Det(H)=\Det(G).$ 
\begin{enumerate}[font=\normalfont]
\item $H \cdot (SL_n(q) \cap G)=G$.
\item If $g \in G$, then there exists $g_1 \in SL_n(q) \cap G$ such that $H^g=H^{g_1}.$ 
\end{enumerate}
\end{Lem}
\begin{proof}
Let $g \in G$. Since $\Det(H)=\Det(G)$, there exists $h \in H$ such that $\det(h)=\det(g).$ Therefore, 
$h^{-1}g \in SL_n(q) \cap G$, so $g=h \cdot (h^{-1}g)\in H \cdot (SL_n(q) \cap G).$

Let $g_1=h^{-1}g,$ so $H^{g_1}=g^{-1}hHh^{-1}g=g^{-1}Hg=H^g.$
\end{proof}

\begin{Lem}\label{irrtog}
Let $H$ be a subgroup of $GL_n(q)$ of shape \eqref{stup}. If for every $H_i$ there exists $x_i \in GL_{n_i}(q)$ (respectively $SL_{n_i}(q)$) such that the intersection $H_i \cap H_i^{x_i}$ consists of upper triangular matrices, then there exist $x,y \in GL_n(q)$ (respectively $SL_n(q)$) such that 
$$(H\cap H^x) \cap (H\cap H^x)^y \le D(GL_n(q)).$$
\end{Lem}
\begin{proof}
Let $x=\diag(x_k, \ldots, x_1)$ and let $y$ be $$\diag(\sgn(\sigma),1 \ldots,1) \cdot \per(\sigma)$$ where  $$\sigma =(1,n)(2, n-1) \ldots ([n/2], [n/2+3/2]).$$ Here $[r]$ is the integer part of a positive number $r$. Since $\det(\per(\sigma))=\sgn(\sigma),$ $\det(y)=1.$ Since
$H\cap H^x$ consists of upper triangular matrices and $(H\cap H^x)^y$ consists of lower triangular matrices, 
\begin{equation*}
(H\cap H^x) \cap (H\cap H^x)^y \le D(GL_n(q)). 
\end{equation*} 
If $x_i \in SL_{n_i}(q)$, then $x \in SL_n(q).$ 
\end{proof}

\begin{Def}
Let $\{v_1, v_2, \ldots,  v_n \}$ be a basis of $V$. The equality 
$$v=\alpha_1v_1+ \alpha_2v_2+  \ldots+ \alpha_nv_n $$ 
is the {\bf decomposition} \index{decomposition} of $v \in V$ with respect to $\{v_1, v_2, \ldots,  v_n \}$.
 \end{Def}

\begin{Lem}\label{diag}
Assume that matrices in $H\le GL_n(q)$ have shape \eqref{stup} with respect to the  basis $$\{v_1, v_2, \ldots,  v_n \}$$ and $n_1 < n,$ so $H$ stabilises $$U=\langle v_{n-n_1+1}, \ldots, v_n \rangle.$$  Then there exists $z \in SL_n(q)$ such that $D(GL_n(q)) \cap H^z \le Z(GL_n(q)).$
\end{Lem} 
\begin{proof}
Let $m=n_1.$ Define vectors 
\begin{equation*}
\begin{split}
u_1 & = v_1 + \ldots + v_{n-m} + v_{n-m+1}  \\
u_2 & = v_1 + \ldots + v_{n-m} + v_{n-m +2}\\
u_3 & = v_1 + \ldots + v_{n-m}  + v_{n-m +3}\\
\vdots \\
u_m & = v_1 + \ldots + v_{n-m} + v_{n}.
\end{split}
\end{equation*}
Let $z \in GL_n(q)$ be such that 
\begin{equation}\label{ui}
(v_{n-m+i})z=u_i, \text{ } i=1, \ldots, m.
\end{equation}
 Such  $z$ exists since $u_1, \ldots, u_m$ are linearly independent and we can assume $z\in SL_n(q)$ since $m<n.$ So \eqref{ui} implies that 
$$Uz = \langle u_1, \ldots, u_m \rangle $$ is $H^z$-invariant.

Let $g \in D(GL_n(q)) \cap H^z.$ So $g$ is $\text{diag}(\alpha_1, \ldots , \alpha_n)$ with respect to the basis $\{v_1, v_2, \ldots, v_{n}\}$. Thus,
$$(u_j)g = \alpha_1 v_1 + \ldots + \alpha_{n-m} v_{n-m} + \alpha_{n-m+j} v_{n-m+j}. $$
Also, $(u_j)g$ must lie in $Uz,$ since $g \in H^z,$ so 
$$(u_j)g = \beta_1 u_1 + \beta_2 u_2 + \ldots + \beta_m u_m.$$
But the decomposition of $(u_j)g$ with respect to    $\{v_1, v_2, \ldots, v_{n}\}$ does not contain $v_{n-m +i}$ for $i \ne j$, so 
$$\beta_1 = \beta_2 = \ldots =\beta_{j-1} =\beta_{j+1}= \ldots = \beta_m =0$$ and $(u_j)g= \beta_j u_j$. Thus, 
$$\alpha_1 = \alpha_2 = \ldots =  \alpha_{n-m}=\alpha_{n-m+j}$$
for $0<j\le m.$
Therefore, $g$ is scalar and lies in $Z(GL_n(q)).$
\end{proof}

 The next three lemmas provide information on $\Gamma(V, {\bf f})$ and its subgroups. Here ${\bf f}$ is unitary or symplectic; by default, we  assume that such a form ${\bf f}$ is non-degenerate.

\begin{Lem}[{\cite[(5.5)]{asch}}] \label{ashb}
Let $H \le \Gamma(V, {\bf f})$, with {\bf f} unitary or symplectic, be irreducible. Let $L$ be a non-scalar normal subgroup of $H$ contained in $GL(V)$. Let $\{V_i \mid 1 \le i\le k\}$ be the homogeneous components of $L$ on
$V$ and assume $k > 1$. One of the following holds:
\begin{enumerate}[font=\normalfont]
\item  $$ \displaystyle V=\mathop{\bot}_{1 \le i\le k} V_i$$ 
 with $V_i$ non-degenerate and isometric to $V_j$ for each $1 \le i \le j \le k $;
\item  $$ \displaystyle V=\mathop{\bot}_{1 \le i\le k/2} U_i$$ with $U_i=V_{2i-1} \oplus V_{2i}$ where $U_i$ is non-degenerate and isometric to $U_j$ for 
$1 \le i \le j \le k / 2$,
and $V_i$ is totally isotropic for each $1 \le i \le k$.
\end{enumerate}
\end{Lem}

\begin{Lem}
\label{uniGamsdp}
Let ${\bf f}$ be a non-degenerate unitary or symplectic form on $V.$ If $\beta$ is a basis of $V$ such that ${\bf f}_{\beta}^{\phi_{\beta}}={\bf f}_{\beta},$ then 
$\Gamma(V,{\bf f})_{\beta}= \Delta(V,{\bf f})_{\beta} \rtimes \langle \phi_{\beta} \rangle.$ 
\end{Lem}
\begin{proof}
Clearly, $\Delta(V,{\bf f})_{\beta} \cap \langle \phi_{\beta} \rangle=1,$ so it suffices to show that $\phi_{\beta}$ normalises $\Delta(V,{\bf f})_{\beta}$ and  is a semisimilarity of $(V, {\bf f}).$ Let $g \in \Delta(V,{\bf f})_{\beta}$, so 
$$g {\bf f}_{\beta} \overline{g}^{\top}=\lambda {\bf f}_{\beta}$$
for some $\lambda \in \mathbb{F}_q^*.$ Therefore,
$$g^{\phi_{\beta}} {\bf f}_{\beta} \overline{(g^{\phi_{\beta}})}^{\top}=g^{\phi_{\beta}} {\bf f}_{\beta}^{\phi_{\beta}} \overline{(g^{\phi_{\beta}})}^{\top}=(g {\bf f}_{\beta} \overline{g}^{\top})^{\phi_{\beta}}= (\lambda {\bf f}_{\beta})^{\phi_{\beta}}= \lambda^p {\bf f}_{\beta}, $$
and $g^{\phi_{\beta}} \in \Delta(V,{\bf f})_{\beta}.$  

Let $v,u \in V$ have coefficients $(\alpha_1, \ldots, \alpha_n)$ and $(\delta_1, \ldots, \delta_n)$ with respect to $\beta$ respectively. Therefore,
\begin{align*}(v\phi_{\beta},u \phi_{\beta}) & =(\alpha_1^p, \ldots, \alpha_n^p){\bf f}_{\beta}(\overline{\delta_1}^p, \ldots, \overline{\delta_n}^p)^{\top}\\
& =(\alpha_1^p, \ldots, \alpha_n^p){\bf f}_{\beta}^{\phi_{\beta}}(\overline{\delta_1}^p, \ldots, \overline{\delta_n}^p)^{\top}\\
& =(v,u)^p,
\end{align*}
so $\phi_{\beta}$ is a semisimilarity.
\end{proof}

\begin{Lem}\label{unist}
Recall that $q=p^f.$ Let $H \le \Gamma(V, {\bf f})$ with ${\bf f}$ unitary or symplectic. There exists a basis $\beta$ such that 
${\bf f}_{\beta}$ is
\begin{equation}\label{fst}
\begin{aligned}
& \left(
\begin{smallmatrix}
        & & & & & & &    & I_{n_1} \\
        & & & & & & &  \reflectbox{$\ddots$}  &\\
        & & & & & &I_{n_k} &    & \\
        & & &I_{n_{k+1}} & & & &    &\\ 
        & & & &\ddots & & &    &\\
        & & & & &I_{n_{k+l}} & &    &\\ 
        & &I_{n_k} & & & & &    & \\
        &\reflectbox{$\ddots$} & & & & & &    &\\
I_{n_1} & & & & & & &    & 
\end{smallmatrix} \right)  \text{ or } \\
 & \left(
\begin{smallmatrix}
        & & & & & & &    & I_{n_1} \\
        & & & & & & &  \reflectbox{$\ddots$}  &\\
        & & & & & &I_{n_k} &    & \\
        & & &J_{n_{k+1}} & & & &    &\\ 
        & & & &\ddots & & &    &\\
        & & & & &J_{n_{k+l}} & &    &\\ 
        & &-I_{n_k} & & & & &    & \\
        &\reflectbox{$\ddots$} & & & & & &    &\\
-I_{n_1} & & & & & & &    & 
\end{smallmatrix} \right) 
\end{aligned}
\end{equation}
 in cases {\bf U} and {\bf S} respectively. Moreover, if $\varphi \in H_{\beta},$ then $\varphi=({\phi_{\beta}})^j g$ with $$j \in \{1, \ldots, {\bf u}f-1\}$$ and 
\begin{equation}\label{gst}
g=\Scale[0.9]{ \begin{pmatrix}
{\tau(g)\gamma_{1}(g)^{\dagger}}&* &\multicolumn{1}{l|}{*} &* & \ldots & *&* &*    &*  \\
        & \ddots &\multicolumn{1}{l|}{} & & &\ddots & &    &\\
0        & & \multicolumn{1}{l|}{\tau(g)\gamma_{k}(g)^{\dagger}}&* &* &* &\ldots &*    &* \\  \cline{1-6}
        & & \multicolumn{1}{l|}{} &\gamma_{k+1}(g) & & \multicolumn{1}{l|}{0}  &* & \ldots   &*\\ 
        & &  \multicolumn{1}{l|}{}& &\ddots &\multicolumn{1}{l|}{}       & & \ddots   &\\
        & &  \multicolumn{1}{l|}{}&0 & &\multicolumn{1}{l|}{\gamma_{k+l}(g)}    & * & \ldots   & *\\ \cline{4-9} 
        & & & & &\multicolumn{1}{l|}{} & \gamma_{k}(g) &  *  &* \\
        & & & & &\multicolumn{1}{l|}{} & & \ddots   &*\\
0        & & & & &\multicolumn{1}{l|}{} &0 &    & \gamma_{1}(g)\\  
\end{pmatrix}}
\end{equation}
where $\tau(g)$ is as in Definition $\ref{taudef}$, $\gamma_i$ is a homomorphism from $H$ to $\GL_{n_i}(q^{\bf u})$ if $i \le k$, and from $H$ to $\GU_{n_i}(q)$ or $\GS_{n_i}(q)$, in cases {\bf U} and {\bf S} respectively, if $i>k.$ Furthermore, $\gamma_i(H)$ is an irreducible subgroup of $\GL_{n_i}(q^{\bf u})$ for every $i$ and $\gamma_i(H \cap GL_n(q^{\bf u})) \le GL_{n_i}(q^{\bf u}).$ 
\end{Lem}
\begin{proof}
If $H$ is an irreducible subgroup of $\GL(V,\mathbb{F}_{q^{\bf u}})$, then by Lemmas \ref{unibasisl} and \ref{sympbasisl} we can take ${\bf f}_{\beta}$ to be $I_n$ or $J_n$ in cases {\bf U} and {\bf S}, and there is nothing  to prove. So assume that there is a proper $H$-invariant subspace $W$ of $V=\mathbb{F}_{q^{\bf u}}^n$ on which $H$ acts irreducibly. Therefore, $W$ is either non-degenerate or totally isotropic. If $V$ has no totally isotropic $H$-invariant subspace, then $V$ is the direct sum of pairwise orthogonal $H$-invariant  non-degenerate subspaces, so $k=0$ and the lemma follows.

Assume that $W$ is totally isotropic.  By Lemma \ref{witt} we can assume that $V$ has a basis $\beta$ as in \eqref{unibasis} such that $$W= \langle e_{(n-n_1+1)}, \ldots, e_{n} \rangle$$
where $n_1=\dim W$. 
Let $U$ be the subspace spanned by 
$$\beta \backslash \{f_{(n-n_1+1)}, \ldots, f_{n}, e_{(n-n_1+1)}, \ldots, e_{n}\}.$$ Notice that $U$ is non-degenerate. Let $\beta_2:=\{v_1, \ldots, v_{n-2n_1}\}$ be a basis of $U$ such that $\beta_2$ is orthonormal in case {\bf U} and as in \eqref{sympbasis} in case {\bf S}. 

Let us define a basis  $$\beta_1:=\{f_{(n-n_1+1)}, \ldots, f_{n}, v_1, \ldots, v_{n-2n_1},   e_{(n-n_1+1)}, \ldots, e_{n}\}.$$
Hence
 \begin{equation*}
{\bf f}_{\beta_1}=
\begin{pmatrix}
  & & I_{n_1}\\
  &\Phi&\\
(-1)^{\bf u}I_{n_1} & &
\end{pmatrix}
\end{equation*} 
with $\Phi$ equal to $I_{n-2n_1}$ and $J_{n-2n_1}$ in cases {\bf U} and {\bf S} respectively. Since $H$ stabilises $W$, it also stabilises $W^{\bot}=\langle v_1, \ldots, v_{n-2n_1},   e_{(n-n_1+1)}, \ldots, e_{n} \rangle$.  By Lemma \ref{uniGamsdp}, if $\varphi \in H_{\beta_1},$ then $\varphi= (\phi_{\beta_1})^j g$ with $g \in GU_n(q, {\bf f}_{\beta_1})$ or $GSp_n(q,{\bf f}_{\beta_1})$ respectively, so, by \eqref{unimatr},   
\begin{equation*}
g=
\begin{pmatrix}
{\tau(g){(g_W})^{\dagger}}&* &*   \\
   0     & g_1 &* \\
0        &0 & g_W    
\end{pmatrix},
\end{equation*}
where $g_1$ is an $(n-2n_1) \times (n-2n_1)$ matrix. If $n=2n_1$, then the lemma follows. We proceed by induction on $n-2n_1$ using the case $n-2n_1=0$ as the base. 

Assume that $n>2n_1.$ Since $g {\bf f}_{\beta_1} \overline{g}^{\top}=\tau(g){\bf f}_{\beta_1},$ 
$$g_1 \Phi \overline{g_1}^{\top}=\tau(g)\Phi.$$
Thus, $g_1$ is a similarity of $$(\langle v_1, \ldots,  v_{n-2n_1} \rangle,{\bf f_1} ),$$
and $$(\phi_{\beta_2})^j g_1 \in \Gamma (\langle v_1, \ldots,  v_{n-2n_1} \rangle,{\bf f_1} )$$
where   ${\bf f_1}$ is the restriction of ${\bf f}$ to $\langle v_1, \ldots,  v_{n-2n_1} \rangle$. Notice that $(\phi_{\beta_2})^j g_1$ is the restriction of $\varphi$ to $W^{\bot}/W.$  So there exists a homomorphism $\psi$  from $H_{\beta_1}$ to $\GU_{n-2n_1}(q)$ in case {\bf U} and $\GS_{n-2n_1}(q)$ in case {\bf S} defined by 
$\psi:g \mapsto (\phi_{\beta_2})^j g_1.$ Applying induction to $\psi(H_{\beta_1}),$ we obtain  the lemma. 
\end{proof}

 The following lemma plays an important role in our proof of Theorems \ref{theorem}, \ref{theoremGU} and \ref{theoremSp}.

\begin{Lem}
\label{scfield}
Let  $\Gamma \in \{\GL_n(q), \GU_n(q), \GS_n(q)\}.$ Let $n\ge 2$ and let $q$ be such that $\Gamma$ is not solvable. Let $\beta$ be a basis of $V$ such that ${\bf f}_{\beta}^{\phi_{\beta}}={\bf f}_{\beta}$ and let $\phi=\phi_{\beta}$ If $H \le \Gamma $ and  $H \cap GL_n(q^{\bf u})$ consists of scalar matrices, then there exists $b \in \Gamma \cap GL_n(q^{\bf u})$ such that every element of $H^b$ has shape $\phi^i g$ for some $i \in \{1, \ldots, {\bf u} f\}$ and $g \in Z(GL_n(q^{\bf u})).$   
\end{Lem}
\begin{proof}
 Let $Z=Z(GL_n(q^{\bf u}) \cap \Gamma).$ Notice that $\Gamma/Z$ is almost simple. Let $G_0$ and $\hat{G}$ be the socle of $\Gamma/Z$ and the group of inner-diagonal automorphisms of $G_0$ respectively. Therefore, $\hat{G}=(\Gamma \cap GL_n(q^{\bf u}))/Z.$ Without loss of generality, we may assume $Z \le H.$ Observe $H\cap GL_n(q^{\bf u}) = Z,$ so $H/Z$ is cyclic and consists of field automorphisms of $G_0.$ Let $\varphi \in H$ be such that $\langle Z \varphi \rangle=H/Z.$   By Lemma \ref{uniGamsdp},
$$\Gamma=(\Gamma \cap GL_n(q^{\bf u})) \rtimes \langle \phi \rangle,$$
so $\varphi \in \phi^i (\Gamma \cap GL_n(q^{\bf u}))$ for some  $i \in \{1, \ldots, {\bf u} f\}$ and $Z\varphi \in (Z\phi^i) \hat{G}$.

By \cite[(7-2)]{conjaut}, $Z \varphi$ and $Z \phi^i$ are conjugate in $\hat{G}$, so   there exists  $Zb \in \Gamma/Z \cap PGL_n(q^{\bf u})$ such that $(Z\varphi)^{Zb}=Z\phi^i$ for some $i \in \{1, \ldots, {\bf u} f\}.$  Therefore, $H^b = Z\langle \varphi^b \rangle = Z\langle \phi^i \rangle = \langle\phi^i \rangle Z.$
\end{proof}

\begin{Lem}\label{al}
For every prime power $q=p^f$ there  exists $\alpha \in \mathbb{F}_{q^2}$ such that $\alpha+ \alpha^q=1.$
\end{Lem}
\begin{proof}
If $p \ne 2$, then  $2^{-1} \in \mathbb{F}_q^{*}$, so $2^{-1} + (2^{-1})^q=2^{-1} + 2^{-1}=1.$

Let $p=2.$ Let $y \in \overline{\mathbb{F}_2}$ be a root of polynomial $x^q+x+1=0$. Hence 
$$y^{q^2}=(y^q)^q=(y+1)^q=y^q+1=y+1+1=y,$$
so $y \in \mathbb{F}_{q^2}.$
\end{proof} 

\begin{Lem}
\label{pj10}
Let $\eta$ be a generator of $\mathbb{F}_{q^2}^*$ and let $\theta=\eta^{q-1}.$ If $\theta^{p^j-1}=1$ for some $j \in \{0,1, \ldots, 2f-1\}$, then $j=0.$
\end{Lem}
\begin{proof}
Notice that $|\theta|=p^f+1.$ Let $j \in \{1, \ldots, 2f\}$ be minimal such that $p^f+1$ divides $p^j-1.$ Hence $p^f+1$ divides $(p^{2f}-1,p^j-1)=p^{(2f,j)}-1.$ Therefore, $(2f,j)>f$, so $j=2f.$
\end{proof}

\begin{Th}[Clifford's Theorem]\label{cliff}
 Let $H$ be a normal subgroup of a finite group $G$ and let $V$ be an irreducible $\mathbb{F}[G]$-module for an arbitrary field $\mathbb{F}.$ Let $W$ be an irreducible $\mathbb{F}[H]$-submodule of $V$. 
\begin{itemize}
\item[$(1)$] $V=W_1 \oplus \ldots \oplus W_k$ where $W_i$ is an irreducible $\mathbb{F}[H]$-submodule of $V$ and each $W_i$ has the form $(W)g_i$ for some $g_i \in G.$
\item[$(2)$] If $L_i$ for $i=1, \ldots, t$ is a homogeneous component of $H$ on $V$  and $t>1$, then $$V=L_1 \oplus \ldots \oplus L_t$$ and $\{L_1, \ldots, L_t\}$ is a system of imprimitivity for $G$. 
\end{itemize} 
\end{Th}
\begin{proof}
See \cite[\S 16]{sup}.
\end{proof}


\section{Singer cycles}
\label{sinsec}

\begin{Def} A  {\bf Singer cycle} \index{Singer cycle} of $GL_n(q)$ is a cyclic subgroup of order $q^n - 1$.
\end{Def}

While  Singer cycles are well known, many related statements are  ``folklore" and are often stated without  proof or  reference.
Therefore, for completeness, we include statements with a proof and a reference for the earliest proof we found.

\begin{Lem}[{\cite{singer38}, \cite[Chapter II, \S 7]{hupp}}] \label{singex} 
 A Singer cycle always exists.
\end{Lem}
\begin{proof}
A field $\mathbb{F}_{q^n}$ can be considered as an $n$-dimensional vector space $V=\mathbb{F}_q^n$ over $\mathbb{F}_q.$  Right multiplication by a generator of  $\mathbb{F}_{q^n}^*$ determines a bijective linear map from $V$ to itself of order $q^n-1$. So  $\mathbb{F}_{q^n}^*$  is isomorphic to a cyclic subgroup of $GL_n(q)$ of order $q^n-1$. 
\end{proof}

Moreover, the action of $\mathbb{F}_{q^n}^*$ on the set $V \backslash \{0\}$ is {\bf regular}, since $x(x^{-1}y)=y$ ({\bf transitivity}) and $xy=y$ if and only if $x=1$ ({\bf semiregularity}) for $x,y \in \mathbb{F}_{q^n}^*$.


\begin{Lem}\label{por}
If $x \in GL_n(q),$ then $|x| \le q^n-1$.
\end{Lem}
\begin{proof}
 Let $\chi_x(t)$ be the characteristic polynomial of $x$, so by the Cayley--Hamilton Theorem $$\chi_x(x)=0.$$ Therefore,  the dimension of the subalgebra $\mathbb{F}_q[\langle x \rangle]$ of $M_n(q)$, generated by $x$, is at most 
$$\deg \chi_x(t) \le n.$$ So $|\mathbb{F}_q[\langle x \rangle]|\le|\mathbb{F}_q|^n=q^n$ and $|x|\le q^n-1.$
\end{proof}

\begin{Lem}\label{irr}
If $T \le GL_n(q)$ is a Singer cycle, then $T$ is irreducible.
\end{Lem}
\begin{proof}
Let $x$ be a generator of $T$. Obviously, $(q^n-1,p)=1$. By Maschke's Theorem $T$ is completely reducible and $x$ is conjugate to a block-diagonal matrix $$y=\diag [y_1, \ldots, y_k],$$
where $y_i \in GL_{n_i}(q)$, $\sum_{i=1}^k n_{i}=n$ and each $y_i$ generates an irreducible subgroup of $GL_{n_i}(q)$. Therefore, $$|x|=|y|={\rm lcm}(|y_1|, \ldots, |y_k|)\le \prod_{i=1}^k|y_i|.$$
By Lemma \ref{por} $|y_i| \le q^{n_i}-1,$ so  
$$q^n-1=|x|\le \prod_{i=1}^k (q^{n_i}-1),$$
which is true only if $k=1,$ so $T$ is irreducible.
\end{proof}

\begin{Lem}\label{field}
If $T \le GL_n(q)$ is a Singer cycle, then $T \cup \{0\}$  is a field with the usual matrix addition and multiplication.
\end{Lem}
\begin{proof}
Let $x$ be a generator of $T$ and let $\mu_x(t)$ be the minimal polynomial of $x.$ If $0 \ne v \in V$, then $\{v,vx,vx^2, \ldots, vx^{\deg{\mu_x(t)}-1}\}$ spans an $x$-invariant subspace, so, since $x$ acts irreducibly on $V$ by Lemma \ref{irr},  $\deg{\mu_x(t)}=n.$ Therefore, $\mathbb{F}_q[\langle x \rangle]$ has dimension $n$ and $|\mathbb{F}_q[\langle x \rangle]|=q^n,$ so $\mathbb{F}_q[\langle x \rangle]=T \cup \{0\},$ since $\mathbb{F}_q[\langle x \rangle]$ contains $T$ by definition. Thus, $T \cup \{0\}$ is closed under taking sums and multiplication, every non-zero element has an inverse, and the remaining axioms of a field are evident.      
\end{proof}

\begin{Th}[Noether-Skolem Theorem]
Let $R$ be a simple Artinian ring with center $F$ and let $A$, $B$ be simple subalgebras of $R$ which contain $F$ and are finite-dimensional over it. If $\phi: A \to B$ is an isomorphism  which fixes $F$ elementwise,  then there exists an invertible $x \in R$ such that $(a)\phi=x^{-1}ax$ for all $a \in A.$
\end{Th}
\begin{proof}
See \cite[Chapter 4]{herst}.  
\end{proof}

\begin{Th}
\label{isoextth}
Let $ \psi :{F}_1 \to {F}_2$ be a field isomorphism. If $E_1$ and $E_2$ are isomorphic algebraic field extensions of $F_1$ and $F_2$ respectively, then there is an isomorphism $\phi : E_1 \to E_2$ such that $\phi |_{_{F_1}}=\psi.$ 
\end{Th}
\begin{proof}
See \cite[Theorem 3.20]{extth}.  
\end{proof}  

\begin{Lem}\label{singconj} 
All Singer cycles are conjugate in $GL_n(q).$ 
\end{Lem}
\begin{proof}
Let $T_1$ and $T_2$ be  Singer cycles, and let $A_i$   be the $\mathbb{F}_q$-subalgebra of $M_n(q)$ generated by $T_i$. By Lemma \ref{field}, $A_1$ and  $A_2$ are fields of order $q^n$ and each of them contains  the subalgebra of scalar matrices $F=Z(M_n(q)) \cong \mathbb{F}_q.$ By Theorem \ref{isoextth} there is an isomorphism $\phi: A_1 \to A_2$ fixing $F$ elementwise. Therefore, by the Noether-Skolem Theorem, there exists $x \in GL_n(q)$ such that $A_1^x=A_2$ and in particular $T_1^x=T_2.$
\end{proof}


\begin{Lem}\label{centr} 
If $A$ is an abelian regular permutation group of degree $n$, then $C_{\Sym(n)}(A)=A.$
\end{Lem}
\begin{proof}
Let $c \in C:= C_{\Sym(n)}(A)$ be such that $1^c=1$ and let $a \in A$ be such that $1^a=i.$ Notice
$$i^c=1^{ac}=1^{ca}=1^{a}=i,$$ so $c$ is trivial. Therefore, 
$$|C|=|C:\Fix_{C}(1)|.$$
By the Fundamental Counting Principle, $|C:\Fix_{C}(1)|$ is the size of the orbit of $C$ containing $1,$ so $|C:\Fix_{C}(1)| \le n=|A|$.
 The claim follows from the inclusion $A \le C$. 
\end{proof}

\begin{Lem}
A Singer cycle is primitive as a linear group.
\end{Lem}
\begin{proof}
Assume that $T$ has a system of imprimitivity
$V_1 \oplus \ldots \oplus V_k,$
so $k \ne 1$ divides $n$ and for each $t \in T$ and $i \in \{1,\ldots, k\}$
 $$(V_i)t=V_j$$
for some $j \in \{1,\ldots k\}.$ Let $v_1$ and $v_2$ be non-zero vectors from $V_1$ and $V_2$ respectively. Since $T$ acts regularly on the set of non-zero vectors, there exists $t \in T$ such that $(v_1)t=v_1+v_2$. Therefore, $(v_1)t$ does not lie in any $V_i$ which contradicts the assumption.  
\end{proof}

\begin{Lem}\label{sin2}
Let $T \le GL_n(q)$ be a Singer cycle, let $g \in GL_n(q)$, and let $\phi$ be the field automorphism such that $(h^{\phi})_{i,j}=(h_{i,j}^p)$ for $h \in GL_n(q)$. Let  $\psi=g \phi^j$, where $$j \in \{0,1, \ldots , f-1\}.$$    If $T^{\psi}=T$, then $\psi$ acts on $T \cup \{0\}$ as a field automorphism.
\end{Lem}
\begin{proof}
Let $x \in \{g, \phi\}.$ It suffices to show that  $x$ induces an isomorphism of fields between $T \cup \{0\}$ and $T^x \cup \{0\}$. Since the action induced by $x$ is clearly a group isomorphism, we only need to show that $x$ preserves sums.

If $x=g \in GL(n,q)$ then the statement follows from properties of matrix addition and multiplication. If $x=\phi$, then  \begin{equation*}(t_1 +t_2)^x=((t_1 +t_2)(i,j))^p=(t_1(i,j))^p+(t_2(i,j))^p=t_1^x+t_2^x. \qedhere
\end{equation*}

\end{proof}

\begin{Th}[{\cite[Chapter II, \S 7]{hupp}}]\label{sin1} 
If $T$ is a Singer cycle of $GL_n(q)$, then $$C_{GL_n(q)}(T) = T$$
and $N_{GL_n(q)} (T) /T$ is cyclic of order $n$. Moreover, $N_{GL_n(q)} (T)=T \rtimes \langle \varphi \rangle$, where $t^{\varphi}=t^q$ for $t \in T.$ 
\end{Th}
\begin{proof}
Consider $T$ as a cyclic regular subgroup of $\Sym(\mathbb{F}_q^n \backslash \{0\}).$  Lemma \ref{centr} implies the first claim  since $GL_n(q)$ is isomorphic to a subgroup of  $\Sym(\mathbb{F}_q^n \backslash \{0\})$.

Let $h \in N_{GL_n(q)} (T)$, so $h$ induces a field automorphism $\varphi$ of $K=T \cup \{0\}$ by Lemma \ref{sin2}. Since $h \in GL_n(q)$, $\varphi$ stabilises the subfield $F \cong \mathbb{F}_q$ of $K$ consisting of scalar matrices, so $|\varphi|$ divides $|\mathbb{F}_{q^n}:\mathbb{F}_q|=n$ and $t^h=t^{\varphi}=t^{q^{(n/|\varphi|)}}.$ If $h_1$ and $h_2$ induce the same automorphism $\varphi$, then $h_2h_1^{-1} \in C_{GL_n(q)}(T)=T,$ so 
$$|N_{GL(n,q)} (T) /T|\le |\mathbb{F}_{q^n}:\mathbb{F}_q|=n.$$

Now we show that every such automorphism is induced by some element of $GL_n(q).$ Recall that we can identify $T$ with $\mathbb{F}_{q^n}^*$ acting on $\mathbb{F}_{q^n}$ as on an $n$-dimensional vector space over $\mathbb{F}_q$ by multiplication, so we can consider vectors in $\mathbb{F}_q^n$ as elements of $\mathbb{F}_{q^n}.$   
Let $\varphi \in \Gal(\mathbb{F}_{q^n}:\mathbb{F}_q),$ so for $v_1, v_2 \in \mathbb{F}_q^n,$ $\lambda \in \mathbb{F}_q$ 
$$(v_1+ \lambda v_2) \varphi = (v_1) \varphi+ \lambda^{\varphi} (v_2)\varphi=(v_1) \varphi+ \lambda (v_2)\varphi.$$

Therefore, $\varphi$ acts linearly on $V=\mathbb{F}_q^n$, so $\varphi \in GL_n(q).$  
\end{proof}


\begin{Lem}\label{irrsub}
A proper subgroup $C$ of a Singer cycle $T \le GL_n(q)$ is irreducible if and only if $|C|$ does not divide $q^r-1$ for every proper divisor $r$ of $n$. 
\end{Lem}
\begin{proof}
Since all subgroups of a cyclic group are cyclic, $C=\langle \sigma \rangle$ for some $\sigma \in T.$

Assume that $C$ is reducible, so there is a non-zero proper  subspace $V_1$ of $V=\mathbb{F}_q^n$ such that $$(V_1)\sigma=V_1.$$ Therefore, $V_1 \backslash \{0\}$ is a collection of $m$ orbits of $C$ acting on all non-zero vectors. Since $T$ is regular on that set, $C$ is semiregular and all orbits have size $|C|.$ Thus, $$q^r-1 =|V_1 \backslash \{0\}|=m|C|$$ for some  $r<n$. Thus, $|C|$ divides $q^r-1$ and $r$ divides $n$ by Clifford's Theorem. 

Let $r$ be the smallest divisor of $n$ such that $|C|$ divides $q^r-1.$ Then $C$ lies in the unique subgroup $S=\langle x^s \rangle$ of $T$ of order $q^r-1$, where $x$ is a generator of $T$ and $s=(q^n-1)/(q^r-1).$ Since $S$ is  unique, $S \cup \{0\}$ is a subfield of $T \cup \{0\}$. Therefore, $\dim (\mathbb{F}_q[\langle \sigma \rangle]) \le r$, which implies $\deg \mu_{\sigma}(t) \le r$ and  
$$\{v,v\sigma,v\sigma^2, \ldots, v\sigma^{\deg{\mu_{\sigma}(t)}-1}\}$$
spans a proper $\sigma$-invariant subspace of $V.$
\end{proof}

\begin{Lem}\label{semis}
An irreducible cyclic subgroup  of $GL_n(q)$ is  contained in some Singer cycle. 
\end{Lem}
\begin{proof}
Let $C=\langle \sigma \rangle$ be an irreducible cyclic subgroup  of $GL_n(q)$. For  non-zero $v \in V$,  as in the proof of Lemma \ref{field}, $$\{v,v\sigma,v\sigma^2, \ldots, v\sigma^{\deg{\mu_{\sigma}(t)}-1}\}$$
spans a $\sigma$-invariant subspace of $V.$ Hence $\deg{\mu_{\sigma}(t)}=n$ and
$\dim (\mathbb{F}_q[\langle \sigma \rangle])=n.$   Moreover, all non-zero matrices in $\mathbb{F}_q[\langle \sigma \rangle]$ are invertible. Indeed,
assume that $$\alpha_0 + \alpha_1 \sigma + \ldots + \alpha_{n-1}\sigma^{n-1}$$
is non-zero and non-invertible. Therefore, there exists a non-zero $v \in V$ such that 
$$v(\alpha_0 + \alpha_1 \sigma + \ldots + \alpha_{n-1}\sigma^{n-1})=0,$$
so $$\{v, v \sigma, \ldots , v\sigma^{n-2}\}$$ 
spans an $\sigma$-invariant subspace. This subspace is spanned by $n-1$ vectors, so it is the zero subspace since $\sigma$ is irreducible. Hence $v=0$, which  contradicts  our choice.

Thus, $\mathbb{F}_q[\langle \sigma \rangle]$ is a field, so $C$ is a subgroup of its multiplicative group which is a Singer cycle.
\end{proof}

\begin{Cor}\label{semis1}
There exists a cyclic irreducible subgroup of order $m$ in $GL_n(q)$ if
and only if $m$ divides $q^n - 1$ but $m$ does not divide $q^d - 1$ for every positive
integer $d < n$.
\end{Cor}
\begin{proof}
The corollary follows from Lemmas \ref{irrsub} and \ref{semis} and the fact that 
\begin{equation*}(q^n-1,q^d-1)=q^{(n,d)}-1.
 \qedhere 
\end{equation*}  
\end{proof}

\begin{Lem}[{\cite[(2.6)]{sim}}]
A proper subgroup $C$ of a Singer cycle $T \le GL_n(q)$ is primitive if and only if $|C|$ does not divide $r(q^{n/r}-1)$ for every prime divisor $r$ of $n$. 
\end{Lem}
\begin{proof}
If $C$ is reducible, then $|C|$ divides $q^{n/r}-1$ for some prime $r$ by Lemma \ref{irrsub}. So let $C$ be irreducible.

Suppose $C$ has a system of imprimitivity $\{V_1, \ldots, V_k\},$ where $k>1.$ Then $C$ permutes the $V_i$ transitively because of irreducibility. So $k$ divides $|C|$ and $n$, but $k$ is not divisible by $p$. Let $\pi: C \to \Sym(k)$ be the homomorphism arising from the action of $C$ on  
$\{V_1, \ldots, V_k\},$ so $\pi(C)$ is cyclic of order $k.$ If $k$ is not prime, then let $k/r=l>1$ for some prime $r$ and let $\langle h \rangle=\pi(C).$ So $\{V_1, \ldots, V_k\}$ is a disjoint union of  $\langle h^r \rangle$-orbits $\{V_{i_1}, \ldots, V_{i_l}\}$ for $i=1, \ldots, r.$ Define $U_i=V_{i_1} \oplus \ldots \oplus V_{i_l}.$ Now
$$\{U_1, \ldots, U_r\}$$ is a $C$-system of imprimitivity. Therefore, $C$ always has a system of imprimitivity of prime size and we assume $k=r$, so $\dim V_i =n/r$. Consider the stabiliser $D$ of $V_1$ in $C$, so $D=\text{ker }\pi$  consists of all $g \in C$ such that $(V_1)g=V_1.$ Hence $|D|=|C|/|\text{Im }\pi|=|C|/r.$ Since $V$ is a faithful irreducible $\mathbb{F}_q[C]$-module, $V_1$ must be a faithful irreducible $\mathbb{F}_q[D]$-module. Therefore, $|D|$ divides $q^{n/r}-1$ by Lemma \ref{semis} and $|C|$ divides $r(q^{n/r}-1).$

Suppose  that  $r$ is a prime divisor of $n$ such that $|C|$ divides ${r(q^{n/r}-1)}.$ If $r$ does not divide $|C|,$ then $|C|$ divides $q^{n/r}-1,$ so $C$ is reducible. Let $r$ divide $|C|$ and let $D$ be the unique subgroup of index $r$ in $C,$ so $C/D=\{D,Dc, \ldots, Dc^{r-1}\}$ for some  $c \in C.$ Since $|D|$ divides $q^{n/r}-1$, by Corollary \ref{semis1} the dimension of an irreducible $\mathbb{F}_q[D]$-submodule $V_1$ of $V$ is at most $n/r$ and, by the proof of Lemma \ref{irrsub}, it divides $n/r.$
Since $D$ is normal in $C$, $(V_1)g$ is also an irreducible $\mathbb{F}_q[D]$-submodule of $V$ for every $g \in C$. Therefore, either $V_1=(V_1)g$ or $V_1 \cap (V_1)g=0$ because the intersection is also a $\mathbb{F}_q[D]$-submodule and $V_1$ is irreducible. Thus, the sum $V_1 + (V_1)g$ is direct. Notice that for every $g \in C$ there exists $d \in D$ and $i \in \{0, \ldots, r-1\}$ such that  $g = dc^i.$ Hence $(V_i)g$ has the form $(V_1)c^i.$  It is easy to see that $$V_1 \oplus V_1c \oplus \ldots \oplus V_1c^{r-1} $$
is a $C$-invariant subspace. Therefore, since $C$ is irreducible, 
$$V=V_1 \oplus V_1c \oplus \ldots \oplus V_1c^{r-1}.$$
Hence $\dim V_1 = n/r$ and $C$ permutes the set $\{V_1, V_1c, \ldots,V_1c^{r-1}\}$, so it is a system of imprimitivity. 
\end{proof}

\begin{Lem}
If $C$ is an irreducible subgroup of a Singer cycle $T \le GL_n(q)$ then $$C_{GL_n(q)}(C)=T \text{ and } N_{GL_n(q)}(C)=N_{GL_n(q)}(T).$$
\end{Lem}
\begin{proof}
Let $C= \langle \sigma \rangle.$  Since $\mathbb{F}_q[\langle \sigma \rangle]$ is a subalgebra of $T \cup \{0\}$, every non-zero element of $\mathbb{F}_q[\langle \sigma \rangle]$ is invertible, so $\mathbb{F}_q[\langle \sigma \rangle]$ is a subfield of $T \cup \{0\}$. Therefore, $\mathbb{F}_q[\langle \sigma \rangle]^*$ has order $q^m-1$ for some divisor $m$ of $n$. Lemma \ref{irrsub} implies that  $|\sigma|$ does not divide $q^m-1$ for any proper divisor $m$ of $n$, so $m=n.$ Thus, every element of $T$ can be represented as a linear combination of powers of $\sigma.$ Hence every element centralising $\sigma$ must centralise $T$ and every element normalising $\sigma$ must normalise $T.$ The inclusions 
$$C_{GL_n(q)}(C) \ge T \text{ and } N_{GL_n(q)}(C) \ge N_{GL_n(q)}(T)$$
are straightforward.
\end{proof}


\begin{Lem}\label{sin}
Let $T \le GL_n(q)$ be a Singer cycle and let $t_1, t_2 \in T$. Let $N$ be the normaliser of $T$ in $GL_n(q).$
If $t_2 =t_1^g$ for $g \in GL_n(q),$ then there exists $g_1 \in N$ such that $t_2=t_1^{g_1}.$
\end{Lem}
\begin{proof}
Let $\overline{G}$ be the algebraic group $GL_n( \overline{\mathbb{F}}),$ where $\overline{\mathbb{F}}$ is the algebraic closure of the field $\mathbb{F}_q$, and let $$\sigma : \overline{G} \to \overline{G}$$
be the Frobenius map $$ (a_{ij}) \mapsto (a_{ij}^q).$$
 By Lemma \ref{sin1}, $T$ is the set  of $\sigma$-fixed points of some maximal $\sigma$-stable torus $\overline{T} $ of  $\overline{G}.$ 
Since all maximal tori are conjugate in $\overline{G}$, there exists $h \in \overline{G}$ such that $\overline{T}=\overline{D}^h$, where $\overline{D}$ is the maximal torus consisting of diagonal matrices. Since $T$ is cyclic,  \cite[Lemma 1.2 and Proposition 2.1]{buturl} imply that  $T=(\overline{D}_{\sigma w})^{h},$ where $w$ is a permutation matrix representing a cycle of length $n$ and acting on $\overline{D}$ by conjugation. For $\alpha \in \Aut(\overline{G})$ and $\alpha$-invariant subgroup $\overline{H} \le \overline{G}$   we denote the subgroup of $\alpha$-invariant elements by $\overline{H}_{\alpha}.$ Without loss of generality, we can assume that $w$ represents the cycle $(1, 2, \ldots, n).$ Therefore, by \cite[Lemma 1.3]{buturl},
\begin{equation}\label{sopr1}
\overline{D}_{\sigma w}= \{\diag(\lambda, \lambda^q, \ldots, \lambda^{q^{n-1}}) \mid \lambda^{q^n-1} = 1, \lambda \in \overline{\mathbb{F}}\}  
\end{equation}
and
\begin{equation*}
N=(\overline{D}_{\sigma w} \rtimes \langle w \rangle)^h.
\end{equation*}

Since $t_1$ and $t_2$ are conjugate, they have the same eigenvalues, so $t_1^{h^{-1}}$ and $t_2^{h^{-1}}$ are diagonal matrices in $\overline{D}_{\sigma w}$ with the same entries up to permutation. Therefore,
\begin{align*}t_1^{h^{-1}} & = \diag(\beta, \beta^q, \ldots, \beta^{q^{n-1}})\\
t_2^{h^{-1}} &  =  \diag(\delta, \delta^q, \ldots, \delta^{q^{n-1}})
\end{align*}
where $\beta, \delta \in \mathbb{F}_{q^n}^*$ and $\delta = \beta^{q^k}$ for some $1\le k<n.$ Thus, 
$$t_2^{h^{-1}}=\diag( \beta^{q^{k}}, \beta^{q^{k+1}}, \ldots, \beta^{q^{n-1}}, \beta, \beta^{q}, \ldots, \beta^{q^{k-1}}  ); \text{  } \beta^{q^n-1} = 1.$$
So $t_1^{h^{-1}}$ and $t_2^{h^{-1}}$ are conjugate by a power of $w$ which lies in $N^{h^{-1}}$. Therefore, $t_1$ and $t_2$ are conjugate in $N$. 
\end{proof}

Notice that, by \eqref{sopr1}, if $t \in GL_n(q)$ generates a Singer cycle, then $\det(t)$ generates $\mathbb{F}_q^*,$ so $\langle t \rangle SL_n(q)=GL_n(q)$ by Lemma \ref{supSL}.


\section{Fixed point ratios and elements of prime order}
\label{fprsec}

\begin{Def}
If a group $G$ acts on a set $\Omega$, then   $C_{\Omega}(x)$ is the set of points in $\Omega$  fixed
by   $x \in G$. If $G$ and $\Omega$ are finite, then  the {\bf fixed point ratio} \index{fixed point ratio} of $x$, 
denoted  by $\fpr(x)$, is the proportion of points in $\Omega$ fixed by $x$, i.e. $\fpr(x) = |C_{\Omega}(x)|/|\Omega|.$
\end{Def}

For completeness, we include a proof of the following well-known result. 

\begin{Lem}\label{fpr}
If  $G$ acts transitively on a set $\Omega$  and $H$ is a point stabiliser,  then 
$$\fpr(x)= \frac{|x^{G} \cap H|}{|x^G|}$$ for  $x \in G.$
\end{Lem}

\begin{proof}

Let $\{1=g_1, g_2, \ldots , g_k\}$ be a right transversal for  $H$ in $G$. The action is transitive,
so  $\{H, H^{g_2}, \ldots , H^{g_k}\}$ is the set of stabilisers of all points. Observe
$$|C_{\Omega}(x)|=|\{i \in [1, \ldots ,k ] \mid x \in H^{g_i}\}|= \frac{|\{g \mid x^{g^{-1}} \in H\}|}{|H|}=  \frac{|x^G \cap H||C_{G}(x)|}{|H|}.$$
On the other hand, $|\Omega|=|G:H|=\frac{|x^G||C_{G}(x)|}{|H|}.$ Thus,
\begin{equation*}
\fpr(x)=\frac{|C_{\Omega}(x)|}{|\Omega|}=\frac{|x^G \cap H|}{|x^G|}. \qedhere 
\end{equation*}
\end{proof}

 In \cite{fpr, fpr2, fpr3, fpr4} Burness studies   fixed point ratios in classical groups. Recall some  observations from \cite{fpr}.
Let a group $G$ act faithfully   on the set $\Omega$ of right cosets of a subgroup $H$ of $G.$  
Let $Q(G, c)$ be the probability that
a randomly chosen $c$-tuple of points in $\Omega$ is not a base for $G$, so $G$ admits a base of size $c$ if and
only if $Q(G, c) < 1$. Of course, a $c$-tuple is not a  base if and only if it is fixed by  $x \in G$ of prime order, and the probability that a random $c$-tuple is fixed
by $x$ is equal to $\fpr(x)^c$. Let $\mathscr{P}$ be the set of elements of prime order in ${G}$, and let ${x}_1, \ldots, {x}_k$
be representatives for the ${G}$-classes of elements in $\mathscr{P}$. Since  fixed point
ratios are constant on conjugacy classes (see Lemma \ref{fpr}), 
\begin{equation}\label{ver}
Q(G,c) \le \sum_{{x} \in \mathscr{P}}\fpr({x})^c = \sum_{i=1}^{k}|{x_i}^{{G}}|\cdot\fpr({x}_i)^c=:\widehat{Q}(G,c).
\end{equation}

\begin{Lem}[{\cite[Lemma 2.1]{burness}}]
\label{fprAB}
Let $G$ act faithfully and transitively on $\Omega$ and let $H$ be a point stabiliser. If  $x_1,\ldots,x_k$ represent distinct $G$-classes such that $\sum_{i=1}^k |x_i^G \cap  H| \le A$ and $|x_i^G| \ge B$ for all $i \in \{1, \ldots, k\},$ then
$$\sum_{i=1}^m |x_i^G| \cdot \fpr (x_i)^c \le B \cdot (A/B)^c.$$ 
for all $c \in \mathbb{N}.$
\end{Lem}

If there exists $\xi \in \mathbb{R}$ such that $\fpr({x})\le |{x}^{{G}}|^{-\xi}$ for every  ${x} \in \mathscr{P}$,  then  
$$\widehat{Q}(G,c)\le \sum_{i=1}^{k}|{x}_i^{{G}}|^{1-c\xi}.$$

\begin{Def}
Let $\mathscr{C}$ be the set of conjugacy classes of prime order elements in ${G}$.
For $t \in \mathbb{R},$ 
$$\eta_G(t):= \sum_{C \in \mathscr{C}}|C|^{-t}.$$
If $Z(G)=1,$ then there exists $T_G \in \mathbb{R}$ such that $\eta_G(T_G) = 1.$
\end{Def}

\begin{Lem}\label{11}
	If $G$ acts faithfully and transitively on $\Omega$ and  $\fpr({x})\le|{x}^{{G}}|^{-\xi}$ for all  ${x} \in \mathscr{P}$ and $T_G < c\xi - 1$, then $b(G) \le c$.
\end{Lem}
\begin{proof}
We follow the proof of \cite[Proposition 2.1]{burness}. Let ${x}_1, \ldots, {x}_k$ be representatives of the ${G}$-classes
of prime order elements in ${G}$. 
By \eqref{ver}, 
$$Q(G,c) \le  \sum_{i=1}^{k}|{x_i}^{{G}}|\cdot \fpr({x}_i)^c \le \eta_G(c\xi - 1).$$
The result follows since $\eta_G(t) < 1$ for all $t>T_G.$
\end{proof}

We fix the following notation for the rest of the section. Let $\overline{G}$ be an adjoint simple algebraic group of  type $A_{n-1}$ or $C_{n/2}$ over the algebraic closure of ${\mathbb{F}_p}$. Let $\overline{G}_{\sigma}=\{g \in \overline{G} \mid g^{\sigma}=g\}$ where $\sigma$ is a Frobenius morphism of $\overline{G}.$ Let $\overline{G}$ be such that $G_0=O^{p'}(\overline{G}_{\sigma})'$ is a finite simple group. Here  $O^{p'}(G)$ is the subgroup of a finite group $G$ generated by all $p$-elements of $G$. Therefore, $\overline{G}_{\sigma}=PGL_n^{\varepsilon}(q)$ and $G_0=PSL_n^{\varepsilon}(q)$ for type $A_{n-1}$; also $\overline{G}_{\sigma}=PGSp_n(q)$ and $G_0=PSp_n(q)'$ for type $C_{n/2}.$ Let  $G$ be a finite almost simple  group with  socle $G_0.$  

As proved in \cite[Proposition 2.2]{burness},  if  $n \ge 6$, then $T_G$ exists and $T_G<1/3$. 
Thus, if for such $G$  
\begin{equation}\label{0} 
\fpr(x)<|x^G|^{-\frac{4}{3c}}
\end{equation}
 for all $x \in \mathscr{P} $, then $\xi \ge {4}/{(3c)}$ and $c\xi -1 \ge 1/3 >T_G$ and $G$ has a base of size $c$.

Therefore, Lemma \ref{11} allows us to estimate the base size by calculating bounds for $|x^G|$ and $|x^G \cap H|$ for elements $x$ of prime order.

\begin{Def}
\label{H1EE0}
Let $A$ be a group and let $\varphi: A \to A$ be a homomorphism. Let $H^1(\varphi,A)$  denote the set of equivalence classes of $A$ corresponding to the equivalence relation:
$$x \sim y \text{ if and only if } y=z^{-1}xz^{\varphi} \text{ for some } z \in A.$$
Let $x \in \overline{G}_{\sigma}$ and let $E=C_{\overline{G}}(x).$ Notice that $\sigma$ induces a homomorphism $\sigma : E/E^0 \to E/E^0$, where $E^0$ is the connected component of $E$ containing $1.$ Let $H^1(\sigma, E/E^0)$ be the set of equivalence classes of elements of $E/E^0$ corresponding to that induced homomorphism.
\end{Def}


\begin{Def}
\label{nudef}
Let $x \in PGL(V)=PGL_n(q)$. Let $\overline{\mathbb{F}}$ be the algebraic closure of $\mathbb{F}_q$, and let $\overline{V}=\overline{\mathbb{F}} \otimes V.$ Let $\hat{x}$ be the preimage of $x$ in $GL(n,q).$ Define
$$ \nu_{V, \overline{\mathbb{F}}}(x):= \min\{\dim [\overline{V}, \lambda \hat{x}] : \lambda \in \overline{\mathbb{F}}^* \}. $$ Here $[V,g]$ for a vector space $V$ and $g \in GL(V)$ is the commutator in $V \rtimes GL(V)$.  Therefore, $\nu_{V, \overline{\mathbb{F}}}(x)$ is  the minimal codimension of an eigenspace of $\hat{x}$  on $\overline{V}$.  Sometimes we denote this number by $\nu(x)$ and $\nu_{V, \overline{\mathbb{F}}}(\hat{x})$. 
\end{Def}

\begin{Lem}[{\cite[Lemma 3.11]{fpr2}}]\label{prost}  
Let $x \in PGL^{\varepsilon}_n(q)$ have  prime order $r.$ One of the following holds: 
\begin{enumerate}[font=\normalfont]
\item $x$ lifts to  $\hat{x} \in GL^{\varepsilon}_n(q)$ of order $r$ such that $|x^{PGL^{\varepsilon}_n(q)}|=|\hat{x}^{GL^{\varepsilon}_n(q)}|;$ \label{prost1}
\item $r$ divides both $q-\varepsilon$ and $n$, and $x$ is $PGL_n(\overline{\mathbb{F}})$-conjugate to the image of $$\diag[I_{n/r}, \omega I_{n/r}, \ldots , \omega^{r-1} I_{n/r} ],$$ where $\omega \in \overline{\mathbb{F}}$ is a primitive $r$-th root of unity. \label{prost2}
\end{enumerate} 
\end{Lem}

\begin{Rem}
Lemma 3.11 from \cite{fpr2} is formulated for   all classical groups, but only for $r \ne 2$. It is easy to see  from its proof  that the condition $r\ne 2$ is necessary only for orthogonal and symplectic cases; if $x \in PGL^{\varepsilon}_n(q)$ then the statement is true for arbitrary prime $|x|.$  
\end{Rem}

\begin{Lem}\label{xGoGs}
Let $x \in \overline{G}_{\sigma}$ have prime order.
\begin{enumerate}[font=\normalfont]
\item If $x$ is semisimple, then $x^{\overline{G}_{\sigma}}=x^{G_0}.$
\item If $x$ is unipotent and $G_0=PSL_n^{\varepsilon}(q)$, then $|x^{\overline{G}_{\sigma}}| \le \min\{n,p\}|x^{G_0}|.$
\item If $x$ is unipotent, $p\ne 2$ and $G_0=PSp_n(q)$, then $|x^{\overline{G}_{\sigma}}| \le 2|x^{G_0}|.$  
\end{enumerate} 
\end{Lem}
\begin{proof}
See \cite[4.2.2(j)]{gorlyo} for the proof of $(1)$ and \cite[Lemma 3.20]{fpr2} for $(2)$ and $(3).$
\end{proof}
Notice that if $p=2$, $(n,q)\ne (4,2)$ and $\overline{G}_{\sigma}$ is symplectic, then $\overline{G}_{\sigma}=G_0.$

\begin{Lem}
Let $x \in G$ have prime order $r$ and $s:=\nu(x).$
\begin{enumerate}[font=\normalfont]
\item In case ${G_0}=PSL_n^{\varepsilon}(q)$:
\begin{equation}\label{5uni}
 |x^G| > 
\begin{cases}
  \frac{1}{2t} \left(\frac{q}{q+1} \right)^{as/(n-s)} q^{ns} \ge \frac{1}{2t} \left(\frac{q}{q+1} \right)^{as/(n-s)} q^{n^2/2}  & \text{ for } s \ge n/2;\\
  \frac{1}{2t} \left(\frac{q}{q+1} \right)^{a} q^{2s(n-s)} \ge  
\frac{1}{2t} \left(\frac{q}{q+1} \right)^{a}  q^{(3/8)n^2}  & \text{ for }  n/4 \le s < n/2,
\end{cases}
\end{equation}
where $t = \min\{r,n\}$ and $a=(1/2)(1- \varepsilon 1)$.
\item In case $G_0=PSp_n(q)'$:
\begin{equation}\label{5simp}
|x^G| > 
  \frac{1}{8} \left(\frac{q}{q+1} \right) \max(q^{s(n-s)}, q^{(ns/2)}).
\end{equation}
\end{enumerate}
\end{Lem}
\begin{proof}
The statement follows by Lemma \ref{xGoGs} and \cite[Propositions 3.22 and 3.36, Lemmas 3.34 and 3.38]{fpr2}.
\end{proof}


\section{Computations using {\sf GAP} and {\sc Magma}}
\label{gapsec}

We use the  computer algebra systems {\sf GAP} \cite{GAP4} and {\sc Magma} \cite{magma}  to check the inequality $b_S(S \cdot(SL_n(q^{\bf u}) \cap G))\le c$ for  particular cases where $S \le G$ is a maximal solvable subgroup of $G \in \{GL_n(q), GU_n(q), GSp_n(q) \}$ and $q$ and $n$ are small. In this section we discuss how the statement $b_S(G)\le c$ can be checked for given  $S \le G$ and integer $c$, and how to construct desired $S$ and $G$ using {\sf GAP} and {\sc Magma}.

Since a nontrivial transitive permutation group on $\Omega$ has base size $c>0$ if and only if the stabiliser of  $\omega \in \Omega$ acting on $\Omega \backslash \{\omega\}$ has base size $c-1$, it is enough to check that the action of the stabiliser of $S$ in $G/S_G$ on $(\Omega \backslash \{S\})^{c-1}$ has a regular point. Here $\Omega$ is the set of right cosets of $S$ in $G$. The following code in {\sf GAP} checks the existence of such an orbit  and returns {\sf true} if $b_S(G) \le c$: 

\lstset{
basicstyle=\ttfamily}  
{\small
\begin{lstlisting}
 gs:=RightCosets(G,S);;
 hom:=Action(G,gs,OnRight);
 SS:=Stabilizer(hom,1);
 m:=Size(SS);
 k:=Order(G)/Order(S);
 Orb:=OrbitsDomain(SS,Arrangements([2..k],c-1),OnTuples);;
 OrbReg:=Filtered(Orb,x->(Size(x)=m));;
 Size(OrbReg)>0;
\end{lstlisting}}

 This procedure works well when $|G:S|$ is relatively small. For example, if $|G:S|$ is at most $4000$, then this code executes in {\sf GAP} 4.11.1 in at most 345 seconds using the default memory allocation of 256 MB  on a machine with a 2.6 GHz processor. Another way to establish  $b_S(G)\le c$, that we use in most situations, is to find $a_1, \ldots, a_{c-1}$ such that 
\begin{equation}
\label{gapint}
S \cap S^{a_1} \cap \ldots \cap S^{a_{c-1}} =S_G=\cap_{g \in G} S^g.
\end{equation}
 The function {\sf Random}$(G)$, in both {\sf GAP} and {\sc Magma}, allows us to find such $a_i$ in most cases we consider. 

A more difficult  task is to construct desired (or all up to conjugation) maximal solvable subgroups of $G$.  Generating sets for primitive maximal solvable subgroups of $GL_n(q)$ for small $n$ (and for subgroups containing primitive maximal solvable subgroups of $GL_n(q)$ in the general case) can be found in \cite{short} and \cite[\S 21]{sup}. The function {\sf IrreducibleSolvableGroupMS}$(n,p,i)$ in the {\sf GAP} package {\sf PrimGrp} \cite{PrimGrp} realises the results of \cite{short}. It returns a representative of the $i$-th conjugacy class of irreducible solvable subgroups of $GL_n(p)$, where $n>1$, $p$ is a prime, and $p^n < 256$.   While constructing a specific subgroup can be difficult, it is usually not necessary:  for $M>S$ if
\begin{equation}
\label{gapintM}
M \cap M^{a_1} \cap \ldots \cap M^{a_{c-1}} =S_G,
\end{equation}
then \eqref{gapint} holds, so it is enough to construct an overgroup $M$ of $S$ such that \eqref{gapintM} holds.   The function {\sf ClassicalMaximals} in  {\sc Magma} realises the results of \cite{maxlow}. It  constructs all maximal subgroups of a classical group ($GL_n(q)$, $GU_n(q)$ or $GSp_n(q)$ for example) up to conjugation for $n \le 17$. If a maximal subgroup $M$ is too big and \eqref{gapintM} does not hold, then we use the function {\sf MaximalSubgroups} to obtain all maximal subgroups of $M$ up to conjugation and check if \eqref{gapintM} holds for them. In practice, at most three iterations are needed to obtain \eqref{gapintM} for some overgroup of the solvable subgroup under investigation.

As an example, consider $G=GU_7(2)$ and check in  {\sc Magma} that \eqref{gapintM} with $c=3$ holds for all irreducible maximal subgroups:
{\small
\begin{lstlisting}
> G:=GU(7, 2);
> M:=ClassicalMaximals("U",7,2:classes:={2..9}, general:=true);
> #M;
2
> #ClassicalMaximals("U",7,2:classes:={2..9}, general:=true,
>   novelties:=true);
0
> for i in [1..#M] do
>    H := M[i];
>    x := Random(G);
>    K := M[i]^x;
>    repeat
>       y := Random(G);
>       L := M[i]^y;
>       I := H meet K meet L;
>    until #I eq 3;
>    "Now shown intersection is central for i = ", i;
> end for;
Now shown intersection is central for i =  1
Now shown intersection is central for i =  2 
\end{lstlisting}}

Hence the intersection of three conjugates of an irreducible  maximal subgroup has order $3$ which is exactly the order of $Z(GU_7(2)).$ Since $M[i]$ contains $Z(GU_7(2)),$ 
$$M[i] \cap M[i]^x \cap M[i]^y =Z(GU_7(2))$$
for $i=1,2.$

We often write ``$b_S(G)<c$ is verified by computation''\index{computation} to imply that the statement is verified using one of the procedures described above.   

We often carry out calculations in {\sc Magma} of the following kind: given specified building blocks, we construct block-diagonal matrices to define a subgroup $S$ of $GL_n(q)$ where $n \le 4$ and $q$ is small, often 2; we then show  that the intersection of a specific number of conjugates of $S$ satisfies particular order bounds.  We often write ``calculations show'' to summarise such routine calculations.

%% file: ch3.tex
\chapter{Intersection of conjugate irreducible solvable subgroups}
\label{ch2}

 Recall from the introduction that $b_S(G)$ is the minimal number such that there exist $$x_1, \ldots, x_{b_S(G)} \in G  \text{ with } S^{x_1} \cap \ldots \cap S^{x_{b_S(G)}}=S_G$$ where $S_G= \cap _{g \in G}S^g$.  Let ${G}$ be  $GL_n(q),$ $GU_n(q)$ or $GSp_n(q)$ in cases {\bf L}, {\bf U} and {\bf S} respectively and let $S$ be an irreducible maximal solvable subgroup of $G$.  The goal of this chapter is to obtain upper bounds for $b_S(S \cdot (SL_n(q^{\bf u}) \cap {G}))$. These bounds play an important role in the proof of Theorems \ref{theorem}, \ref{theoremGU} and \ref{theoremSp} in Chapter \ref{ch3}.  While, with some exceptions,  $b_S(S \cdot (SL_n(q^{\bf u}) \cap {G})) \le 4$ follows by Theorem \ref{bernclass}, it is not sufficient for our purposes. In this chapter we prove that  $b_S(S \cdot (SL_n(q^{\bf u}) \cap {G})) \le 2$ in case {\bf L} and $b_S(S \cdot (SL_n(q^{\bf u}) \cap {G}))\le 3$ in cases {\bf U} and {\bf S} with a detailed list of exceptions.

\section{Primitive and quasi-primitive subgroups}

    We start our study with a special case:  $S$ is a primitive maximal solvable  subgroup for the case {\bf L} and $S$ is quasi-primitive solvable for cases {\bf U} and {\bf S}. In the next section  we use these results to obtain bounds for $b_S(S \cdot (SL_n(q^{\bf u}) \cap {G}))$ where $S$ is irreducible.  


\begin{Def}
Let $H \le GL(V).$ An irreducible $\mathbb{F}_q[H]$-module $V$ is {\bf quasi-primitive} \index{quasi-primitive} if it is a homogeneous $\mathbb{F}_q[N]$-module for all $N \trianglelefteq H.$ A subgroup $H$ of $GL(V)$ is {\bf quasi-primitive} if $V$ is a {quasi-primitive}
$\mathbb{F}_q[H]$-module.
\end{Def}

We use Lemma \ref{supirr} to extend our results from a primitive subgroup  to an irreducible subgroup in the case {\bf L}. In cases {\bf U} and {\bf S}, Lemma \ref{ashb} does not guarantee that an irreducible subgroup of ${G}$ lies in the wreath product of a primitive subgroup of a general unitary  or symplectic group of smaller degree with a subgroup of a symmetric group. However, it gives us a decomposition of $V$ which allows us to use induction if $S$ is not quasi-primitive.

To prove results about $b_S(SL_n(q^{\bf u}) \cap {G})$, we need  information about  primitive and quasi-primitive solvable groups, upper bounds for $|S|$ and lower bound for $\nu(x)$ (the codimension of a largest eigenspace of $x$, see Definition \ref{nudef}), where $x$ is a prime order element of the image of $S$ in $PGL_n(q^{\bf u}).$  This information is needed to apply the probabilistic method described in Section \ref{fprsec}.
A primitive subgroup of $GL_n(q)$ is quasi-primitive by Clifford's Theorem.  If $S \le GL_n(q)$ is solvable and { quasi-primitive}, then every  normal abelian subgroup of $S$ is cyclic by 
\cite[Lemma 0.5]{manz}. Such groups are  studied in \cite{manz}; we collect the main results in the following lemma.

\begin{Lem}[{\cite[Corollary 1.10]{manz}}] \label{olaf}
Suppose $S \le GL_n(q)$ is nontrivial solvable and every normal abelian subgroup of $S$ is cyclic. Let $F = {\bf F}(S)$ be the Fitting subgroup of $S$ and let $Z$ be the socle of the cyclic
group $Z(F)$. Set $C=C_S(Z)$. Then there exist normal subgroups $E$ and $T$ of $S$ satisfying the following: 
\begin{enumerate}[font=\normalfont]
\item $F=ET$, $Z=E \cap T$ and $T=C_F(E);$
\item $E/Z = E_1/Z \times \ldots \times E_k/Z$ for chief factors $E_i/Z$ of $G$ with 
$E_i \le C_S(E_j)$ for  $i \ne j$;
\item For each $i$, $Z(E_i) = Z$, $|E_i/Z| = p_i^{2k_i}$ for a prime $p_i$ and an integer
$k_i$, and $E_i= O_{p_i'} (Z) \cdot F_i$ for an extra-special group $F_i = O_{p_i} (E_i) \trianglelefteq S$ of order $p_i^{2k_i+1};$
\item There exists $U \le  T$ of index at most $2$ with $U$ cyclic, $U \trianglelefteq S$ and
$C_T(U)=U$;
\item $T = C_S(E)$ and $F = C_C(E/Z)$;
\item If $C_i$ is the centraliser of $E_i/Z$ in C, then $C/C_i$ is isomorphic to a subgroup of $ Sp_{2k_i}(p_i)$.
\end{enumerate}
\end{Lem}

\begin{Rem}\label{eqrem}
In the notation of Lemma \ref{olaf}, let $e$ be a positive integer such that $|E/Z|=e^2,$ so $e= \prod_{i=1}^k p_i^{k_i}.$ Since $E_i$ has the subgroup $F_i$ of order $p_i^{2k_i+1}$ and $|E_i/Z|=p_i^{2k_i},$ for each $p_i$ there must exist an element of order $p_i$ in $Z$ (and in $U$, since $Z \le U$). In other words, each $p_i$ divides $|U|.$ 
\end{Rem}

The following lemma collects properties of primitive maximal solvable subgroups from \cite[\S \S 19 -- 20]{sup} and \cite[\S 2.5]{short}.  For $A \le GL_m(q)$ and $B \le GL_n(q)$, $A \otimes B = \{a \otimes b \mid a \in A, b \in B\}$ where $a \otimes b$ is the Kronecker product defined in Section \ref{secnot}.  

\begin{Lem}\label{suplem}
Let $S \le GL_n(q)$ be a primitive maximal solvable subgroup. Then $S$
admits the unique chain of subgroups 
\begin{equation}\label{ryad}
S \trianglerighteq C \trianglerighteq F \trianglerighteq A
\end{equation}
 where $A$ is the unique maximal abelian normal subgroup,  $C =C_S(A)$, and  $F$ is the full preimage in $C$ of the maximal  abelian normal subgroup of $S/A$ contained in $C/A.$ The following hold: 
\begin{enumerate}[font=\normalfont]
\item[$a)$] $A$ is the multiplicative group of a field extension $K$ of the field of scalar matrices $ \Delta = \{\alpha I_n: \alpha \in \mathbb{F}_q\}$ and $m:=|K : \Delta|$ divides  $n$.
\item[$b)$]  $S/C$ is isomorphic to a subgroup of the Galois group of the extension $K:\Delta,$ so 
$$S=C.\sigma,$$
where $\sigma$ is cyclic of order dividing $m$. 
\item[$c)$] $C \le GL_e(K)$ where $e=n/m$.
\item[$d)$] If
\begin{equation}\label{r}
 e=p_1^{l_1} \cdot \ldots \cdot p_t^{l_t},
\end{equation}
where $t \in \mathbb{N},$ and the $p_i$ are distinct primes, then each $p_i$ divides $|A|=q^m-1.$
\item[$e)$] $|F/A|=e^2$ and $F/A$ is the direct product of elementary abelian groups. 
\item[$f)$] If $Q_i$ is the full preimage in $F$ of the Sylow $p_i$-subgroup of $F/A,$ then
$$Q_i= \langle u_1 \rangle \langle v_1 \rangle \ldots \langle u_{l_i} \rangle \langle v_{l_i} \rangle A,$$
where 
\begin{equation}\label{quv}
[u_j,v_j]=\eta_j, \text{ } \eta_j^{p_i}=1, \text{ } \eta_j \ne 1, \text{ } \eta_j \in A; \text{ } u_j^{p_i},v_j^{p_i} \in A
\end{equation}
and elements from distinct pairs $(u_i,v_i)$ commute.
\item[$g)$] In a suitable $K$-basis of $K^e$, 
$$F=\tilde{Q_1} \otimes \ldots \otimes \tilde{Q_t},$$
where $\tilde{Q_i}$ is an absolutely irreducible subgroup of $GL_{p_i^{l_i}}(K)$ isomorphic to $Q_i.$
\item[$h)$] If $N=N_{GL_e(K)}(F)$ and $N_i=N_{G_i}(\tilde{Q_i}),$ then
 $$N=N_1 \otimes \ldots \otimes N_t.$$
\item[$i)$] $N_i/\tilde{Q_i}$ is isomorphic to a completely reducible subgroup of $Sp_{2l_i}(p_i).$
\end{enumerate}
\end{Lem}

\begin{Rem}\label{pderem}
The structure of a quasi-primitive solvable group is similar to that of a primitive maximal solvable one.  Nevertheless, we need Lemma \ref{suplem} to obtain better estimates for $|S|$ in the case {\bf L} and to deal with  particular cases when $n$ is small. Our notation in Lemma \ref{suplem} is not consistent with that of \cite{sup}, but similar to that of \cite{short}.
\end{Rem}

\begin{Th}[{\cite[Theorem 3.5]{manz}}] \label{lowboundsol}
If $S$ is a completely reducible solvable  subgroup of $GL_n(q)$, then $|S| < q^{9n/4}/2.8.$
\end{Th}

Let $S$ be a primitive maximal solvable  subgroup of $GL_n(q)$. If $A,$ $F$ and $C$ are as in  \eqref{ryad}, then by Lemma \ref{suplem}
\begin{equation*}
\begin{split}
|A|&=q^m-1;\\
|F:A|&=(n/m)^2=e^2;\\
|S:C|&\le m.\\
\end{split}
\end{equation*}
By  Lemma \ref{suplem} i) and Theorem \ref{lowboundsol}
\begin{equation}
\label{CFbound}
|C:F| \le 
\begin{cases}
\prod_{i=1}^t |Sp_{2l_i}(p_i)|;\\
\prod_{i=1}^t ((p_i^{2l_i})^{9/4}/2.8)<e^{9/2},
\end{cases}
\end{equation}
where $e=n/m$ is as in \eqref{r}. Notice that 
$$|Sp_{2l_i}(p_i)|=p_i^{l_i^2} \prod_{j=1}^{l_i}(p_i^{2j}-1)\le p_i^{l_i^2} \prod_{j=1}^{l_i}p_i^{2j}=(p_i^{l_i})^{(2l_i+1)}. $$
Denote $\log_2(e)$ by $l$, so $l_i \le l$ for $i=1, \ldots, t.$ Therefore,
$$|C:F| \le \prod_{i=1}^t (p_i^{l_i})^{(2l+1)}=e^{(2l+1)} $$ 
and
\begin{equation*}
|S|\le (q^m-1)\frac{n^2}{m^2}m \cdot \min \{ e^{(2l+1)}, e^{9/2}\}=(q^m-1)m \cdot \min \{e^{2(l+1)}, e^{13/2}\},
\end{equation*}
so
\begin{equation}\label{H}
|S/Z(GL_n(q))|\le \left( \frac{q^m-1}{q-1} \right) m\cdot \min \{e^{2(l+1)}, e^{13/2}\}.
\end{equation}

\begin{Lem}\label{simpcycl}
Let $C$ be a non-scalar cyclic subgroup of $G \in \{Sp_{n}(q), GU_n(q)\}$ such that $V$ is $\mathbb{F}_{q^{\bf u}}[C]$-homogeneous. Recall that ${\bf u}=2$ if $G=GU_n(q)$ and ${\bf u}=1$ otherwise.  If $W \subseteq V$ is a $\mathbb{F}_{q^{\bf u}}[C]$-irreducible $C$-invariant subspace with $\dim W = m,$ then  $|C|$ divides $(q^{\bf u})^{m/2}+1.$ Moreover, $m$ is even if $G=Sp_{n}(q)$ and $m$ is odd if  $G=GU_n(q)$.
\end{Lem}
\begin{proof}
Since $W$ is $\mathbb{F}_{q^{\bf u}}[C]$-irreducible, it is either non-degenerate or totally isotropic. If $W$ is non-degenerate, then the lemma follows by \cite[Satz 4 and 5]{zing}. 

Let $W$ be totally isotropic. We consider here the proof for $G=Sp_{n}(q),$ the proof for $GU_n(q)$ is  analogous. By \cite[(5.2)]{asch}, we can assume that there exist $W_1, W_2 \subseteq V$ such that $W_1=W$ and $W_2$ is totally isotropic and $\mathbb{F}_{q^{\bf u}}[C]$-irreducible, and $W_1 \oplus W_2$ is non-degenerate. 

Let $g \in Sp(W_1 \oplus W_2)$ be the restriction of a generator of $C$ to $W_1 \oplus W_2.$ Since $V$ is $\mathbb{F}_{q^{\bf u}}[C]$-homogeneous, there exist bases $\beta_i$ of $W_i$ for $i=1,2$ (let $\beta=\beta_1 \cup \beta_2$) such that 
$$g_{\beta}=
\begin{pmatrix}
g_1 & 0\\
0 & g_1
\end{pmatrix} 
\text{ and } 
{({\bf f}|_{_{W_1 \oplus W_2}})}_{\beta}=
\begin{pmatrix}
0 & A \\
-A^{\top} & 0
\end{pmatrix}
$$   for some $g_1, A \in GL_m(q).$ Here $\langle g_1 \rangle \le GL_m(q)$ is irreducible. Since $g$ is an isometry of $W_1 \oplus W_2,$ $$g_{\beta}{({\bf f}|_{_{W_1 \oplus W_2}})}_{\beta}(g_{\beta})^{\top}={({\bf f}|_{_{W_1 \oplus W_2}})}_{\beta},$$ so $g_1 A g_1^{\top}=A$ and $g_1^A=(g_1^{-1})^{\top}.$ Therefore, the set of eigenvalues of $g_1$ is closed under taking inverses. Moreover, the multiplicity of the eigenvalue $\mu$ is equal to the multiplicity of the eigenvalue $\mu^{-1}$ for every $\mu \in \overline{\mathbb{F}_q}^*.$

By \cite[Lemma 1.3]{buturl}, $g_1$ is conjugate in $GL_m(\overline{\mathbb{F}_q})$ to 
$$\diag(\lambda,\lambda^q, \ldots, \lambda^{q^{m-1}}), \text{ where } \lambda^{q^m-1}=1.$$
Clearly, $|\lambda| = |g_1|=|C|.$ Let $r$ be the minimal natural number such that $\lambda^{q^r}=\lambda^{-1}.$ If $r=0$, so $\lambda=\lambda^{-1}=\pm 1,$ then $g_1= \lambda I_m$ and $C$ is a group of scalars, since it is homogeneous. Assume $r>0$, so $\lambda^{q^r+1}=1$ and $(q^r+1)$ divides $(q^m-1).$  
Since $\langle g_1 \rangle$ is irreducible, $|g_1|$ does not divide $q^l-1$ for every proper divisor $l$ of $m.$ Notice that $|g_1|$ divides $(q^{2r}-1),$ so $|g_1|$ divides $(q^{2r}-1, q^m-1)=q^{(2r,m)}-1$ which is  divisible by $q^r+1.$ Therefore, $(2r,m)>r$ and $2r=m.$ Hence $m$ is even and $|C|$ divides $q^{m/2}+1.$
\end{proof}

\begin{Cor}\label{simpcyclcor}
Let $C$ be a non-scalar cyclic subgroup of $GSp_{n}(q)$ such that $V$ is $\mathbb{F}_{q}[C]$-homogeneous.   If $W \subseteq V$ is a $\mathbb{F}_{q}[C]$-irreducible  submodule of dimension $m,$ then $m$ is even and $|C|$ divides $(q^{m/2}+1)(q-1).$
\end{Cor} 
\begin{proof}
Let  $C=\langle c \rangle$ and $\tau(c)=\lambda \in \mathbb{F}_q$ where $\tau$ is as in Definition \ref{taudef}. So
$$(uc,vc)=\lambda (u,v) \text{ for all } u,v \in V.$$ Notice that
$$(uc^{|\lambda|},vc^{|\lambda|})=\lambda^{|\lambda|} (u,v) =(u,v) \text{ for all } u,v \in V.$$
Therefore, $c^{|\lambda|} \in Sp_n(q).$ Let $c_1$ be the restriction of $c$ to $W$, so $c=\diag[c_1, \ldots, c_1]$ in some basis of $V$ since $V$ is $\mathbb{F}_q[C]$-homogeneous.  

We claim that $\langle c_1^{|\lambda|} \rangle$ is an irreducible subgroup of $GL(W).$ Assume the  opposite, so there exists a $\langle c_1^{|\lambda|} \rangle$-invariant subspace of $W$ of dimension $r$ dividing $m$. Hence $r$ is even and $|c_1|/{|\lambda|}$ divides $(q^{r/2}+1)$ by Lemma   \ref{simpcycl} since $c^{|\lambda|} \in Sp_n(q).$ Also $|\lambda|$ divides $(q-1)$ and $(q-1)$ divides $(q^{r/2}-1),$ so $|c_1|=|\lambda|\cdot|c^{|\lambda|}|$ divides $q^r-1.$ By Lemma \ref{irrsub}, $\langle c_1 \rangle$ is a reducible subgroup of $GL(W)$ which is a contradiction. 

Thus, $W$ is $\langle c_1^{|\lambda|} \rangle$-irreducible and $|c^{|\lambda|}|$ divides $(q^{m/2}+1)$ by Lemma \ref{simpcycl}, so $|C|$ divides $(q^{m/2}+1)(q-1).$
\end{proof}

We adopt the notation of Lemma \ref{olaf} in the following statement.

\begin{Lem}\label{uniqorder}
Let $G \in \{GSp_{n}(q), GU_n(q)\}$. Let $S$ be a quasi-primitive solvable  subgroup of  $G$. Let $W$ be an $m$-dimensional  irreducible $U$-submodule of $V$ and let $e$ be a positive integer such that $e^2=|E/Z|$.  The following hold:
\begin{enumerate}[font=\normalfont]
\item $em$  divides $n$;
\item $|S| \le \min\{|U|^2e^{13/2}/2, |U|me^{13/2} \};$
\item if $e=1$, then $n=m$ and $S$ is a subgroup of the normaliser of a Singer cycle of $GL_n(q^{\bf u});$
\item if $m=1$ then $|S| \le |Z(G)|e^{13/2}.$
\end{enumerate}
\end{Lem}
\begin{proof}
Since $EU \trianglelefteq S,$ (1) follows by Clifford's Theorem and \cite[Corollary 2.6]{manz}.

It is easy to see that 
$$|S|=|S/C|\cdot |T| \cdot|C/F| \cdot|F/T|.$$
By the proof of \cite[Corollary 3.7]{manz}, $|S/C| \cdot |T| \le |U|^2,$ and $|C/F| \le e^{9/2}/2$ and $|F/T|=e^2,$ which gives us the first bound of (2). To obtain the second bound, we claim that $|S/C|\le m.$ Indeed, the linear span ${\mathbb{F}_{q^{\bf u}}}[Z]$ is the field extension $K$ of the field of scalar matrices $\Delta=\mathbb{F}_{q^{\bf u}}\cdot I_n$ of degree $m_1 = \dim W_1,$ where $W_1 \le V$ is an irreducible $\mathbb{F}_{q^{\bf u}}[Z]$-module, so $m_1$ divides $n$ since $Z$ is homogeneous, and $m_1 \le m$ since $Z \le U$. Consider the map 
$$f:S \to \mathrm{Gal}(K/\Delta), \text{ } g \mapsto \sigma_g,$$
where $\sigma_g : K \to K,$ $x^{\sigma_g}=x^g$ for $x \in K.$ Since $\ker(f)=C,$
$$S/C \cong \mathrm{Im}(f) \le {\mathrm{Gal}}(K/\Delta), $$
so $|S/C|$ divides $m_1$ and the second bound follows.


If $e=1$, then $F=T$ and $S$ is a subgroup of the normaliser of a Singer cycle of $GL_n(q^{\bf u})$ by \cite[Corollary 2.3]{manz}, so $U$ is self-centralising. By \cite[Lemma 2.2]{manz} $U$ is irreducible, so $m=n$ 
 and (3) follows.

If $m=1,$ then $U \le Z(G)$ since $U$ is homogeneous, so  $|S/C|=1$ which implies (4). 
\end{proof}



\begin{Lem}\label{nuuni}
Let $S$ be a quasi-primitive solvable subgroup of $GL_n(q)$ and let $H$  be the image of $S$ under the natural homomorphism from $GL_n(q)$ to $PGL_n(q).$ If $x \in H$ has prime order, then $\nu(x) \ge n/4.$
\end{Lem}
\begin{proof}
Let $\hat{x}$ be a preimage of $x$ in $GL_n(q),$ so $\hat{x}=\mu g$, where $\mu \in Z(GL_n(q))$ and $g \in S \backslash \{1\}.$ Observe that $\nu(x)=\nu(\hat{x})=\nu(g),$ so it suffices to prove that $\nu(g)\ge n/4$ for all nontrivial $g \in S.$

If $g \in U$, then, since $U$ is abelian and $V$ is $U$-homogeneous, $g$ is conjugate in $GL_n(\overline{\mathbb{F}_q})$ to 
$$\diag(\lambda, \lambda^q, \ldots, \lambda^{q^{m_1-1}}, \ldots, \lambda, \lambda^q, \ldots, \lambda^{q^{m_1-1}}); \text{ } \lambda \in \overline{\mathbb{F}_q}$$
by \cite[Lemma 1.3]{buturl}. Here $m_1$ is the smallest possible integer such that  $\lambda^{q^{m_1}}=\lambda.$ Therefore, $\nu(g)=n-n/m_1\ge n/2.$ Moreover, if $z \in U$ is nontrivial, then $C_V(z)=\{0\}.$

Let $\lambda \in \overline{\mathbb{F}_q}^*.$ If $g \in S \backslash C,$ then 
$$[\lambda g, z]=[g,z] \in Z \backslash \{1\}$$ for some $z \in Z \le U.$ Notice 
$$C_V((\lambda g)^{-1}) \cap C_V(z^{-1} \lambda g z) \subseteq C_V([\lambda g, z])=\{0\}$$
and $\dim(C_V((\lambda g)^{-1}))= \dim(C_V((\lambda g)^{z})),$ so $ \dim (C_V((\lambda g))) \le n/2$ for every $\lambda \in \overline{\mathbb{F}_q}^*$. Hence $\nu(g)\ge n/2.$

If $g \in F \backslash T,$ then, by $(1)$ and $(5)$ of Lemma \ref{olaf}, there exists $h \in E$ such that $$[ \lambda g, h]= [g,h] \in Z \backslash \{1\}$$ and $\nu(g) \ge n/2$ as above.

If $g \in T \backslash U$, then $[\lambda g, u]= [g,u] \in U \backslash \{1\}$ for some $u \in U$ by $(4)$ of Lemma \ref{olaf}, so $\nu(g) \ge n/2$ as above.

If $g \in C \backslash F$, then $[\lambda g, h] \in E \backslash Z \subseteq F \backslash T$ for some $h \in E$ by $(5)$ of Lemma \ref{olaf}. Therefore, 
$$\dim (C_V([\lambda g, h])) \le n/2 \text{ and } \dim (C_V(\lambda g)) \le 3n/4$$
for every $\lambda \in \overline{\mathbb{F}_q}^*$,
so $\nu(g) \ge n/4.$
\end{proof}

The above result  holds for a primitive solvable  subgroup of $GL_n(q).$ However, we prefer to state it in the notation of Lemma \ref{suplem}.

\begin{Lem}\label{6}
 Let $S$  be a primitive maximal  solvable subgroup of $GL_n(q)$ and let $H$  be the image of $S$ under the natural homomorphism from $GL_n(q)$ to $PGL_n(q).$ If $x \in H$ has prime order $k$, so its preimage  $\hat{x}$ lies in $S \backslash Z(GL_n(q))$, then 
the following hold:
\begin{enumerate}[font=\normalfont]
\item if $\hat{x} \in A$, then $\nu(x) \ge n/2; $ \label{6it1}
\item if $\hat{x}  \in F \backslash A$, then $\nu(x) \ge n/2; $ \label{6it2}
\item if $\hat{x}  \in C \backslash F$, then $\nu(x) \ge n/4; $ \label{6it3}
\item if $\hat{x}  \in S \backslash C$, then $\nu(x)=n-n/k \ge n/2. $ \label{6it4}
\end{enumerate}
\end{Lem}
\begin{proof}

The proof of \eqref{6it1} -- \eqref{6it3} is  as in Lemma \ref{nuuni}.

 Let us prove \eqref{6it4}.  If  \eqref{prost2} holds in Lemma \ref{prost}, then $\nu(x)=n-n/k\ge n/2.$ 
Suppose \eqref{prost1} holds in Lemma \ref{prost}, so $\hat{x} \in S$ has order $k.$
 If $\hat{x}  \in S \backslash C$ then, by $b)$ of Lemma \ref{suplem} and \cite[7-2]{conjaut}, $\hat{x}$ is conjugate to a field automorphism $\sigma_1 \le \langle \sigma \rangle$ of $GL_e(K)$ of order $k$.
Such an element acts as a permutation with $(n/k)$ $k$-cycles on a suitable basis of $V$; therefore, 
\begin{equation*}
\nu(x)=n-n/k\ge n/2. \qedhere
\end{equation*} 
\end{proof}



\begin{Th}\label{sch} 
 $\phantom{gg}$
\begin{enumerate}[font=\normalfont]
\item Let $n \ge 6. $ If $S$ is a primitive maximal solvable  subgroup of $GL_n(q)$, then $$b_S(S \cdot SL_n(q)) = 2.$$ 
\item Let $n \ge 2$ and $q$ be such that $(n,q)$ is not any of $(2,2)$, $(2,3)$ or $(3,2)$. If $S$ is a quasi-primitive maximal solvable  subgroup of $GU_n(q)$, then $$b_S(S \cdot SU_n(q)) \le 3.$$ 
\item Let $n \ge 6. $ If $S$ is a quasi-primitive maximal solvable subgroup of $GSp_n(q)$, then $$b_S(S \cdot Sp_n(q)) \le 3.$$ 
\end{enumerate}
\end{Th}
\begin{proof}

Let $\hat{G}$ be $S \cdot SL_n(q),$ $S \cdot SU_n(q)$ and $S \cdot Sp_n(q),$ for  cases $(1),$ $(2)$ and $(3)$ respectively.  Let $G=\hat{G}/Z(\hat{G})\le PGL_n(q^{\bf u})$ and let $H$ be $S/Z(\hat{G})\le G.$ Obviously, 
$$b_S(\hat{G})=b_H(G).$$

If $n \ge 6$ and for all $x \in G$ of prime order   
$$|x^G \cap H| < |x^G|^{(3c-4)/(3c)},$$
then $b_H(G) \le c$ by Lemma \ref{fpr}  and \eqref{0}. Therefore, if $n \ge 6$, then it suffices to show this inequality for $c=2$ in $(1)$ and for $c=3$ in cases $(2)$ and $(3)$. 




Let $s:=\nu(x).$ We use bounds  \eqref{5uni} and \eqref{5simp} for $|x^G|.$
In most cases the bound $|x^G \cap H| \le |H|$ is sufficient. 

Part $(1)$ of the lemma follows from \eqref{0}, Lemma \ref{6} and bounds  \eqref{5uni}, \eqref{H} for all $q$ for $n \ge 16$. The list of cases when these bounds  are not sufficient for $6 <n \le 15$ is finite. By  Lemma \ref{suplem} $d)$,   $(q^m-1)$ must be divisible by $p_i$ for all $p_i,$ $i=1, \ldots, t.$ Using this statement and the bound for $|C:F|$ obtained in \eqref{CFbound} by using the precise orders of the $Sp_{2l_i}(p_i)$, we reduce this list to  cases 1--6 in Table~\ref{tab}. 
  For cases 2--6 the lemma is verified by computation. 
\begin{table}[h]
\centering
\caption{Exceptional cases in proof of Theorem \ref{sch} for $(1)$}
\label{tab}
\begin{tabular}{|l|l|l|l|}
\hline
\textbf{Case} & $n$ & $e$ & $q$   \\ \hline
\textbf{1}  & 6   & 1   & any   \\ \hline
\textbf{2}  & 6   & 2   & 3     \\ \hline
\textbf{3}  & 6   & 3   & 2,4 \\ \hline
\textbf{4}  & 7   & 1   & 2,3,4 \\ \hline
\textbf{5}  & 8   & 1   & 2     \\ \hline
\textbf{6}  & 8   & 8   & 3,5    \\ \hline
\end{tabular}
\end{table} 

 Consider the case $n=6,$ $e=1,$ so $S$ is the normaliser of a Singer cycle and $S \cdot SL_n(q)=GL_n(q).$ First, we find a better 	
estimate for $|x^G \cap {H}|. $ We  represent $\hat{x}$ as $(\lambda, j) \in \mathbb{F}_{q^n}^{*} \rtimes \mathbb{Z}_n,$
where $(\lambda, 0)^{(1, j)}=(\lambda^{q^j},0).$

If $k$ does not divide $n$, then $\hat{x} \in \mathbb{F}_{q^n}^*$ has order $k$ by Lemma \ref{prost}. Let $\hat{x}^g$  be an element of $\hat{x}^{GL_n(q)} \cap S$ for some $g \in GL_n(q).$
Then $\hat{x}^g$ can be written in the form $(\lambda_1, i) \in \mathbb{F}_{q^n}^{*}\rtimes \mathbb{Z}_n$ for suitable $\lambda_1 \in \mathbb{F}_{q^n}^*$ and $i \in \mathbb{Z}_n$, so $(\hat{x}^g)^k=(\lambda_2, ki)$ and $ki \equiv 0 \pmod{n}$. Thus, $i\equiv 0 \pmod{n}$ and $\hat{x}^g \in \mathbb{F}_{q^n}^*$ since $(k,n)=1$. 
By  Lemma \ref{sin},  if $\hat{x}, \hat{x}^g \in \mathbb{F}_{q^n}^*$, then there exists $g_1 \in S$ such
that $\hat{x}^g=\hat{x}^{g_1}.$ Let $g_1 =(\mu,s ) \in \mathbb{F}_{q^n}^{*}\rtimes \mathbb{Z}_n$, so $\hat{x}^{(\mu, s)}= \hat{x}^{(1,s)}$ since $\hat{x} \in \mathbb{F}_{q^n}^{*}.$ Thus, 
\begin{equation}\label{8}
|x^G \cap H| = |\hat{x}^{GL_n(q)} \cap S|\le |\mathbb{Z}_n|=n.
\end{equation}

If $k=3$, then   $(\lambda,j)^3=(\lambda^{q^{2(n-j)} +q^{n-j}+1},3j)=(\delta,0)$, where $\delta \in \mathbb{F}_q,$ so   $j=4,2,0.$ 
 If $j=4$, then $\lambda^{q^4+q^2+1}=\delta$. If 
$j=2$, then $\lambda^{q^8+q^4+1}=\delta$. Notice that $q^8 +q^4 +1= (q^6-1)q^2 + (q^4+q^2+1)$. Therefore, in both cases  $$\lambda^{q^4+q^2+1}=\delta.$$ Let $\theta$ be a generator of $\mathbb{F}_{q^n}^*$. Since  $(q^6-1)=(q^4+q^2+1)(q^2-1)$, we deduce that  $\lambda = \theta^{m(q+1)}$, where $m=1, \ldots, (q^4 +q^2+1)(q-1).$ 
If $j=0$, then 
$x$ is an element of  $\mathbb{F}_{q^n}^*/\mathbb{F}_q^*$, and there are only two such elements of order 3. 
 In total, the number of elements of order 3 in $H$  is at most $2(q^4+q^2+1)(q-1)+2$,
 so 
\begin{equation}\label{9}
|x^G \cap H|\le 2(q^4+q^2+1)(q-1)+2.
\end{equation}

If $k=2$, then $(\lambda, j)^2=(\lambda^{q^{(n-j)}+1}, 2j)=(\delta,0)$, where $\delta \in \mathbb{F}_q,$ so $j=3,0.$  If $j=3$, then $$\lambda^{q^3+1}=\delta.$$ Thus, $\lambda= \theta^{m\cdot(q^2+q+1)}$ where $m=1, \ldots, (q^3+1)(q-1).$
If $j=0$, then 
$x$ is an element of  $\mathbb{F}_{q^n}^*/\mathbb{F}_q^*$, and there is only one such element of order 2.
 In total,  the number of involutions in $H$ is at most  $(q^3+1)(q-1)+1$,
 so 
\begin{equation}\label{10}
|x^G \cap H|\le (q^3+1)(q-1)+1.
\end{equation}
 Bounds \eqref{8} -- \eqref{10} and \eqref{5uni} are sufficient for $n=6.$ 

Part $(2)$ of the lemma follows from \eqref{0}, Lemma \ref{nuuni}, bounds  \eqref{5uni}, Lemma \ref{simpcycl} and $(2)$ of Lemma \ref{uniqorder} for all $q$ and for $n \ge 10$. These bounds do not suffice when $6 \le n \le 9$ and $q=2$; here the lemma is verified by computation.  
For  $n \le 5$  part $(2)$ follows by \cite[Table 2]{burness} with a finite number of exceptions verified by computation. 

Part $(3)$ of the lemma follows from \eqref{0}, Lemma \ref{nuuni}, bounds  \eqref{5uni}, Corollary \ref{simpcyclcor} and $(2)$ of Lemma \ref{uniqorder} for all $q$ and for $n \ge 16$. The list of cases when these bounds are not sufficient for $6 \le n \le 14$ is finite. Using Remark \ref{pderem} we reduce this list to  $6 \le n \le 8$ and $q \in \{2,3,5,7\}$; here the lemma is verified by computation.  
\end{proof}

We use the notation of Lemma $\ref{olaf}$ in the following lemma.

\begin{Lem}
\label{c6small}
 Let $S \le \hat{G} \in \{GL_n(q), GU_n(q), GSp_n(q)\}$ be a quasi-primitive maximal solvable subgroup.  Recall that $q=p^f.$ If $e=n=r^l$ for some integer $l$ and prime $r$, then the following hold:
\begin{enumerate}[font=\normalfont]
\item $T=Z(F)=C_S(E)=Z(\hat{G})$;
\item $S=S_1 \cdot Z(\hat{G})$ where $S_1=S \cap GL_n(p^t)$,   $t$ divides $f$, and $S_1$ lies in the normaliser $M$ in $GL_n(p^t)$ of an absolutely irreducible symplectic-type subgroup of $GL_n(p^t)$. So $M$ is a maximal group of $\hat{G} \cap GL_n(p^t).$  
\end{enumerate} 
\end{Lem}
\begin{proof}
By \cite[Lemma 2.10]{manz}, $(1)$ follows. Let $W \le V$ be an irreducible $\mathbb{F}_{q^{\bf u}}[F]$-submodule. By Theorem \ref{olaf}, $F=O_{r'}(Z) \cdot F_1$ where $F_1$ is extra-special of order $r^{2l+1},$ so $W$ is a faithful irreducible $\mathbb{F}_{q^{\bf u}}[F_1]$-module. Therefore, by \cite[Proposition 4.6.3]{kleidlieb}, $\dim W=r^{l}$, and $F_1$ is an absolutely irreducible subgroup of $\hat{G} \cap GL_n(p^t)$  where $\mathbb{F}_{p^t}$ is the smallest field over which such a representation of $F_1$ can be realised. By  \cite[Theorem 2.4.12]{short}, $S=S_1 \cdot Z(\hat{G})$ where $S_1 \le N_{GL_n(p^t)}(F_1)=M.$ 
\end{proof}

\subsection{ Primitive and quasi-primitive maximal solvable subgroups for $n \le 5$}\label{sec5}

Since Theorem \ref{sch} gives us sufficient results for $n \le 5$ in case {\bf U}, in this section we consider cases {\bf L} and {\bf S} only.

 In  case {\bf L}  we assume that $S$ is a  primitive maximal solvable subgroup of $GL_n(q)$ of  degree $n \le 5.$ The equation  $b_S(S \cdot SL_n(q))=2$ does not always hold for such $n$. However, in view of Lemmas 
\ref{irrtog}, \ref{diag}   and \ref{prtoirr}, for every primitive maximal  solvable subgroup $S$ it suffices  to find $x \in SL_n(q)$  such that
\begin{equation*}
 S \cap S^x \le RT(GL_n(q))
\end{equation*}
 to prove that $b_M(M \cdot SL_n(q)) \le 5$ for every maximal solvable subgroup $M$.  Therefore, if $b_S(S \cdot SL_n(q))>2$, then we decide if there exists $x \in SL_n(q)$ such that  $S \cap S^x$ lies in $D(GL_n(q))$ or $RT(GL_n(q)).$ Recall that $D(GL_n(q))$ and $RT(GL_n(q))$ denote the subgroups of all diagonal and all upper-triangular matrices in $GL_n(q)$ respectively. In particular, we prove in this section that if $q>7$, then such $x$ always exists, so  $$b_M(M \cdot SL_n(q)) \le 5$$
for every maximal solvable subgroup $M$ of $GL_n(q)$, for every $n \ge 2$ and $q>7.$ Also we use this information in the proof of Theorem \ref{theorem} in Section \ref{secproof}.

In case {\bf S} we assume that $S$ is a quasi-primitive solvable subgroup of $GSp_n(q)$ such that $Z(GSp_n(q)) \le S$ and $S$ is not contained in any larger solvable subgroup of $GSp_n(q).$ Our aim is to prove that $b_S(S \cdot Sp_n(q)) \le 3,$ which is sufficient for the proof of Theorem \ref{theoremSp}.

Let us consider $n \in \{3,5\}$ first since 
 we only need to deal with  case {\bf L}.

\medskip

{\bf Degree 3.} If $n=3$, then by \cite[\S 21.3]{sup} either $S$ is an absolutely irreducible subgroup such that $S/Z(GL_3(q))$ is isomorphic to $3^2.Sp_2(3)$ or $S$ is the normaliser of a Singer cycle.  The  first case arises only if 3 divides  $q-1$; here  $b_S(S \cdot SL_3(q))=2$  by \cite[Table 2]{burness} and Lemma \ref{c6small}.  In the second case $b_S(S \cdot SL_3(q))=2$ for $q>2$ by \cite[Table 2]{burness} and computation. 
If $q=2$ and $S$ is the normaliser of the Singer cycle generated by the matrix 
$$ \begin{pmatrix}
0    & 0 & 1  \\
1 &0    & 0  \\
  0&    1    &   1      
\end{pmatrix},$$ then computation shows that 
$$ S\cap S^x = \left\langle
 \begin{pmatrix}
1    & 1 & 1  \\
0 &1    & 0  \\
  1&    0    &   0      
\end{pmatrix} 
\right\rangle $$
has order 3 where $$
x= \begin{pmatrix}
1    & 0 & 1  \\
1 &1    & 1  \\
  0&    0    &   1      
\end{pmatrix}.$$
\medskip

{\bf Degree 5.} If $n=5$, then by \cite[\S 21.3]{sup} either $S$ is the normaliser of a Singer cycle or $S$ is as in Lemma \ref{c6small}. 
 The  second case arises only if 5 divides  $q-1$. In both cases  $b_S(S \cdot SL_5(q))=2$  by \cite[Table 2]{burness}. 
\medskip

{\bf Degree 2.} Notice that $GSp_2(q)=GL_2(q).$ Let $\hat{G}$ be $GL_2(q)$.  If $q \in \{2,3\},$ then $GL_2(q)$ is solvable. If $q>9$, then $b_S(\hat{G}) \le 3$  by \cite[Table 2]{burness}. For $4 \le q \le 9$ the inequality  $b_S(\hat{G}) \le 3$ is verified by computation.

 Let us consider the case {\bf L} more closely now. If $n=2$ and $q>3$, then by \cite[\S 21.3]{sup} either $S$ is the normaliser of a Singer cycle, or $S$  is as in Lemma \ref{c6small}. The  second case arises only if $q$ is odd. For such $S$,  $b_S(S \cdot SL_2(q))=2$ if $q \ne 7$ and $b_S(S \cdot SL_2(7))=3$ by \cite[Table 2]{burness}. 

Suppose  that $S$  is the normaliser of a Singer cycle, so $S \cdot SL_2(q)=GL_2(q).$ First, let $q$ be odd and let $a \in \mathbb{F}_q$ have no square roots in $\mathbb{F}_q$ (there are $(q-1)/2$ such elements in $\mathbb{F}_q$ if $q$ is odd). Consider 
\begin{equation}\label{thesin}
 S_a= \left\{
\alpha \begin{pmatrix}
1    & 0  \\
      0    &   1       
\end{pmatrix} + 
\beta \begin{pmatrix}
0    & 1  \\
a    &  0       
\end{pmatrix}
\right\} \backslash \{0\}.
\end{equation}
 Notice that 
$$\det \begin{pmatrix}
\alpha    & \beta  \\
    a  \beta    & \alpha       
\end{pmatrix} =0$$ if, and only if, $a=(\alpha/ \beta)^2,$ so all matrices in $S_a$ are invertible. 
 Calculations show that $S_a$ is an abelian subgroup of $GL_2(q)$ of order $q^2-1$ and $S_a \cup \{0\}$ under usual matrix addition and multiplication is a field, so $S_a$ is a Singer cycle. Notice that 
$$\varphi = \begin{pmatrix}
-1    & 0  \\
0    &  1       
\end{pmatrix}$$
normalises $S_a.$ Therefore, we can view $S$ as   $N_{GL_2(q)}(S_a)=S_a \rtimes \langle \varphi \rangle$. Moreover, $N_{\GL_2(q)}(S_a)$ is solvable and 
$$N_{\GL_2(q)}(S_a) = S_a \rtimes \left\langle \phi \begin{pmatrix}
a^{-(p-1)/2}    & 0  \\
0    &  1       
\end{pmatrix} \right\rangle$$ 
where $\phi : \lambda v_i \mapsto \lambda^p v_i$ for $\lambda \in \mathbb{F}_q$ and $\{v_1, v_2\}$ is the basis of $V$ with respect to which matrices from $S_a$ have shape \eqref{thesin}. 

 It is easy to see that if there exist $a, b \in \mathbb{F}_q$ such that $a \ne b$
and neither $a$ nor $b$ has square roots in $\mathbb{F}_q$, then $$S_a \cap S_b \le Z(GL_2(q)).$$ 
If, in addition, $a \ne -b$, then calculations show that 
\begin{equation}
\label{2sindiagodd}
 N_{\GL_2(q)}(S_a) \cap N_{\GL_2(q)}(S_b) =\langle \varphi \rangle Z(GL_n(q)) \le D(GL_2(q)).
\end{equation} 
It is possible to find such $a$ and $b$ if $q>5,$ so in this case there exists $x \in SL_2(q)$ such that $$S \cap S^x \le D(GL_2(q)),$$ since all Singer cycles are conjugate in $GL_2(q)$ and $\Det(S_a)=\Det(GL_2(q)).$

Let $q$ be even and let $a \in \mathbb{F}_q $ be such that there are no roots of $x^2+x+a$ in $\mathbb{F}_q$ (there are $q/2$ such elements in $\mathbb{F}_q$). 
 Consider  
\begin{equation} \label{thesineven}
S_a= \left\{
\alpha \begin{pmatrix}
1    & 0  \\
      0    &   1       
\end{pmatrix} + 
\beta \begin{pmatrix}
0    & 1  \\
a    &  1       
\end{pmatrix}
\right\} \backslash \{0\}.
\end{equation} Notice that 
$$\det \begin{pmatrix}
\alpha    & \beta  \\
    a  \beta    & \alpha +\beta      
\end{pmatrix} =0$$ if, and only if, $a=(\alpha/ \beta)^2 +(\alpha/\beta),$ so all matrices in $S_a$ are invertible. 
 Calculations show that $S_a$ is an abelian subgroup of $GL_2(q)$ of order $q^2-1$ and $S_a \cup \{0\}$ under usual matrix addition and multiplication is a field, so $S_a$ is a Singer cycle. Notice that 
$$\varphi = \begin{pmatrix}
1    & 0  \\
1    &  1       
\end{pmatrix}$$
normalises $S_a.$  Therefore, we can view $S$ as   $N_{GL_2(q)}(S_a)=S_a \rtimes \langle \varphi \rangle$. Moreover, $N_{\GL_2(q)}(S_a)$ is solvable and 
$$N_{\GL_2(q)}(S_a)= (S_a) \rtimes \left\langle \phi \begin{pmatrix}
a^{-1}    & 0  \\
1    &  a^1       
\end{pmatrix} \right\rangle$$ 
where $\phi : \lambda v_i \mapsto \lambda^p v_i$ for $\lambda \in \mathbb{F}_q$ and $\{v_1, v_2\}$ is the basis of $V$ with respect  to which matrices from $S_a$ have shape \eqref{thesineven}.

 It is easy to see that if there exist $a, b \in \mathbb{F}_q$ such that $a \ne b$
and neither $x^2+x+a$ nor $x^2+x+b$ has roots in $\mathbb{F}_q$, then $$S_a \cap S_b \le Z(GL_2(q)).$$
Calculations show that  
\begin{equation}
\label{2sindiageven}
N_{\GL_2(q)}(S_a) \cap N_{\GL_2(q)}(S_b)=\langle \varphi \rangle Z(GL_n(q))
\end{equation} consists of  lower triangular matrices. It is possible to find such $a$ and $b$ if $q \ge 4,$ so in this case there exists $x \in SL_2(q)$ such that all matrices in $S \cap S^x$ are  lower triangular,   since all Singer cycles are conjugate in $GL_2(q)$ and $\Det(S_a)=\Det(GL_2(q)).$
\medskip

{\bf Degree 4.} We claim that $b_S(S \cdot SL_4(q))=2$ for all $q$ in case ${\bf L}$ and $b_S(S \cdot Sp_4(q))\le 3$ for all $q$ in case {\bf S}. We provide here the proof for  {\bf S}. The proof for {\bf L} is  analogous. 

\medskip

Let $S$ be a quasi-primitive maximal solvable subgroup  of $GSp_4(q)$ and $\hat{G}=S \cdot Sp_4(q).$ We use notation from Lemmas \ref{olaf} and \ref{uniqorder}. The proof splits into several cases depending on values of $e$, $m,$ and $q$.

\medskip

\underline{\it Case {$e=4.$}} If $e=4$, then $m=1$, $q$ is odd by Remark \ref{eqrem} and $S$ is as in Lemma \ref{c6small}. By \cite[Table 2]{burness}, $b_S(S \cdot Sp_4(q)) \le 3$  for $q>3;$ for $q=3$ the statement  $b_S(S \cdot Sp_4(q)) \le 3$ is established by computation.

\medskip

 Let $G$ and $H$ be the image of $\hat{G}$ and $S$ in $PGSp_4(q)$ under the natural homomorphism respectively. In the remaining cases 
 we claim  that $Q(G,3)<1$ in \eqref{ver}.  
Denote by $k_{s,r}$ the number of conjugacy classes of  $x \in PGSp_4(q)$ of prime order $r$  with $\nu(x)=s$. By \cite[Propositions 3.24 and 3.40]{fpr2}
\begin{equation}\label{40s}
k_{s,p} \le p^{s/2} \text{ and }
k_{s,r}  \le 
\left\{ 
\begin{aligned}
&q^{\xi s}, &\text{ if } s<n/2;\\
&q^{\xi(s+1)}, &\text{ otherwise},
\end{aligned}
\right.
\end{equation}
where $r \ne p$, $\xi=1$ in the case {\bf L} and $\xi=1/2$ in the case {\bf S}. Let $A(s,r)$ and $B(s,r)$ be  lower bounds for $|x^{PGSp_4(q)}|$ and $|x^G|$ respectively. Let $C(s,r)$ be an upper bound for $|x^G \cap H|.$   Notice that $A(s,r)$, $B(s,r)$ and $C(s,r)$ also depend on $n$ and $q$. Therefore, 
\begin{equation}\label{QABC}
Q(G,3) \le \sum_{{x} \in \mathscr{P}}\fpr({x})^3 \le \sum_{r \text{ divides } |H|}
\sum_{s=1}^{n-1} k_{s,r} \cdot A(s,r) \left(\frac{C(s,r)}{B(s,r)}  \right)^3. 
\end{equation}

\medskip

\underline{\it Case {$q$} is even.} In this case $PGSp_n(q)$ is $PSp_n(q)$ by \cite[Proposition 2.4.4]{kleidlieb}, so $B(s,r)=A(s,r).$  If $q$ is even, then $e=1$ by Remark \ref{eqrem}, so $n=m$ and $S$ is a subgroup of the normaliser $N=T \rtimes \langle \varphi \rangle$ of a Singer cycle $T$ of $GL_4(q)$ by Lemma \ref{uniqorder}. Here $T$ and $\varphi$ are as in Lemma \ref{sin1}.

 Therefore, $|S|$ divides $|T|\cdot|\langle \varphi \rangle|$, where $|T|=q^4-1$ and $|\langle \varphi \rangle|=4$, so $\langle \varphi \rangle$ is a $2$-Sylow subgroup of $N$ and it contains only one element of order $2$. 

 If $r \ne 2$, then
\begin{equation} \label{41s}
 C(r,s):= |x^G \cap H|\le 4
\end{equation}
 by  Lemma \ref{sin} and 
\begin{equation}\label{42s}
A(s,r)=B(s,r):=
(1/2) \max(q^{s(n-s)},q^{ns/2})
\end{equation}
by  \cite[Lemma 3.34 and Proposition 3.36]{fpr2}.

  By the proof of Lemma \ref{nuuni}, if $S$ is a subgroup of the normaliser of a Singer cycle of $GL_n(q)$, then $\nu(x) \ge n/2$ for all $x \in H,$ since $C=F=U.$ Therefore, $\nu(x) \in \{2,3\}$.

 If $r=2$ then, as  mentioned above, $x$ is conjugate to the image of $\varphi^2.$ Recall the identification of $\mathbb{F}_{q^n}$ with $\mathbb{F}_q^n$ from Lemma \ref{singex}. Represent $N$ in the form $\mathbb{F}_{q^n}^{*} \rtimes \mathbb{Z}_n$ and consider the subspace $W$ of $\mathbb{F}_q^n$ consisting of vectors fixed by $\varphi^2$. By Lemma \ref{sin1}, $v^{\varphi^2}=v^{(1,2)}=v^{q^2};$ therefore, $W$ consists of  $v\in \mathbb{F}_q^n \cong \mathbb{F}_{q^n}$ (as vector spaces) such that
$v^{q^2-1}=0,$  so $\dim W=2$ and  $\nu(x)=s=2$. 

Since all elements of order $2$ are conjugate in $H$, $$|x^G \cap H|=|x^H|\le|(\varphi^2)^N|= |N|/ |C_N(\varphi^2)|.$$ Let us compute the order of $C_N(\varphi^2).$ Assume that $\lambda \in T$  and $ \lambda \varphi^i$ centralises $\varphi^2$ where $i=1, \ldots, 4$, so
$$(\lambda \varphi^i)^{\varphi^2}= \lambda^{q^2}\varphi^i=\lambda \varphi^i.$$
Therefore, $\lambda^{q^2-1}=1$ and $i \in \{1, \ldots, 4\}$. Thus,   $|C_N(\varphi^2)|=4(q^2-1)$, so $$C(2,2):=|x^H|\le q^2+1.$$

 By \cite[Proposition 3.22]{fpr2}, 
$$A(2,2)>\frac{1}{4}\left(\frac{q}{q+1} \right) \max(q^{s(4-s)},q^{4s/2})=\frac{1}{4}\left(\frac{q}{q+1} \right)(q^{4}).$$ 
The intersection $T \cap Sp_4(q)$ has at most $q^2+1$ elements by Lemma \ref{simpcycl}. So there are fewer than $\log_2(q^2+1)$ distinct odd prime divisors of $|H|$. Thus,
\begin{equation*}
\begin{aligned}
Q(G,3)  & \le \sum_{r \text{ divides } |H|} \left(
\sum_{s=1}^{n-1} k_{s,r} \cdot A(s,r) \left(\frac{C(r,s)}{B(s,r)}  \right)^3 \right)\\  
  & <  \frac{(q^2+1)^3}{(1/4(q/(q+1))q^4)^2}  
+\log_2 \left( q^2+1 \right) \left(\frac{4^3 q^{3/2}}{((1/2) \cdot q^4)^2}  + \frac{4^3q^2}{((1/2) \cdot q^{6})^2}  \right),
\end{aligned}
\end{equation*}
so $Q(G,3)<1$ and $b_S(\hat{G})\le 3$ for  $q > 4.$ If $q=2, 4$ then $b_S(\hat{G})\le 3$ is established by computation. 

\medskip

\underline{\it Case ${q}$ is odd and ${e=1}$.} Now let $q$ be odd and $e=1$, so $m=4$ and $S$ is again a subgroup of the normaliser $N$ of a Singer cycle $T$ of $GL_4(q)$. Since $q$ is odd, there is no element of order $p$ in $H$. Let us compute the number of elements $(\lambda \phi^i) \in S\le N$ such that $(\lambda \varphi^i)^2$ is scalar, so $(\lambda \varphi^i)^2 \in Z(\hat{G})= Z(GSp_n(q)).$ Notice that $|Z(GSp_n(q))|=|Z(GL_n(q))|=q-1.$ Since $$(\lambda \phi^i)^2=\lambda^{q^{4-i}+1}\phi^{2i} \in Z(\hat{G}),$$
 there are two possibilities: $i=4$ and $i=2$. If $i=4$, then $\lambda^2 \in Z(\hat{G})$, so there are $2(q-1)$ such elements in $T$. 
In the second case $\lambda^{q^2+1} \in Z(\hat{G}),$ so there are $(q^2+1)(q-1)$ such elements. Therefore, there are at most $((q^2+1)+2)-1=q^2+2$ elements of order two in $H$.

Thus, if $r=2$, then $C(s,r):=|x^G \cap H|\le (q^2+2)$. Also by \cite[Table 3.8]{fpr2} (and since $\nu(x) \ge n/2$ by the proof of Lemma \ref{nuuni}), $s$ can be only  $2$ and $k_{2,2}=4.$ So
\begin{equation}\label{polup2s}
A(s,r)=B(s,r):=|x^G|> q^{4}/4
\end{equation}
by \cite[Proposition 3.37]{fpr2}. 

For $r\ne 2$ we use \eqref{41s} and \eqref{42s}. Thus
\begin{equation*}
\begin{aligned}
Q(G,3) & \le  \sum_{r \text{ divides } |H|} \left(
\sum_{s=1}^{n-1} k_{s,r} \cdot A(s,r) \left(\frac{C(s,r)}{B(s,r)}  \right)^3 \right)\\
 & <   \frac{4(q^2+2)^3}{((1/4)q^4)^2}  
+\log_2 \left( q^2+1 \right) \left(\frac{4^3 q^{3/2}}{(1/2 \cdot q^4)^2}  + \frac{4^3q^2}{(1/2 \cdot q^{6})^2}  \right),
\end{aligned}
\end{equation*}
so $Q(G,3)<1$ and $b_S(\hat{G})\le 3$ for  $q \ge 9.$ If $q<9$ then $b_S(\hat{G})\le 3$ is established by computation.

\medskip

\underline{\it Case ${q}$ is odd and ${e=2}$.}
 Since $$|S|= |S/C|\cdot|T/U |\cdot|U|\cdot|C/F |\cdot|F/T |,$$ by $(6)$ of Lemma \ref{olaf}, $|S|$ divides
$2 \cdot 2 \cdot (q^2-1) \cdot |Sp_2(2)| \cdot e^2.$
Therefore,  $|H|$ divides $96(q+1).$ Let $x \in H$ have prime order $r$.  Let $Q_1,$ $Q_2$ and $Q_3$ be  $$\sum_{{x} \in \mathscr{P}; r \mid (q+1)}\fpr({x})^3, \text{ } \sum_{{x} \in \mathscr{P}; r=2}\fpr({x})^3 \text{ and } \sum_{{x} \in \mathscr{P}; r=3}\fpr({x})^3$$ respectively, so $Q(G,3)\le Q_1 + Q_2 +Q_3$. We  find upper bounds for $Q_i,$ $i \in \{1,2,3\}.$

If $r \ne 2, 3$ then $r$ divides $q+1$ and, since $r$ does not divide $n$, by Lemma \ref{prost} $x$ has a preimage $\hat{x} \in U$ of order $r$. Since $U$ is a normal cyclic subgroup of order dividing $q+1$, its image in $PGL_n(q)$ contains the unique Sylow cyclic $r$-subgroup of $S.$ Therefore, the number of elements of order $r$ in $H$ is at most $r-1\le q+1-1=q,$ so $C(s,r):=|x^G \cap H|\le q$. There are fewer than $\log_2(q+1)$ prime divisors of $q+1$, so using \eqref{42s} for $A(s,r)$ and $B(s,r)$ and  using \eqref{40s} for $k_{s,r}$ we obtain
\begin{equation*}
\begin{aligned}
Q_1& \le \sum_{r \mid (q+1)} 
\sum_{s=1}^{n-1} k_{s,r} \cdot A(s,r) \left(\frac{C(s,r)}{B(s,r)}  \right)^3\\
&<    
\log_2(q+1) \left(q^{3/2}\left(\frac{q^3}{((1/2) \cdot q^4)^2} \right) + q^2\left(\frac{q^3}{((1/2) \cdot q^{6})^2} \right) \right).
\end{aligned}
\end{equation*}

 Our arguments  to estimate $Q_2$ and $Q_3$ are more complex and require more work. Our analysis splits into two subcases: $S$ is imprimitive and $S$ is primitive. We use results summarised in the following remark to find $C(s,r)$ for $r \in \{2,3\}.$

\medskip

\begin{Rem}\label{remmi}
Let $r\in \{2,3\}$. For  $L \le GL_n(q)$ let $$c_r(L)=|\{g \in L : g^r \in Z(GL_n(q))\}|.$$ Notice that $|x^G \cap H| \le c_r(S)/|Z(Sp_4(q))|$ for $x \in S/Z(Sp_4(q))$ of prime order. Here we compute $c_r(L)$ for some specific groups.  

By \cite[\S21, Theorem 6]{sup} and \cite[Chapter 5]{short}, a primitive maximal solvable subgroup of $GL_2(q)$ is conjugate to either the normaliser of a Singer cycle or to a certain subgroup of order $24(q-1).$ We follow \cite[Chapter 5]{short} and denote the normaliser of a Singer cycle of $GL_2(q)$ by $M_2$ and the primitive maximal solvable subgroup of order $24(q-1)$ by   $M_3$ and $M_4$ for $q \equiv 3 \bmod 4$ and $q \equiv 1 \bmod {4}$ respectively.  Explicit generating sets of $M_3$ and $M_4$ are listed in \cite[\S 5.2]{short}. It is routine to check that $M_3$ and $M_4$ contain $Z(GL_2(q))$, $c_2(M_i)=10(q-1)$ and $c_3(M_i)=9(q-1)$ for both $i=3,4.$ Notice that $c_3(M_2)=(3,q+1) \cdot (q-1),$ since all  $g \in M_2$ such that $g^3 \in Z(GL_2(q))$ lie in the Singer cycle which is a normal subgroup of $M_2.$ Using the same method as for the case $m=4$ (when $S$ lies in the normaliser of a Singer cycle), we obtain $c_2(M_2)=(q+3)(q-1)$.
\end{Rem}

\bigskip

{\bf Subcase 1.} Assume that $S$ is imprimitive, so there exists a system of imprimitivity
$$V=V_1 \oplus \ldots \oplus V_k.$$ Let $k$ be the maximum possible for $S,$ so $k \in \{2,4\}.$

Let $k=2,$ so $N:=\Stab_S(V_1)=\Stab_S(V_2)$ and $N$ is normal in $S$. Hence $V$ is $\mathbb{F}_q[N]$-homogeneous. Notice that, by \cite[\S 15, Lemma 5]{sup}, in some basis of $V$, $S$ must be a subgroup of $S_1 \wr \Sym(2)$ where $S_1=\Stab_S(V_1)|_{_{V_1}}$. So $N$ is not a group of scalar matrices, since in that case $k$ must be $4.$ Since $S$ is irreducible, $V_i$ is either totally isotropic  or non-degenerate  for both $i=1,2.$  If $V_i$ is totally isotropic, then $S$ lies in a maximal group of $GSp_4(q)$ of type $GL_2(q).2$ and $b_S(\hat{G}) \le 3$ by  \cite[Table 2]{burness}. If $V_i$ is non-degenerate, then  either $V_1 \bot V_2$ and  $S$ lies in a larger solvable subgroup $S_1 \wr \Sym(2) \cap GSp_4(q)$ (which is not quasi-primitive) of $Sp_4(q)$, which contradicts the assumption, or $V_2 \ne V_1^{\bot}.$

Assume that $V_2 \ne V_1^{\bot}$ and consider projection operators $\pi_1$ and $\pi_2$ on $V_1$ and $V_1^{\bot}$ respectively with respect to the decomposition $V=V_1 \oplus V_1^{\bot}.$ Notice that $(V_2)\pi_1$ is $\mathbb{F}_q[N]$-irreducible, so $(V_2)\pi_1=V_1$  since $V_2 \ne V_1^{\bot}.$  Therefore, $(V_2)\pi_2=V_1^{\bot}$ since otherwise $V_1 \cap V_2 \ne 0.$ 

Fix a basis $\beta_1=\{f_1, e_1\}$ of $V_1$ as in \eqref{sympbasis} and basis $\beta_2=\{f_1+w_1, e_1 +w_2\}$ of $V_2$, where $w_1, w_2 \in V_1^{\bot}$ such that $(f_1+w_1), (e_1 +w_2) \in V_2.$ Let ${\bf f}_1$ and ${\bf f}_2$ be the restrictions of {\bf f} to $V_1$ and $V_2$ respectively, so ${\bf f}_i$ is a non-degenerate symplectic form since $V_i$ is non-degenerate for $i=1,2.$  Denote   $({\bf f}_i)_{\beta_i}$ by $\Phi_i$ for $i=1,2.$ Let ${\bf f}(w_1,w_2)= \alpha \in \mathbb{F}_q$,  $\delta=1+\alpha$ and $\beta=\beta_1 \cup \beta_2$. Notice that 
$$\Phi_1=
\begin{pmatrix}
0&1 \\
-1&0
\end{pmatrix}; \text{ }
\Phi_2=
\begin{pmatrix}
0&\delta \\
-\delta&0
\end{pmatrix}; \text{ }
{\bf f}_{\beta}=
\begin{pmatrix}
\Phi_1&\Phi_1 \\
\Phi_1&\Phi_2
\end{pmatrix}.
$$
If $g \in N$, then, in basis $\beta$, $g=\diag[g_1, g_2]$ with $g_1 \in S_1 \le GSp_2(q),$ $g_2 \in GL_2(q).$ By \eqref{unimatr}, $g {\bf f}_{\beta}g^{\top}=\tau(g){\bf f}_{\beta},$ so 
$$
\begin{pmatrix}
g_1\Phi_1 g_1^{\top}&g_1\Phi_1 g_2^{\top} \\
g_2\Phi_1 g_1^{\top}&g_2\Phi_2 g_2^{\top}
\end{pmatrix}=\tau(g)
\begin{pmatrix}
\Phi_1&\Phi_1 \\
\Phi_1&\Phi_2
\end{pmatrix}.
$$
In particular, $$\tau(g)\Phi_1=g_1 \Phi_1 g_2^{\top}=g_1 \Phi_1g_1^{\top}(g_1^{\top})^{-1} g_2^{\top}=\tau(g)\Phi_1(g_1^{\top})^{-1} g_2^{\top},$$
so $(g_1^{\top})^{-1} g_2^{\top}=1$ and $g_1=g_2.$


Hence $N$ consists of matrices $\diag[g_1,g_1]$ with $g_1 \in S_1$ and $S_1$ is a primitive subgroup of $GSp_2(q)$ since otherwise $k=4.$ Therefore, $S_1$ is a subgroup of a primitive maximal solvable subgroup $M$ of $GL_2(q),$ so, by \cite[Chapter 5]{short}, either $M$ is $M_2$ or $M$ is  $M_i$ as in Remark \ref{remmi} with $i=3,4$  and has order $24(q-1).$ It is clear that $N$ is a subgroup of index $2$ in $S$, so, if $M=M_2$, then  $|H|$ divides $(q+1) \cdot 2 \cdot 2$. If $M$ is $M_i$ for $i=3,4$, then $|H|$ divides $24 \cdot 2$.

 Recall that $x \in H$ has prime order $r$ and $\hat{x}$ is a preimage of $x$ in $S$. Notice that if $r\ne 2,$ then $\hat{x} \in N.$ If $\hat{x} \in N$, so $\hat{x}=\diag[g_1,g_1]$ where $g_1 \in S_1,$ then $\nu(g_1)=1$ and $\nu(x)=2.$  If $r=2,$ then $\nu(x)=2$ and $k_{2,2}=4$ by \cite[Table 3.8]{fpr2}. 
So, using \eqref{polup2s} and $C(s,r) := |H|$,
$$
Q_2  \le \sum_{{x} \in \mathscr{P}; |x|=2}\fpr({x})^3
<   4\frac{|H|^3}{((1/4)q^4)^2}. 
$$

If $M$ is $M_2$, then either $p\ne 3$ or there is no element of order $3$ in $H$ since all elements of odd prime order lie in the Singer cycle. If $p \ne 3$, then it is clear that there are two elements of order $3$ in $H,$ so, using \eqref{42s},
 $$Q_{3}  \le \sum_{{x} \in \mathscr{P}; |x|=3}\fpr({x})^3  \le 2^3/((1/2)q^4)^2.$$

If $M$ is $M_i$ with $i=3,4,$ then $|x^G \cap H|\le c_3(M_i)/(q-1)=9$ and there are at most two conjugacy classes of elements of order $3$ in $H$, since a Sylow $3$-subgroup of $H$ has order $3.$  If $p=3$, then by \cite[Proposition 3.22]{fpr2} and Lemma \ref{xGoGs} we can take
\begin{equation}\label{ptri}
2B(s,p)=A(s,p):=\frac{1}{2}\left(\frac{q}{q+1} \right) \max(q^{s(n-s)},q^{ns/2}).
\end{equation} If $p\ne 3$, then we use \eqref{polup2s}. Therefore,
  $$Q_{3}  \le \sum_{{x} \in \mathscr{P}; r=3}\fpr({x})^3 \le 2 \cdot 9^3 \cdot 8/((1/2)(q/(q+1))q^4)^2.$$ 
Computations show that $Q(G,3)\le Q_1+Q_2+Q_3 <1$ for $q>9.$ If $q \le 9$ then  $b_S(S \cdot Sp_4(q)) \le 3$ is established by computation. 

If $k=4$, then  $S$ is a group of monomial matrices in some basis of $V.$ Thus, $D=S \cap D(GL_n(q))$ is normal in $S$, so $V$  is $\mathbb{F}_q[D]$-homogeneous. Hence $D \le Z(\hat{G}).$ Therefore, if $a,b \in S$ correspond to the same permutation (every monomial matrix is a product of a diagonal matrix and a permutation matrix which are unique), then $ab^{-1} \in Z(\hat{G}).$ So $H \cong \Sym(4)$ and there are $8$ elements of order $3$ and $9$ elements of order $2$ in $H.$ Hence $$|x^G \cap H| \le 
\begin{cases}
 9 \text{ if } r=2;\\
8 \text{ if } r=3.
\end{cases}
$$   
These bounds and \eqref{polup2s} show that
$$
Q_2 
<   4\frac{9^3}{((1/4)q^4)^2},
$$ since, by \cite[Table 3.8]{fpr2}, $s$ can be only  $2$ (recall that $q$ is odd) and $k_{2,2}=4.$ 

We use \eqref{42s} and \eqref{ptri} when $p\ne 3$ and $p=3$ respectively.  Therefore, if $p \ne 3$, then \eqref{40s}  shows that
\begin{equation*}
\begin{aligned}
Q_{3} \le \left(q^{1/2}\left(\frac{8^3}{((1/2) \cdot q^3)^2} \right) +q^{3/2}\left(\frac{8^3}{((1/2) \cdot q^4)^2} \right) + q^2\left(\frac{8^3}{((1/2) \cdot q^{6})^2} \right) \right).
\end{aligned}
\end{equation*}
If $p = 3$, then

$$\Scale[0.99]{Q_3 \le \left(\left(\frac{(8)^3 \cdot 8}{(1/2 (q/(q+1)) \cdot q^3)^2} \right) +3\left(\frac{(8)^3 \cdot 8}{(1/2 (q/(q+1)) \cdot q^4)^2} \right) + 3^{3/2}\left(\frac{(8)^3 \cdot 8}{(1/2 (q/(q+1)) \cdot q^{6})^2} \right) \right).} 
$$ 

Computations show that $Q(G,3)\le Q_1+ Q_2+Q_3 <1$ for $q>5.$ If $q \le 5$ then  $b_S(Sp_4(q)) \le 3$ is established by computation.

\bigskip

{\bf Subcase 2.} Now let $S$ be primitive, so $S$ lies in a  primitive maximal solvable  subgroup $M$ of $GL_n(q).$ Since $e=2,$ $M=M_6$ (see \cite[\S 8.1]{short} for the definition). By \cite[Proposition 8.2.1]{short}, 
$$M=M_2 \otimes M_i$$
where $M_i$ is defined in Remark \ref{remmi} for $i=2,3,4.$ Recall also the values of $c_r(M_i)$ for $r=2,3$ from Remark \ref{remmi}.  Now, since $M_2 \otimes I_2$ and $I_2 \otimes M_i$ contain $Z(GL_4(q)),$ we deduce $c_2(M)=(q+3) \cdot 10(q-1)$  and $c_3(M)=(3, q+1) \cdot 9(q-1).$ So there are $9(3, q+1)$ elements $g$ of $M/Z(GL_4(q))$ such that $g^3=1$ and, therefore,  $9(3, q+1)-1$ elements of order $3$. Similarly, there are $10(q+3)-1$ elements of order $2$ in $M/Z(GL_4(q)).$ Since $H$ is isomorphic to a subgroup of $M/Z(GL_4(q))$,
$$|x^G \cap H| \le
\begin{cases}
10(q+3)-1 &\text{ if } r=2;\\
 9(3,q+1)-1 &\text{ if } r=3.
\end{cases}
$$ 

These bounds and \eqref{polup2s} show that
$$
Q_2 
\le   4\frac{(10(q+3)-1)^3}{((1/4)q^4)^2},
$$ since, by \cite[Table 3.8]{fpr2}, $s$ can be only  $2$ (recall that $q$ is odd) and $k_{2,2}=4.$ 

We use \eqref{42s} and \eqref{ptri} when $p\ne 3$ and $p=3$ respectively. Therefore, if $p \ne 3$, then \eqref{40s}  shows that
\begin{equation*}
\begin{aligned}
\Scale[0.99]{
Q_{3} \le  \left(q^{1/2}\left(\frac{(9(3,q+1)-1)^3}{(1/4 \cdot q^3)^2} \right) +q^{3/2}\left(\frac{(9(3,q+1)-1)^3}{(1/4 \cdot q^4)^2} \right) + q^2\left(\frac{(9(3,q+1)-1)^3}{(1/4 \cdot q^{6})^2} \right) \right).}
\end{aligned}
\end{equation*}
If $p = 3$, then
$$\Scale[0.99]{Q_3 \le \left(\left(\frac{(8)^3 \cdot 8}{(1/2 (q/(q+1)) \cdot q^3)^2} \right) +3\left(\frac{(8)^3 \cdot 8}{(1/2 (q/(q+1)) \cdot q^4)^2} \right) + 3^{3/2}\left(\frac{(8)^3 \cdot 8}{(1/2 (q/(q+1)) \cdot q^{6})^2} \right) \right).} 
$$  

Computations show that $Q(G,3)\le Q_1+Q_2+Q_3 <1$ for $q>11.$ If $q \le 11$, then  $b_S(Sp_4(q)) \le 3$ is established by computation.

\medskip

\medskip

Notice that $b_S(GL_n(q))\le 3$ in the case {\bf L} for all $n \le 5$ and a maximal solvable subgroup $S$ of $GL_n(q)$ with $(n,q)$ neither $(2,2)$ nor $(2,3)$ by \cite[Table 2]{burness}.

\medskip

 We now summarise the results of this section.
\begin{Th}
\label{sec311lem}
Let ${G}$ be  $GL_n(q),$ $GU_n(q)$ or $GSp_n(q)$ in cases {\bf L}, {\bf U} and {\bf S} respectively. Let $S$ be a primitive maximal solvable subgroup of $G$ in case ${\bf L}$, and let $S$ be a quasi-primitive maximal solvable subgroup of $G$  in cases {\bf U} and {\bf S}. In each case let $(n,q)$ be such that $G$ is not solvable.
\begin{itemize}
\item In case {\bf L}, either $b_S(S \cdot SL_n(q))=2$, or $b_S(S \cdot SL_n(q))=3$ and one of the following holds:
\begin{enumerate}
\item[$(1)$] $n =2$,  $q>3$ is odd, and $S$ is the normaliser of a Singer cycle. If $q>5$, then there exists $x\in GL_n(q)$ such that  $S \cap S^x \le D(GL_n(q));$ 
\item[$(2)$] $n =2$,  $q\ge 4$ is even, and $S$ is the normaliser of a Singer cycle. In this case there exists $x\in GL_n(q)$ such that  $S \cap S^x \le RT(GL_n(q));$
\item[$(3)$] $n=2$, $q=7$, and $S$  is an absolutely irreducible subgroup such that $S/Z(GL_n(q))$ is isomorphic to $2^2.Sp_2(2);$
\item[$(4)$] $n=3$, $q=2$,  and $S$ is the normaliser of a Singer cycle. 
\end{enumerate}  
\item In case {\bf U}, $b_S(S \cdot SU_n(q))\le 3.$
\item In case  {\bf S}, $b_S(S \cdot Sp_n(q)) \le 3.$
\end{itemize}
\end{Th}

\section{Imprimitive irreducible subgroups}

 We commence by obtaining a result about the  groups of monomial matrices in $GL_n(q)$ and $GU_n(q).$ Recall that by default we assume that $GU_n(q)$ is $GU_n(q, I_n)$, the general unitary group with respect to an orthonormal basis of $V$. We combine this result with those of the previous section to obtain 
an upper bound to $b_S(S \cdot (SL_n(q^{\bf u}) \cap G))$ for those maximal solvable subgroups $S$ of $G \in \{GL_n(q), GU_n(q), GSp_n(q)\}$ which are neither primitive nor quasi-primitive.   

\medskip

For $a \in \mathbb{F}_{q^{\bf u}}^*$ let $A(n), B(n,a)$ and $C(n,a)$  be the following $n \times n$ matrices:
\begin{equation}\label{igrek}
\begin{aligned}
A(n) & = \begin{pmatrix}
1      & -1     & 1      & -1     & \ldots & (-1)^{n+1}         \\
0      &   1    & -1     & 1      & \ldots &  (-1)^{n}         \\
0      &    0   & 1      & -1     & \ldots &  (-1)^{n-1}        \\
       &        &        & \ddots &\ddots  &          \\
       &        &        &        & 1      & -1        \\
       &        &        &        & 0      &  1      
\end{pmatrix}, \text{ so } \\
A(n)^{-1} & =
\begin{pmatrix}
1      & 1      & 0      & 0      &        &          \\
0      &   1    & 1      & 0      &        &           \\
0      &    0   & 1      & 1      &        &          \\
       &        &        & \ddots &\ddots  &          \\
       &        &        &        & 1      & 1        \\
       &        &        &        & 0      &  1      
\end{pmatrix};
\end{aligned}
\end{equation}
 
\begin{gather*}
B(n,a) = \begin{pmatrix}
a         & a          & 0          & 0       &0     & \ldots & 0         \\
a^2       & -a^2       & a          & 0       &0     & \ldots & 0         \\
a^3       & -a^3       & -a^2       & a       &0     & \ldots &           \\
          &            &            &\ddots   &      &\ddots  &           \\
a^{(n-2)} & -a^{(n-2)} & -a^{(n-3)} & \ldots  & -a^2 & a      & 0         \\
a^{(n-1)} &-a^{(n-1)}  &-a^{(n-2)}  &\ldots   & -a^3 & -a^2   & a         \\
a^{(n-1)} &-a^{(n-1)}  &-a^{(n-2)}  &\ldots   & -a^3 & -a^2   & -a        
\end{pmatrix};
\end{gather*} 
\begin{gather*}
C(n,a) = \begin{pmatrix}
\sqrt{a^2+1}      & a      & 0          &   0   & \ldots & 0         \\
\sqrt{a^2+1}      & (a+a^{-1})      & a^{-1}          &   0   & \ldots &  0         \\
\sqrt{a^2+1}      & (a+a^{-1})      & (a+a^{-1})          &   a   & \ldots &  0        \\
                  &        & \vdots     &\vdots &\vdots  &           \\
\sqrt{a^2+1}      & (a+a^{-1})      & \ldots     &(a+a^{-1})      & (a+a^{-1})      &  a^{(-1)^{n}}     \\ 
\alpha     & \delta      & \ldots     &\delta & \delta      &  \delta 
\end{pmatrix}.
\end{gather*} 
Here $n\ge 3$ for $B(n,a)$ and $C(n,a).$ 
 Denote by $B'(n,a)$ the matrix $B(n,a)\pi$, where $\pi$ is the permutation matrix for the permutation $(1,n)(2,n-1) \ldots ([n/2], [n/2+3/2]).$ If $n$ is even, then $\alpha=a$ and $\delta=\sqrt{a^2+1}$ in $C(n,a).$ If $n$ is odd, then $\alpha=1$ and $\delta= a^{-1}\sqrt{a^2+1}.$

\begin{Lem} \label{omnom}
Let $M_n(q)$ be the group of all monomial matrices in $GL_n(q)$ and $MU_n(q):=M_n(q^2) \cap GU_n(q)$. 
\begin{enumerate}[font=\normalfont]
\item $M_n(q) \cap M_n(q)^{A(n)}=Z(GL_n(q))$. \label{omnom1}
\item If $q$ is odd and $a \in \mathbb{F}_{q^2}$ satisfies $a^{q+1}=2^{-1},$ then 
$$MU_n(q) \cap MU_n(q)^{B(n,a)} \cap MU_n(q)^{B'(n,a)} \le Z(GU_n(q))$$ for $n \ge 3.$ \label{omnom2}
\item If $q$ is even and $1 \ne a \in \mathbb{F}_q^*$ (so $q>2$), then
$$MU_n(q) \cap MU_n(q)^{C(n,a)}  \le Z(GU_n(q))$$ for $n \ge 3.$ \label{omnom3}
 \item If $q\ge 4$, then $b_{MU_2(q)}(GU_2(q))\le 3.$ \label{omnom4}
 \item If $n=4$, then $b_{MU_n(2)}(GU_n(2))=4.$ If $n>4$, then $b_{MU_n(2)}(GU_n(2))\le 3.$ \label{omnom5}
\end{enumerate}
\end{Lem}
\begin{proof}
\eqref{omnom1} Consider $g \in M_n(q) \cap M_n(q)^{A(n)},$ so $g=(\diag(d_1, \ldots, d_n)s)^{A(n)}$, where $s \in \Sym(n)$ and $d_i \in \mathbb{F}_q^{*}.$  If $s$ does not fix the point $n$, so $(n)s=l$ for $l<n$, then the last row of the matrix $g$ is equal to the $l$-th row of the matrix $A(n)$ multiplied by $d_l$, which contains more than one non-zero entry. Therefore, $s$ stabilises $n$.  Assume that $s$ stabilises the last ${n-j}$ points and $(j)s=i \le j.$
The  $j$-th row of $g$ is 
\begin{equation*}
\begin{split}
\bigl( \overbrace{0, \ldots, 0}^{i-1},  \overbrace{d_i,-d_i,  \ldots,(-1)^{j-i}d_i}^{j+1-i} , \ldots \bigr).
\end{split}
\end{equation*}
Therefore, $s$ must stabilise $j$ and, by induction, $s$ is trivial.  It is easy to check that the $j$-th row of  $g$ is 
\begin{equation*}
\begin{split}
\bigl( \overbrace{0, \ldots, 0}^{j-1},  d_{j},   
 ( d_{j+1} -d_{j})   ,  \ldots, (-1)^{k-j+1}((d_{j+1} - d_{j})) \bigr)
\end{split}
\end{equation*}
for all $1\le j < n$. Therefore,  $d_{i+1}=d_{i}$ for all $1\le i<n$, since $g \in M_n(q).$ So $g \in Z(GL(n,q)).$

\eqref{omnom2} Since $q$ is odd, there always exists $a \in \mathbb{F}_{q^2}$ such that $a^{q+1}=2^{-1}$. Indeed, let $\eta$ be a generator of $\mathbb{F}_{q^2}^*$ and $\theta=\eta^{q+1},$ so $\theta$ is a generator of $\mathbb{F}_q^*.$ Thus, $\theta^k=2^{-1}$ for some integer  $k$. Therefore, $(\eta^k)^{q+1}=2^{-1}.$
It is routine to check that  $B(n,a)$ and $B'(n,a)$ lie in $GU_n(q)$ for such $a.$ 



Consider $g \in MU_n(q) \cap MU_n(q)^{B(n,a)}$,  so 
\begin{equation*}
g=\diag(g_1, \ldots, g_n)r,
\end{equation*}
where $r \in \Sym(n)$ and $g_i \in \mathbb{F}_{q^2}^{*}$. 

Let $\beta=\{v_1, \ldots, v_n\}$ be the orthonormal basis of $V$ such that $GU_n(q)=GU_n(q, {\bf f}_{\beta})$. Since $g$ is monomial, it stabilises the decomposition
$$\langle v_1\rangle \oplus \ldots \oplus \langle v_n \rangle.$$ Since $g \in GU_n(q)^{B(n,a)},$ it stabilises the decomposition 
$$\langle (v_1)B(n,a)\rangle \oplus \ldots \oplus \langle (v_n) B(n,a) \rangle.$$ We write $w_i$ for $(v_i)B(n,a).$ Notice that 
\begin{equation*}
w_i=
\begin{cases}
av_1 +a v_2 & \text{ if } i=1; \\
a^{i}v_1  -a^{i}v_2  -a^{i-1}v_3  \ldots   -a^2v_i + av_{i+1} & \text{ if } 1<i<n; \\
a^{n-1}v_1  -a^{n-1}v_2  -a^{n-2}v_3  \ldots   -a^2v_{n-1} - av_{n} & \text{ if } i=n.\\          
\end{cases}
\end{equation*}
Since $g$ is monomial, $w_i$ and $(w_i)g$ have the same number of non-zero entries (which is $i+1$ for $i \ne n$ and $n$ for $i=n$) in the decomposition with respect to $\beta.$ Therefore, $(w_i)g \in \langle w_i \rangle$ for $i <n-1,$ so $r$ must fix $\{1,2\}$ and points $3, \ldots, n.$  
Thus, $(w_{n-1})g$ is either $ \delta w_{n-1}$ or $\delta w_{n}$ for some $\delta \in \mathbb{F}_{q^2}.$ If $r$ fixes the point 1, then  $\delta=g_1=g_2= \ldots =g_{n-1}= \pm g_n;$ if $(1)r=2$, then 
$\delta=-g_1=-g_2=g_3= \ldots =g_{n-1}= \pm g_n.$  It is easy to see that $MU_n(q) \cap MU_n(q)^{B(n,q)}$  lies in 
\begin{equation*}
\begin{split}
\{ \diag((-1)^i \alpha,(-1)^i \alpha,\alpha, \ldots,\alpha, \pm \alpha) \cdot (1,2)^i \mid \alpha \in \mathbb{F}_{q^2};  i \in \{0,1\}\}.
\end{split}
\end{equation*}
Therefore, $MU_n(q) \cap MU_n(q)^{B(n,q)\pi}  =  (MU_n(q) \cap MU_n(q)^{B(n,q)})^{\pi}$ lies in 
\begin{equation*}
\begin{aligned}
 & \,   \{ \diag((-1)^i \alpha,(-1)^i \alpha,\alpha, \ldots,\alpha, \pm \alpha) \cdot (1,2)^i\mid \alpha \in \mathbb{F}_{q^2};  i \in \{0,1\}\}^{\pi} \\
 \subseteq & \, \{ \diag(\pm \alpha,\alpha, \ldots,\alpha,  (-1)^i \alpha,(-1)^i \alpha) \cdot (n,n-1)^i\mid \alpha \in \mathbb{F}_{q^2};  i \in \{0,1\}\}
\end{aligned}
\end{equation*} 
and $$MU_n(q) \cap MU_n(q)^{B(n,q)} \cap MU_n(q)^{B(n,q)\pi}\le Z(GU_n(q)).$$

\eqref{omnom3} Let $1 \ne a \in \mathbb{F}_q$. Since $\phi : \mathbb{F}_q \to \mathbb{F}_q$ mapping $x$ to $x^2$ is a Frobenius automorphism of $\mathbb{F}_q$, every element of $\mathbb{F}_q$ has a unique square root in $\mathbb{F}_q$. Therefore, the matrix $C(n,a)$ exists and lies in $GU_n(q).$ 

Suppose that $n\ge 3$ is odd and consider $g \in MU_n(q) \cap MU_n(q)^{C(n,a)}$, so 
\begin{equation*}
g=\diag(g_1, \ldots, g_n)r.
\end{equation*}
Let $\beta=\{v_1, \ldots, v_n\}$ be the orthonormal basis as in \eqref{omnom2}. Since $g$ is monomial, it stabilises the decomposition
$$\langle v_1\rangle \oplus \ldots \oplus \langle v_n \rangle.$$ Since $g \in MU_n(q)^{C(n,a)},$ it stabilises the decomposition 
$$\langle (v_1)C(n,a)\rangle \oplus \ldots \oplus \langle (v_n) C(n,a) \rangle.$$ We write $w_i$ for $(v_i)C(n,a).$ Notice that 
\begin{equation*}
\begin{aligned}
w_1 & =\sqrt{a^2+1} v_1 +a v_2;\\
w_n & =\alpha v_1 + \delta v_2+    \ldots  + \delta v_{n-1} + \delta v_{n},
\end{aligned}
\end{equation*}
and if $1<i<n$, then 
\begin{equation*}
w_i=
\begin{cases}
\sqrt{a^2+1} v_1 + (a+a^{-1})v_2 +   \ldots+   (a+a^{-1})v_i + av_{i+1} & \text{ if $i$ is odd;}  \\ 
\sqrt{a^2+1} v_1 + (a+a^{-1})v_2+    \ldots +  (a+a^{-1})v_i + -av_{i+1} & \text{ if $i$ is even.} \\
       
\end{cases}
\end{equation*}
Since $g$ is monomial, $w_i$ and $(w_i)g$ have the same number of non-zero entries (which is $i+1$ for $i \ne n$ and $n$ for $i=n$) in the decomposition with respect to $\beta.$ Therefore, $(w_i)g \in \langle w_i \rangle$ for $i <n-1,$ so $r$ must fix $\{1,2\}$ and points $3, \ldots, n.$
Assume that $(1)r=2,$ so $$(w_1)g= g_1 \sqrt{a^2+1}v_2 + g_2av_1.$$ Since $(w_1)g \in \langle w_1 \rangle,$  
$$g_1 \sqrt{a^2+1}v_2 + g_2av_1= \gamma (\sqrt{a^2+1} v_1 +a v_2)$$
for some $\gamma \in \mathbb{F}_{q^2}.$ Calculations show that $g_2(g_1)^{-1}=1+a^{-2.}$ Notice that $g_1^{q+1}=g_2^{q+1}=1$, since $g \in MU_n(q).$ Hence $(g_2(g_1)^{-1})^{q+1}$ must be 1. However,
$$(1+a^{-2})^{q+1}=(1+a^{-2})^2=1+a^{-4} \ne 1.$$
So $r$ must fix the points 1 and 2.

Since $(w_{n-2})g \in \langle w_{n-2} \rangle$, we obtain $g_1= \ldots = g_{n-1}.$ Assume that $(w_{n-1})g= \gamma w_n$ for some $\gamma \in \mathbb{F}_{q^2}.$ Then $g_1= \gamma \sqrt{a^2+1} $ and $g_n= \gamma(\sqrt{a^2+1})^{-1}.$ Since $(g_i)^{q+1}=1$ for all $i=1, \ldots, n,$ 
$$\gamma^{q+1}(a^2+1)=\gamma^{q+1}(a^2+1)^{-1}=1.$$
Therefore, $a^2+1$ must be equal to $(a^2+1)^{-1},$ which is not true since $$(a^2+1)^2=a^4+1 \ne 1.$$
Thus, $(w_{n-1})g= \gamma w_{n-1}$ and $g$ is a scalar.

The proof of \eqref{omnom3} for even $n$ is  analogous to that for  odd $n$. 

\medskip

\eqref{omnom4} For $q=4,5$ the statement is verified by computation. For $q>5$ the statement follows from \cite[Table 2]{burness}.

\medskip

\eqref{omnom5}  For $n < 7 $  the statement is verified by computation. Assume $n\ge 7.$ Let $a$ be a generator of $\mathbb{F}_4^*,$ so $a^2=a+1$ and $a^3=1.$ Let $\beta=\{v_1, \ldots, v_n\}$ be the orthonormal basis of $V$ such that $GU_n(q)=GU_n(q, {\bf f}_{\beta})$. Let $E(n,a) \in GU_n(2)$ be defined as follows. If $n$ is even, then
\begin{equation*}
(v_i)E(n,a)=w_i=
\begin{cases}
v_1 & \text{ if } i=1; \\
\sum_{j=2}^n v_j & \text{ if } i=2; \\
(a+1)v_i +a v_{i+1} + \sum_{j=i+2}^n v_j & \text{ if } i \text{ is odd and }3\le i\le n; \\
 a v_i +(a+1) v_{i+1} + \sum_{j=i+2}^n v_j & \text{ if } i \text{ is even and }3\le i\le n.
\end{cases}
\end{equation*}
If $n$ is odd, then 
\begin{equation*}
(v_i)E(n,a)=w_i=
\begin{cases}
\sum_{j=1}^n v_j & \text{ if } i=1; \\
(a+1)v_i +a v_{i+1} + \sum_{j=i+2}^n v_j & \text{ if } i \text{ is even and } 2\le i\le n; \\
 a v_i +(a+1) v_{i+1} + \sum_{j=i+2}^n v_j & \text{ if } i \text{ is odd and } 2\le i\le n.
\end{cases}
\end{equation*}
For example,  $E(8,a)$ is
$$
\begin{pmatrix}
1 & 0 & 0 & 0 & 0 & 0 & 0 & 0\\
0 & 1 & 1 & 1 & 1 & 1 & 1 & 1\\
0 & a+1 & a & 1 & 1 & 1 & 1 & 1\\
0 & a & a+1 & 1 & 1 & 1 & 1 & 1\\
0 & 0 & 0 & a+1 & a & 1 & 1 & 1\\
0 & 0 & 0 & a & a+1 & 1 & 1 & 1\\
0 & 0 & 0 & 0 & 0 & a+1 & a & 1\\
0 & 0 & 0 & 0 & 0 & a & a+1 & 1\\
\end{pmatrix}.$$ We obtain $E(7,a)$ by deleting the first row and the first column in $E(8,a).$ It is routine to verify that $E(n,a) \in GU_n(2).$

 Let $\pi \in GU_n(2)$ be the permutation matrix corresponding to the permutation
\begin{align*}
(1,n)(3,n-1)(5,n-2)(6, 7, \ldots, n-3) & \text{ if } n \text{ is even;}\\
(1,n)(3,n-1)(5,6, 7, \ldots, n-3) & \text{ if } n \text{ is odd.}
\end{align*}
We claim that $MU_n(2) \cap MU_n(2)^{E(n,a)} \cap MU_n(2)^{E(n,a)^{\pi}}\le Z(GU_n(2)).$ We prove this for even $n$; the proof is analogous for odd $n$.

\medskip

Consider $g \in MU_n(2) \cap MU_n(2)^{E(n,a)}$,  so 
\begin{equation*}
g=\diag(g_1, \ldots, g_n)r,
\end{equation*}
where $r \in \Sym(n)$ and $g_i \in \mathbb{F}_{4}^{*}$.  Since $g$ is monomial, it stabilises the decomposition
$$\langle v_1\rangle \oplus \ldots \oplus \langle v_n \rangle.$$ Since $g \in MU_n(2)^{E(n,a)},$ it stabilises the decomposition 
$$\langle w_1\rangle \oplus \ldots \oplus \langle w_n \rangle.$$ 
Since $g$ is monomial, $w_i$ and $(w_i)g$ have the same number of non-zero entries in the decomposition with respect to $\beta.$ Therefore, $(w_i)g \in \langle w_{n-2}, w_{n-1}, w_{n} \rangle$ for $n-2 \le i \le n,$ so $r$ must fix $\{n-2,n-1,n\}.$ If $i \in \{n-4,n-3\},$ then $(w_i)g \in \langle w_{n-4}, w_{n-3} \rangle$, so $r$ must fix $\{n-4,\ldots,n\}$ and, therefore, $\{n-4, n-3\}$. Continuing this process, we obtain that $r$ fixes 
\begin{equation}
\label{monq2new}
\{1\},\{2,3\},\{4,5\}, \ldots, \{n-4, n-3\}, \{n-2, n-1, n\}.
\end{equation}
Now assume $g \in MU_n(2) \cap MU_n(2)^{E(n,a)^{\pi}}$. The above arguments show that $r$ must fix
$$\{(1)\pi\},\{(2)\pi,(3)\pi\},\{(4)\pi,(5)\pi\}, \ldots, \{(n-4)\pi, (n-3)\pi\}, \{(n-2)\pi, (n-1)\pi, (n)\pi\}$$ which are
$$\{n\},\{2,n-1\},\{4,n-2\}, \{7,8\}, \ldots, \{n-3, 6\}, \{1, 3, 4\}.$$ Combining this with \eqref{monq2new} we obtain that $r$ is a trivial permutation, so $g$ is diagonal.

Observe that $(w_2)g = (0, g_2, g_3, \ldots, g_n)$ with respect to $\beta$. Since  $ (w_2)g \in \langle w_2, w_3, w_4 \rangle$, 
$$g_4 = g_5 = \ldots =g_n.$$ So $g=\diag(g_1, g_2, g_3, \lambda, \ldots, \lambda)$ for some $\lambda \in \mathbb{F}_4^*.$  Let $u_i = (v_i)E(n,a)^{\pi}.$ Therefore, 
\begin{align*}
u_2 & =(1, \ldots, 1, 0)\\
u_{n-1} & = (1, a+1, 1, \ldots, 1, a, 0)\\
u_{4} & = (1, a, 1, \ldots, 1, a+1, 0)
\end{align*}
are the only vectors in $\{u_1, \ldots, u_n\}$ that have $n-1$ non-zero entries in the decomposition with respect to $\beta.$ Hence  $(u_2)g$ lies in $\langle u_2\rangle$, $\langle u_{n-1}\rangle$, or $\langle u_4 \rangle$ since $g \in MU_n(2)^{E(n,a)^{\pi}}$ and stabilises the decomposition
$$\langle u_1\rangle \oplus \ldots \oplus \langle u_n \rangle.$$ Notice that  $(u_2)g= (g_1, g_2, g_3, \lambda, \ldots, \lambda,0)$  Hence $(u_2)g \in \langle u_2 \rangle$ and $g_1=g_2=g_3=\lambda.$ So $g = \lambda I_n \in  Z(GU_n(q))$. Therefore, 
\begin{equation*}
MU_n(2) \cap MU_n(2)^{E(n,a)} \cap MU_n(2)^{E(n,a)^{\pi}}\le Z(GU_n(2)). \qedhere
\end{equation*}
\end{proof}

\begin{Rem} \label{omnomSL}
In Lemma \ref{omnom} each statement of \eqref{omnom1}-\eqref{omnom5} can be written as $$S \cap S^{x_1} \cap \ldots \cap S^{x_t}\le Z(\widehat{G})$$ for $x_i \in \widehat{G}$ with suitable $S$ and $\widehat{G}=GL_n(q)$ for \eqref{omnom1} and $\widehat{G}=GU_n(q)$ for \eqref{omnom2}--\eqref{omnom5}. In each case we can assume $x_i \in SL_n(q^{\bf u}) \cap \widehat{G}.$ Indeed, if $\det(x_i)\ne 1$, then $a_i=\diag(\det(x_i)^{-1},1\ldots, 1) \in S$ since $S$ is the group of all monomial matrices in $\hat{G},$ so $S^{a_ix_i}=S^{x_i}.$
\end{Rem}

\begin{Lem}\label{igrekl}
Let $H \le X \wr Y,$ where $X \le GL_m(q),$ $Y \le \Sym(k).$  
Let $A(k)=(y_{ij})$ be as in \eqref{igrek} and let $x_i$ for $i =1, \ldots, k$ be arbitrary elements of $X$.
Define $x \in GL_{mk}(q)$ to be
\begin{gather*}
\begin{pmatrix}
y_{11}x_1      & y_{12}x_1      & \ldots     & y_{1k}x_1            \\
y_{21}x_2      &  y_{22}x_2     & \ldots     & y_{2k}x_2             \\
\vdots      &             &            & \vdots               \\
y_{k 1}x_{k}   & y_{k 2}x_{k}   &  \ldots    & y_{k k}x_{k}               
\end{pmatrix},
\end{gather*} 
 Let $h=\diag[D_1, \ldots, D_k] \cdot s  \in H$, where $D_i \in X$ and $s \in Y,$ so $h$ is obtained from the permutation matrix $s$ by replacing $1$ in the $j$-th line by the $(m \times m)$ matrix $D_j$ for $j =1, \ldots, k$ and replacing each zero by an $(m\times m)$ zero matrix. 
If $h^x \in H$, then $s$ is trivial and $D_j^{x_j}=D_{j+1}^{x_{j+1}}$ for $j=1, \ldots, k-1.$ 
\end{Lem}
\begin{proof}
   If $s$ does not stabilise the point $k$, then there is more than one non-zero $(m \times m)$-block in the last $(m \times m)$-row of $h^{x}$ and, thus, $h^{x}$ does not lie in   $ X \wr Y.$
Assume that $s$ stabilises the last ${k-j}$ points and $(j)s=i \le j.$ The  $j$-th $(m \times m)$-row of $h^{x}$ is 
\begin{equation*}	
\begin{split}
\bigl( \overbrace{0, \ldots, 0}^{i-1},  \overbrace{{x_j}^{-1}(D_{j}){x_i},-{x_j}^{-1}(D_{j}){x_i},  \ldots,(-1)^{j-i} {x_j}^{-1}(D_{j}){x_i}}^{j+1-i}, \ldots \bigr).
\end{split}
\end{equation*}
Therefore, $h^x$ does not lie in $X \wr Y$ if $i \ne j,$ since the $j$-th  $(m \times m)$-row contains more than one non-zero $(m \times m)$-block in that case. So $i=j$ and the  $j$-th $(m \times m)$-row of $h^{x}$ is
\begin{equation*}
\begin{split}
\bigl( \overbrace{0, \ldots, 0}^{i-1},  (D_{j})^{x_j},   
( ( D_{j+1})^{x_{j+1}} -(D_{j})^{x_j}  ) ,  \ldots, (-1)^{k-j+1}((D_{j+1})^{x_{j+1}} - (D_{j})^{x_j}  ) \bigr).
\end{split}
\end{equation*}
 So, if $h^x \in H$, then $s$ stabilises $j$ and $D_{j}^{{x_j}}=D_{j+1}^{x_{j+1}}.$
\end{proof}


\begin{Lem}\label{prtoirr}
Let $T(GL_n(q))$ be one of the following subgroups: $Z(GL_n(q))$, $D(GL_n(q))$ or $RT(GL_n(q))$. Let $H$ be a subgroup of $GL_n(q)$ such that
$$H \le H_1 \wr \Sym(k), $$
where $n=mk$, $H_1 \le GL_m(q)$. If there exist $g_1, \ldots, g_b \in GL_m(q)$ ({respectively } $SL_m(q))$ such that 
$$H_1 \cap H_1^{g_1} \cap \ldots \cap H_1^{g_b} \le T(GL_m(q)),$$
then there exist   $x_1, \ldots, x_b \in GL_n(q)$ ({respectively} $SL_n(q))$ such that 
\begin{equation}\label{irreq}
H \cap H^{x_1} \cap \ldots \cap H^{x_b} \le T(GL_n(q)).
\end{equation}
\end{Lem}
\begin{proof}
Define $x_i$ to be 
\begin{gather*}
\begin{pmatrix}
y_{11}g_{i}      & y_{1 2}g_{i }      & \ldots     & y_{1 k}g_{i }            \\
y_{21}g_{i}      &  y_{22}g_{i}     & \ldots     & y_{2k}g_{i}             \\
\vdots      &             &            & \vdots               \\
y_{k 1}g_{i }   & y_{k 2}g_{i}   &  \ldots    & y_{kk}g_{i}               
\end{pmatrix},
\end{gather*} 
 where $y=(y_{ij})=A(n).$ Let us show that \eqref{irreq} holds for such $x_i.$  Let  $h=\diag[D_1, \ldots, D_k] \cdot s  \in H$, where $D_i \in H_1$ and $s \in \Sym(k).$
If $$h \in H \cap H^{x_1} \cap \ldots \cap H^{x_b}$$  then, by  Lemma \ref{igrekl}, $s$ is trivial and $D_i=D_j$ for all $1 \le i, j \le {k}$. Thus, $$D_i \in H_1 \cap H_1^{g_1} \cap \ldots \cap H_1^{g_b} \le T(GL_m(q))$$
and $h \in  T(GL_n(q)).$ 

Notice that $\det(y)=1$ and if $\det(g_i)=1,$ then 
\begin{equation*}
\det(x_i)=\det(g_i \otimes y)=\det(g_i)^k \cdot \det({y})^m=1. \qedhere
\end{equation*} 
\end{proof}

\begin{Cor}
Let $S$ be a maximal solvable subgroup of $GL_n(q)$ and assume that matrices in $S$ have shape \eqref{stup}; so $S_i \gamma_i(S)= P_{i} \wr \Gamma_i$,
where $P_{i}$ is a primitive solvable subgroup of $GL_{m_i}(q)$, $\Gamma_i$ is a transitive solvable
subgroup of the symmetric group $\Sym(k_i)$, and $k_im_i = n_i$. If, for every $P_{i}$, there exist $x_i \in SL_{m_i}(q)$ such that
\begin{equation*}
P_{i} \cap P_{i}^{x_i}\le RT(GL_{m_i}(q)),
\end{equation*}
then $b_S(S \cdot SL_n(q))\le5.$
\end{Cor}
\begin{proof}
The statement follows  from Lemmas \ref{irrtog},  \ref{prtoirr} and  \ref{diag}.
\end{proof}


\begin{Th}\label{irred}
Let $S$ be an irreducible maximal solvable  subgroup of $GL_n(q)$ with $n \ge 2$, and $(n,q)$ is neither  $(2,2)$ nor $(2,3)$. For every such $S$, $$b_S(S \cdot SL_n(q)) \le 3.$$ Moreover, $$b_S(S \cdot SL_n(q))=2$$ for all such $S$ except the following cases: 
\begin{enumerate}[font=\normalfont]
\item $n =2$,  $q>3$ is odd, and $S$ is the normaliser of a Singer cycle. If $q>5$, then there exists $x\in GL_n(q)$ such that  $S \cap S^x \le D(GL_n(q)).$ \label{irred11}
\item $n =2$,  $q\ge 4$ is even, and $S$ is the normaliser of a Singer cycle. In this case there exists $x\in GL_n(q)$ such that  $S \cap S^x \le RT(GL_n(q)).$ \label{irred12}
\item $n=2$, $q=7$, and $S$  is an absolutely irreducible subgroup such that $S/Z(GL_n(q))$ is isomorphic to $2^2.Sp_2(2).$ \label{irred13}
\item $n=3$, $q=2$,  and $S$ is the normaliser of a Singer cycle. In this case $b_S(GL_n(q))=3.$ \label{irred14} 
\item $n=4$, $q=3$, and $S=GL_2(3) \wr \Sym(2)$. In this case there exists $x \in SL_4(3)$ such that $S \cap S^x \le RT(GL_4(3)).$ \label{irred15} 
\end{enumerate} 
\end{Th}
\begin{proof}
Let $S$ be an  irreducible maximal solvable subgroup of $GL_n(q).$ The statement  follows by Lemmas \ref{supirr} and \ref{prtoirr}, Theorem \ref{sch} and Section  \ref{sec5}
 for all cases except groups  conjugate to $S_1 \wr \Gamma$, where 
$n=kl$, $\Gamma$ is a transitive maximal  solvable subgroup of $\Sym(l)$, and $S_1$ is  one of the following groups:
\begin{enumerate}[label=\alph*)]
\item $k=3, q=2$ and $S_1$ is the normaliser of a Singer cycle of $GL_3(2);$ \label{irred1}
\item $k=2$, $q=2,3$ and $S_1=GL_2(q);$ 
\item $k=2$, $q=3$ and $S_1=GL_2(3)$; 
\item $k=2$, $q=7$ and $S_1/Z(GL_2(7))$ is isomorphic to $2^2.Sp_2(2);$ \label{irred4}
\item $k=2$, $q>3$ is odd and $S_1$ is the normaliser of a Singer cycle of $GL_2(q);$ \label{irred6}
\item $k=2$, $q$ is even and $S_1$ is the normaliser of a Singer cycle of $GL_2(q).$ \label{irred5}

\end{enumerate}
If $l \ge 3$, then $b_S(S \cdot SL_n(q))=2$ by \cite[Theorem 3.1]{james}. 

We verify by computation that $b_S(S \cdot SL_n(q))=2$ for $l=2$ for all cases \ref{irred1} -- \ref{irred4} except  $S_1=GL_2(3)$. If $S_1=GL_2(3)$, then there exists $x \in SL_4(3)$ such that $S \cap S^x \le RT(GL_4(3))$. 

Consider  case \ref{irred6}: so we assume $S=S_1 \wr \Sym(2)\ \le GL_4(q)$, where $S_1$ is the normaliser of a Singer cycle $S_a$ in $GL_2(q)$ and $q>3$ is odd, as in \eqref{thesin}. Hence $S \cdot SL_n(q)=GL_n(q)$ by Lemma \ref{supSL} since the determinant of a generator of a Singer cycle generates $\mathbb{F}_q^*$. 


Let $$s=
\begin{pmatrix}
0 & 0& 1 & 0\\
0 & 0& 0 & 1\\
1 & 0& 0 & 0\\
0 & 1& 0 & 0
\end{pmatrix},
$$  
so  $g \in S$ has shape 
$$
\begin{pmatrix}
\alpha_1' & \beta_1& 0 & 0\\
a \beta_1' &  \alpha_1 & 0 & 0 \\
0 & 0& \alpha_2' & \beta_2\\
0 & 0& a \beta_2' & \alpha_2
\end{pmatrix} \cdot s^i,
$$
where $i =0,1,$ $\alpha_j' = \pm \alpha_j$, $\beta_j' = \pm \beta_j$ and $\alpha_j'\beta_j' = \alpha_j \beta_j.$ Consider $g^y$ where $y=A(n)$ is as in \eqref{igrek}. First let $i=0,$ so 
$$g^y = \Scale[0.95]{
\begin{pmatrix}
\alpha_1' +a \beta_1' &\alpha_1+ \beta_1 - \alpha_1' -a\beta_1' & \alpha_1' +a\beta_1'- \alpha_1 - \beta_1 & \alpha_1 + \beta_1 -\alpha_1' -a\beta_1'\\
a \beta_1' &  \alpha_1 -a\beta_1' & a\beta_1' - \alpha_1 + \alpha_2' & \alpha_1 + \beta_2 -a\beta_1' - \alpha_2' \\
0 & 0& \alpha_2' +a\beta_2' & \alpha_2+\beta_2-\alpha_2' -a\beta_2'\\
0 & 0& a \beta_2' & \alpha_2 - a\beta_2'
\end{pmatrix}}.
$$ 
Assume that $g^y \in S,$ so ${g^y}_{1,3}={g^y}_{1,4}={g^y}_{2,3}={g^y}_{2,4}=0.$ Thus, $0={g^y}_{2,3}+{g^y}_{2,4}=\beta_2.$ Also, ${g^y}_{1,2}= - {g^y}_{1,3} = 0$, but, since the left upper $(2 \times 2)$ block must lie in $N_{GL_2(q)}(S_a)$, 
$$({g^y}_{1,2})a = \pm {g^y}_{2,1},$$
so $0 ={g^y}_{2,1} = a\beta_1'$ and $\beta_1=0.$ Therefore, ${g^y}_{1,2}= \alpha_1 -\alpha_1',$ so
$$\alpha_1 =\alpha_1'.$$ Now, since $\beta_2=0,$ ${g^y}_{3,4} = \pm ({g^y}_{4,3}) a^{-1}=0,$  
$$\alpha_2=\alpha_2'$$
and since ${g^y}_{2,4}=0$ we obtain $\alpha_1=\alpha_2,$ so $g^y$ is a scalar.  

Now let $i=1$, so 
$$g^y = \Scale[0.95]{
\begin{pmatrix}
0 & 0 & \alpha_1' + a\beta_1'& \alpha_1 +\beta_1 - \alpha_1' - a\beta_1'\\
\alpha_2' &  \beta_2'-\alpha_2' & a\beta_1' - \beta_2' + \alpha_2' & \alpha_1 - a\beta_1' +a\beta_2 - \alpha_2' \\
a \beta_2' + \alpha_2' & \alpha_2 + \beta_2 -a\beta_2' -\alpha_2'& a\beta_2' + \alpha_2 - \alpha_2 -\beta_2 & \alpha_2+ \beta_2 - a\beta_2' \alpha_2'\\
a\beta_2' & \alpha_2 -a \beta_2'& a\beta_2' - \alpha_2 & \alpha_2 - a\beta_2'
\end{pmatrix}}.
$$ 
If $g^y \in  S$, then ${g^y}_{2,1}={g^y}_{2,2}=0.$ So $\alpha_2 = \beta_2= 0$ which  contradicts the invertibility of $g.$
Therefore, $S \cap S^y  \le Z(GL_4(q)).$ 

Consider \ref{irred5},  so we  assume $S=S_1 \wr \Sym(2)\ \le GL_4(q)$, where $S_1$ is the normaliser of a Singer cycle $S_a$ in $GL_2(q)$ as in \eqref{thesineven}. If $q>2$, so we can choose $a\ne 1$, then arguments similar to those in case \ref{irred6} show that $$S \cap S^y \le Z(GL_n(q)).$$ For $q=2$ the statement $b_S(GL_4(q))=2$ is verified  by computation.
\end{proof}

\begin{Th} \label{uniind}
Let $Z(GU_m(q)) \le H \le GU_m(q)$. Assume that there exist  $a,b \in GU_m(q)$ such that 
$$H \cap H^a \cap H^b \le Z (GU_m(q)).$$
Let $M=(\mathbb{F}_{q^2}^*)^{q-1} \wr \Gamma$ for $\Gamma \le \Sym(k)$, so $M$ is a subgroup of monomial matrices in $GU_k(q).$ Assume that there exist  $x,y \in GU_k(q)$ such that 
$$M \cap M^x \cap M^y \le Z (GU_k(q)).$$
Denote 
\begin{align*}
&X=I_m \otimes x & & A= a \otimes I_k\\
&Y=I_m \otimes y & & B= b \otimes I_k.
\end{align*} 
If $n=mk$ and $S = H \wr \Gamma \le GU_n(q)$, then 
$$S \cap S^{AX} \cap S^{BY} \le Z(GU_{n}(q)).$$
\end{Th}
\begin{proof}
Consider $h \in S\cap S^{AX},$ so $h=g^{AX}$ where $g  \in S.$ Hence 
$$g^A = \diag[g_1, \ldots , g_k] \cdot \pi,$$ where $g_i \in H^a$ and $\pi=I_m \otimes \pi_1 $ for some $\pi_1 \in \Sym(k).$ If 
\begin{equation*}
x=
\begin{pmatrix}
x_{11} & \ldots & x_{1k}\\
\vdots & & \vdots \\
x_{k1} & \ldots & x_{kk}
\end{pmatrix}, \text{ then }
x^{-1}=
\begin{pmatrix}
x_{11}^q & \ldots & x_{k1}^q\\
\vdots & & \vdots \\
x_{1k}^q & \ldots & x_{kk}^q
\end{pmatrix},
\end{equation*}
since $x \in GU_k(q)$, and
\begin{equation*}
X=
\begin{pmatrix}
x_{11} I_m & \ldots & x_{1k} I_m\\
\vdots & & \vdots \\
x_{k1} I_m & \ldots & x_{kk} I_m
\end{pmatrix}. 
\end{equation*}
Here $x_{ij} \in \mathbb{F}_{q^2.}$ The $i$-th $(k\times k)$-row of $X^{-1}g^A$  is equal to 
\begin{equation}\label{garow}
(x_{(1)\pi_1^{-1}i}^q g_{(1)\pi_1^{-1}}, \ldots, x_{(k)\pi_1^{-1}i}^q g_{(k)\pi_1^{-1}}).
\end{equation} 
Let  $j$ be such that the $(i,j)$-th $(m \times m)$-block of $h$ is not zero (there is only one such $j$ for given $i$ since $h \in S$). Consider the system of linear equations with variables $Z_1, \ldots, Z_k \in H^a$
\begin{equation}\label{sys}
\begin{split}
&x_{11}Z_1+x_{21}  Z_2 + \ldots + x_{k1}Z_k=0\\
&\vdots\\
&\underline{x_{1j} Z_1+x_{2j}Z_2 + \ldots + x_{kj}Z_k=0} \\
& \vdots\\
&x_{1k}Z_1+x_{2k}  Z_2 + \ldots + x_{kk}Z_k=0,\\
\end{split}
\end{equation}
where we exclude the (underlined) $j$-th equation. Thus,  \eqref{sys} consist of $k-1$ linearly independent equations. If we fix $Z_k$ to be some matrix from $GL_n(q^2)$, then $Z_i$ for $i=1, \ldots, k-1$ are  determined uniquely. It is routine to check that 
$$(x_{1j}^q D, \ldots, x_{kj}^q D) \text{ where } D\in GL_n(q^2)$$
is a solution for the system \eqref{sys}. 

Notice that the row \eqref{garow} must be a solution of \eqref{sys}, since $X^{-1}g^AX=h \in S.$ Therefore, by fixing $Z_k$ to be $x_{kj}^q D_i:=x_{(k)\pi_1^{-1}i}^q g_{(k)\pi_1^{-1}}$, we obtain 
$$(x_{(1)\pi_1^{-1}i}^q g_{(1)\pi_1^{-1}}, \ldots, x_{(k)\pi_1^{-1}i}^q g_{(k)\pi_1^{-1}})=(x_{1j}^q D_i, \ldots, x_{kj}^q D_i)$$
for some $D_i \in \alpha H^a$, $\alpha \in \mathbb{F}_{q^2}^*$, since $g_i \in H^a.$ Thus, 
$$h=\diag[h_1, \ldots, h_k] \cdot \sigma$$
where $h_i = D_i$. Therefore, $\alpha^{q+1}=1$ and $D_i \in H^a$, since $h_i \in GU_m(q).$ So $h_i\in H \cap H^a$ and $\sigma \in I_m \otimes \Sym(k).$ 

Assume that $h \in S\cap S^{BY},$  so $h =(g')^{BY}$ for $g' \in S$. The same argument as above shows that 
$$h=\diag[h_1, \ldots, h_k] \cdot \sigma$$
where $h_i \in H \cap H^b$ and $\sigma \in I_m \otimes \Sym(k).$ 

Therefore, if $h \in S \cap S^{AX} \cap S^{BY}$ then  $h_i = \lambda_i I_m \in H \cap H^a \cap H^b$ for some $\lambda_i \in \mathbb{F}_{q^2}^*$ with $\lambda_i^{q+1}=1$. So $g^A, g'^B \in  I_m \otimes M$ and 
\begin{equation*}
h \in  I_m \otimes (M \cap M^x \cap M^y) \le Z(GU_n(q)). \qedhere
\end{equation*} 
\end{proof}

\begin{Rem}\label{uniindSL}
If $a,b \in SU_m(q)$ and $x,y \in SU_k(q)$ in Theorem \ref{uniind}, then $AX, BY \in SU_n(q),$ since $AX=a \otimes x$ and $BY=b \otimes y.$
\end{Rem}

\begin{Lem} \label{m2isotr}
Let $k \in \{2,4,6,8\}$. Let $H$ be an irreducible subgroup of $GU(V)$  that stabilises the decomposition 
$$V=V_1 \oplus \ldots \oplus V_k$$
as in $(2)$ of Lemma $\ref{ashb}$, so each $V_i$ is totally isotropic  and $\dim V_i=m,$ where $n=km.$ Denote $ \Stab_H(V_1)|_{_{V_1}} \le GL(V_1)$ by $H_1.$ 
If there exist $a,b \in GL(V_1)$ $(\text{respectively } SL(V_1))$ such that 
$$H_1 \cap H_1^a \cap H_1^b \le Z(GL(V_1)),$$
then there exist $A, B \in GU(V)$ $(\text{respectively } SU(V))$ such that
$$H \cap H^A \cap H^B \le Z(GU(V)).$$ 
\end{Lem}
\begin{proof}
Let $\alpha \in \mathbb{F}_{q^2}^*$ be such that $\alpha + \alpha^q=0.$ Such $\alpha$ always exists. Indeed,
if $q$ is even, then $\alpha$ can be an arbitrary element of $\mathbb{F}_q^*.$ Assume that $q$ is odd and $\eta$ is a generator of $\mathbb{F}_{q^2}^*,$ so  $\eta^{(q^2-1)/2}=-1$ is the unique element of order 2 in $\mathbb{F}_{q^2}^*.$ Let $\alpha =\eta^{(q+1)/2},$ therefore, $\alpha^{q-1}=-1$ and $\alpha^q=-\alpha$, so $\alpha+\alpha^q=0.$

Assume $k=2.$ Let $\beta$ be a basis as in \eqref{unibasis}. Since $V_1$ is totally isotropic, we can assume that $V_1=\langle f_1, \ldots, f_m \rangle$ by Lemma \ref{witt}. Every $v \in V$ has a unique decomposition $v=v_1+v_2,$ $v_i \in V_i.$ Define the projection operators $\pi_i: V \to V_i$ by $(v)\pi_i=v_i$ for $i=1,2$. Notice that $$(f_i,e_j)=(f_i, (e_j)\pi_1+(e_j)\pi_2)=(f_i,(e_j)\pi_2)$$
since $V_1$ is totally isotropic. Also 
$((e_i)\pi_2, (e_j)\pi_2)=0$
since $V_2$ is totally isotropic. Therefore, the form {\bf f} has matrix ${\bf f}_{\beta_1}= J_{2m}$ with respect to the basis $ \beta_1=\{f_1, \ldots, f_m, (e_1)\pi_2, \ldots, (e_m)\pi_2\}.$ In other words, we can assume  that $$V_2= \langle e_1, \ldots, e_m \rangle.$$  

Therefore, applying the  above argument to each $U_i=V_{2i-1} \oplus V_{2i}$ for $i \in \{1,2,3\}$, we obtain  a basis $$\beta= \{ f_{11}, \ldots, f_{1m}, e_{11}, \ldots, e_{1m}, \ldots, f_{(k/2)1}, \ldots, f_{(k/2)m}, e_{(k/2)1}, \ldots, e_{(k/2)m} \}$$ of $V$ such that ${\bf f}_{\beta}=J_{2m} \otimes I_{k/2}$ and $V_{2i-1}=\langle f_{i1}, \ldots, f_{im} \rangle,$ $V_{2i}=\langle e_{i1}, \ldots, e_{im} \rangle.$ 

Recall that $g^{\dagger}=(\overline{g}^{\top})^{-1}$ for $g \in GU_n(q).$ For $x \in GL_m(q^2),$ denote by $X(k,x)$ the initial $(k \times k)$-submatrix  of the  matrix
\begin{equation*}
X(x)= \begin{pmatrix} 
\alpha x& \multicolumn{1}{c|}{x} & \alpha x& \multicolumn{1}{c|}{x}&0&\multicolumn{1}{c|}{0}&0&0\\
0&\multicolumn{1}{c|}{(\alpha x)^{\dagger}}& 0&\multicolumn{1}{c|}{0}&0&\multicolumn{1}{c|}{0}&0&0\\ \cline{1-2}
0& -x& 0 &\multicolumn{1}{c|}{x}&0&\multicolumn{1}{c|}{0}&0&0\\
0 & 0&x^{\dagger}&\multicolumn{1}{c|}{-(\alpha x)^{\dagger}}&x^{\dagger}&\multicolumn{1}{c|}{-(\alpha x)^{\dagger}}&0&0\\ \cline{1-4}
0&0&0&0&\alpha x&\multicolumn{1}{c|}{x}&0&0\\
0& (\alpha x)^{\dagger}&0&-(\alpha x)^{\dagger}&0&\multicolumn{1}{c|}{(\alpha x)^{\dagger}}&x^{\dagger}&-(\alpha x)^{\dagger}\\ \cline{1-6}
0&0&0&0&0&0&\alpha x& x\\
0&0&0&0&x^{\dagger}&-(\alpha x)^{\dagger}&0& (\alpha x)^{\dagger}
\end{pmatrix}.
\end{equation*}

It is routine to check that $X(k,x) \in GU_n(q, {\bf f}_{\beta})$. If $g \in H_{\beta} \cap H_{\beta}^{X(k,a)},$  then $g$ stabilises both 
\begin{equation}\label{2m2dec1}
V=V_1 \oplus \ldots \oplus V_k,
\end{equation}
and
\begin{equation}\label{2m2dec2}
V=(V_1){X(k,a)} \oplus \ldots \oplus (V_k) {X(k,a)}.
\end{equation}

Let $k=8$ and $v \in V$. Since $g$ stabilises \eqref{2m2dec1}, $(v)g$ and $v$ have the same number of non-zero projections on the $V_i$.  Hence $g$ stabilises $(V_2)X(k,a)$ and $V_2$ because $(V_2)X(k,a)$ is the only subspace in \eqref{2m2dec2} which has only one non-zero projection on the $V_i$. Therefore, $g$ stabilises  $V_1$ and $(V_2)X(k,a)$ because they are the only subspaces which are not orthogonal to $V_2$  and $(V_2)X(k,a)$ respectively in decompositions \eqref{2m2dec1} and \eqref{2m2dec2}. Since $g$ stabilises $V_2$, it stabilises $(V_3)X(k,a)$, so $g$ also stabilises $V_3$, $V_4$ and $(V_4)X(k,a).$ Since $g$ stabilises $(V_4)X(k,a)$, it stabilises $V_5 \oplus V_6$, so it stabilises $(V_5)X(k,a)$ and $(V_6)X(k,a).$ Now it is easy to see that $g$ must stabilise $V_5$ and $V_6.$ Since $g$ stabilises $(V_7)X(k,a) \oplus (V_8)X(k,a)$, it stabilises $(V_8)X(k,a),$ so $g$ stabilises $V_8$ and $V_7.$ Therefore, $g$ stabilises all subspaces in \eqref{2m2dec1} and \eqref{2m2dec2}, so $g=\diag[g_1, g_1^{\dagger}, \ldots, g_{k/2}, g_{k/2}^{\dagger}]$ with $g_i \in H_1.$ 

Since $X(x)^{-1}={\bf f}_{\beta}\overline{X(x)}^{\top}{\bf f}_{\beta},$ 
\begin{equation*}
X(x)^{-1}= \Scale[0.93]{\begin{pmatrix} 
(\alpha x)^{-1}& \multicolumn{1}{c|}{\overline{x}^{\top}} & 0& \multicolumn{1}{c|}{-\overline{x}^{\top}}&0&\multicolumn{1}{c|}{0}&0&0\\
0&\multicolumn{1}{c|}{(\overline{\alpha x})^{\top}}& 0&\multicolumn{1}{c|}{0}&0&\multicolumn{1}{c|}{0}&0&0\\ \cline{1-2}
0& \overline{x}^{\top^{\phantom{1}}}& -(\alpha x)^{-1} &\multicolumn{1}{c|}{\overline{x}^{\top}}&-(\alpha x)^{-1}&\multicolumn{1}{c|}{0}&0&0\\
0 & (\overline{\alpha x})^{\top}&x^{-1}&\multicolumn{1}{c|}{0}&0&\multicolumn{1}{c|}{0}&0&0\\ \cline{1-4}
0&0&-(\alpha x)^{-1}&0&(\alpha x)^{-1}&\multicolumn{1}{c|}{\overline{x}^{\top}}&-(\alpha x)^{-1}&0\\
0& 0&x^{-1}&0&0&\multicolumn{1}{c|}{(\overline{\alpha x})^{\top}}&x^{-1}&0\\ \cline{1-6}
0&0&0&0&-(\alpha x)^{-1}&0&(\alpha x)^{-1}& \overline{x}^{\top}\\
0&0&0&0&x^{-1}&0&0& (\overline{\alpha x})^{\top}
\end{pmatrix}}.
\end{equation*}

A similar argument to the above shows that if $h =g^{X(k,a)^{-1}} \in H_{\beta} \cap H_{\beta}^{X(k,a)^{-1}}$, then $h=\diag[h_1, h_1^{\dagger}, \ldots, h_{k/2}, h_{k/2}^{\dagger}]$ with $h_i \in H_1.$ Calculations show that if the equation
$h^{X(k,a)}=g$  holds, then 
\begin{equation*}
\begin{cases}
g_i &  =  g(2i-1,2i-1)= h_i^a \text{ for } i \in \{1,2,3,4\} \\
0 & =  g(1,2)=(\alpha a)^{-1} h_1 a + \overline{a}^{\top} h_1^{\dagger} (\alpha a)^{\dagger}= {\alpha}^{-1}(h_1^a - (h_1^a)^{\dagger});\\
0 & =  g(1,3)=(\alpha a)^{-1} h_1 (\alpha a) - \overline{a}^{\top}h_2^{\dagger} a^{\dagger}=h_1^a - (h_2^a)^{\dagger};\\
0 & =  g(1,5)= - \overline{x}^{\top} h_2^{\dagger} a^{\dagger} +(\alpha a)^{-1} h_3 (\alpha a)= h_3^a - (h_2^a)^{\dagger};\\
0 & =  g(8,5)= a^{-1} h_3 (\alpha a) + (\overline{\alpha a})^{\top}h_4^{*} a^{\dagger}= \alpha (h_3^a - (h_4^a)^{\dagger}).
\end{cases}
\end{equation*}
So $h_1^a=g_1=g_1^{\dagger} \ldots = g_{k/2}= g_{k/2}^{\dagger}$ and $g_1 \in H_1 \cap H_1^a.$ 

 If $g \in H_{\beta} \cap H_{\beta}^{X(k,a)} \cap H_{\beta}^{X(k,b)},$ then the same argument   with $a$ replaced by $b$ shows that 
$g=\diag[g_1, \ldots, g_1]$ and $g_1 \in H_{1} \cap H_{1}^{a} \cap H_{1}^{b},$ so $g \in Z(GU_{2m}(q, {\bf f}_{\beta}))$.

The proof for $k \in \{2,4,6\}$ is analogous.

Calculations show that if $x \in SL_m(q)$, then 
$$\det(X(k,x))=\begin{cases}
1 & \text{ for } k=4,8;\\
(-1)^m & \text{ for } k=2,6.
\end{cases}
$$ 
Consider $A=\diag[\alpha I_m, \alpha^{\dagger} I_m, I_m, \ldots, I_m] \in GU_n(q, {\bf f}_{\beta})$. Notice that $\det(A)=(-1)^m.$ Repeating the arguments above, one can show that  $$H_{\beta} \cap H_{\beta}^{AX(k,a)} \cap H_{\beta}^{AX(k,b)} \le Z(GU_n(q, {\bf f}_{\beta})).$$  Notice that $X(k,a), X(k,b) \in SU_n(q, {\bf f}_{\beta})$ for $k=4,8$ and $AX(k,a), AX(k,b) \in SU_n(q, {\bf f}_{\beta})$ for $k=2,6.$ 
\end{proof}

\begin{Lem}\label{isklwr}
Let $n=mk$ for integers $m \ge 3$ and $k \ge 2$. 
\begin{enumerate}[font=\normalfont]
\item If $S=GU_2(2) \wr \Sym(k)$, then  $b_S(S \cdot SU_{2k}(2)) \le 3.$ \label{isklwr1}
\item If $S=GU_2(3) \wr \Sym(k)$, then $b_S(S \cdot SU_{2k}(3))\le 3.$ \label{isklwr2}
\item If $S=GU_3(2) \wr \Sym(k)$, then $b_S(S \cdot SU_{3k}(2))\le 3.$ \label{isklwr3}
\item Let $N$ be a  quasi-primitive maximal solvable  subgroup of $GU_m(2)$. If $S=N \wr \Sym(k)$ with $k \in \{2,3,4\}$, then $b_S(S \cdot SU_{km}(2))\le 3.$ \label{isklwr4}
\item Let $N$ be a  quasi-primitive maximal  solvable  subgroup of $GU_m(3)$. If $S=N \wr \Sym(2)$, then $b_S(S \cdot SU_{2m}(3))\le 3.$ \label{isklwr5}
\end{enumerate}
\end{Lem}
\begin{proof}
Notice that $S \cdot SU_n(q)=GU_n(q)$ for \eqref{isklwr1} -- \eqref{isklwr2}.  

\eqref{isklwr1} Notice that $\lambda^{q+1}=1$ for $q=2$ and $\lambda \in \mathbb{F}_{q^2}^*.$
Therefore, since  a row $v$ of a matrix in $GU_n(q)$ satisfies ${\bf f}(v,v)=1,$
every matrix in $GU_2(2)$ is  monomial. Thus, $$S=GU_2(2) \wr \Sym(k)\le MU_{2k}(q)$$ and  the statement for $k>2$ follows by \eqref{omnom5} of Lemma \ref{omnom}. The case $k=2$ is verified by computation.

\medskip

\eqref{isklwr2} For $k \in \{2,3\}$ we verify the statement by computation, so assume $k \ge 4.$ Let $\beta =\{v_{11}, v_{12}, v_{21}, v_{22}, \ldots, v_{k1}, v_{k2}\}$ be an orthonormal basis of $V$ such that $S$ stabilises the decomposition $V_1 \oplus \ldots \oplus V_k$ with $V_i=\langle v_{i1}, v_{i2} \rangle.$ Define a basis $\beta_1=\{w_{11}, w_{12}, w_{21}, w_{22}, \ldots, w_{k1}, w_{k2}\}$ by the following rule:
\begin{equation*}
\begin{split}
(w_{11}, w_{21}, \ldots, w_{k1})&=(v_{12}, v_{21}, \ldots, v_{k1})B(k,a); \\
(w_{12}, w_{22}, \ldots, w_{k2})&=(v_{22}, v_{32}, \ldots, v_{k2}, v_{12})B(k,a).
\end{split}
\end{equation*}
Here $a$ and $B(k,a)$ are as in \eqref{omnom2} of Lemma \ref{omnom}. Denote the change-of-basis matrix from $\beta_1$ to $\beta$ by $y$. For example, if $k=4$, then  
\begin{equation*}
y=
\begin{pmatrix}
a& \multicolumn{1}{c|}{}&a&\multicolumn{1}{c|}{}&&\multicolumn{1}{c|}{}&& \\
& \multicolumn{1}{c|}{}&&\multicolumn{1}{c|}{a}&&\multicolumn{1}{c|}{a}&& \\ \hline
a^2& \multicolumn{1}{c|}{} &-a^2&\multicolumn{1}{c|}{}&a&\multicolumn{1}{c|}{}&& \\
  &\multicolumn{1}{c|}{} &&\multicolumn{1}{c|}{a^2}&&\multicolumn{1}{c|}{-a^2}&&a \\\hline
a^3&\multicolumn{1}{c|}{} &-a^3&\multicolumn{1}{c|}{} &-a^2&\multicolumn{1}{c|}{}&a& \\
 &\multicolumn{1}{c|}{a}&&\multicolumn{1}{c|}{a^3}&&\multicolumn{1}{c|}{-a^3}&&-a^2 \\\hline
a^3&\multicolumn{1}{c|}{} &-a^3&\multicolumn{1}{c|}{}&-a^2&\multicolumn{1}{c|}{}&-a& \\
& \multicolumn{1}{c|}{-a}&&\multicolumn{1}{c|}{a^3}&&\multicolumn{1}{c|}{-a^3}&&-a^2 \\ 
\end{pmatrix}.
\end{equation*}
We use blanks instead of zeroes in the matrix. It is routine to verify that $\beta_1$ is orthonormal, so $y \in GU_{2k}(3)$. Observe $g \in S \cap S^y$ stabilises the decompositions
$$V_1 \oplus \ldots \oplus V_k \text{ and  } W_1 \oplus \ldots \oplus W_k, $$
where $W_i=\langle w_{i1}, w_{i2} \rangle.$

Notice that $w_{11}$ has non-zero entries only in two $V_i$-s, so $(w_{11})g$ also must have non-zero entries only in two $V_i$-s. It is easy to see that a vector from $W_j$ for $j>1$ has non-zero entries  in at least three $V_i$-s. Thus, $g$ stabilises $W_1.$ The same argument shows that $g$ must stabilise $W_i$ for $i=1, \ldots, k-2.$ Notice that $(w_{ij})g$ lies either in $\langle w_{i1} \rangle$ or $\langle w_{i2} \rangle$ for $i=1, \ldots, k-2,$ since otherwise it would have non-zero entries in more $V_i$-s than $w_{ij}$.

Assume that $(w_{11})g \in \langle w_{12}\rangle$. Therefore, $(w_{12})g \in \langle w_{11}\rangle$, so either
 $$(V_1)g=V_2, (V_2)g=V_3, (V_3)g=V_1,$$
 or
 $$(V_1)g=V_3, (V_2)g=V_2, (V_3)g=V_1.$$
In both cases $(w_{22})g$ cannot lie in either  $\langle w_{21} \rangle$ or $\langle w_{22} \rangle$, which is a contradiction. So $g$ stabilises $\langle w_{11} \rangle$ and, therefore, it stabilises $\langle w_{12} \rangle$, since $\langle w_{12} \rangle$ is the orthogonal complement of $\langle w_{11} \rangle$ in $W_1.$ Therefore, $g$ stabilises $V_1,$ $V_2$ and $V_3.$ The same argument shows that $g$
stabilises $V_1, \ldots, V_{k-1},$ so $g$ stabilises $V_k$ as well. Thus, $g$ stabilises $$\langle w_{11} \rangle, \langle w_{12} \rangle, \ldots, \langle w_{(k-2)1} \rangle, \langle w_{(k-2)2} \rangle,$$ which implies that $g$ stabilises   $\langle v_{11} \rangle, \langle v_{12} \rangle, \ldots, \langle v_{k1} \rangle, \langle v_{k2} \rangle$. So $$g=\diag(g_{11}, g_{12}, \ldots, g_{k1}, g_{k2}).$$
Since $(w_{11})g \in \langle w_{11} \rangle$, $g_{11}=g_{21}.$ Applying the same argument to all $w_{ij}$ for $i=1, \ldots, k-2$ and $j= 1,2$ we obtain that $g$ is scalar.

\medskip

\eqref{isklwr3} For $k \le 3$ the statement is verified by computation, so assume $k \ge 4.$ Fix $$\beta=\{v_{11}, v_{12}, v_{13}, v_{21}, \ldots, v_{k3}\}$$ to be the initial orthonormal basis, so $g \in S$ stabilises the decomposition
\begin{equation}\label{dec1}
V=V_1 \oplus \ldots \oplus V_k,
\end{equation}
where $V_i=\langle v_{i1}, v_{i2}, v_{i3} \rangle$.
 Let $x$ be the permutation matrix for the permutation $(1,2, \ldots, n)$ where $n=3k.$ Consider $g \in S \cap S^x$. We claim that $g$ is monomial. Indeed, since $g \in S^x$, it stabilises the decomposition
\begin{equation} \label{dec2}
V=(V_1)x \oplus \ldots \oplus (V_k)x=\langle v_{12}, v_{13}, v_{21} \rangle \oplus \ldots \oplus \langle v_{k2}, v_{k3}, v_{11} \rangle,
\end{equation} 
so it permutes subspaces $\langle v_{11} \rangle, \langle v_{21} \rangle, \ldots, \langle v_{k1} \rangle$ and  $\langle v_{12}, v_{13} \rangle, \langle v_{22}, v_{23} \rangle, \ldots, \langle v_{k2}, v_{k3} \rangle.$ Thus,  $g$ consists of $(1 \times 1)$ and $(2 \times 2)$ blocks which lie in $GU_1(2)$ and $GU_2(2)=MU_2(2)$, respectively. 

Define a basis $\beta_1=\{w_{11}, w_{12}, w_{13}, w_{21}, \ldots, w_{k3}\}$ as follows:
\begin{equation*}
\begin{split}
w_{11}&= (\sum_{i=1}^{k-1} \sum_{j=1}^3 v_{ij}) +v_{k1} +((1+(-1)^k)/2)v_{k2}  \\ 
w_{12}&= v_{11} + v_{12} + v_{k3}\\
w_{13}&= v_{11} + v_{21} + v_{k3}\\
w_{s1}&= v_{11} + v_{(s-1)3} + v_{k3}\\
w_{s2}&= v_{11} + v_{s2} + v_{k3}\\
w_{s3}&= v_{11} + v_{(s+1)1} + v_{k3}\\
w_{k1}&= v_{11} + v_{(k-1)3} + v_{k3}\\
w_{k2}&= v_{11} + v_{k2} + v_{k3}\\
w_{k3}&= ((1+(-1)^k)/2)v_{12} +v_{13} + (\sum_{i=2}^{k} \sum_{j=1}^3 v_{ij}). \\
\end{split}
\end{equation*}
Here $1<s<k.$  Denote the change-of-basis matrix from $\beta_1$ to $\beta$ by $y$. For example, if $k=3$, then
\begin{equation*}
y=
\begin{pmatrix}
1&1& \multicolumn{1}{c|}{1}&1&1&\multicolumn{1}{c|}{1}&1&0&0 \\
1&1& \multicolumn{1}{c|}{0}&0&0&\multicolumn{1}{c|}{0}&0&0&1 \\
1&0& \multicolumn{1}{c|}{0}&1&0&\multicolumn{1}{c|}{0}&0&0&1 \\ \cline{1-9}
1&0& \multicolumn{1}{c|}{1}&0&0&\multicolumn{1}{c|}{0}&0&0&1 \\
1&0& \multicolumn{1}{c|}{0}&0&1&\multicolumn{1}{c|}{0}&0&0&1 \\
1&0& \multicolumn{1}{c|}{0}&0&0&\multicolumn{1}{c|}{0}&1&0&1 \\\cline{1-9}
1&0& \multicolumn{1}{c|}{0}&0&0&\multicolumn{1}{c|}{1}&0&0&1 \\
1&0& \multicolumn{1}{c|}{0}&0&0&\multicolumn{1}{c|}{0}&0&1&1 \\
0&0& \multicolumn{1}{c|}{1}&1&1&\multicolumn{1}{c|}{1}&1&1&1 \\
\end{pmatrix}.
\end{equation*}

It is routine to verify that $\beta_1$ is orthonormal, so $y \in GU_{3k}(2)$. If $g \in S \cap S^x \cap S^y$, then $g$ stabilises decompositions \eqref{dec1}, \eqref{dec2} and 
\begin{equation*}
V=W_1 \oplus \ldots \oplus W_k,
\end{equation*}
where $W_i=\langle w_{i1}, w_{i2}, w_{i3} \rangle.$ Since $g$ is monomial, $w_{ij}$ and $(w_{ij})g$ have the same number of non-zero entries in the decomposition with respect to $\beta.$ Therefore, $(w_{11})g$ can lie either in $W_1$ or in $W_k.$ Assume that $(w_{11})g \in W_k,$ so $(W_1)g=W_k.$ If a vector in $W_k$ has the same number of non-zero entries  in the decomposition with respect to $\beta$ as $w_{11}$, then its first $1+((1+(-1)^k)/2)$ entries are zero. So $g$ must permute subspace $\langle v_{11}, v_{12} \rangle$ with $\langle v_{k2}, v_{k3} \rangle$ for $k$ odd (respectively $\langle v_{11} \rangle$ with $\langle v_{k3} \rangle$ for $k$ even) which contradicts the fact that $g$ stabilises decompositions \eqref{dec1} and \eqref{dec2}. Therefore, $g$ stabilises $W_1$ and $\langle w_{11} \rangle$ in particular. It is easy to see now that $g$ stabilises  $\langle w_{12} \rangle$ and $\langle w_{13} \rangle$, since $g$ stabilises \eqref{dec1}. Thus, $g$ stabilises $V_1$, $V_k$ and $V_2$, so it stabilises  $\langle v_{11} \rangle$,
$\langle v_{12} \rangle$, $\langle v_{13} \rangle$, $\langle v_{21} \rangle$, $\langle v_{k3} \rangle$. Using the same argument, we obtain that $g$ is diagonal. Since $g$ stabilises $\langle w_{11} \rangle$ and $\langle w_{k3} \rangle$, all non-zero entries of $g$ must be equal, so $g$ is scalar.

\medskip

\eqref{isklwr4}-\eqref{isklwr5}  For  $n<12$ we verify the statement by computation. For larger $n$ we prove the statement by checking \eqref{0} with $c=3$ for the elements of prime order of $H \le PGU_n(q)$, where $H$ and $G$ are the images of $S$ and $S \cdot SU_n(q)$ respectively in $PGU_n(q).$  

 Let $B$ be the image in $G$ of the block-diagonal subgroup $$GU_m(q) \times \ldots \times GU_m(q) \le GU_{km}(q).$$ 
Let $x \in H$ have prime order.

 If  $(|x| > 3$ for $k \in \{3,4\})$ or $(|x| > 2$ for $k=2)$, then $x^G \cap H \subseteq B.$ 
In this case there exists a preimage $\hat{x}=\diag[\hat{x_1}, \ldots ,\hat{x_k}]$ of $x$ in $GU_{km}(q)$ such that $\hat{x_i} \in N.$  If $x_i\ne 1$, then $\nu(x_i)\ge m/4$ by Lemma \ref{nuuni}. We can assume that $x$ is such that the number $l(x)$ of $x_i$ not equal to 1 is maximal for elements in $x^G \cap H.$ Therefore, 
$$\nu(x) \ge (1/4)l(x)m \text{ and } |x^G \cap H| \le \tbinom{k}{l(x)} |N|^{l(x)}.$$
These bounds together with bounds from Lemma \ref{uniqorder} for $|N|$ and \eqref{5uni} for $|x^G|$ are  sufficient to show \eqref{0} holds for $mk \ge 12.$

\medskip
Now consider the case where $x^G \cap H$ is not a subset of $B$. For such $x$, we use the bounds for $|x^G|$ and $|x^G \cap H|$ given in \cite[Propositions 2.5 and 2.6]{fpr3}. These propositions give corresponding bounds when $H=(GU_m(q) \wr \Sym(k))/Z(GU_n(q)),$ so they are applicable in our situation. For $mk \ge 12$, these bounds  are sufficient to show that \eqref{0} holds. 

We briefly outline how to extract the corresponding bounds. The proofs of the propositions split into several cases depending on $|x|,$ $m$, $k$ and 
$|H^1(\sigma, E/E^0)|$ (see Definition \ref{H1EE0}). Notice that if $x$ is semisimple, then $|H^1(\sigma, E/E^0)|=(|x|,q+1)$ by \cite[Lemma 3.35]{fpr}.

 Assume that $k=2$, so $|x|=2.$ If $q=2$, then we use the bounds from {\bf Case 2.2} of the proof of  \cite[Proposition 2.6]{fpr3} for unipotent $x$. 
If $q=3$, then we use bounds from {\bf Case 2.4} of the proof of  \cite[Proposition 2.5]{fpr3} for semisimple $x.$

Assume that $k \in \{3,4\}$, so $q=2.$ If $|x|=2$, then we use the bounds from {\bf Case 2.2} of the proof of  \cite[Proposition 2.6]{fpr3} for unipotent $x$. 
Let  $|x|=3$. We use bounds  from {\bf Case 2.2} (if $|H^1(\sigma, E/E^0)|=1$) and {\bf Case 2.3} (if $|H^1(\sigma, E/E^0)|=3$)  of the proof of  \cite[Proposition 2.5]{fpr3}. 
\end{proof}

\begin{Th} \label{irredGU}
Let $(n,q)$ be such that $GU_n(q)$ is not solvable. If $S$ is an irreducible maximal solvable  subgroup of $GU_n(q)$,  then $b_S(S \cdot SU_n(q)) \le 3$ or $(n,q)=(4,2)$. If $(n,q)=(4,2)$  and $b_S(GU_n(q))>3$, then  $S$ is conjugate to $GU_1(q) \wr \Sym(4)=MU_4(2)$ (so $S \cdot SU_n(q)=GU_n(q)$) and $b_S(GU_n(q))=4.$
\end{Th}
\begin{proof}
Let us fix $q$ and consider a minimal counterexample to the statement $$b_S(S \cdot SU_n(q)) \le 3,$$ so $(n,S)$ is such that  $n$ is the smallest integer satisfying the conditions of the theorem: namely,  $GU_n(q)$ is not solvable and $GU_n(q)$ has an irreducible maximal solvable  subgroup $S$ with $b_S(S \cdot SU_n(q))>3.$ 
By Theorem \ref{sch}, $S$ is not quasi-primitive, so  $S$ has a normal subgroup $L$ such that $V$ is not homogeneous as $\mathbb{F}_q[L]$-module. Therefore, $S$ and $L$ satisfy the conditions of Lemma \ref{ashb}. So $S$ stabilises a decomposition
\begin{equation}\label{decirr}
V=V_1 \oplus \ldots \oplus V_k, \text{ } k \ge 2
\end{equation} 
such that $(1)$ or $(2)$ of Lemma \ref{ashb} holds. 
  Let us fix \eqref{decirr} to be such a decomposition with  the largest possible $k.$  

If $(1)$ of Lemma \ref{ashb} holds, then consider $S_1:= \Stab_S(V_1)|_{_{V_1}} \le GU(V_1).$  By Clifford's  Theorem $S_1$ acts irreducibly on $V_1.$ Notice that we can assume $S_1$ to be quasi-primitive in that case. Indeed, if $S_1$ is not quasi-primitive, then $S_1$ stabilises a decomposition  
$$V_1=V_{11} \oplus \ldots \oplus V_{1t}$$ for some $t \ge 2$ such that  $(1)$ or $(2)$ of Lemma \ref{ashb} holds. Therefore, since $S$ is irreducible, it stabilises the  decomposition 
$$V=V_{11} \oplus \ldots \oplus V_{1t} \oplus \ldots \oplus V_{k1} \oplus \ldots \oplus V_{kt},$$
for which  $(1)$ or $(2)$ of Lemma \ref{ashb} holds contradicting the maximality of $k$ in \eqref{decirr}.

  If $\dim V_1=1$ then $S$ can be represented as a group of monomial matrices with respect to an orthonormal basis. By Lemma \ref{omnom}, this is possible if and only if $(n,q)=(4,2)$ and $S=MU_4(q).$
Assume $\dim V_1=m \ge 2.$ If $q\in \{2,3\}$ and $S$ is conjugate to a subgroup of one of the groups listed in Lemma \ref{isklwr}, then we obtain a contradiction. If $S$ is not conjugate to a subgroup of a group from Lemma \ref{isklwr}, then $GU_m(q)$ is not solvable. Therefore, $S_1\le GU(V_1)$ satisfies the condition of the theorem  
 and $b_{S_1}(S_1 \cdot SU(V_1)) \le 3$ since $(n,S)$ is a minimal counterexample. Thus, there exist $a,b \in SU(V_1)$ such that $S_1 \cap S_1^a \cap S_1^b \le Z(GU(V_1)).$ Applying Theorem \ref{uniind} and Lemma \ref{omnom} we obtain 
$b_S(S \cdot SU_n(q))\le 3.$ 

Finally, let us assume part $(2)$ of  Lemma \ref{ashb} holds. Let $U_i=V_{2i-1} \oplus V_{2i}$, so 
$$V=U_1 \bot \ldots \bot U_{k/2}$$
 and $S$ transitively permutes the $U_i$. Indeed, since $S$ acts on $V$ by isometries, $(V_{2i-1})g$ and $(V_{2i})g$ cannot be mutually orthogonal for $g \in S$, which is possible if and only if      $(V_{2i-1})g$ and $(V_{2i})g$ lie in the same $U_j$ for some $j =1, \ldots, k/2.$ Transitivity follows from the irreducibility of $S.$ Consider $S_1:= \Stab_S(U_1)|_{_{U_1}}.$ 
Notice that $\dim U_1 = 2m \ge 2.$ 

If $(2m, q) \in \{(2,2), (2,3)\}$, then, since $S$ is a maximal solvable subgroup of $GU(V),$ $S$ must be conjugate to $GU_{2m}(q) \wr \Gamma$ with $\Gamma \le \Sym(k/2)$. In this case the theorem follows by Lemma \ref{isklwr}. Otherwise $GU(U_1)$ is not solvable. If $k>8$, then  $S_1\le GU(U_1)$ satisfies the condition of the theorem 
 and $b_{S_1}(GU(U_1)) \le 3$ since $(n,S)$ is a minimal counterexample. Thus, there exist $a,b \in SU(U_1)$ such that $S_1 \cap S_1^a \cap S_1^b \le Z(GU(U_1))$. Applying Theorem \ref{uniind} and Lemma \ref{omnom} we obtain 
$b_S(S \cdot SU_n(q))\le 3.$
Let $k \le 8$. Consider $P:= \Stab_S(V_1)|_{_{V_1}} \le GL(V_1),$ so $P$ is an irreducible solvable  subgroup of $GL(V_1)$ and  there exist $a,b \in SL(V_1)$ such that $P \cap P^a \cap P^b \le Z(GL(V_1))$ by Theorem \ref{irred}. Applying Lemma \ref{m2isotr} 
we obtain $b_S(S \cdot SU(V))\le 3,$ which contradicts the assumption.
\end{proof}

\begin{Th}\label{sympirr}
Let $n \ge 2$. If $S \le GSp_n(q)$ is an irreducible maximal solvable  subgroup, then one of the following holds:
\begin{enumerate}[font=\normalfont]
\item  there exist $x,y \in Sp_n(q)$ such that $S \cap S^x \cap S^y \le Z(GSp_n(q))$; \label{sympirr1}
\item \label{sympirr2} $n=4$, $q \in \{2,3\}$, and $S$ is the stabiliser of decomposition $V =V_1 \bot V_2$ with $V_i$ non-degenerate and there 
exist $x,y,z \in Sp_n(q)$ such that $S \cap S^x \cap S^y \cap S^z \le Z(GSp_n(q))$; 
\item $n=2$, $q \in \{2,3\},$ and $S=GSp_2(q).$ \label{sympirr3}
\end{enumerate}
\end{Th}
\begin{proof}
The following is verified by computation: if $n=4$ and $q \in \{2,3\},$ then either $b_S(S \cdot Sp_n(q)) \le 3$ or $S$ is as in \eqref{sympirr2}.

Assume that $n$ is minimal such that there exists a counterexample to the theorem: namely,  $(S,n,q)$  is such that  $b_S(Sp_n(q))>3$ and neither  \eqref{sympirr2} nor \eqref{sympirr3}  hold.

If $S$ is quasi-primitive, then it is not a counterexample by Theorem \ref{sch} and Section \ref{sec5}.

 Assume that $S$ is not quasi-primitive, so $S$ has a normal subgroup $L$ such that $V$ is not $\mathbb{F}_q[L]$-homogeneous by Lemma \ref{ashb}. Therefore, $S$ stabilises a decomposition 
\begin{equation}
\label{BCvisp}
V=V_1 \oplus \ldots \oplus V_k
\end{equation}
such that $\dim V_i=m$ for $i \in \{1, \ldots, k\},$ $k>1$ and one of the following holds:
\begin{itemize}
\item[{\bf Case 1.}] $V=V_1 \bot \ldots \bot V_k$ with $V_i$ non-degenerate for $i=1, \ldots, k;$
\item[{\bf Case 2.}] $V=U_1 \bot \ldots \bot U_{k/2}$ with $U_i=V_{2i-1} \oplus V_{2i}$ non-degenerate and $V_i$ totally isotropic. 
\end{itemize} Let us fix \eqref{BCvisp} to be such a decomposition with
the largest possible $k$.   {\bf Case 1} splits into two subcases: $k \ge 3$ and $k=2$. 

\medskip

{\bf Case (1.1).} Assume that the $V_i$ are non-degenerate and $k\ge 3$ in \eqref{BCvisp}. Let $H$ be $\Stab_S(V_1)|_{_{V_1}} \le GSp_m(q).$ Notice that $H$ is an irreducible maximal solvable  subgroup of $GSp_m(q).$ If $H$ is as $S$ in \eqref{sympirr2}, then $V_i = V_{i1} \bot V_{i2}$ with non-degenerate $V_{ij}$ and $S$ stabilises the decomposition $$V=(V_{11} \bot V_{12}) \bot \ldots \bot (V_{k1} \bot V_{k2}).$$   Since $m<n$ and $H$ is not a counterexample, we can assume that either there exist $x_1, x_2 \in Sp_m(q)$ such that $H \cap H^{x_1} \cap H^{x_2} \le Z(GSp_m(q))$ or  $H=GSp_2(q)$ with $q \in \{2,3\}$. In the latter case we take $x_1=x_2=I_2.$ 

Let $\{v_1, u_1, \ldots, v_k, u_k\}$ be a basis of a $2k$-dimensional vector space over $\mathbb{F}_q.$ Let $y_1$ and $z_1$ be the matrices of the linear transformations of this space defined by the formulae:
\begin{equation*}
\begin{split}
(v_i)y_1 &=v_i - v_{i+1} \text{ for } i \in \{1, \ldots, k-1\};\\
(v_k)y_1 &= v_k;\\
(u_i)y_1 &= \sum_{j=1}^i u_i \text{ for } i \in \{1, \ldots, k\};
\end{split}
\end{equation*}
and  
\begin{align*}
(v_1)z_1 &=v_1; &(u_1)z_1 &= u_1+v_2;\\
(v_i)z_1 &=v_i - v_{i+1} \text{ for } i \in \{1, \ldots, k-2\}; &(u_i)z_1 &= \sum_{j=1}^i u_i \text{ for } i \in \{2, \ldots, k-1\};\\
(v_{k-1})z_1 &=v_1 + u_{k}; &(u_k)z_1 &= u_k;\\
(v_k)z_1 &= v_1 + v_k + \sum_2^{k-1}u_i.
\end{align*} 
For example, if $k=4$, then 
\begin{equation*} y_1=
\begin{pmatrix}
1      & \multicolumn{1}{c|}{0}& -1& \multicolumn{1}{c|}{0}&0  &\multicolumn{1}{c|}{0}&0&0   \\
0      & \multicolumn{1}{c|}{1}& 0 & \multicolumn{1}{c|}{0}&0  &\multicolumn{1}{c|}{0}&0&0 \\ \cline{1-8}
0      & \multicolumn{1}{c|}{0}& 1 & \multicolumn{1}{c|}{0}& -1&\multicolumn{1}{c|}{0}&0&0   \\
0      & \multicolumn{1}{c|}{1}& 0 & \multicolumn{1}{c|}{1}& 0 &\multicolumn{1}{c|}{0}&0&0   \\\cline{1-8}
0      & \multicolumn{1}{c|}{0}& 0 & \multicolumn{1}{c|}{0}& 1 &\multicolumn{1}{c|}{0}&-1&0 \\       
0      & \multicolumn{1}{c|}{1}& 0 & \multicolumn{1}{c|}{1}& 0 &\multicolumn{1}{c|}{1}&0&0   \\\cline{1-8}
0      & \multicolumn{1}{c|}{0}& 0 & \multicolumn{1}{c|}{0}& 0 &\multicolumn{1}{c|}{0}& 1 &        0\\
0      & \multicolumn{1}{c|}{1}& 0 & \multicolumn{1}{c|}{1}& 0 &\multicolumn{1}{c|}{1}& 0 &        1
\end{pmatrix}, \text{ }
z_1=
\begin{pmatrix}
1      & \multicolumn{1}{c|}{0}& 0 & \multicolumn{1}{c|}{0}&0  &\multicolumn{1}{c|}{0}&0&0   \\
0      & \multicolumn{1}{c|}{1}& 1 & \multicolumn{1}{c|}{0}&0  &\multicolumn{1}{c|}{0}&0&0 \\ \cline{1-8}
0      & \multicolumn{1}{c|}{0}& 1 & \multicolumn{1}{c|}{0}& -1&\multicolumn{1}{c|}{0}&0&0   \\
1      & \multicolumn{1}{c|}{0}& 0 & \multicolumn{1}{c|}{1}& 0 &\multicolumn{1}{c|}{0}&0&0   \\\cline{1-8}
0      & \multicolumn{1}{c|}{0}& 0 & \multicolumn{1}{c|}{0}& 1 &\multicolumn{1}{c|}{0}&0&1 \\       
1      & \multicolumn{1}{c|}{0}& 0 & \multicolumn{1}{c|}{1}& 0 &\multicolumn{1}{c|}{1}&0&0   \\\cline{1-8}
1      & \multicolumn{1}{c|}{0}& 0 & \multicolumn{1}{c|}{1}& 0 &\multicolumn{1}{c|}{1}& 1 &        0\\
0      & \multicolumn{1}{c|}{0}& 0 & \multicolumn{1}{c|}{0}& 0 &\multicolumn{1}{c|}{0}& 0 &        1
\end{pmatrix}.
\end{equation*}
Let $\beta_i$ be a basis of $V_i$ of shape \eqref{sympbasis} for $i \in \{1, \ldots, k\}$ and let $\beta$ be $\beta_1 \cup \ldots \cup \beta_k$. Let $y=I_{m/2} \otimes y_1$ and $z=I_{m/2} \otimes z_1$.  It is routine to check that $y,z \in Sp_n(q, {\bf f}_{\beta}).$

Let $W_i=V_i(x_1 \otimes I_k)y$ for $i \in \{1, \ldots, k\}.$ Consider $g \in S \cap S^{(x_1 \otimes I_k)y},$ so $g$ stabilises decompositions $V=V_1 \bot \ldots \bot V_k$ and $V=W_1 \bot \ldots \bot W_k.$ Notice that $W_i$ has non-zero projection on exactly $i+1$ of the $V_j$ for $i \in \{1, \ldots, k-2\}.$ So $g$ cannot map such $W_i$ to others and, therefore, $g$ stabilises subspaces $W_1, \ldots, W_{k-2}$ and $\{W_{k-1}, W_k\}.$ Thus, $g$ stabilises $\{V_1, V_2\}$ and $V_3, \ldots, V_k.$ Notice that $\dim(V_1 \cap W_1)=m/2$ and $\dim(V_2 \cap W_1)=0.$ So $(V_1)g=V_1$ since $(V_1 \cap W_1)g=(V_1)g \cap (W_1)g=(V_1)g \cap W_1 \ne \{0\}.$ The same argument for $V_n \cap W_n$ shows that $(W_n)g=W_n$. Hence $g$ stabilises all $V_i$ and $W_i$ for $i \in \{1, \ldots, k\}$; in particular $$g= \diag[g_1, \ldots, g_k]$$
for $g_i \in H.$

Now let us show that  $g_i \in H \cap H^{x_1}$ for all $i \in \{1, \ldots, k\}.$ Since $g \in S \cap S^{(x_1 \otimes I_k)y},$ $g=h^{(x_1 \otimes I_k)y},$ where $h \in S \cap S^{((x_1 \otimes I_k)y)^{-1}}.$ The same arguments as above show that $h = \diag[h_1, \ldots, h_k]$ with $h_i \in H.$ Denote $h_i^{x_1}$ by $\hat{h}_i$ and $h^{x_1 \otimes I_k}=\diag[\hat{h}_1, \ldots, \hat{h}_k]$ by $\hat{h}$, so $g= \hat{h}^y.$ Let $\hat{h}_i = 
\left( \begin{smallmatrix}
h_{(i,1)} & h_{(i,2)}\\
h_{(i,3)} & h_{(i,4)}
\end{smallmatrix} \right),
$ where $h_{(i,j)} \in GL_{m/2}(q).$

Consider the last $(m \times m)$-row of $g=\hat{h}^y.$ Calculations show that it is $$(0, \ldots, 0, g_k)=(A_1, A_2, \ldots, A_k)$$ with 
{\small
\begin{equation*}
\begin{aligned}
A_i= & \begin{pmatrix} 0 & h_{(k,2)} \\ 0& h_{(k,4)}-h_{(k-1,4)} \end{pmatrix} & \text{ for } i \in \{1, \ldots, k-2\}, \\
 A_{k-1}= &  
 \begin{pmatrix} 0 & h_{(k,2)} \\ -h_{(k-1,3)}& h_{(k,3)}-h_{(k-1,4)} \end{pmatrix},& \\
 A_k = & \begin{pmatrix} h_{(k,1)} & h_{(k,2)} \\ h_{(k,3)}-h_{(k-1,3)}& h_{(k,4)} \end{pmatrix}.&
\end{aligned}
\end{equation*}
}
 So, $h_{(k-1,4)}=h_{(k,4)}$; $h_{(k,2)}=h_{(k-1,3)}=0$ and $\hat{h}_k=g_k.$
Consider the $(k-1)$-th $(m \times m)$-row of $g=\hat{h}^y.$ As above, we obtain
\begin{align*} h_{(k-2,4)}&=h_{(k-1,4)};\\ h_{(k-2,1)}&=h_{(k-1,1)};\\ h_{(k-2,3)}&=h_{(k-1,2)}=0
\end{align*} and $\hat{h}_{k-1}=g_{k-1}.$

Continuing in the same way we obtain for all $i,j \in \{1, \ldots, k\}:$
\begin{equation}\label{hhati4}
\begin{split}
\hat{h}_i&=g_i;\\
h_{(i,1)}&=h_{(j,1)};\\
h_{(i,4)}&=h_{(j,4)}.
\end{split}
\end{equation} 
Also $$
\hat{h}_1 = \begin{pmatrix} h_{(1,1)} & h_{(1,2)} \\ 0& h_{(1,4)}\end{pmatrix}; \text{ }
\hat{h}_k = \begin{pmatrix} h_{(k,1)} & 0 \\ h_{(k,3)} & h_{(k,4)}\end{pmatrix}$$ { and }
$$\hat{h}_i = \begin{pmatrix} h_{(i,1)} & 0 \\ 0 & h_{(i,4)}\end{pmatrix} \text{ for } 1<i<k. 
$$ Hence $g_i \in H \cap H^{x_1}.$ 

Assume now that $g \in S \cap S^{(x_2 \otimes I_k)z},$ $g=t^{(x_2 \otimes I_k)z},$ where $t \in S \cap S^{((x_2 \otimes I_k)z)^{-1}}.$ Using similar arguments to above, we obtain 
$$g=\diag [g_1, \ldots, g_k]= \diag[\hat{t}_1, \ldots, \hat{t}_k],$$
where $\hat{t}_i \in H^{x_2}$ is defined analogously to $\hat{h}_i \in H^{x_1}.$ In addition,
\begin{equation}\label{thati1}
 t_{(k-1,4)}=t_{(k,1)} \text{ and } t_{(k,3)}=t_{(1,2)}=0.
\end{equation}

Therefore, if $g \in S \cap S^{(x_1 \otimes I_k)y} \cap S^{(x_2 \otimes I_k)z},$ then, by \eqref{hhati4} and \eqref{thati1},  $g = \diag[g_1, \ldots, g_k]$ and $g_i= \diag[\delta,\delta] \in  H \cap H^{x_1} \cap H^{x_2}$  with $\delta \in GL_{m/2}(q)$ for all $ i \in \{1, \ldots, k\}.$ If $m=2$, then $g$ is scalar, so
$S \cap S^{(x_1 \otimes I_k)y} \cap S^{(x_2 \otimes I_k)z}\le Z(GSp_n(q))$.
If $m>2$, then $H$ is not a counterexample, so  $g_i \in H \cap H^{x_1} \cap H^{x_2} \le Z(GSp_m(q))$ and $g \in Z(GSp_n(q, {\bf f}_{\beta})).$

\medskip

{\bf Case (1.2).} Assume that the $V_i$ are non-degenerate and $k= 2$ in \eqref{BCvisp}. Let $H$ be $\Stab_S(V_1)|_{_{V_1}} \le GSp_m(q).$ Thus, either $H$ is not a counterexample, so there exist $x_1, x_2 \in Sp_m(q)$ such that $H \cap H^{x_1} \cap H^{x_2} \le Z(GSp_m(q))$, or  $H=GSp_2(q)$ with $q \in \{2,3\}$. The latter was discussed at the beginning of the proof. Assume the former holds.   
Let 
\begin{equation*} y= I_{m/2} \otimes
\begin{pmatrix}
1      & \multicolumn{1}{c|}{0}& -1& 0   \\
0      & \multicolumn{1}{c|}{1}& 0 & 0 \\ \cline{1-4}
0      & \multicolumn{1}{c|}{0}& 1 & 0   \\
0      & \multicolumn{1}{c|}{1}& 0 & 1
\end{pmatrix}, \text{ }
z= I_{m/2} \otimes
\begin{pmatrix}
1      & \multicolumn{1}{c|}{0}& 0 & 0   \\
0      & \multicolumn{1}{c|}{1}& 1 & 0 \\ \cline{1-4}
0      & \multicolumn{1}{c|}{0}& 1 & 0   \\
1      & \multicolumn{1}{c|}{0}& 0 & 1
\end{pmatrix}.
\end{equation*}
It is routine to check that $y,z \in Sp_n(q, {\bf f}_{\beta}).$
Denote $(V_i)y$ by $W_i$ and $(V_i)z$ by $U_i$ for $i=1,2.$ We claim that if $g \in S \cap S^{(x_1 \otimes I_2)y} \cap S^{(x_2 \otimes I_2)z}$, then $g$ stabilises $V_i$, $i=1,2.$ Assume the opposite, so $(V_1)g=(V_2).$ Therefore, $(W_1)g=W_2$ and $(U_1)g=U_2.$ Thus, 
$$(V_1 \cap W_1)g= (V_1)g \cap (W_1)g= (V_2 \cap W_2)$$ 
and 
$$(V_1 \cap U_1)g= (V_1)g \cap (U_1)g= (V_2 \cap U_2).$$
 Notice that $(V_2 \cap W_2)=(V_2 \cap U_2)$ but $(V_1 \cap W_1) \ne (V_1 \cap U_1)$ which is a contradiction since $g$ is invertible. Therefore, $g= \diag[g_1, g_2]$ where $g_i \in H.$  Also, $g=h^y$ where $h \in S^{y^{-1}} \cap S^{(x_1 \otimes I_2)}$ and $g=t^z$ where $t \in S^{z^{-1}} \cap S^{(x_2 \otimes I_2)}$. It is routine to check, using arguments as above, that $h=\diag[h_1, h_2]$ with $h_i \in H^{x_1}$ and $t=\diag[t_1,t_2]$ with $t_i \in H^{x_2}.$

Now calculations as in {\bf Case (1.1)} show that 
$$
g=
\begin{pmatrix}
h_{(1,1)}      & \multicolumn{1}{c|}{h_{(1,2)}}& 0& 0   \\
0      & \multicolumn{1}{c|}{h_{(1,4)}}& 0 & 0 \\ \cline{1-4}
0      & \multicolumn{1}{c|}{0}& h_{(2,1)} & 0   \\
0      & \multicolumn{1}{c|}{0}& h_{(2,3)} & h_{(2,4)}
\end{pmatrix}=
\begin{pmatrix}
t_{(1,1)}      & \multicolumn{1}{c|}{0}& 0& 0   \\
t_{(1,3)}      & \multicolumn{1}{c|}{t_{(1,4)}}& 0 & 0 \\ \cline{1-4}
0      & \multicolumn{1}{c|}{0}& t_{(2,1)} & 0   \\
0      & \multicolumn{1}{c|}{0}& t_{(2,3)} & t_{(2,4)}
\end{pmatrix}
$$
for some $h_{(i,j)}, t_{(i,j)} \in GL_{m/2}(q)$  with 
\begin{align*}
h_{(1,1)}=h_{(2,1)};& &t_{(1,1)}=t_{(2,4)}; \\
h_{(1,4)}=h_{(2,4)};& &t_{(1,4)}=t_{(2,1)}.
\end{align*}
So $$
g=
\begin{pmatrix}
g_{(1,1)}      & \multicolumn{1}{c|}{0}& 0& 0   \\
0      & \multicolumn{1}{c|}{g_{(1,1)}}& 0 & 0 \\ \cline{1-4}
0      & \multicolumn{1}{c|}{0}& g_{(1,1)} & 0   \\
0      & \multicolumn{1}{c|}{0}& g_{(2,3)} & g_{(1,1)}
\end{pmatrix}; \text{ }
g_1 =\begin{pmatrix}
g_{(1,1)}      & {0}   \\
0      & {g_{(1,1)}}
\end{pmatrix}; \text{ }
g_2=
\begin{pmatrix}
 g_{(1,1)} & 0   \\
 g_{(2,3)} & g_{(1,1)}
\end{pmatrix}
$$
with $g_1, g_2 \in H \cap H^{x_1} \cap H^{x_2}.$ Since $H \cap H^{x_1} \cap H^{x_2} \le Z(GSp_m(q))$, we obtain $g \in Z(GSp_n(q)).$ 

\medskip

{\bf Case 2.} Assume that the $V_i$ are totally isotropic, so $(2)$ of Lemma \ref{ashb} holds. If $k>2$, then $S$ stabilises the decomposition $V=U_1 \bot \ldots \bot U_{k/2}$ with $U_i$ non-degenerate, so {\bf Case 1} applies.

 Now assume $k=2.$
Let $H$ be $\Stab_S(V_1)|_{_{V_1}} \le GL_m(q).$ Thus, by Theorem \ref{irred}, either $H=GL_2(q)$ with $q \in \{2,3\}$ or  there exist $x_1, x_2 \in GL_m(q)$ such that $H \cap H^{x_1} \cap H^{x_2} \le Z(GL_m(q)).$ In the first case the theorem is verified by computation, so assume that the second case holds.

Fix $\beta$ to be a basis of $V$ as in \eqref{sympbasis} with $\langle f_1, \ldots, f_m \rangle=V_1$ and $\langle e_1, \ldots, e_m \rangle=V_2$. Let $y$ be $I_m \otimes \left( \begin{smallmatrix} 1& -1\\ 0& 1 \end{smallmatrix} \right),$ and let $X_i$ be $\diag[x_i, (x_i^{-1})^{\top}]$ for $i=1,2.$ Notice that $y, X_1, X_2 \in Sp_n(q,{\bf f}_{\beta}).$ Consider $g \in S \cap S^{X_1 y} \cap S^{X_2}.$ By the proof of Lemma \ref{igrekl}, $g=\diag[g_1,g_1]$ with $g_1 \in H \cap H^{x_1} \cap H^{x_2},$ so $g \in Z(GL_n(q)) \cap GSp_n(q)=Z(GSp_n(q)).$
\end{proof}

%% file: ch4.tex
\chapter{The general case}
\label{ch3}
 We prove Theorems (\ref{theorem}-\ref{theoremGR}), \ref{theoremGU} and (\ref{theoremSp}--\ref{theoremSpGR}) in Sections $\ref{ch3}.1$, $\ref{ch3}.2$ and $\ref{ch3}.3$ respectively. If $S \cap \GL_n(q^{\bf u})$ is irreducible, then  their proofs are either consequences of Theorem \ref{bernclass} or of the results obtained in Chapter \ref{ch2}.   So the main obstacle  is the situation when $S\cap \GL_n(q^{\bf u})$ stabilises a non-zero proper subspace of $V$. Our general strategy is to obtain three or four  conjugates of $S$ such that their intersection consists of elements of shape $(\phi_{\beta})^j \diag(\alpha_1, \ldots, \alpha_n)$ for some basis $\beta$ of $V$ (here $\phi_{\beta}$ is as defined in \eqref{defphibet}), and then use a technique similar to that used in the proof of Lemma \ref{diag} to construct another conjugate of $S$ such that the intersection of all of these conjugates consists of scalar matrices. This task is  particularly tricky  when $q \in \{2,3\}$ and  leads to case-by-case considerations. In case ${\bf L}$, since in general $b_S(S \cdot SL_n(q)) \le 5$, we also construct five distinct regular orbits in $\Omega^5$ to show that $\Reg_S(S \cdot SL_n(q),5)\ge 5$.  In cases {\bf U} and {\bf S} (apart from the situation $(n,q)=(5,2)$ verified by computation in Theorem \ref{theoremGU}) we show that $b_S(S \cdot (SL_n(q^{\bf u}) \cap X))\le 4$; Lemma \ref{base4} now implies that $\Reg_S(S \cdot (SL_n(q^{\bf u}) \cap X),5)\ge 5$.

\section{Linear groups}\label{secproof}

We prove Theorems \ref{theorem} and \ref{theoremGR} in Sections \ref{sec411} and \ref{sec412} respectively.

\subsection{Solvable subgroups contained in $\GL_n(q)$}
\label{sec411}

In this section $S$ is a maximal solvable subgroup of $\GL_n(q),$ $G=S \cdot SL_n(q)$, $H=S/Z(GL_n(q))$ and $\overline{G}=G/Z(GL_n(q)).$ Our goal is to prove the following theorem.

\begin{T1}
Let $X=\GL_n(q)$, $n \ge 2$ and $(n,q)$ is neither $(2,2)$ nor $(2,3).$ If $S$ is a maximal solvable subgroup of $X$, 
 then $\Reg_S(S \cdot SL_n(q),5)\ge 5$, in particular $b_S(S \cdot SL_n(q)) \le 5.$
\end{T1}

Before we start the proof, let us discuss the structure of a maximal solvable subgroup $S \le \GL_n(q)$ and fix some notation. 

Let $0< V_1 \le V$ be such that $(V_1)S=V_1$ and $V_1$ has no non-zero proper $S$-invariant subspace. It is easy to see that $S$ acts semilinearly on $V_1$. Let $$\gamma_1:S \to \GL_{n_1}(q),$$ where $n_1=\dim V_1,$ be the homomorphism defined by $\gamma_1:g \mapsto g|_{_{V_1}}.$  Since $V_1$ is $S$-invariant, $S$ acts (semilinearly) on $V/V_1.$ Let $V_1 < V_2 \le V$ be such that $(V_2/V_1)S=V_2/V_1$ and $V_2/V_1$ has no non-zero proper $S$-invariant subspace.  Observe that $S$ acts semilinearly on $V_2/V_1$. Let $\gamma_2:S \to \GL_{n_2}(q)$, where $n_2=\dim (V_2/V_1),$ be the homomorphism defined by $\gamma_2:g \mapsto g|_{_{(V_2/V_1)}}.$ 
Continuing this procedure we obtain the chain of subspaces 
\begin{equation}
\label{Gchain}
0=V_0<V_1< \ldots < V_k=V
\end{equation}
 and a sequence of homomorphisms $\gamma_1, \ldots, \gamma_k$ such that $V_{i}/V_{i-1}$ is $S$-invariant and has no  non-zero proper $S$-invariant  subspaces and $\gamma_i(S) \le \GL_{n_i}(q)$ is the restriction of $S$ to $V_{i}/V_{i-1}$ for $i \in \{1, \ldots, k\}.$

 \begin{Lem}
 \label{GammairGL}
 If $k=1$, then $S \cap GL_n(q)$ lies in an irreducible solvable subgroup of $GL_n(q).$
 \end{Lem}
 \begin{proof}
 Since $k=1,$  $V$ has no  non-zero proper $S$-invariant subspace. Let $M$ be $S \cap GL_n(q).$ Assume that $M$ is reducible, so there exists $0<U_1 <V$ such that $(U_1)M=M$ and $U_1$ is $\mathbb{F}_q[M]$-irreducible. Let $\varphi \in S$ be such that $\varphi M$ is a generator of $S/M.$  Let $U_2$ be $(U_1)\varphi,$ so, for $g \in M$,
$$(U_2)g=(U_1)\varphi g=(U_1)g^{\varphi^{-1}}\varphi=(U_1)\varphi=U_2$$
since $g^{\varphi^{-1}} \in M.$  Thus, $U_2$ is $M$-invariant and $U_1 \cap U_2=\{0\}$ since $(U_1 \cap U_2)M=(U_1 \cap U_2)$, $U_1$ is $\mathbb{F}_q[M]$-irreducible and $U_1 \ne U_2$. Here $U_1 \ne U_2$ since otherwise $V$ has an $S$-invariant non-zero proper subspace and $k>1$. Let $M_i$ be the restriction of $M$ on $U_i$. Since $M_2=M_1^{\varphi},$ $U_2$ is $\mathbb{F}_q[M]$-irreducible. Let $m=\dim U_1$. For $i \in \{1, \ldots, n/m\}$, the same argument shows  that $(U_1) \varphi^i$ is an $\mathbb{F}_q[M]$-irreducible submodule of $V$ and 
$$(U_1 \oplus \ldots \oplus U_{i-1}) \cap U_i=\{0\}.$$ So $M$ stabilises the decomposition
$$V=U_1 \oplus \ldots \oplus U_{n/m}.$$  In particular, $M$ lies in an imprimitive irreducible maximal solvable subgroup of $GL_n(q).$
 \end{proof}
 
 We start the proof of Theorem \ref{theorem} with the case $k=1.$

\begin{Th}
Theorem {\rm \ref{theorem}} holds for $k=1.$ 
\end{Th}
\begin{proof}
Let $M = S \cap GL_n(q)$. By Lemma \ref{GammairGL}, $M$ lies in an irreducible solvable subgroup of $GL_n(q).$
If $M$ is not a subgroup of one of the groups listed in  \eqref{irred11} -- \eqref{irred15} of Theorem \ref{irred}, then there exists $x \in SL_n(q)$ such that $M \cap M^x \le Z(GL_n(q)).$ So $H \cap H^{\overline{x}}$ is a cyclic subgroup of $\overline{G}$ and by Theorem \ref{zenab} there exists $\overline{y} \in \overline{G}$ such that 
$$(H \cap H^{\overline{x}}) \cap (H \cap H^{\overline{x}})^{\overline{y}}=1.$$ Hence $b_S(G) \le 4$ and $\Reg_S(G,5) \ge 5$ by Lemma \ref{base4}.
If  $M$ is a subgroup of one of the groups in \eqref{irred13} -- \eqref{irred15} of Theorem \ref{irred}, then $S=M$ and $b_S(G) \le 3.$ If  $M$ is a subgroup of one of the groups in \eqref{irred11} -- \eqref{irred12}, then $b_S(G) \le 3$ by \cite[Table 3]{burness} for $q>4$ and by computation for $q=4.$ So $\Reg_S(G,5) \ge 5$ by Lemma \ref{base4}.
\end{proof}

 For the rest of the section we assume that $k>1.$ Our proof for $q=2$ naturally splits into two cases.  If $q=2$ and $S_0$ is the normaliser of a Singer cycle of $GL_3(2)$, then there is no $x, y \in GL_3(2)$ such that $S_0^{x} \cap S_0^{y}$ is contained in $RT(GL_3(2))$ (see Theorem \ref{irred}). So, in Theorem \ref{lem41}, we assume that if $q=2$, then there is no  $i \in \{1, \ldots, k\}$ such that $\gamma_i(S)$ is the normaliser of a Singer cycle of $GL_3(2).$ In Theorem \ref{lem42} we address the case where there exists such an $i$.

\begin{Th}
\label{lem41}
Let $k>1.$ Theorem {\rm \ref{theorem}} holds if
\begin{itemize}
\item $q \ge 3,$ or
\item  $q=2$ and there is no  $i \in \{1, \ldots, k\}$ such that $\gamma_i(S)$ is the normaliser of a Singer cycle of $GL_3(2).$
\end{itemize}
\end{Th} 
\begin{proof}

Since $k>1,$ there exists a nontrivial $S$-invariant subspace $U<V$ of dimension $m<n$. Let $M$ be $S \cap GL_n(q).$ We fix  
\begin{equation}
\label{basGLG}
\beta=\{v_1, v_2, \ldots, v_{n-m+1}, v_{n-m+2}, \ldots, v_n \} 
\end{equation}
such that the last $ \left( \sum_{j=1}^i n_j \right)$  
 vectors in $\beta$ form a basis of $V_i$. So  $g \in M$ has  shape 
\begin{equation}
\label{stupG}
\begin{pmatrix}
\gamma_k(g)     & * & \ldots & * & * \\
0          &    \gamma_{k-1}(g)  & \ldots & * & * \\
     &     & \ddots & \ddots  &    \\
   0    &    \ldots& 0 & \gamma_{2}(g) & * \\
      0    &   \ldots        & \ldots& 0 & \gamma_{1}(g) 
\end{pmatrix}
\end{equation}
 with respect to $\beta.$ Recall that $q=p^f$, so if $f=1,$ then $\GL_n(q)=GL_n(q)$ and $S=M.$ Let  $\phi \in \GL_n(q)$ be such that \begin{equation}
\label{phidef}
\phi : \lambda v_i \mapsto \lambda^p v_i \text{ for all } i \in \{1, \ldots, n\} \text{ and } \lambda \in \mathbb{F}_q.
\end{equation}

Assume that $f>1.$  Let $x= \diag[x_k, \ldots, x_1]$, where $x_i\in SL_{n_i}(q)$ is such that $\gamma_i(M) \cap (\gamma_i(M))^{x_i} \le RT(GL_{n_i}(q))$. Such $x_i$ exist by Theorem \ref{irred} since $\gamma_i(M)$ lies in an irreducible solvable subgroup of $GL_{n_i}(q)$ by Lemma \ref{GammairGL}. Let $y$ be as in the proof of Lemma \ref{irrtog}. Recall that  we defined $\gamma_i$ only on $S$, but it is easy to see that $\gamma_i$ can be extended to $\Stab_{\GL_n(q)}(V_i, V_{i-1})$ since $\gamma_i$ is the restriction on $V_i/V_{i-1}.$

  We present the following piece of the proof as a proposition for easy reference. 
\begin{Prop} 
\label{clm44}
There exists $\beta$ as in \eqref{basGLG} such that for every $\varphi \in (S \cap S^x) \cap (S \cap S^x)^y$, 
$$\varphi=\phi^j g \text{ for } j \in \{0,1, \ldots,f-1\}$$
where $g \in GL_n(q)$ is diagonal.
 Moreover, if  $\gamma_i(M) \cap \gamma_i(M)^{x_i} \le Z(GL_{n_i}(q)),$ then $\gamma_i(g)$ is scalar. 
\end{Prop}
\begin{proof}
Consider $\varphi \in S \cap S^x.$ If $n_i=2$ and $\gamma_i(S)$ lies in the normaliser of a Singer cycle of $GL_2(q)$ in $\GL_2(q)$, then by \eqref{2sindiagodd} and \eqref{2sindiageven} we can choose $x_i \in SL_2(q)$ such that $\gamma_i(S) \cap \gamma_i(S)^{x_i}\le RT(GL_2(q))$.  If $\varphi$ acts linearly on $V_i/V_{i-1},$ then $\varphi$ acts linearly on $V$, so $S \cap S^x \le GL_n(q)$.   Otherwise, by Lemma \ref{scfield}, there exists a basis of $V_i/V_{i-1}$ such that $\gamma_i(\varphi)= \phi^{j_i} g_i$ with scalar $g_i$ and $j_i \in \{0,1, \ldots , f-1 \}.$ Hence we can choose $\beta$ in \eqref{basGLG} such that if $\varphi \in S \cap S^x,$ then $\gamma_i(\varphi)= \phi^{j_i} g_i$ with scalar $g_i.$ Since $(\phi^{j_1})^{-1} \varphi$ acts on $V_1\ne 0$ linearly,   $(\phi^{j_1})^{-1} \varphi \in GL_n(q),$ so $j_i=j_l$ for all $i,l \in \{1, \ldots, k\}$
and 
\begin{equation}
\label{phitri}
\varphi = \phi^j 
\begin{pmatrix}
g_k     & * & \ldots & * & * \\
0          &    g_{k-1}  & \ldots & * & * \\
     &     & \ddots & \ddots  &    \\
   0    &    \ldots& 0 & g_{2} & * \\
      0    &   \ldots        & \ldots& 0 & g_{1}
\end{pmatrix}
\end{equation}
 with $g_i \in Z(GL_{n_i}(q))$ if $n_i \ne 2$ and $g_i \in RT(GL_2(q))$ if $n_i=2.$ 


Since $\varphi \in (S \cap S^x)^y,$
$$
\varphi = \phi^j 
\begin{pmatrix}
g_1'     & 0 & \ldots & 0 & 0 \\
*          &    g_{k-1}'  & \ldots & 0 & 0 \\
     &     & \ddots & \ddots  &    \\
   *    &    \ldots& * & g_{2}' & 0 \\
      *    &   \ldots        & \ldots& * & g_{1}'
\end{pmatrix}
$$ with $g_i$ either scalar or lower-triangular. So 
\begin{equation}
\label{phidiagG}
\varphi =\phi^j g \text{ where } g=\diag[g_k, \ldots, g_1]
\end{equation}
 with $g_i$ scalar if $n_i \ne 2$ and $g_i$ diagonal otherwise. 
\end{proof}

\medskip

We now resume our proof of Theorem \ref{lem41}.
Let $\Omega$ be the set of right $S$-cosets in $G$ and let $\overline{\Omega}$ be the set of right $H$-cosets in $\overline{G}$, where the action is given by  right multiplication.  Since $S$ is maximal solvable and $$S_{G}:=\cap_{g \in G} S^g=Z(GL_n(q)),$$ $\Reg_S(G,5)$ is the number of $\overline{G}$-regular orbits on  $\overline{\Omega}^5$.   Therefore, $$\overline{\omega}=(H\overline{g_1},H\overline{g_2},H\overline{g_3},H\overline{g_4},H\overline{g_5}) \in \overline{\Omega}^5$$ is regular under the action of $\overline{G}$ if and only if the stabiliser of $$ \omega =(Sg_1, Sg_2, Sg_3,Sg_4, Sg_5) \in \Omega^5$$ under the action of $G$ is equal to $Z(GL_n(q))$, where $g_i$ is a preimage in $\GL_n(q)$ of $\overline{g_i}.$  
In our proof we say that   $\omega \in \Omega^5$ is regular if it is stabilised only by elements from $Z(GL_n(q))$, so $\overline{\omega}$ is regular in terms of Definition \ref{def1} under the induced  action of $\overline{G}.$  

The proof of Theorem \ref{lem41} splits into five cases. To show $\Reg_S(G,5) \ge 5$, in each case we find five regular orbits in $\Omega^5$ or show that $b_S(G)\le 4,$ so $\Reg_S(G,5) \ge 5$ by Lemma \ref{base4}. Recall that $n_i = \dim V_i/V_{i-1}$ for $i \in \{1, \ldots, k\}.$  Different cases arise according to the number of $n_i$ for $i \in \{1, \ldots, k\}$ which equal 2  and in what rows of  $g \in M=S\cap GL_n(q)$ the submatrix $\gamma_i(g)$ is located for $n_i=2.$ We now specify the cases.
\begin{description}[before={\renewcommand\makelabel[1]{\bfseries ##1}}]
\item[{\bf Case 1.}] Either $f>1$, or $f=1$ and the number of $i \in \{1, \ldots, k\}$ with $n_i=2$ is not one;
\item[{\bf Case 2.}] $f=1, $ $n$ is even, there exist exactly one $i \in \{1, \ldots, k\}$ with $n_i=2$ and $\gamma_i(g)$ appears in rows  $\{n/2,n/2+1\}$ of $g \in M$ for such $i$;
\item[{\bf Case 3.}] $f=1, $ $n$ is odd, there exist exactly one $i \in \{1, \ldots, k\}$ with $n_i=2$ and $\gamma_i(g)$ appears in rows  $\{(n+1)/2,(n+1/2)+1\}$ of $g \in M$ for such $i$;
\item[{\bf Case 4.}] $f=1, $ $n$ is odd, there exist exactly one $i \in \{1, \ldots, k\}$ with $n_i=2$ and $\gamma_i(g)$ appears in rows  $\{(n-1)/2,(n+1)/2+1\}$ of $g \in M$ for such $i$;
\item[{\bf Case 5.}] $f=1,$ there exist exactly one $i \in \{1, \ldots, k\}$ with $n_i=2$ and none of {\bf Cases 2 -- 4} holds. 
\end{description}


Before we proceed with the proof of Theorem \ref{lem41}, let us resolve two situations which both arise often in our analysis of these cases and can be readily settled.

\begin{Prop}
\label{ni1}
Let $S$ be a maximal solvable subgroup of $\GL_n(q),$ let $\beta$ be as in \eqref{basGLG},  $n \ge 2$ and $(n,q)$ is neither $(2,2)$ nor $(2,3).$ If  all $V_i$ in \eqref{Gchain} are such that $n_i=1,$ then there exist $x,y,z \in SL_n(q)$ such that $S \cap S^x \cap S^y \cap S^z \le Z(GL_n(q)).$ 
\end{Prop}
\begin{proof}
Let $\sigma =(1, n)(2, n - 1) \ldots ([n/2], [n/2 + 3/2]) \in \Sym(n).$ We define $x,y,z \in SL_n(q)$ as follows:
\begin{itemize}
 \item $x = \diag(\sgn(\sigma), 1\ldots, 1) \cdot \per (\sigma);$ 
\item $(v_i)y=v_i$ for $i \in \{1, \ldots, n-1\}$ and $(v_n)y= \sum_{i=1}^n v_i$;
\item $(v_i)z=v_i$ for $i \in \{1, \ldots, n-1\}$ and $(v_n)z= \theta v_1 + \sum_{i=2}^n v_i$, 
\end{itemize}
where $\theta$ is a generator of $\mathbb{F}_q^*.$ Let $\varphi \in S \cap S^x \cap S^y \cap S^z.$ Since $\varphi \in S \cap S^x,$ it stabilises $\langle v_i \rangle $ for all $i \in \{1, \ldots, n\}.$ So $\varphi =\phi^j \diag(\alpha_1, \ldots, \alpha_n)$ for $j \in \{0,1, \ldots, f-1\}$ and $\alpha_i \in \mathbb{F}_q^*.$ Since $\varphi \in S^y,$ it stabilises $(V_1)y =\langle v_n \rangle y= \langle v_1 + \ldots +v_n \rangle.$ So $\alpha_1=\alpha_i$ for $i \in \{2, \ldots, n\}.$  Since $\varphi \in S^z,$ it stabilises $(V_1)z =\langle v_n \rangle z= \langle \theta v_1 + v_2 + \ldots +v_n \rangle.$ So $\theta^{p^j}=\theta$ and $j=0.$ Hence $\varphi \in Z(GL_n(q)).$ 
\end{proof}

\begin{Prop}
\label{n3GL}
Fix a basis $\beta$ as in \eqref{basGLG}. Let $n \in \{2,3\}$ and $(n,q)$ is neither $(2,2)$ nor $(2,3)$. Let $S$ be a maximal solvable subgroup of $\GL_n(q).$ If $S$ stabilises a  non-zero proper subspace $V_1$ of $V$, then there exist $x,y,z \in SL_n(q)$ such that 
 $$S \cap S^x \cap S^y \cap S^z \le Z(GL_n(q)).$$
\end{Prop}
\begin{proof}
We can assume that $V_1$ has no non-zero proper $S$-invariant subspaces. Let $V_i$ be as in \eqref{Gchain}.

If $n=2$, then $\dim V_1=1$ and $n_1=n_2=1$, so the statement follows by Proposition \ref{ni1}

Assume $n=3$.  If $k=3,$ so $n_i=1$ for all $i$, then the statement follows by Proposition \ref{ni1}, so we assume $k=2.$

Let $n_1=1,$ so $n_2=2.$ Let $f>1.$ Assume that $\gamma_2(S)$ is a subgroup of the normaliser of a Singer cycle of $GL_2(q)$ in $\GL_2(q).$ Therefore, by \eqref{2sindiagodd} and \eqref{2sindiageven}, there exists $x_1 \in SL_2(q)$ such that $\gamma_2(S) \cap (\gamma_2(S))^{x_1} \le \langle \varphi \rangle Z(GL_2(q))$ where $\varphi = \diag(-1,1)$ if $q$ is odd and $\varphi = \begin{pmatrix}
1 & 0\\
1 & 1
\end{pmatrix}$ if $q$ is even. Hence, if $x=\diag[x_1,1],$ then $S \cap S^x \le GL_3(q)$ and matrices in $S \cap S^x$ have shape 
$$
 \begin{pmatrix}
\alpha_1 \varphi & *\\
0 & \alpha_2
\end{pmatrix}
$$ with $\alpha_1, \alpha_2 \in \mathbb{F}_q^*.$ Let $y=\per((1,2,3))$ and 
$$
z=
\begin{cases}
\begin{pmatrix}
1&0&0\\
0&1&0\\
1&1&1
\end{pmatrix} \text{ if $q$ is odd;}\\
\begin{pmatrix}
0&1&0\\
0&0&1\\
1&0&0
\end{pmatrix} \text{ if $q$ is even.}\\
\end{cases}
$$
Calculations show that $S \cap S^x \cap S^y \cap S^z \le Z(GL_n(q)).$

If $\gamma_2(S)$ does not normalise a Singer cycle of $GL_2(q),$ then, by Theorem \ref{irred}, 
   there exists $x_1 \in SL_2(q)$ such that $\gamma_2(S) \cap (\gamma_2(S))^{x_1} \cap GL_2(q) \le Z(GL_2(q)).$ Let $x=\diag[x_1,1]$ so $\varphi \in S \cap S^x$ has shape \eqref{phitri} with $g_2 \in Z(GL_2(q))$ and $g_1 \in \mathbb{F}_q^*.$ In particular, let 
$$
\varphi =\phi^j
\begin{pmatrix}
\alpha_1 & 0 & \delta_1\\
 0&     \alpha_1 & \delta_2\\
0 & 0& \alpha_2  
\end{pmatrix}.
$$
 Let $y=
\begin{pmatrix}
0&1&0\\
1&0&0\\
\theta&1&1
\end{pmatrix}$ and consider $\varphi \in (S \cap S^x) \cap (S \cap S^x)^y.$ Since $\varphi \in S^y,$ it stabilises $$(V_1)y=\langle v_3 \rangle y= \langle \theta v_1 +v_2 +v_3 \rangle.$$ Therefore,
$$(\theta v_1 +v_2 +v_3)\varphi = \theta^{p^j} \alpha_1 v_1 + \alpha_1 v_2 + (\alpha_2 +\delta_1 +\delta_2)v_3 \in \langle \theta v_1 +v_2 +v_3 \rangle,$$
so $\theta^{p^j} \alpha_1= \theta \alpha_1$ and $j=0$. Thus, $(S \cap S^x) \cap (S \cap S^x)^y \le GL_3(q).$ Now calculations  show that $(S \cap S^x) \cap (S \cap S^x)^y \le Z(GL_3(q)).$

Now assume $f=1$ (we  continue to assume that $n_1=1$). Let $\sigma=(1,3) \in \Sym(3)$ and $x=\diag(\sgn(\sigma),1,1) \cdot {\rm perm}(\sigma) \in SL_3(q)$. Matrices from $S \cap S^x$ have shape  
  \begin{equation*} 
  \begin{pmatrix}
*    & 0 & 0  \\
* &*    & *  \\
  0&    0    &   *      
\end{pmatrix}.
\end{equation*}
Let   
  \begin{equation*} 
y=  \begin{pmatrix}
0    & -1 & 0  \\
1 &0    & 0  \\
  1&    1    &   1      
\end{pmatrix} \in SL_3(q).
\end{equation*}
It is easy to check that $$(S\cap S^x) \cap (S\cap S^x)^y \le Z(GL_3(q)).$$

To complete the proof, we assume  $n_1=2,$ so $n_2=1.$ We consider the cases  $f>1$ and $f=1$ separately. First assume $f>1.$ Assume that $\gamma_1(S)$ is a subgroup of the normaliser of a Singer cycle of $GL_2(q)$ in $\GL_2(q).$ Therefore, by \eqref{2sindiagodd} and \eqref{2sindiageven}, there exists $x_1 \in SL_2(q)$ such that $\gamma_1(S) \cap (\gamma_1(S))^{x_1} \le \langle \varphi \rangle Z(GL_2(q))$ where $\varphi = \diag(-1,1)$ if $q$ is odd and $\varphi = \begin{pmatrix}
1 & 0\\
1 & 1
\end{pmatrix}$ if $q$ is even. Hence, if $x=\diag[1,x_1],$ then $S \cap S^x \le GL_3(q)$ and matrices in $S \cap S^x$ have shape 
$$
 \begin{pmatrix}
\alpha_1 & *\\
0 & \alpha_2 \varphi
\end{pmatrix}
$$ with $\alpha_1, \alpha_2 \in \mathbb{F}_q^*.$ Let 
$$
y=
\begin{cases}
\begin{pmatrix}
0&-1&0\\
1&0&0\\
1&1&1
\end{pmatrix} \text{ if $q$ is odd;}\\
\begin{pmatrix}
0&0&1\\
1&1&1\\
1&0&0
\end{pmatrix} \text{ if $q$ is even.}\\
\end{cases}
$$ Calculations show that $(S \cap S^x) \cap (S \cap S^x)^y \le Z(GL_3(q)).$ 

If $\gamma_1(S)$ does not normalise a Singer cycle of $GL_2(q),$ then, by Theorem \ref{irred}, 
   there exists $x_1 \in SL_2(q)$ such that $\gamma_1(S) \cap (\gamma_1(S))^{x_1} \cap GL_2(q) \le Z(GL_2(q)).$ Let $x=\diag[1,x_1]$ so $\varphi \in S \cap S^x$ has shape \eqref{phitri} with $g_1 \in Z(GL_2(q))$ and $g_2 \in \mathbb{F}_q^*.$ In particular, let 
$$
\varphi = \phi^j
\begin{pmatrix}
\alpha_1 & \delta_1 & \delta_2\\
 0&     \alpha_2 & 0\\
0 & 0& \alpha_2  
\end{pmatrix}.
$$ Notice that $S \cap S^x$ stabilises $\langle v_2\rangle$, $\langle v_3\rangle$ and $\langle v_2+v_3\rangle.$
 Let $y=
\begin{pmatrix}
0&0&-1\\
1&0&0\\
0&1&\theta
\end{pmatrix}$ and consider $\varphi \in (S \cap S^x) \cap (S \cap S^x)^y.$ Since $\varphi \in (S \cap S^x)^y,$ it stabilises $$\langle  v_3 \rangle y= \langle  v_2 + \theta v_3 \rangle.$$ Therefore,
$$(v_2 + \theta v_3)\varphi =  \alpha_2 v_2 + \theta^{p^j} \alpha_2 v_3 \in \langle  v_2 + \theta v_3 \rangle,$$
so $\theta^{p^j} \alpha_2= \theta \alpha_2$ and $j=0$. Thus, $(S \cap S^x) \cap (S \cap S^x)^y \le GL_3(q).$ Now calculations  show that $(S \cap S^x) \cap (S \cap S^x)^y \le Z(GL_3(q)).$
  
Finally, assume $n_1=2$ and $f=1$. Let $x \in SL_3(q)$ be $\diag(-1,1, 1) \per(\sigma)$ with $\sigma=(1,3)$ and let 
  \begin{equation*} 
y=  \begin{pmatrix}
1    & 0 & 0  \\
1 &1    & 1 \\
  0&    0    &   1      
\end{pmatrix} \in SL_3(q).
\end{equation*}
Calculations show that 
\begin{equation*} (S\cap S^x) \cap (S\cap S^x)^y \le Z(GL_3(q)).
 \qedhere
\end{equation*}
\end{proof} 

We now consider {\bf Cases 1} -- {\bf 5}, resolving each in turn via a proposition.

\begin{Prop}
\label{case1prop}
Theorem $\ref{lem41}$ holds in {\bf Case 1}.
\end{Prop} 
\begin{proof}
In this case, there are four conjugates of $M=S \cap GL_n(q)$ whose intersection lies in $D(GL_n(q)).$ Indeed, if $f>1,$ or $f=1$ and there is no $i \in \{1, \ldots, k\}$ (here $k$ is as in \eqref{Gchain}) such that $n_i=2$, then 
 there exist $x,y \in SL_n(q)$ such that 
$$(M \cap M^x) \cap (M \cap M^x)^y \le D(GL_n(q))$$
by Proposition \ref{clm44}  or Lemma \ref{irred}  and Lemma \ref{irrtog}.

Assume $f=1$ (so $S=M$) and there exists at least two $i \in \{1, \ldots, k\}$ such that $n_i=2.$ Let $t\ge 2$ be the number of such $i$-s and let $j_1, \ldots, j_t \in \{1, \ldots, n\}$ be such that    $(2 \times 2)$ blocks (corresponding to $V_i/V_{i-1}$ of dimension $2$) on the diagonal in matrices of $M$ occur in the rows $$(j_1, j_1+1), (j_2, j_2 +1), \ldots, (j_t, j_t+1).$$ 
Let $\tilde{x}=\diag(\sgn(\sigma), 1, \ldots, 1) \cdot {\rm perm}(\sigma) \in SL_n(q),$ where  $$\sigma=(j_1, j_1+1, j_2, j_2 +1, \ldots, j_t, j_t+1)(j_1, j_1+1).$$  Let $x= \diag[x_k, \ldots, x_1]\tilde{x}$, where $x_i\in SL_{n_i}(q)$ is such that $\gamma_i(M) \cap (\gamma_i(M))^{x_i} \le RT(GL_{n_i}(q))$ if $n_i \ne 2$ and $x_i$ is the identity matrix if $n_i=2.$ Such $x_i$ exist by Lemma \ref{irred}. If $y$ is as in Lemma \ref{irrtog}, then calculations show that  
$$(M \cap M^x) \cap (M \cap M^x)^y \le D(GL_n(q)).$$


Let $0<U <V$ be an $S$-invariant subspace and $\dim U=m.$ In general, we take $U=V_i$ for some $i<k$ (in most cases it is sufficient to take $U=V_1$). We fix $\beta$ to be as in \eqref{basGLG} such that Proposition \ref{clm44} holds for $f>1.$ So 
$$U = \langle v_{n-m+1}, v_{n-m+2}, \ldots, v_n \rangle.$$ Since in the proof we reorder $\beta,$ let us fix a second notation for  $$v_{n-m+1}, v_{n-m+2}, \ldots, v_n,$$ namely 
$w_1, \ldots, w_m$
respectively. So $U = \langle w_1, \ldots, w_m \rangle.$

  Our proof splits into the following four  subcases:
\begin{description}[before={\renewcommand\makelabel[1]{\bfseries ##1}}]
\item[{\bf Case (1.1)}] $n-m\ge 2 $ and $m \ge 2$;
\item[{\bf Case (1.2)}] $n-m=m=1$;
\item[{\bf Case (1.3)}] $n-m \ge 2$ and $m=1$;
\item[{\bf Case (1.4)}] $n-m=1$  and $m \ge 2$.
\end{description}

{\bf Case (1.1).} Let $n-m\ge 2 $ and $m \ge 2$, so $n \ge 4$. We claim that there exist $z_i \in SL_n(q)$ for $i=1, \ldots, 5$ such that the points $\omega_i=(S,Sx,Sy,Sxy,Sz_i)$ lie in distinct regular orbits. 

 Recall that, if either $f=1$ or  $f>1$ and  $\gamma_i(S)$ normalises a Singer cycle of $GL_2(q)$ for some $i \in \{1, \ldots, k\}$, then $$(S \cap S^x) \cap (S \cap S^x)^y \le GL_n(q).$$ Also, we can assume that $n_l \ge 2$ for some $l \in \{1, \ldots, k\}$ in \eqref{Gchain}: otherwise Theorem~\ref{lem41} follows by Proposition \ref{ni1}. Let $s \in \{1, \ldots, n\}$ be such that $$V_l = \langle v_s, v_{s+1}, \ldots, v_{s+n_l-1}, V_{l-1} \rangle.$$

 If $(S \cap S^x) \cap (S \cap S^x)^y$ does not lie in $GL_n(q)$, then   $\varphi \in (S \cap S^x) \cap (S \cap S^x)^y$ stabilises $W=\langle v_s, v_{s+1}, \ldots, v_{s+n_l-1} \rangle$ and the restriction $\varphi|_{_{W}}$ is $\gamma_l(\varphi)=\phi^j g_l$ where $g_l =\lambda I_{n_l}$ for some $\lambda \in \mathbb{F}_q^*$ by \eqref{phidiagG}. We now relabel vectors in $\beta$ as follows: $v_s, v_{s+1}, \ldots, v_{s+n_l-1}$ becomes $v_1, v_{2}, \ldots, v_{n_l}$ respectively; $v_1, v_{2}, \ldots, v_{n_l}$  becomes $v_s, v_{s+1}, \ldots, v_{s+n_l-1}$  respectively;  the remaining labels are unchanged.

  Let $z_1, \ldots, z_5 \in SL_n(q)$ be such that $$(w_{j})z_i=u_{(i,j)}$$ for $i=1, \ldots, 5$ and $j=1, \ldots, m,$ where

\begingroup
\allowdisplaybreaks
\begin{align*}
& \left\{ 
\begin{aligned}
u_{(1,1)} & =\theta v_1+v_2;\\
u_{(1,2)} & =\theta v_1+v_3+v_4;\\
u_{(1,2+r)} & =\theta v_1+v_{4+r}\phantom{;} \text{ for } r \in \{1, \ldots, m-3\};\\
u_{(1,m)} & =\theta v_1 +\sum_{r=2+m}^{n}v_r;
\end{aligned}
\right. \\
& \left\{ 
\begin{aligned}
u_{(2,1)} & =\theta v_1+v_3;\\
u_{(2,2)} & =\theta v_1+v_2+v_4;\\
u_{(2,2+r)} & =\theta v_1+v_{4+r}\phantom{;} \text{ for } r \in \{1, \ldots, m-3\};\\
u_{(2,m)} & =\theta v_1 +\sum_{r=2+m}^{n}v_r;\\
\end{aligned}
\right. \\
& \left\{ 
\begin{aligned}
u_{(3,1)} & =\theta v_1+v_4;\\
u_{(3,2)} & =\theta v_1+v_2+v_3;\\
u_{(3,2+r)} & =\theta v_1+v_{4+r}\phantom{;} \text{ for } r \in \{1, \ldots, m-3\}; \\
u_{(3,m)} & =\theta v_1 +\sum_{r=2+m}^{n}v_r;
\end{aligned}
\right. \stepcounter{equation}\tag{\theequation} \label{orb}  \\
& \left\{ 
\begin{aligned}
u_{(4,1)} & =v_2+v_3;\\
u_{(4,2)} & =\theta v_1+v_2+v_4;\\
u_{(4,2+r)} & =v_2+v_{4+r}\phantom{;} \text{ for } r \in \{1, \ldots, m-3\};\\
u_{(4,m)} & =v_2 +\sum_{r=2+m}^{n}v_r;
\end{aligned}
\right. \\
& \left\{ 
\begin{aligned}
u_{(5,1)} & =v_2+v_4;\\
u_{(5,2)} & =\theta v_1+v_2+v_3;\\
u_{(5,2+r)} & =v_2+v_{4+r}\phantom{;} \text{ for } r \in \{1, \ldots, m-3\};\\
u_{(5,m)} & =v_2 +\sum_{r=2+m}^{n}v_r.
\end{aligned}
\right.  
\end{align*}
\endgroup

Recall that $\theta$ is a generator of $\mathbb{F}_q^*.$ Such $z_i$ always exist since $m<n$ and $u_{(i,1)}, \ldots, u_{(i,m)}$ are linearly independent  for every $i =1 ,\ldots, 5$. 
Consider $\varphi \in (S \cap S^x) \cap (S \cap S^x)^y \cap S^{z_1}$, so $\varphi=\phi^jg$ as in \eqref{phidiagG} with $g=\diag({\alpha_1, \ldots, \alpha_n})$ for $\alpha_i \in \mathbb{F}_q^*.$ Notice that $\varphi$ stabilises the subspace $Uz_1.$ 
Therefore, 
$$(u_{(1,1)})\varphi =
\begin{cases}
 \theta^{p^j} \alpha_1 v_1 + \alpha_2 v_2\\
 \mu_1 u_{(1,1)} + \ldots + \mu_m u_{(1,m)}
\end{cases}
$$ for $\mu_i \in \mathbb{F}_q^*.$ The first line does not contain $v_i$ for $i>2$, so $\mu_i=0$ for $i>1$ and $(u_{(1,1)})\varphi=\mu_1 u_{(1,1)}= \mu_1(\theta v_1 +v_2).$ Therefore, $\mu_1= \alpha_2.$  Assume  that $\varphi \notin GL_n(q)$, so, by the arguments before \eqref{orb}, $\alpha_1 =\alpha_2$,  $\theta^{p^j-1}=1$ and $j=0$ which is a contradiction. Thus, we can assume $\varphi=g \in GL_n(q).$
 The same arguments as in the proof of Lemma \ref{diag} show that $$(u_{(1,j)} )g = \delta_j u_{(1,j)} \text{ for all } j\in \{1,\ldots, m\}$$ for some $\delta_j \in \mathbb{F}_q,$. Since all $u_{(1,j)}$ have $v_1$ in the decomposition \eqref{orb}, all $\delta_j$ are equal and $g$ is a scalar matrix, so $\omega_1$ is a regular point in $\Omega^5.$ It is routine to check that every $\omega_i \in \Omega^5$ is regular. 

Assume that $\omega_1$ and $\omega_2$ lie in the same orbit in $\Omega^5$, so there exists $\varphi \in \GL_n(q)$ such that $\omega_1 \varphi =\omega_2$. This implies $$\varphi \in (S \cap S^x) \cap (S \cap S^x)^y \cap z_1^{-1}Sz_2.$$
In particular, $\varphi= \phi^j g$ with $g =\diag(\alpha_1, \ldots, \alpha_n) \in D(GL_n(q))$ and $g=z_1^{-1}\psi z_2$, $\psi \in S$. Consider $(u_{(1,1)})\varphi$.  Firstly, 
\begin{equation}\label{orr}
(u_{(1,1)})\varphi=(\theta v_1+v_2)\varphi= \theta^{p^j} \alpha_1 v_1 + \alpha_2 v_2;
\end{equation}
on the other hand,
$$(u_{(1,1)})\varphi =(u_{(1,1)})z_1^{-1}\psi z_2=(w_{1})\psi z_2=(\eta_{1} w_{1} + \ldots +\eta_{m} w_{m})z_2 \in Uz_2,$$
for some $\eta_i \in \mathbb{F}_q$. So $(u_{(1,1)})\varphi= \delta u_{(2,2)}$ for some $\delta \in \mathbb{F}_q$ since there are no $v_i$ for $i>2$ in decomposition \eqref{orr}. However, 
$$\theta^{p^j} \alpha_1 v_1 + \alpha_2 v_2=\delta u_{(2,2)}$$
if and only if $(u_{(1,1)})\varphi=0$, which contradicts the fact that $\varphi$ is invertible. Hence  there is no such $\varphi$, so points $\omega_1$ and $\omega_2$ are in  distinct regular orbits on $\Omega^5.$ The same arguments show that all $\omega_i$ lie in  distinct regular orbits on $\Omega^5.$   

\medskip

{\bf Case (1.2).} If $n-m=m=1$, then $n=2$ and  Theorem \ref{lem41} follows by Proposition \ref{n3GL}.

\medskip

{\bf Case (1.3).} Let $m=1$ and $n-m \ge 2$, so $n \ge 3.$ 
If $n=3$, then Theorem \ref{lem41} follows by Proposition \ref{n3GL}, so we assume that $m=1$ and $n \ge 4.$ If $k\ge 4$ in \eqref{stup}, then $U = \langle v_{n_1+n_2+1}, \ldots, v_n \rangle$ is an $S$-invariant subspace with $m=n_3 + \ldots + n_k \ge 2$ and $n-m=n_1 +n_2 \ge 2$, so the proof as in {\bf Case (1.1)} using \eqref{orb} works.

 Assume that $k=3.$ If $n_3 \ge 2$, then    $U = \langle v_{n_3+1}, \ldots, v_n \rangle$ is an $S$-invariant subspace with $m=n_1 + n_2 \ge 2$ and $n-m=n_3 \ge 2$, so the proof as in {\bf Case (1.1)} using \eqref{orb} also works.

Let $n_3=1.$ Since $m=1$, $n_1=1.$ If $(n_2,q)=(2,2)$ or $(n_2,q)=(2,3)$, then Theorem~\ref{lem41} is verified by computation. Otherwise, by Theorem \ref{irred} and Lemma \ref{scfield}, there exist $\tilde{x}_1,\tilde{x}_2 \in SL_{n_2} (q)$  such that $$\gamma_2(S) \cap (\gamma_2(S))^{\tilde{x}_1} \cap (\gamma_2(S))^{\tilde{x}_2} \le \langle \phi^i \rangle Z(GL_{n_2}(q))$$ for some $i \in \{0, 1, \ldots, f-1\}$. Let $x_1, x_2 \in GL_n(q)$ be the matrices $\diag[1,\tilde{x}_1,1]$ and $\diag[1,\tilde{x}_2,1]$ respectively.  Let $\tilde{x}= \diag(\sgn(\sigma),1 \ldots, 1) \cdot \per(\sigma) \in SL_n(q)$ with $\sigma=(1,n)$. Calculations show that $(S \cap S^{x_1} \cap S^{x_2 \tilde{x}}) \cap GL_n(q) \le D(GL_n(q)).$ So, with respect to a basis $\beta$ as in \eqref{basGLG},  $\varphi \in S \cap S^{x_1} \cap S^{x_2 \tilde{x}}$ has shape
$$\phi^j \diag(\alpha_1, \alpha_2, \ldots, \alpha_2, \alpha_3)$$ with $j \in \{0,1, \ldots, f-1\}$ and $\alpha_1, \alpha_2, \alpha_3 \in \mathbb{F}_q^*.$ Let 
$$y=
\begin{pmatrix}
1& 0      & 0\ldots  & 0 &0 \\
0& 1      & 0 \ldots  & 0 &0\\
&       &\ddots &  &  &\\
&       & & \ddots &  &\\
0& 0      & 0& \ldots & 1 &0 \\
1& \theta & 1& \ldots & 1 &1
\end{pmatrix}.$$
Calculations show that $S \cap S^{x_1} \cap S^{x_2 \tilde{x}} \cap S^y \le Z(GL_n(q)).$ 

Now let $k=2$, so $n_2=n-m=n-1$ and $n_2>2$. If $(n,q)$ is $(4,2)$ or $(5,3)$, then Theorem \ref{lem41} is verified by  computation.  Otherwise,  by Theorem \ref{irred} and Lemma \ref{scfield},
there exists $x \in SL_n(q)$ such that $S \cap S^x$ consists of elements of  shape (with respect to a basis $\beta$ as in \eqref{basGLG})
\begin{equation*} 
\phi^j
  \begin{pmatrix}
\alpha_1 &0 & \dots    & 0 & \delta_1  \\
0 & \alpha_1 & \dots  & 0 & \delta_2 \\
 &  & \ddots  &  &  \\
 0& \dots & 0  & \alpha_{1} & \delta_{n-1} \\
0 & 0 & \dots  & 0 & \alpha_2 \\
\end{pmatrix}
\end{equation*}
where $\alpha_i, \delta_i \in \mathbb{F}_q$ and $j \in \{0,1, \ldots, f-1\}.$
Let $\sigma =(1, n)(2, n - 1) \ldots ([n/2], [n/2 + 3/2]) \in \Sym(n).$ Let $$y = \per(\sigma) \cdot \diag(\sgn(\sigma), 1, \ldots, 1, \theta, \theta^{-1}).$$ Calculations show that $(S \cap S^x) \cap (S \cap S^x)^y \le Z(GL_n(q)).$
Thus, $b_S(G)\le 4$ and $\Reg_{S}(G,5) \ge 5.$

\medskip

{\bf Case (1.4).} Now let $n-m=1$  and $m \ge 2$,  so $n\ge 3.$ If $n=3$, then Theorem \ref{lem41} follows by Proposition \ref{n3GL}, so we assume $n>3.$
If $k>2$, then the proof is  as in  {\bf Case (1.3)}, so we assume $k=2.$
If $(n,q)$ is $(4,2)$ or $(5,3)$, then Theorem \ref{lem41} is verified by computation. 
 Otherwise,  by Theorem \ref{irred} and Lemma \ref{scfield}, there exists $x \in SL_n(q)$ such that $S \cap S^x$ consists of elements of  shape (with respect to a basis $\beta$ as in \eqref{basGLG})
\begin{equation*} 
\phi^j
  \begin{pmatrix}
\alpha_1 &\delta_2 & \dots    & \delta_{n-1} & \delta_n  \\
0 & \alpha_2 & \dots  & 0 & 0 \\
 &  & \ddots  &  &  \\
 0& \dots & 0  & \alpha_{2} & 0 \\
0 & 0 & \dots  & 0 & \alpha_2 \\
\end{pmatrix}
\end{equation*}
where $\alpha_i, \delta_i \in \mathbb{F}_q$ and $j \in \{0,1, \ldots, f-1\}.$
Let $\sigma =(1, n)(2, n - 1) \ldots ([n/2], [n/2 + 3/2]) \in \Sym(n).$ Let $$y = \per(\sigma) \cdot \diag(\sgn(\sigma), 1, \ldots, 1, \theta, \theta^{-1}).$$ Calculations show that $(S \cap S^x) \cap (S \cap S^x)^y \le Z(GL_n(q)).$ Thus, $b_S(G)\le 4$ and $\Reg_{S}(G,5) \ge 5.$ This concludes the proof of Proposition \ref{case1prop}.
\end{proof}

\bigskip

Recall that $f=1$ for {\bf Cases 2 -- 5}, so $\GL_n(q)=GL_n(q)$ and $S=M.$ Therefore, $\beta$ and $\gamma_i$ are as in Lemma \ref{supreduce} and matrices from $S$ have shape \eqref{stup}. Denote $\gamma_i(S)$ by $S_i$.

\begin{Prop}
\label{case2prop}
 Theorem $\ref{lem41}$ holds in {\bf Case 2}.
\end{Prop} 
\begin{proof}

 Recall that $n$ is even, there exists exactly one $i \in \{1, \ldots, k\}$ with $n_i=2$, and $\gamma_i(g)$ appears in rows  $\{n/2,n/2+1\}$ of $g \in M$ for such $i$.  Let $x= \diag[x_k, \ldots, x_1]$, where $x_i\in SL_{n_i}(q)$ is such that $S_i \cap S_i^{x_i} \le RT(GL_{n_i}(q))$ if $n_i \ne 2$ and $x_i$ is the identity matrix if $n_i=2.$ If $y$ is as in Lemma \ref{irrtog}, then   calculations show that  
$$(S \cap S^x) \cap (S \cap S^x)^y$$ consists of matrices of  shape 
\begin{equation}\label{sered} 
  \begin{pmatrix} 
\alpha_1 &0       & \dots          &              &               &              &\dots  &0  \\
0        & \ddots &                &              &               &              &       &  \\
         &        & \alpha_{n/2-1} &              &               &              &       &   \\
         &        &                &\alpha_{n/2}  &\beta_{n/2}    &              &       &   \\
         &        &                &\beta_{n/2+1} &\alpha_{n/2+1} &              &       &   \\
         &        &                &              &               &\alpha_{n/2+2}&       & \\
         &        &                &              &               &              &\ddots & 0 \\
 0       & \dots  &                &              &               &     \dots    &  0    & \alpha_n 
\end{pmatrix}.
\end{equation}
Assume that $n \ge 8$. We take as the $S$-invariant subspace $U$ the subspace with  basis $\{v_{n/2+2}, \ldots, v_n\}$, so
$m=n/2-1 \ge 3$ and $n-m=n/2+1\ge 5.$  Let us rename some basis vectors for convenience, so denote vectors 
$$v_1, \ldots v_{n/2-1}, v_{n/2+2}, \ldots, v_n$$ by $$w_1, w_2, \ldots, w_{n-2}$$ respectively.
Let  $z_1, \ldots, z_5 \in SL_n(q)$ be such that $$(v_{n-m+j})z_i=u_{(i,j)}$$ for $i=1, \ldots, 5$ and $j=1, \ldots, m,$ where

\begin{gather}\label{orb21}
\left\{ 
\begin{aligned}
u_{(1,1)} & =v_{n/2}+w_1+w_2;\\
u_{(1,2)} & =v_{n/2+1}+w_1+w_3;\\
u_{(1,2+r)} & =w_1+w_{3+r}\phantom{;} &\text{ for } r \in \{1,\ldots, m-3\};\\
u_{(1,m)} & =w_1 +\sum_{r=2+m}^{n-2}w_r;
\end{aligned}
\right.
\end{gather}
\begin{gather}\label{orb212}
\left\{ 
\begin{aligned}
u_{(2,1)} & =v_{n/2}+w_2+w_3;\\
u_{(2,2)} & =v_{n/2+1}+w_2+w_4;\\
u_{(2,2+r)} & =w_2+w_{(2,r)}\phantom{;} &\text{ for } r \in \{1,\ldots, m-3\};\\
u_{(2,m)} & =w_2 +\sum_{r=m-2}^{n-5}w_{(2,r)}.\\
\end{aligned}
\right. 
\end{gather}
Here $w_{(2,1)}, \ldots, w_{(2,n-5)}$ are equal to $w_1, w_5, \ldots, w_{n-2}$ respectively.  
\begin{gather}
\left\{ 
\begin{aligned}
u_{(3,1)} & =v_{n/2}+w_3+w_4;\\
u_{(3,2)} & =v_{n/2+1}+w_3+w_5;\\
u_{(3,2+r)} & =w_3+w_{(3,r)}\phantom{;} &\text{ for } r \in \{1, \ldots, m-3\};\\
u_{(3,m)} & =w_3 +\sum_{r=m-2}^{n-5}w_{(3,r)}.
\end{aligned}
\right.
\end{gather}
Here $w_{(3,1)}, \ldots, w_{(3,n-5)}$ are equal to $w_1, w_2,  w_6, \ldots, w_{n-2}$ respectively.  
\begin{gather} 
\left\{ 
\begin{aligned}\label{orb22}
u_{(4,1)} & =v_{n/2}+w_4+w_5;\\
u_{(4,2)} & =v_{n/2+1}+w_4+w_6;\\
u_{(4,2+r)} & =w_4+w_{(4,r)}\phantom{;} &\text{ for } r \in \{1, \ldots, m-3\};\\
u_{(4,m)} & =w_4 +\sum_{r=m-2}^{n-5}w_{(4,r)}.
\end{aligned}
\right. 
\end{gather}
Here $w_{(4,1)}, \ldots, w_{(4,n-5)}$ are equal to $w_1, w_2, w_3,  w_7, \ldots, w_{n-2}$ respectively.  
\begin{gather} 
\left\{ 
\begin{aligned}\label{orb23}
u_{(5,1)} & =v_{n/2}+w_5+w_6;\\
u_{(5,2)} & =v_{n/2+1}+w_5+w_1;\\
u_{(5,3)} & =w_5+w_{(5,1)}+w_{(5,2)};\\
u_{(5,3+r)} & =w_5+w_{(5,r+2)}\phantom{;} && \text{ for } r \in \{1, \ldots, m-4\};\\
u_{(5,m)} & =w_5 +\sum_{r=m-1}^{n-5}w_{(5,r)}\phantom{;} && \text{ if } m>3;\\
u_{(5,m)} & =w_5 +\sum_{r=m-2}^{n-5}w_{(5,1)}\phantom{;} && \text{ if } m=3.
\end{aligned}
\right. 
\end{gather}
Here $w_{(5,1)}, \ldots, w_{(5,n-5)}$ are equal to $w_2, w_3, w_4, w_7, \ldots, w_{n-2}$ respectively.

 If $m=3$, then there are no  
$u_{(i,2+r)}$ for   $r \in \{1, \ldots, m-3\}$ and $i \in \{1, \ldots, 4\}$; and no $u_{(5,3+r)}$ for   $r \in \{1, \ldots, m-4\}$. Also if $m=3$, we define $u_{(i,3)}$   by $u_{(i,m)}$ in \eqref{orb21} -- \eqref{orb23}.   If $m=4$, then there are no  
$u_{(5,3+r)}$ for   $r \in [1, m-4]$ and $u_{(5,4)}$ is defined by $u_{(5,m)}$ in \eqref{orb23}.

 Let $$\omega_i=(S,Sx,Sy,Sxy,Sz_i).$$ We first show  that the  $\omega_i$ are regular points of $\Omega^5$. Consider $\omega_1$. The regularity of the remaining points can be shown using the same arguments. Let $t \in (S \cap S^x) \cap (S \cap S^x)^y \cap S^{z_1}$, so it takes shape \eqref{sered} for some $\alpha_i, \beta_i \in \mathbb{F}_q$ and stabilises the subspace $Uz_1.$ Thus
\begin{equation}\label{b2s1}
(u_{(1,1)})t=\alpha_{n/2} v_{n/2} + \beta_{n/2} v_{n/2+1} +\alpha_1 w_1 + \alpha_2 w_2,
\end{equation}
since $t \in (S \cap S^x) \cap (S \cap S^x)^y$ and 
\begin{equation}\label{b2s2}
(u_{(1,1)})t=\eta_1 u_{(1,1)} +\eta_2 u_{(1,2)} + \ldots + \eta_m u_{(1,m)},
\end{equation}
since $t$ stabilises  $Uz_1.$ There is no $w_i$ for $i \ge 3$ in decomposition \eqref{b2s1}, so in \eqref{b2s2} we obtain
$$(u_{(1,1)})t=\eta_1 u_{(1,1)}=\alpha_1 u_{(1,1)},$$
in particular $\beta_{n/2}=0$. The same arguments show that 
$$(u_{(1,i)})t=\alpha_1 u_{(1,i)},$$
for the remaining $i \in \{1, \ldots , m\},$ so $t$ is scalar. Therefore, $\omega_1$ is regular.

Now we claim that the $\omega_i$ lie in distinct orbits of $\Omega^5.$ Here we prove that $\omega_1$ does not lie in the orbits containing $\omega_2$ or $\omega_5$; the remaining cases are similar.  First assume that $\omega_1 g =\omega_2$ for $g \in GL_n(q),$ so 
$$g \in (S \cap S^x) \cap (S \cap S^x)^y \cap z_1^{-1}Sz_2.$$
Therefore, $g$ has shape \eqref{sered} and $g=z_1^{-1}hz_2$ for $h \in S,$ so 
\begin{equation}\label{b2s3}
(u_{(1,1)})g=\alpha_{n/2} v_{n/2} + \beta_{n/2} v_{n/2+1} +\alpha_1 w_1 + \alpha_2 w_2,
\end{equation}
and 
\begin{equation}\label{b2s4}
(u_{(1,1)})g=(u_{(1,1)})z_1^{-1}hz_2=(v_{n-m+1})hz_2 = \delta_1 u_{(2,1)} + \ldots + \delta_m u_{(2,m)}.
\end{equation}
At least one of $\alpha_{n/2}$ and $\beta_{n/2}$ in \eqref{b2s3} is non-zero, since $g$ is invertible,   so at least one of $\delta_1$ and $\delta_2$ is non-zero, since only $u_{(2,1)}$ and $u_{(2,2)}$ have $v_{n/2}$ or $v_{n/2+1}$ in the decomposition \eqref{orb212}.  On the other hand, $(u_{(1,1)})g$ in decomposition \eqref{b2s3} does not contain $w_3$ and $w_4$, so $\delta_1$ and $\delta_2$ must be zero, which is a contradiction. Thus, such $g$ does not exist, so $\omega_1$ and $\omega_2$ lie in distinct $\Omega^5$-orbits.

Assume now      that $\omega_1 g =\omega_5$ for $g \in GL_n(q),$ so 
$$g \in (S \cap S^x) \cap (S \cap S^x)^y \cap z_1^{-1}Sz_5.$$
Therefore, $g$ has shape \eqref{sered} and $g=z_1^{-1}hz_5$ for $h \in S,$ so 
$(u_{(1,1)})g$ has decomposition \eqref{b2s3}
and 
\begin{equation}\label{b2s5}
(u_{(1,1)})g=(u_{(1,1)})z_1^{-1}hz_5=(v_{n-m+1})hz_5 = \delta_1 u_{(5,1)} + \ldots + \delta_m u_{(5,m)}.
\end{equation}
At least one of $\alpha_{n/2}$ and $\beta_{n/2}$ in \eqref{b2s3} is non-zero, since $g$ is invertible,   so at least one of $\delta_1$ and $\delta_2$ is non-zero, since only $u_{(5,1)}$ and $u_{(5,2)}$ have $v_{n/2}$ or $v_{n/2+1}$ in the decomposition \eqref{orb23}.  On the other hand, $(u_{(1,1)})g$ in decomposition \eqref{b2s3} does not contain $w_6$, so $\delta_1=1.$ Also, there is no $w_{(5,j)}$ for $j>1$ in \eqref{b2s3}, so only $\delta_2$ can be non-zero in \eqref{b2s5}, but \eqref{b2s3} does not contain $w_5$ and $\delta_2$ must be zero, which  contradicts the invertibility of $g$. Thus, such $g$ does not exist, so $\omega_1$ and $\omega_5$ lie in distinct $\Omega^5$-orbits.

Now let $n < 8,$ so $n=4$ or $n=6$ since $n$ is even. In both cases $k=n-1$, since there must be only one  $(2 \times 2)$ block, so $n_{n/2}=2$ and $n_i=1$ for the remaining $i.$ Therefore, if $x=\diag(\sgn(\sigma),1 \ldots, 1) \per(\sigma)$ with $\sigma = (1, 2, \ldots, n )$, then $S \cap S^x \le RT(GL_n(q))$ and $$(S \cap S^x) \cap (S \cap S^x)^y \le D(GL_n(q)),$$
where $y$ is as in {\bf Case 1}. The rest of the proof is as in {\bf Case 1}. 
\end{proof}

\begin{Prop}
\label{case3prop}
 Theorem $\ref{lem41}$ holds in {\bf Case 3}.
\end{Prop}
\begin{proof}

 Recall that $f=1, $ $n$ is odd, there exists exactly one $i \in \{1, \ldots, k\}$ with $n_i=2$, and $\gamma_i(g)$ appears in rows  $\{(n+1)/2,(n+1/2)+1\}$ of $g \in M$ for such $i$. Let $s := (n + 1)/2$. Let $x= \diag[x_k, \ldots, x_1]$, where $x_i\in SL_{n_i}(q)$ such that $S_i \cap S_i^{x_i} \le RT(GL_{n_i}(q))$ if $n_i \ne 2$ and $x_i$ is the identity matrix if $n_i=2.$ If $y$ is as in Lemma \ref{irrtog}, then  calculations show that  
$$(S \cap S^x) \cap (S \cap S^x)^y$$ consists of matrices of shape 
\begin{equation}\label{sered2} 
  \left(\begin{smallmatrix} 
\alpha_1 &0       & \dots              &                  &                 &                   &  & \dots &0  \\
0        & \ddots &                    &                  &                 &                   &       & &  \\
         &        & \alpha_{s-2} &                  &                 &                   &       & &   \\
         &        &                    &\alpha_{s-1}&\beta_{s-1}&                   &       & &   \\
         &        &                    &                  &\alpha_{s} &                   &               &              &       &   \\
         &        &                    &                  &\beta_{s+1}&\alpha_{s+1} &              &       &   \\
         &        &                    &                    &               &                   &\alpha_{s+2}&       & \\
         &        &                    &                    &               &                   &              &\ddots & 0 \\
 0       & \dots  &                    &                    &               &                   &     \dots    &  0    & \alpha_n 
\end{smallmatrix} \right).
\end{equation}
Assume that $n \ge 9$. We take as the $S$-invariant subspace $U$ the subspace with  basis $\{v_{(n+1)/2}, \ldots, v_n\}$, so
$m=(n+1)/2 \ge 5$ and $n-m=(n-1)/2\ge 4.$  Let us rename some basis vectors for convenience, so denote vectors 
$$v_1, \ldots, v_{(n-1)/2}, v_{(n+2)/2}, \ldots, v_n$$ by $$w_1, w_2, \ldots, w_{n-3}$$ respectively.
Let  $z_1, \ldots, z_5 \in SL_n(q)$ be such that $$(v_{n-m+j})z_i=u_{(i,j)}$$ for $i=1, \ldots, 5$ and $j=1, \ldots, m,$ where
\begin{gather}\label{orb2r1}
\left\{ 
\begin{aligned}
u_{(1,1)} & = v_{s-1}+w_1;\\
u_{(1,2)} & = v_{s+1}+w_1;\\
u_{(1,3)} & =v_{s}+w_1+w_2;\\
u_{(1,3+r)} & =w_1+w_{2+r}\phantom{;} &\text{ for } r \in \{1, \ldots, m-4\};\\
u_{(1,m)} & =w_1 +\sum_{r=m-1}^{n-3}w_r; \, \, \, \, \, \, \,
\end{aligned}
\right.
\end{gather}
\begin{gather}\label{orb2r12}
\left\{ 
\begin{aligned}
u_{(2,1)} & =v_{s-1 }+w_2;\\
u_{(2,2)} & = v_{s+1}+w_2;\\
u_{(2,3)} & =v_{s}+w_2+w_3;\\
u_{(2,3+r)} & =w_2+w_{(2,r)}\phantom{;} &\text{ for } r \in \{1, \ldots, m-4\};\\
u_m^2 & =w_2 +\sum_{r=m-3}^{n-5}w_{(2,r)}.\\
\end{aligned}
\right. 
\end{gather}
Here $w_{(2,1)}, \ldots, w_{(2,n-5)}$ are equal to $w_1, w_4, \ldots, w_{n-3}$ respectively.  
\begin{gather}
\left\{ 
\begin{aligned}
u_{(3,1)} & =v_{s-1 }+w_3;\\
u_{(3,2)} & = v_{s+1}+w_3;\\
u_{(3,3)} & =v_{s}+w_3+w_4;\\
u_{(3,3+r)} & =w_3+w_{(3,r)}\phantom{;} &\text{ for } r \in \{1, \ldots, m-4\};\\
u_{(3,m)} & =w_3 +\sum_{r=m-3}^{n-5}w_{(3,r)}.
\end{aligned}
\right.
\end{gather}
Here $w_{(3,1)}, \ldots, w_{(3,n-5)}$ are equal to $w_1, w_2,  w_5, \ldots, w_{n-3}$ respectively.  
\begin{gather} 
\left\{ 
\begin{aligned}\label{orb2r2}
u_{(4,1)} & =v_{s-1 }+w_4;\\
u_{(4,2)} & = v_{s+1}+w_4;\\
u_{(4,3)} & =v_{s}+w_4+w_5;\\
u_{(4,3+r)} & =w_4+w_{(4,r)}\phantom{;} &\text{ for } r \in \{1,\ldots, m-4\};\\
u_{(4,m)} & =w_4 +\sum_{r=m-3}^{n-5}w_{(4,r)}.
\end{aligned}
\right. 
\end{gather}
Here $w_{(4,1)}, \ldots, w_{(4,n-5)}$ are equal to $w_1, w_2, w_3,  w_6, \ldots, w_{n-3}$ respectively.  
\begin{gather} 
\left\{ 
\begin{aligned}\label{orb2r3}
u_{(5,1)} & =v_{s-1 }+w_5;\\
u_{(5,2)} & = v_{s+1}+w_5;\\
u_{(5,3)} & =v_{s}+w_5+w_6;\\
u_{(5,3+r)} & =w_5+w_{(5,r)}\phantom{;} &\text{ for } r \in \{1, \ldots, m-4\};\\
u_{(5,m)} & =w_5 +\sum_{r=m-3}^{n-5}w_{(5,r)}.
\end{aligned}
\right. 
\end{gather}
Here $w_{(5,1)}, \ldots, w_{(5,n-5)}$ are equal to $w_1, w_2, w_3, w_4, w_7, \ldots, w_{n-3}$ respectively.

Let $$\omega_i=(S,Sx,Sy,Sxy,Sz_i).$$ We first show  that the $\omega_i$ are regular points of $\Omega^5$. Consider $\omega_1$. The regularity of the remaining points can be shown using the same arguments. Let $t \in (S \cap S^x) \cap (S \cap S^x)^y \cap S^{z_1}$, so it takes shape \eqref{sered2} for some $\alpha_i, \beta_i \in \mathbb{F}_q$ and stabilises the subspace $Uz_1.$ Thus
\begin{equation}\label{b2sp1}
(u_{(1,1)})t=\alpha_{s-1} v_{s-1} +  \beta_{s-1}v_{s}+\alpha_1 w_1,
\end{equation}
since $t \in (S \cap S^x) \cap (S \cap S^x)^y$ and 
\begin{equation}\label{b2sp2}
(u_{(1,1)})t=\eta_1 u_{(1,1)} +\eta_2 u_{(1,2)} + \ldots + \eta_m u_{(1,m)},
\end{equation}
since $t$ stabilises  $Uz_1.$ There is no $v_{s+1}$ and $w_i$ for $i \ge 2$ in decomposition \eqref{b2sp1}, so in \eqref{b2sp2} we obtain
$$(u_{(1,1)})t=\eta_1 u_{(1,1)}=\alpha_1 u_{(1,1)},$$
in particular $\beta_{s-1}=0$. The same arguments show that $\beta_{s+1}=0$ and
$$(u_{(1,i)})t=\alpha_1 u_{(1,i)},$$
for the remaining $i \in \{1, \ldots , m\},$
so $t$ is scalar. Therefore, $\omega_1$ is regular.

Now we claim that the $\omega_i$ lie in distinct orbits of $\Omega^5.$ Here we prove that $\omega_1$ does not lie in the orbit containing $\omega_2$; the   remaining cases are similar.  Assume that $\omega_1 g =\omega_2$ for $g \in GL_n(q),$ so 
$$g \in (S \cap S^x) \cap (S \cap S^x)^y \cap z_1^{-1}Sz_2.$$
Therefore, $g$ has shape \eqref{sered2} and $g=z_1^{-1}hz_2$ for $h \in S,$ so 
\begin{equation}\label{b2sp3}
(u_{(1,3)})g=\alpha_{s} v_{s} +\alpha_1 w_1+\alpha_2 w_2,
\end{equation}
and 
\begin{equation}\label{b2sp4}
(u_{(1,3)})g=(u_{(1,3)})z_1^{-1}hz_2=(v_{n-m+3})hz_2 = \delta_1 u_{(2,1)} + \ldots + \delta_m u_{(2,m)}.
\end{equation}
Since $\alpha_{s}$  is non-zero and there are no $v_{s \pm 1}$ in \eqref{b2sp3}, $\delta_3$ is non-zero, so \eqref{b2sp3}  must contain a term  $ \delta_3 w_3$. Thus, such $g$ does not exist. Therefore, $\omega_1$ and $\omega_2$ lie in  distinct orbits. 

Now let $n<9.$ If $n=3,$ then Theorem~\ref{lem41} follows by Proposition \ref{n3GL}. So $n=5$ or $n=7.$
Assume that $n_i=1$ for some $i\le k$. So there is a $(1\times 1)$ block on the line $j \le n$. Let $\tilde{x}$ be $\diag(\sgn(\sigma), 1, \ldots, 1) \cdot \per(\sigma)$ with $\sigma=(s+1,j)$ if $j>s+1$ and $\sigma=(j,s,s+1)$ if $j<s.$ It is easy to see that 
$$S \cap S^{x \tilde x}\le RT(GL_n(q)),$$
so the rest of the proof is as in {\bf Case 1}. If $n=5$ then $s+1=4,$ so there is a $(1 \times 1)$ block on the $n$-th line. If $n=7,$ then  $s+1=5,$ so there are  $(1 \times 1)$ blocks in the rows $n-1$ and $n$, since there must be only one  $(2 \times 2)$ block.       
\end{proof}

\begin{Prop}
\label{case4prop}
 Theorem $\ref{lem41}$ holds in {\bf Case 4}.
\end{Prop}
\begin{proof}

Recall that $f=1, $ $n$ is odd, there exists exactly one $i \in \{1, \ldots, k\}$ with $n_i=2$, and $\gamma_i(g)$ appears in rows  $\{(n-1)/2,(n+1)/2+1\}$ of $g \in M$ for such $i$.   Let $s:=(n-1)/2+1$. Let $x= \diag[x_k, \ldots, x_1]$, where $x_i\in SL_{n_i}(q)$ is such that $S_i \cap S_i^{x_i} \le RT(GL_{n_i}(q))$ if $n_i \ne 2$ and $x_i$ is the identity matrix if $n_i=2.$ If $y$ is as in Lemma \ref{irrtog}, then  calculations show that  
$$(S \cap S^x) \cap (S \cap S^x)^y$$ consists of matrices of shape 
\begin{equation}\label{seredl} 
  \left(\begin{matrix} 
\alpha_1 &0       & \dots              &                  &                 &             \dots &0  \\
0        & \ddots &                    &                  &                 &                   &   \\
         &        & \alpha_{s-1}       &                  &                 &                   &   \\
         &        & \beta_{s-1}        &\alpha_{s}        &\beta_{s+1}      &                   &   \\
         &        &                    &                  &\alpha_{s+1   }  &                   &   \\
         &        &                    &                  &                 &\ddots             & 0  \\
 0       & \dots  &                    &                  &                 &     0             &\alpha_{n} 
\end{matrix} \right).
\end{equation}
Assume that $n \ge 7$. We take as the $S$-invariant subspace $U$ the subspace with the basis $\{v_{s+1}, \ldots, v_n\}$, so
$m=(n-1)/2 \ge 3$ and $n-m=(n+1)/2\ge 4.$  Let us rename some basis vectors for convenience, so denote vectors 
$$v_1, \ldots, v_{s-2}, v_{s+2}, \ldots, v_n$$ by $$w_1, w_2, \ldots, w_{n-3}$$ respectively.
Let  $z_1, \ldots, z_5 \in SL_n(q)$ be such that $$(v_{n-m+j})z_i=u_{(i,j)}$$ for $i=1, \ldots, 5$ and $j=1, \ldots, m,$ where
\begin{gather}\label{orb2l1}
\left\{ 
\begin{aligned}
u_{(1,1)} & = v_{s}+w_1;\\
u_{(1,2)} & = v_{s}+v_{s-1}+v_{s+1}+w_2;\\
u_{(1,2+r)} & =v_{s+1}+w_{2+r}\phantom{;} &\text{ for } r \in \{1, \ldots, m-3\};\\
u_{(1,m)} & =v_{s+1} +\sum_{r=m-1}^{n-3}w_r;
\end{aligned}
\right.
\end{gather}
\begin{gather}\label{orb2l12}
\left\{ 
\begin{aligned}
u_{(2,1)} & =v_{s}+w_1;\\
u_{(2,2)} & = v_{s}+v_{s-1}+v_{s+1} +w_3;\\
u_{(2,2+r)} & =v_{s+1}+w_{(2,r)}\phantom{;} &\text{ for } r \in \{1, \ldots, m-3\};\\
u_{(2,m)} & =v_{s+1} +\sum_{r=m-2}^{n-5}w_{(2,r)}.\\
\end{aligned}
\right. 
\end{gather}
Here $w_{(2,1)}, \ldots, w_{(2,n-5)}$ are equal to $w_2, w_4, \ldots, w_{n-3}$ respectively.  
\begin{gather}
\left\{ 
\begin{aligned}
u_{(3,1)} & =v_{s}+w_1;\\
u_{(3,2)} & = v_{s}+v_{s-1}+v_{s+1}+w_4;\\
u_{(3,2+r)} & =v_{s+1}+w_{(3,m)}\phantom{;} &\text{ for } r \in \{1, \ldots, m-3\};\\
u_{(3,m)} & =v_{s+1} +\sum_{r=m-2}^{n-5}w_{(3,r)}.
\end{aligned}
\right.
\end{gather}
Here $w_{(3,1)}, \ldots, w_{(3,n-5)}$ are equal to $w_2, w_3,  w_5, \ldots, w_{n-3}$ respectively.  
\begin{gather} 
\left\{ 
\begin{aligned}\label{orb2l2}
u_{(4,1)} & =v_{s}+w_2;\\
u_{(4,2)} & = v_{s}+v_{s-1}+v_{s+1}+w_3;\\
u_{(4,2+r)} & =v_{s+1}+w_{(4,r)}\phantom{;} &\text{ for } r \in \{1, \ldots, m-3\};\\
u_{(4,m)} & =v_{s+1} +\sum_{r=m-2}^{n-5}w_{(4,r)}.
\end{aligned}
\right. 
\end{gather}
Here $w_{(4,1)}, \ldots, w_{(4,n-5)}$ are equal to $w_1, w_4, \ldots, w_{n-3}$ respectively.  
\begin{gather} 
\left\{ 
\begin{aligned}\label{orb2l3}
u_{(5,1)} & =v_{s}+w_2;\\
u_{(5,2)} & = v_{s}+v_{s-1}+v_{s+1}+w_4;\\
u_{(5,2+r)} & =v_{s+1}+w_{(5,r)}\phantom{;} &\text{ for } r \in \{1, \ldots, m-3\};\\
u_{(5,m)} & =v_{s+1} +\sum_{r=m-2}^{n-5}w_{(5,r)}.
\end{aligned}
\right. 
\end{gather}
Here $w_{(5,1)}, \ldots, w_{(5,n-5)}$ are equal to $w_1, w_3, w_5, \ldots, w_{n-3}$ respectively.

Let $$\omega_i=(S,Sx,Sy,Sxy,Sz_i).$$ We  first show that the $\omega_i$ are regular points of $\Omega^5$. Consider $\omega_1$. The regularity of the remaining points can be shown using the same arguments. Let $t \in (S \cap S^x) \cap (S \cap S^x)^y \cap S^{z_1}$, so it takes shape \eqref{seredl} for some $\alpha_i, \beta_i \in \mathbb{F}_q$ and stabilises the subspace $Uz_1.$ Thus
\begin{equation}\label{b2l1}
(u_{(1,1)})t= \beta_{s-1}v_{v-1}+\alpha_s v_s +\beta_{s+1}v_{v+1} + \alpha_1 w_1,
\end{equation}
since $t \in (S \cap S^x) \cap (S \cap S^x)^y$ and 
\begin{equation}\label{b2l2}
(u_{(1,1)})t=\eta_1 u_{(1,1)} +\eta_2 u_{(1,2)} + \ldots + \eta_m u_{(1,m)},
\end{equation}
since $t$ stabilises  $Uz_1.$ There is no  $w_i$ for $i \ge 2$ in decomposition \eqref{b2l1}, so in \eqref{b2l2} we obtain
$$(u_{(1,1)})t=\eta_1 u_{(1,1)}=\alpha_s u_{(1,1)},$$
in particular $\beta_{s-1}=\beta_{s+1}=0$. The same arguments show that 
$$(u_{(1,i)})t=\alpha_s u_{(1,i)},$$
for the remaining $i \in \{1, \ldots , m\},$
so $t$ is scalar. Therefore, $\omega_1$ is regular.

Now we claim that the $\omega_i$ lie in distinct orbits of $\Omega^5.$ Here we prove that $\omega_1$ does not lie in the orbits containing $\omega_2$ and $\omega_4$; the remaining cases are similar.  Assume that $\omega_1 g =\omega_2$ for $g \in GL_n(q),$ so 
$$g \in (S \cap S^x) \cap (S \cap S^x)^y \cap z_1^{-1}Sz_2.$$
Therefore, $g$ has shape \eqref{seredl} and $g=z_1^{-1}hz_2$ for $h \in S,$ so 
\begin{equation}\label{b2l3}
(u_{(1,2)})g=\alpha_{s}v_{s}+\beta_{s-1}v_{s-1}+\beta_{s+1}v_{s+1} +\alpha_{s-1}v_{s-1}+\alpha_{s+1}v_{s+1} + \alpha_2 w_2,
\end{equation}
and 
\begin{equation}\label{b2l4}
(u_{(1,2)})g=(u_{(1,2)})z_1^{-1}hz_2=(v_{n-m+3})hz_2 = \delta_1 u_{(2,1)} + \ldots + \delta_m u_{(2,m)}.
\end{equation}
Since there are no $w_1$ and $w_3$ in \eqref{b2l3}, $\delta_1=\delta_2=0$. Thus, $\alpha_s=0$ in \eqref{b2l3}, which is a contradiction, since $g$ must be invertible, so such $g$ does not exist. Therefore, $\omega_1$ and $\omega_2$ lie in  distinct orbits. 

Assume that $\omega_1 g =\omega_4$ for $g \in GL_n(q),$ so 
$$g \in (S \cap S^x) \cap (S \cap S^x)^y \cap z_1^{-1}Sz_4.$$
Therefore, $g$ has shape \eqref{seredl} and $g=z_1^{-1}hz_4$ for $h \in S,$ so 
\begin{equation}\label{b2l13}
(u_{(1,1)})g=\alpha_{s}v_{s}+\beta_{s-1}v_{s-1}+\beta_{s+1}v_{s+1}  +\alpha_1 w_1,
\end{equation}
and 
\begin{equation}\label{b2l14}
(u_{(1,1)})g=(u_{(1,1)})z_1^{-1}hz_4=(v_{n-m+3})hz_4 = \delta_1 u_{(4,1)} + \ldots + \delta_m u_{(4,m)}.
\end{equation}
Since there are no $w_2$ and $w_3$ in \eqref{b2l13}, $\delta_1=\delta_2=0$. Thus, $\alpha_s=0$ in \eqref{b2l13}, which is a contradiction, since $g$ must be invertible, so such $g$ does not exist. Therefore, $\omega_1$ and $\omega_4$ lie in  distinct orbits. 

Now let $n \le 5.$ If $n=3,$ then Theorem~\ref{lem41} follows by Proposition \ref{n3GL}, so consider the case $n=5.$
Assume that $n_i=1$ for some $i\le k$. So there is a $(1\times 1)$ block on the line $j \le n$. Let $\tilde{x}$ be $\diag(\sgn(\sigma),1, \ldots, 1) \cdot \per(\sigma)$ with $\sigma=(s+1,j)$ if $j>s+1$ and  $\sigma=(j,s,s+1)$ if $j<s.$ It is easy to see that 
$$S \cap S^{x \tilde x}\le RT(GL_n(q)),$$
so the rest of the proof is  as in {\bf Case 1}. Since $n=5$,  $s-1=2,$ so there is a $(1 \times 1)$ block on the first line. 
\end{proof}


\begin{Prop}
\label{case5prop}
Theorem $\ref{lem41}$ holds in {\bf Case 5}.
\end{Prop}
\begin{proof}
 Let $i \in \{1, \ldots, k\}$  be  such that $n_i=2$. Our proof  splits into three subcases:
\begin{description}[before={\renewcommand\makelabel[1]{\bfseries ##1}}]
\item[{\bf Case (5.1)}] $n\ge 9$ and $\gamma_i(g)$ appears in rows  $\{l,l+1\}$ of $g \in M$ where $l>n/2$ if $n$ is even and $l>(n+1)/2$ if $n$ is odd;
\item[{\bf Case (5.2)}] $n \ge 9$ and $\gamma_i(g)$ appears in rows $\{l-1,l\},$ where $l \le n/2$ if $n$ is even and $l \le (n-1)/2$ if $n$ is odd;
\item[{\bf Case (5.3)}] $n<9$.
\end{description} 

{\bf Case (5.1).} Consider the case when  the only $(2\times 2)$ block is in  rows $\{l,l+1\},$ where $l>n/2$ if $n$ is even and $l>(n+1)/2$ if $n$ is odd. Denote $s:=n-l$. Let $x= \diag[x_k, \ldots, x_1]$, where $x_i\in SL_{n_i}(q)$ is such that $S_i \cap S_i^{x_i} \le RT(GL_{n_i}(q))$ if $n_i \ne 2$ and $x_i$ is the identity matrix if $n_i=2.$ If $y$ is as in Lemma \ref{irrtog}, then  calculations show that  
$$(S \cap S^x) \cap (S \cap S^x)^y$$ consists of matrices of shape 
{
\begin{equation}\label{ugol} 
  \left(\begin{smallmatrix} 
\alpha_1&0     &\dots         &            &              &      &           &             &            &\dots  &0  \\
0       &\ddots&              &            &              &      &           &             &            &       & \\
        &      &\alpha_{s-1}  &            &              &      &           &             &            &       &   \\
        &      &              &\alpha_{s}&\beta_{s}       &      &           &             &            &       &   \\
        &      &              &            &\alpha_{s+1  }&      &           &             &            &       &   \\
        &      &              &            &              &\ddots&           &             &            &       &   \\
        &      &              &            &              &      &\alpha_{l} &             &            &       &      \\
        &      &              &            &              &      &\beta_{l+1}&\alpha_{j+1} &            &       &   \\
        &      &              &            &              &      &           &             &\alpha_{l+2}&       & \\
        &      &              &            &              &      &           &             &            &\ddots & 0 \\
0       &\dots &              &            &              &      &           &             &     \dots  &  0    & \alpha_n 
\end{smallmatrix} \right).
\end{equation}
}
Assume that $n \ge 9$. We take as the $S$-invariant subspace $U$ the subspace with  basis $\{v_{l}, \ldots, v_n\}$, so
$m\ge 2$ and $n-m\ge [n/2] \ge 4.$  Let us rename some basis vectors for convenience, so denote vectors 
$$v_1, \ldots, v_{s-1}, v_{s+2}, \ldots, v_{l-1}, v_{l+2}, \ldots, v_{n}$$ by $$w_1, w_2, \ldots, w_{n-4}$$ respectively.
Let $z_1, \ldots, z_5 \in SL_n(q)$ be such that $$(v_{n-m+j})z_i=u_{(i,j)}$$ for $i=1, \ldots, 5$ and $j=1, \ldots, m,$ where
\begin{gather}\label{orb2t1}
\left\{ 
\begin{aligned}
u_{(1,1)} & = v_{s} + v_{l+1};\\
u_{(1,2)} & = v_{s}+ v_{s+1}+v_{l}+w_1;\\
u_{(1,2+r)} & =v_{s+1}+w_{1+r}\phantom{;} &\text{ for } r \in \{1, \ldots, m-3\};\\
u_{(1,m)} & =v_{s+1}+\sum_{r=m-1}^{n-4}w_r;
\end{aligned}
\right.
\end{gather}
\begin{gather}
\left\{ 
\begin{aligned}
u_{(2,1)} & = v_{s} + v_{l+1};\\
u_{(2,2)} & = v_{s}+v_{s+1}+v_{l}+w_2;\\
u_{(2,2+r)} & =v_{s+1}+w_{(2,r)}\phantom{;} &\text{ for } r \in \{1, \ldots, m-3\};\\
u_{(2,m)} & =v_{s+1}+\sum_{r=m-2}^{n-5}w_{(2,r)}.
\end{aligned}
\right. 
\end{gather}
Here $w_{(2,1)}, \ldots, w_{(2,n-5)}$ are equal to $w_1, w_3, \ldots, w_{n-4}$ respectively.  
\begin{gather}
\left\{ 
\begin{aligned}
u_{(3,1)} & = v_{s} + v_{l+1};\\
u_{(3,2)} & = v_{s}+ v_{s+1}+v_{l}+w_3;\\
u_{(3,2+r)} & =v_{s+1}+w_{(3,r)}\phantom{;} &\text{ for } r \in \{1, \ldots, m-3\};\\
u_{(3,m)} & =v_{s+1}+\sum_{r=m-2}^{n-5}w_{(3,r)}.
\end{aligned}
\right.
\end{gather} 
Here $w_{(3,1)}, \ldots, w_{(3,n-5)}$ are equal to $w_1, w_2, w_4 \ldots, w_{n-4}$ respectively.  
\begin{gather} 
\left\{ 
\begin{aligned}
u_{(4,1)} & = v_{s} + v_{l+1};\\
u_{(4,2)} & = v_{s}+ v_{s+1}+v_{l}+w_4;\\
u_{(4,2+r)} & =v_{s+1}+w_{(4,r)}\phantom{;} &\text{ for } r \in \{1, \ldots, m-3\};\\
u_{(4,m)} & =v_{s+1}+\sum_{r=m-2}^{n-5}w_{(4,r)}.
\end{aligned}
\right. 
\end{gather}
Here $w_{(4,1)}, \ldots, w_{(4,n-5)}$ are equal to $w_1, w_2, w_3, w_5 \ldots, w_{n-4}$ respectively.  
\begin{gather} 
\left\{ 
\begin{aligned}\label{orb2t3}
u_{(5,1)} & = v_{s} + v_{l+1};\\
u_{(5,2)} & = v_{s}+ v_{s+1}+v_{l}+w_5;\\
u_{(5,2+r)} & =v_{s+1}+w_{(5,r)}\phantom{;} &\text{ for } r \in \{1, \ldots, m-3\};\\
u_{(5,m)} & =v_{s+1}+\sum_{r=m-2}^{n-5}w_{(5,r)}.
\end{aligned}
\right. 
\end{gather}
Here $w_{(5,1)}, \ldots, w_{(5,n-5)}$ are equal to $w_1, w_2, w_3, w_4, w_6, \ldots, w_{n-4}$ respectively.

 Notice that $n \ge 9$ and $n-m \ge 3$ is necessary for such a definition of $u_{(i,j)}$. 

Let $$\omega_i=(S,Sx,Sy,Sxy,Sz_i).$$ We first show that the $\omega_i$ are regular points of $\Omega^5$. Consider $\omega_1$.  The regularity of the remaining points can be shown using the same arguments. Let $t \in (S \cap S^x) \cap (S \cap S^x)^y \cap S^{z_1}$, so it takes shape \eqref{ugol} for some $\alpha_i, \beta_i \in \mathbb{F}_q$ and stabilises the subspace $Uz_1.$ Thus
\begin{equation}\label{b2t1}
(u_{(1,1)})t=\alpha_{s}v_s +\beta_s v_{s+1} + \alpha_{l+1} v_{l+1} +\beta_{l+1}v_l,
\end{equation}
since $t \in (S \cap S^x) \cap (S \cap S^x)^y$ and 
\begin{equation}\label{b2t2}
(u_{(1,1)})t=\eta_1 u_{(1,1)} +\eta_2 u_{(1,2)} + \ldots + \eta_m u_{(1,m)},
\end{equation}
since $t$ stabilises  $Uz_1.$ There is no $w_i$ for $i \ge 1$ in decomposition \eqref{b2t1}, so in \eqref{b2t2} we obtain
$$(u_{(1,1)})t=\eta_1 u_{(1,1)}=\alpha_s u_{(1,1)},$$
in particular $\beta_{s}=\beta_{l+1}=0$, so $t$ is diagonal. Therefore,
\begin{equation}\label{2b2t2}
(u_{(1,2)})t=\alpha_{s}v_s +\alpha_{s+1}v_{s+1}+ \alpha_{l} v_{l} +\alpha_{1}w_1,
\end{equation}
must also lie in $\langle u_{(1,1)}, \ldots, u_{(1,m)} \rangle$. There are no $v_{l+1}$ and $w_i$ for $i \ge 2$ in \eqref{2b2t2}, so 
$$(u_{(1,2)})t= \alpha_s u_{(1,1)}.$$ If $i>2$ then
 the same arguments show that 
$$(u_{(1,i)})t=\alpha_{s+1} u_{(1,i)}=\alpha_s u_{(1,i)},$$
for the remaining $i \in \{1, \ldots , m\},$ so $t$ is scalar. Therefore, $\omega_1$ is regular.

Now we claim that  the $\omega_i$ lie in distinct orbits of $\Omega^5.$ Here we prove that $\omega_1$ and $\omega_2$  lie in distinct orbits;  the remaining cases are similar.  Assume that $\omega_1 g =\omega_2$ for $g \in GL_n(q),$ so 
$$g \in (S \cap S^x) \cap (S \cap S^x)^y \cap z_1^{-1}Sz_2.$$
Therefore, $g$ has shape \eqref{ugol} and $g=z_1^{-1}hz_2$ for $h \in S,$ so 
\begin{equation}\label{b2t3}
(u_{(1,2)})g=\alpha_{s}v_s +\beta_s v_{s+1} + \alpha_{s+1} v_{s+1} +\alpha_{l}v_l+\alpha_1 w_1,
\end{equation}
and 
\begin{equation*}
(u_{(1,2)})g=(u_{(1,2)})z_1^{-1}hz_2=(v_{n-m+2})hz_2 = \delta_1 u_{(2,1)} + \ldots + \delta_m u_{(2,m)}.
\end{equation*}
Since there is no $w_2$ in \eqref{b2t3}, $\delta_2$ is zero. Therefore, there must be no $v_l$ in \eqref{b2t3}, so $\alpha_l$ must be zero, which contradicts the existence of such invertible $g.$ Thus, $\omega_1$ and $\omega_2$ lie in  distinct orbits. 

\medskip

{\bf Case (5.2).} Consider the case when $n \ge 9$ and the only $(2\times 2)$ block is in  rows $\{l-1,l\},$ where $l \le n/2$ if $n$ is even and $l \le (n-1)/2$ if $n$ is odd. We take as the $S$-invariant subspace $U$ the subspace with  basis $\{v_{l+1}, \ldots, v_n\}$, so
$m\ge [n/2]$ and $n-m \ge 2.$  If $n-m \ge 3$, then the proof is  as in {\bf Case (5.1).} 

 Consider the case  $n-m=2,$ so $n_k=2$. Let $x= \diag[x_k, \ldots, x_1]$, where $x_i\in SL_{n_i}(q)$ is such that $S_i \cap S_i^{x_i} \le RT(GL_{n_i}(q))$ if $n_i \ne 2$ and $x_i$ is the identity matrix if $n_i=2.$ Let $y$ be as in Lemma \ref{irrtog}. If $k=2$ in \eqref{stup}, then there exists $x_1 \in SL_{n_1}(q)$ such that $S_1 \cap S_1^{x_1} \le D(GL_{n_1}(q))$, by 
Theorem \ref{irred}, since $n_1=n-2\ge 7.$ Thus, $(S\cap S^x) \cap (S\cap S^x)^y \le D(GL_n(q))$ and the rest of the proof is as in {\bf Case 1}.

Let $k>2$ and $n_1 \ne 1$. So $n_1 >2$, since the only $(2\times 2)$ block is in  rows $\{l-1,l\},$ where $l \le n/2$. 
We can take $U$ to be the subspace with  basis $\{v_{n -n_1+1}, \ldots, v_n\}$. Therefore, $n-m\ge 3$  and the proof is the same as the one using \eqref{orb2t1} -- \eqref{orb2t3}. 

Let $k>2$ and $n_1=1.$ Denote $\diag(-1,1\ldots,1) \cdot \per((2,n))$ by $\tilde{x}$. Calculations show that  $S \cap S^{x \tilde{x}} \le RT(GL_n(q))$, so $(S \cap S^{x \tilde{x}}) \cap (S \cap S^{x \tilde{x}})^y\le D(GL_n(q))$. The rest of the proof is  as in {\bf Case 1}, since we can take $U=V_{k-1}=\langle v_3, \ldots, v_n \rangle,$ so $n-m=2,$ $m\ge 2.$

\medskip

{\bf Case (5.3).}  Now let $n<9.$ If $n=3,$ then Theorem~\ref{lem41} follows by Proposition \ref{n3GL}, so $n \in \{4,5,6,7,8\}$.

Assume that $n_i=1$ for some $i\le k$. So there is a $(1\times 1)$ block in the row $j \le n$ of $g \in S$. Let $\tilde{x}$ be $\diag(\sgn(\sigma), 1, \ldots, 1) \cdot \per(\sigma)$ with $\sigma=(s+1,j)$ if $j>s+1$ and  $\sigma=(j,s,s+1)$ if $j<s.$ It is easy to see that 
$$S \cap S^{x \tilde x}\le RT(GL_n(q)),$$
so the rest of proof is  as in {\bf Case 1}.

If we exclude the cases previously resolved, then we obtain the following list of possibilities:\medskip\\
$n=5, k=2, n_1=2, n_2=3;$\\
$n=5, k=2, n_1=3, n_2=2;$\\
$n=6, k=2, n_1=2, n_2=4;$\\
$n=6, k=2, n_1=4, n_2=2;$\\
$n=7, k=2, n_1=2, n_2=5;$\\
$n=7, k=2, n_1=5, n_2=2;$\\
$n=8, k=2, n_1=2, n_2=6;$\\
$n=8, k=2, n_1=6, n_2=2;$\\
$n=8, k=3, n_1=2, n_2=3, n_3=3;$\\
$n=8, k=3, n_1=3, n_2=3, n_3=2.$
\medskip

If $(n,q) \ne (6,3),$ then, by Lemma \ref{irred}, for all $i$ such that $n_i\ne 2$  there exist $x_i \in SL_{n_i}(q)$ such that $S_i \cap S_i^{x_i} \le D(GL_{n_i}(q))$. Let $x=\diag[x_k, \ldots, x_1]$, where $x_i$ is the identity matrix  if $n_i=2$. It is easy to check directly that $$(S \cap S^x) \cap (S \cap S^x)^y \le D(GL_n(q))$$
for all cases above, so the rest of the proof is  as in {\bf Case 1}.

  For $(n,q)=(6,3)$ the result is established by computation. This concludes the proof of Proposition \ref{case5prop}.
\end{proof}
  
 Theorem \ref{lem41}  now follows from     Propositions \ref{case1prop} -- \ref{case5prop}.
\end{proof}

\begin{Th}
\label{lem42}
If $k>1$,  $q=2$ and  there exists  $i \in \{1, \ldots, k\}$ such that $\gamma_i(S)$ is the normaliser of a Singer cycle of $GL_3(2),$ then Theorem {\rm \ref{theorem}} holds.
\end{Th} 
\begin{proof}
Notice that $GL_n(2)=SL_n(2),$ so $G=GL_n(2).$  

Since $k>1,$ $S$ stabilises a nontrivial invariant subspace $U$ of dimension $m<n$. Assume that matrices in $S$ take shape \eqref{stup} in the  basis $$\{v_1, v_2, \ldots, v_{n-m+1}, v_{n-m+2}, \ldots, v_n \}.$$

The main difference from the case $q>2$ is  that if $S_0$ is the normaliser of a Singer cycle of $GL_3(2)$, then there is no $x, y \in GL_3(2)$ such that $S_0^{x} \cap S_0^{y}$ is contained in $RT(GL_3(2))$ (see Theorem \ref{irred}). Since all Singer cycles are conjugate in $GL_3(2),$ we assume that $S_0$ is the normaliser of $$
\left\langle
\begin{pmatrix}
0&0&1\\
1&0&0\\
0&1&1
\end{pmatrix}
 \right\rangle.
$$ If $$x_0=\begin{pmatrix}
0&1&0\\
1&0&1\\
0&0&1
\end{pmatrix},
$$ then 
\begin{equation*}
S_0 \cap S_0^{x_0} = P =  \left\{
\begin{pmatrix}
1 & 1& 1\\
0& 1 & 0\\
1 & 0 & 0
\end{pmatrix},
\begin{pmatrix}
0 & 0& 1\\
0& 1 & 0\\
1 & 1 & 1
\end{pmatrix},
\begin{pmatrix}
1 & 0& 0\\
0& 1 & 0\\
0 & 0 & 1
\end{pmatrix} \right\}. 
\end{equation*}

By Lemma \ref{irred} for $n_i \ge 3$   there exist $x_i\in GL_{n_i}(2)$ such that $S_i \cap S_i^{x_i} \le Z(GL_{n_i}(2))$ if $S_i$ is not conjugate to $S_0$ and $S_i \cap S_i^{x_0} \le P$ if $S_i$ is conjugate to $S_0$. Notice that $Z(GL_{n_i}(2))=1.$  Let $x=\diag[x_k, \ldots, x_1]$, where $x_i$ is an identity matrix if $n_i = 2.$  Therefore, matrices in $S \cap S^x$ are upper triangular except for $(2 \times 2)$ and  $(3 \times 3)$ blocks on the diagonal. Let $y$ be the permutation matrix corresponding to the permutation $$(1,n)(2, n-1) \ldots (n/2, n/2+1)$$ if $n$ is even, and $$(1,n)(2, n-1) \ldots ((n-1)/2, (n+1)/2)$$ if $n$ is odd.

 Fix some $(3\times 3)$ block on the diagonal of matrices in $S$ such that $S_i=S_0$ is in this block, so $P$ is in this block in $S \cap S^x$.
 This block intersects  one, two or three blocks in $(S \cap S^x)^y.$  
If it intersects  at least two blocks, then the  matrices in $(S \cap S^x) \cap (S \cap S^x)^y$ have the following restriction  to the chosen $(3 \times 3)$ block:
$$
\begin{pmatrix}
*&*&0\\
*&*&*\\
*&*&*
\end{pmatrix}. 
$$ 
Since  the only such matrix in $P$ is the identity, every matrix in $(S \cap S^x) \cap (S \cap S^x)^y$ has the identity submatrix in this block.  

If the chosen $(3 \times 3)$ block intersects  a bigger block in $(S \cap S^x)^y$, then it must lie in the block with  $S_i \cap S_i^{x_i}$ for some $i$ such that $n_i>3$ and  all such matrices are scalar.

Let the chosen $(3 \times 3)$ block intersect another $(3 \times 3)$ block in $(S \cap S^x)^y$ which consists of matrices in $P$. By Lemma \ref{irred}, $b_{S_0}(GL_3(2))=3,$ so there exists $y_0$ such that $$S_0 \cap S_0^{x_0} \cap S_0^{y_0} \le Z(GL_{3}(2)).$$ Let $\tilde{y}=\diag[1, \ldots, 1, y_0, 1, \ldots, 1]$ where $y_0$ is in the chosen block. Therefore,
 $$(S \cap S^x) \cap (S \cap S^x)^{y\tilde{y}}$$
consists of matrices which are diagonal except, possibly, for $(2 \times 2)$ blocks.

The rest of the proof is  as in Theorem \ref{lem41}. 
\end{proof}

Theorem \ref{theorem} now follows by Theorems \ref{lem41} and \ref{lem42}.

\subsection{Solvable subgroups not contained in $\GL_n(q)$}
\label{sec412}


Recall that $V=\mathbb{F}_q^n$ and let $\beta=\{v_1, \ldots, v_n\}$ be a basis of $V$. Let $\Gamma= \Gamma L _n(q)= GL_n(q) \rtimes \langle  \phi_{\beta} \rangle$
and $A=A(V)=A(n,q):=\Gamma \rtimes \langle \iota_{\beta} \rangle$ where $\iota_{\beta}$ is the inverse-transpose  map of $GL_n(q)$ with respect to $\beta.$  Our goal is to prove the following theorem.

\begin{T4}
Let $n\ge 3.$ If $S$ is a maximal solvable subgroup of $A(n,q)$ not contained in $\Gamma,$ then one of the following holds:
\begin{enumerate}
\item[$(1)$] $b_S(S \cdot SL_n(q))\le 4$;
\item[$(2)$] $(n,q)=(4,3)$, $S$ is the normaliser in $A(n,q)$ of the stabiliser in $\GL_n(q)$ of a $2$-dimensional subspace of $V$, $b_S(S \cdot SL_n(q))=5$ and $\Reg_S(S \cdot SL_n(q),5)\ge 5.$
\end{enumerate}
\end{T4} 

Before we start the proof, let us discuss the structure of a maximal solvable subgroup
$S$ of $A(n,q)$ and fix some notation. In this section, we assume  $S$ is not contained in $\Gamma.$

Consider the action of $\Gamma$ on the set $\Omega_1$ of subspaces of $V$ of dimension $m<n.$ This action is transitive and equivalent to the action of $\Gamma$ on the set
$$\Omega_1'=\{\mathrm{Stab}_{\Gamma}(U) \mid U<V, \dim U= m\}$$ by conjugation. Let $\Omega_2$ be the set of subspaces of $V$ of dimension $n-m$ and let  $$\Omega_2'=\{\mathrm{Stab}_{SL_n(q)}(W) \mid W<V, \dim W= n-m\}.$$ It is easy to see that $\Gamma$ acts on $\Omega=\Omega_1 \cup \Omega_2$ with orbits $\Omega_1$ and $\Omega_2$ (respectively, on $\Omega'=\Omega_1' \cup \Omega_2'$ with orbits $\Omega_1'$ and $\Omega_2'$). We can extend this action to an action of $A$ on $\Omega'$ by conjugation which is equivalent to the following action of $A$ on $\Omega$:  for $U,W \in \Omega \text{ and } \varphi \in A$,
$$ U\varphi =W \text{ if and only if } 
(\mathrm{Stab}_{\Gamma}(U))^{\varphi}=\mathrm{Stab}_{\Gamma}(W).$$ 
In particular, if $U= \langle v_{n-m+1}, v_{n-m+2}, \ldots, v_n \rangle$, then $U\iota_{\beta}=\langle v_1, \ldots, v_{n-m}\rangle:=U'$ of dimension $n-m.$ Moreover, if $\varphi=\iota_{\beta}g$ with $g \in \Gamma$ and $h \in \Gamma,$ then 
$$(Uh)\varphi=U\iota_{\beta}h^{\iota_{\beta}}g=U'h^{\iota_{\beta}}g.$$  So elements of $A\backslash \Gamma$ permute $\Omega_1$ and $\Omega_2.$

Let us now define the action of $A$ on the pairs of subspaces $(U,W)$ of $V$ with $\dim U=m \le n/2$ and $\dim W=n-m$ where either
\begin{itemize}
\item  $U<W$, or
\item $U\cap W= \{0\}.$
\end{itemize}
 Here we let $(U,W) \varphi =(U\varphi, W \varphi).$ The pair $(U, W)$ is not ordered, but for convenience we usually list first the subspace of smaller dimension.  Notice that this action is equivalent to the action of $A$ by conjugation on 
 \begin{itemize}
\item {$\{\mathrm{Stab}_{\Gamma}((U,W)) \mid U \le W < V, \dim U=m, \dim W=n-m\}$};
\item $\{\mathrm{Stab}_{\Gamma}((U,W)) \mid U , W < V, U \cap W = \{0\}, \dim U=m, \dim W=n-m\}$ 
 \end{itemize}
respectively
where $\mathrm{Stab}_{\Gamma}((U,W))=\mathrm{Stab}_{\Gamma}(U) \cap \mathrm{Stab}_{\Gamma}(W).$

\begin{Def}
\label{deftypepm}
Let $M$ be the stabiliser in $A$ of a pair $(U,W)$ of subspaces of $V$ where $\dim U=m \le n/2$, $\dim W=n-m$. Assume that $M$ is a maximal subgroup of $A.$
\begin{itemize}
\item If $U \le W$, then $M$ is a maximal subgroup {\bf of type} $P_{m,n-m};$
\item If $U\cap W=\{0\}$, then $M$ is a maximal subgroup {\bf of type} $GL_m(q) \oplus GL_{n-m}(q).$
\end{itemize}
\end{Def}
Let $S$  be a maximal solvable subgroup of $A$ not contained in $\Gamma$ such that $S \le M$ where $M$ is a maximal subgroup of $A$ contained in Aschbacher's class $\mathit{C_1}.$ By \cite[\S 4.1]{kleidlieb}, $M$ is as in Definition \ref{deftypepm}.
In \cite[\S 4.1]{kleidlieb} the type $P_{m,n-m}$ is used only when $m<n/2;$ when $m=n/2$ such $M$ are labelled $P_m$. We let $m \le n/2$ and use the label $P_{m,n-m}$ since it allows us to use more uniform statements.

Let us fix $M\ge S$ as above and let $\beta=\{v_1, \ldots, v_n\}$ be a basis of $V$ such that $U=\langle v_{n-m+1}, \ldots, v_n \rangle$ and 
$$W= \begin{cases}
\{v_{m+1}, \ldots, v_n\} &\text{ if } M \text{ is of type } P_{m, n-m};\\
\{v_1, \ldots, v_{n-m}\} &\text{ if } M \text{ is of type } GL_{m}(q) \oplus GL_{n-m}(q).
\end{cases} 
$$ 
For such $\beta$ we say that it is {\bf associated} with $(U,W).$
 By Definition \ref{deftypepm}, $M$ is the normaliser of $\mathrm{Stab}_{\Gamma}(U,W)$ in $A$. Therefore, with respect to $\beta$,
\begin{equation}
\label{Mstr}
M =\begin{cases}
\mathrm{Stab}_{\Gamma}(U,W) \rtimes \langle \iota_{\beta}a(n,m) \rangle & \text{ if } M \text{ is of type } P_{m, n-m};\\
\mathrm{Stab}_{\Gamma}(U,W) \rtimes \langle \iota_{\beta} \rangle & \text{ if } M \text{ is of type } GL_{m}(q) \oplus GL_{n-m}(q)
\end{cases}
\end{equation}
where $a(n,m) \in GL_n(q)$ is 
\begin{equation}
\label{adef}
\begin{pmatrix}
0 & 0 & I_{m} \\
0 & I_{n-2m} & 0 \\
I_{m} & 0 & 0
\end{pmatrix}.
\end{equation}

\begin{Lem}
\label{GRmmim2}
Let $n\ge 3$ and let $S \le A$ be as above. If $m$ is the least integer such that $S$ lies in $M$ as in \eqref{Mstr}, then 
 $\tilde{S}=(S \cap \Gamma)$ acts  on $U$ irreducibly.
\end{Lem} 
\begin{proof}
Let $\beta$ and $M$ be as in \eqref{Mstr} and let $P$ be $\mathrm{Stab}_{\Gamma}(U,W),$ so 
$$M=P \rtimes \langle \iota_{\beta} a\rangle$$
where $a$ is $a(n,m)$ if $M$ is of type $P_{m,n-m}$ and $a=I_n$ otherwise.

Assume that $\tilde{S}$ acts  reducibly on $U$, so  $\tilde{S}$ stabilises $U_1 <U$ of dimension $s<m.$  Assume that $U_1$ is a minimal such subspace, so $\tilde{S}$ stabilises no proper non-zero subspace of $U_1.$ Let $\varphi \in S \backslash \tilde{S}$ be such that
$S=\langle \tilde{S}, \varphi \rangle$, so $\varphi=\iota_{\beta}a \cdot g$ with $g \in P.$  Hence $\tilde{S}$ stabilises $W_1=U_1 \varphi$ of dimension  $n-s$. Notice that $U_1<W_1$ if $M$ is of type $P_{m,n-m}$ and $W_1 \cap U_1=0$ otherwise.
Since $\varphi^2 \in \Gamma,$ we obtain $\varphi^2 \in \tilde{S}$ and   $$W_1\varphi=U_1\varphi^2=U_1.$$
Therefore, $S$ normalises the stabiliser of $(U_1, W_1)$ in $\Gamma,$ so $S$ lies in a maximal subgroup of $A$ of type $P_{s, n-s}$ or $GL_{s}(q) \oplus GL_{n-s}(q)$  which contradicts the assumption of the lemma. 
\end{proof}

Notice that if we take $U=W=V$, then the proof of Lemma \ref{GRmmim2} implies the following statement.
\begin{Lem}
\label{GRirred2}
If $S \le A$ is not contained  in a maximal subgroup of $A$ from the class $\mathit{C_1}$, then $\tilde{S}=(S \cap \Gamma)$ acts irreducibly on $V$.
\end{Lem}

\begin{Th}
Theorem {\rm \ref{theoremGR}} holds if $S$ is not contained  in a maximal subgroup of $A$ from the class $\mathit{C_1}$.
\end{Th}
\begin{proof}
 Lemmas \ref{GammairGL}  and \ref{GRirred2}  imply that $\hat{S}:=S \cap GL_n(q)$ lies in an irreducible solvable subgroup of $GL_n(q)$. By Theorem \ref{irred}, either there exists $x \in SL_n(q)$ such that $\hat{S} \cap \hat{S}^x \le GL_n(q)$, or $\hat{S}$ lies on one of the groups in $(4)$ -- $(5)$ of Theorem \ref{irred}. In the latter case the statement is verified by computation. 

Otherwise, $\overline{S} \cap \overline{S}^{\overline{x}}$ is an abelian subgroup of $G=(S \cdot SL_n(q))/Z(GL_n(q))$ where $\overline{\phantom{G}}:S \cdot SL_n(q) \to G$ is the natural homomorphism.  By Theorem \ref{zenab}, there exist $\overline{y} \in G$ such that $\overline{S} \cap \overline{S}^{\overline{x}} \cap (\overline{S} \cap \overline{S}^{\overline{x}})^{\overline{y}}=1.$ So $b_S(S \cdot SL_n(q)) \le 4$ and the statement follows. 
\end{proof}

For the rest of the section we assume that $S$ lies in a maximal subgroup of $A$ from the class $\mathit{C_1}$.

 Let $m_1$ be minimal such that $S$ lies in a maximal subgroup $M_1<A$ of type $P_{m_1, n-m_1}$ or $GL_{m_1}(q) \oplus GL_{n-m_1}(q)$ stabilising subspaces $(U_1,W_1)$. Let $V_1$ be $V$;  let $V_2=W_1/U_1$ and $n_2=\dim V_2$. Let $\beta$ be a basis associated with $(U_1,W_1).$ Notice that elements from $\mathrm{Stab}_{\Gamma}(U_1,W_1)$ induce semilinear transformations on $V_2$, and $\iota_{\beta}$ induces the inverse-transpose map on $GL(V_2).$  Hence there exists a homomorphism 
$$\psi: S \to A(V_2) $$ mapping $x \in S$ to the element it induces on $V_2.$ Denote $\psi(S)$ by $S|_{_{V_2}}.$ Let $m_2$ be minimal such that $S|_{_{V_2}}$ lies in a maximal subgroup $M_2<A(V_2)$ of type $P_{m_2, n_2-m_2}$, or $GL_{m_2}(q) \oplus GL_{n_2-m_2}(q)$ stabilising subspaces $(U_2/U_1,W_2/U_1)$ of $V_2$.

 Repeating the arguments above we obtain a chain of subspaces 
\begin{equation}
\label{GRUs}
 0=U_0<U_1< \ldots < U_k<V \text { with } \dim U_i/U_{i-1} =m_i \text{ for } i \in\{1 \ldots, k\},
\end{equation} 
 subspaces $W_i$  for $i \in \{0, 1, \ldots, k\}$ where $W_0=V$, and groups $S|_{V_i}$ where $V_i=W_{i-1}/U_{i-1}$ for $i \in \{1, \ldots, k+1\}.$ Here $S|_{V_{i}}$ stabilises $(U_{i}/U_{i-1}, W_{i}/U_{i-1})$ for $i\in \{1, \ldots, k\}$ and $S|_{V_{k+1}}$ stabilises no subspace of $V_{k+1}.$ If $U_k=W_k,$ then $V_{k+1}$ is a zero space and  $S|_{V_{k+1}}$ is trivial.  Let $\beta=\{v_1, \ldots, v_n\}$ be a basis of $V$ such that  $\beta_i=\{v_1 +U_{i-1}, \ldots, v_n +U_{i-1}\} \cap V_i$ is a basis associated with $(U_{i}/U_{i-1}, W_{i}/U_{i-1})$ for $i\in \{1, \ldots, k\}.$

Let $\varphi \in S$. Hence $\varphi|_{_{V_i}}$ lies in $M_i$ for $i=1, \ldots, k$, so $$\varphi|_{_{V_i}}= (\iota_{\beta_i} a_i)^l \cdot g_i \text{ with } l\in \{0,1\} \text{ and } g_i \in \mathrm{Stab}_{\Gamma L (V_i)}((U_{i}/U_{i-1}, W_{i}/U_{i-1})),$$
where $a_i$ is $a(n_i,m_i)$  as in \eqref{adef} if $M_i$ is of type $P_{m_i, n_i-m_i}$ and $a_i$ is $I_{n_i}$ otherwise. Let $a \in GL_n(q)$ be such that  $a|_{_{V_k}}=a_k$ and 
\begin{equation}
\label{adefGR}
a|_{_{V_i}}= \begin{cases}
\begin{pmatrix}
a|_{_{V_{i+1}}} & 0\\
0 & I_{m_{i}}
\end{pmatrix} &\text{ if } M_i \text{ is of type } GL_{m_i}(q) \oplus GL_{n_i-m_i}(q)\\
\begin{pmatrix}
0 & 0 & I_{m_i}\\
0 & a|_{_{V_{i+1}}} & 0\\
I_{m_i} & 0 & 0
\end{pmatrix} &\text{ if } M_i \text{ is of type } P_{m_i, n_i -m_i}
\end{cases}
\end{equation}
 for $i \in \{1, \ldots, k-1\}.$  Therefore, 
\begin{equation}
\label{GRvarphishape}
\varphi = (\iota_{\beta}a)^l \cdot (\phi_{\beta})^j \cdot g,
\end{equation}
where $g \in GL_n(q)$ and $g|_{_{V_i}}$ stabilises $(U_{i}/U_{i-1}, W_{i}/U_{i-1}).$ More specifically, let $s \le k$ be the number of $i \in \{1, \ldots, k\}$ such that $M_i$ is of type $P_{m_i, n_i-m_i}$ and let $i_1 < \ldots < i_s$ be the corresponding $i$-s. Therefore,
\begin{equation}
\label{GRguppdiag}
g= \begin{pmatrix}
g_{i_1} '&* &\ldots &\ldots &\ldots & &* \\
 & \ddots &* &\ldots &\ldots& \ldots &* \\
 & &g_{i_s}' &* &\ldots & &* \\
 & & &g_{k+1} &* & &* \\
 & & & & g_k &* & *\\
 & & & & & \ddots &*\\
& & & & & & g_1
\end{pmatrix}
\end{equation} 
where $g_i, g_i' \in GL_{m_i}(q)$ and $g_{k+1} \in GL_{n_{k+1}}(q).$

\begin{example}
Let $k=3,$ and let $M_1$ and $M_3$ be of types $P_{m_i, n_i-m_i}$ for $i=1$ and $i=3$ respectively. Let $M_2$ be of type $GL_{m_2}(q) \oplus GL_{n_2-m_2}(q).$ Then 
$$
a=\begin{pmatrix}
 & & & & &  I_{m_1} \\
 &{\cellcolor{gray!50}} &{\cellcolor{gray!50}} &{\cellcolor{gray!50}} I_{m_3} &{\cellcolor{gray!20}} & \\
 & {\cellcolor{gray!50}} &{\cellcolor{gray!50}} I_{n_3-2m_3} &{\cellcolor{gray!50}} &{\cellcolor{gray!20}} & \\
 &{\cellcolor{gray!50}} I_{m_3} &{\cellcolor{gray!50}} & {\cellcolor{gray!50}} &{\cellcolor{gray!20}} & \\
 & {\cellcolor{gray!20}} & {\cellcolor{gray!20}} &{\cellcolor{gray!20}} &{\cellcolor{gray!20}} I_{m_2} & \\
I_{m_1} & & & & & 
\end{pmatrix}
$$
and $$\varphi = (\iota_{\beta}a)^l \cdot (\phi_{\beta})^j \cdot \begin{pmatrix}
 g_1 '&* &* &* &* &* \\
 & g_3' &* &* &0 &* \\
 & &g_4 &* &0 &* \\
 & & &g_3 &0 &* \\
 & & & & g_2 &* \\
 & & & & & g_1
\end{pmatrix}, $$
where $g_1, g_1' \in GL_{m_1}(q),$ $g_2 \in GL_{m_2}(q),$ $g_3, g_3' \in GL_{m_3}(q),$ $g_4 \in GL_{n_3-2m_3}(q).$
\end{example}

\begin{Lem}
\label{GRdiag}
Let $n \ge 3$. Let  $S$ be a maximal solvable subgroup of $A$. Assume that $S$ is contained in a maximal subgroup of $A$ of type $P_{m,n-m}$ or $GL_m(q) \oplus GL_{n-m}(q)$ for some $m \le n/2.$ Let $\beta$ be as described after \eqref{GRUs}.
\begin{enumerate}
\item[$(1)$] If none of $(m_i, q)$ and $(n_i,q)$ lies in $\{(2,2), (2,3)\},$ then there exist $x, y \in SL_n(q)$ such that if $\varphi \in (S \cap S^x \cap S^y) \cap \Gamma,$ then 
\begin{equation}
\label{GRvarphdiagc1}
\varphi=(\phi_{\beta})^j \cdot \diag[g_{i_1}, \ldots,g_{i_s}, g_{k+1}, g_k, \ldots,  g_1]=(\phi_{\beta})^j \cdot \diag(\alpha_1, \ldots, \alpha_n)
\end{equation}
 where $\alpha_i \in \mathbb{F}_q^*$ and $j \in \{0,1, \ldots, f-1\}.$ Moreover, $g_i, g_i' \in Z(GL_{m_i}(q))$ for $i \in \{1, \ldots, k\}$ and $g_{k+1} \in Z(GL_{n_{k+1}}(q))$.
\item[$(2)$] If $q \in \{2,3\}$ and at least one of $m_i$ or $n_i$ is $2$, then there exist $x,y \in SL_n(q)$ such that if $g \in S \cap S^x \cap S^y \cap GL_n(q)$, then $$g=\diag[g_{i_1}, \ldots,g_{i_s}, g_{k+1}, g_k, \ldots,  g_1]$$ with $g_i, g_i' \in GL_{m_i}(q)$ for $i \in \{1, \ldots, k\},$ $g_{k+1} \in GL_{n_{k+1}}(q)$, and one of the following holds:
\begin{enumerate} 
\item[$(2a)$]  for each $i \in \{1, \ldots, k+1\}$, either  $g_i$, $g_i'$ are  scalar matrices over $\mathbb{F}_q$ or $m_i=2$ and  $g_i$, $g_i'$ are upper-triangular matrices in $GL_2(q);$
\item[$(2b)$] there exist exactly one $j \in \{1, \ldots, k+1\} \backslash \{i_1, \ldots, i_s\}$ such that $g_j \in GL_2(q),$ and $g_i$, $g_i'$ are scalar for $i \in \{1, \ldots, t\} \backslash \{j\}.$
\end{enumerate}
\end{enumerate}
\end{Lem}
\begin{proof}
We start with the proof of $(1),$ so neither $(m_1,q)$ nor $(n_2,q)$ lies in $\{(2,2), (2,3)\}.$ 
Let us fix $q$ and assume that $n$ is minimal such that there exists $S\le A$ for which the statement of the lemma does not hold. Let $K$ be $S|_{_{U_1}} \le A(U_1)$ and let $R$ be $S|_{_{V_2}} \le A(V_2).$

We claim that there exist $x_2, y_2 \in SL_{n_2}(q)$ such that $(1)$ holds for $R \le A(V_2)$. Indeed, if $k>1$, then $R$ stabilises $(U_2,W_2)$ and $R$ is not a counterexample to the lemma since $n_2<n.$ If $k=1$, then $\tilde{R}=R \cap \Gamma$ acts irreducibly on $V_2$ by Lemma \ref{GRirred2}.  Hence $R \cap GL_{n_2}(q)$ lies in an irreducible maximal solvable subgroup of $GL_{n_2}(q)$ by  Lemma \ref{lem41}. By Theorem \ref{irred}, there exist $x_2,y_2 \in SL(V_2)$ such that $\tilde{R} \cap \tilde{R}^{x_2} \cap \tilde{R}^{y_2}= Z(GL(V_2)).$ Now the claim follows by Lemma \ref{scfield}.

By the same argument, $\tilde{K}=K \cap \Gamma L_{m_1}(q)$ acts irreducibly on $U_1$, and $K \cap GL_{m_1}(q)$ lies in an irreducible maximal solvable subgroup of $GL_{m_1}(q)$.

Our proof of $(1)$ splits into two cases: when $U_1\cap W_1=\{0\}$ and $U_1 \le W_1$ respectively.
\medskip

{\bf Case (1.1).} Let $U_1 \cap W_1= \{0\},$ so $V=U_1 \oplus W_1$, $n_2=n-m_1$ and $S$ lies in a maximal subgroup $M_1$ of $A$ of type $GL_{m_1}(q) \oplus GL_{n-m_1}(q).$ Recall that, by \eqref{Mstr}, elements in $M_1$ have shape 
\begin{equation}
\label{GROpleldiag}
(\iota_{\beta})^l \cdot (\phi_{\beta})^j \cdot \diag[h_2, h_1]
\end{equation}
with $l \in \{0,1\}$, $j \in \{0, 1, \ldots, f-1\}$, $h_2 \in GL_{n_2}(q)$ and $h_1 \in GL_{m_1}(q).$   By Theorem \ref{irred}, there exist $x_1, y_1 \in GL_{m_1}(q)$ such that 
$$K \cap K^{x_1} \cap K^{y_1} \cap GL_{m_1}(q) \le Z(GL_{m_1}(q)).$$ So, by Lemma \ref{scfield}, we can assume that $$K \cap K^{x_1} \cap K^{y_1} \cap  \Gamma L_{m_1}(q) \le \langle \phi_{\beta_1} \rangle Z(GL_{m_1}(q)),$$ where $\beta_1$ is the basis of $U_1$ consisting of the last $m_1$ vectors of $\beta.$ Let $x=\diag[x_2, x_1]$ and let $y=\diag[y_2, y_1],$ so $x,y \in SL_n(q).$

Consider $h \in S \cap S^x \cap S^y \cap \Gamma.$ By \eqref{GROpleldiag},
$$h= (\phi_{\beta})^j \cdot \diag[h_2,h_1],$$ so
$$h|_{_{V_2}}=(\phi_{\beta_2})^j \cdot h_2 \le (R \cap R^{x_2} \cap R^{y_2}) \cap \Gamma L_{n_2}(q)$$ and 
$$h|_{_{U_1}}=(\phi_{\beta_1})^j \cdot h_1 \le (K \cap K^{x_1} \cap K^{y_1}) \cap \Gamma L_{m_1}(q),$$ 
where $\beta_2=\beta \backslash \beta_1.$ Therefore, $h_1$ is a scalar matrix and  $h_2$ is as $\varphi$ in \eqref{GRvarphdiagc1}, so $h$ has the shape claimed by the lemma.

\medskip

{\bf Case (1.2).} Let $U_1 \le W_1,$ so $n_2=n-2m_1$ and $S$ lies in a maximal subgroup $M_1$ of $A$ of type $P_{m_1, n-m_1}.$ Recall that, by \eqref{Mstr}, elements in $M_1$ have shape 
\begin{equation}
\label{GRPmnmdiag}
(\iota_{\beta} a)^l \cdot (\phi_{\beta})^j \cdot 
\begin{pmatrix}
h_1' & * & *\\
 0   & h_2 & *\\
 0   & 0 & h_1
\end{pmatrix}
\end{equation}
with $l \in \{0,1\}$, $j \in \{0, 1, \ldots, f-1\}$, $h_2 \in GL_{n_2}(q)$ and $h_1, h_1' \in GL_{m_1}(q).$ Let $K'$ be the restriction of $S$ on $V/W_1$ and let $\hat{K}$ and $\hat{K'}$ be $K \cap GL_{m_1}(q)$ and $K' \cap GL_{m_1}(q)$ respectively.

 Recall that $S=\langle \tilde{S}, \varphi \rangle,$ where $\varphi = \iota_{\beta} a \cdot g$ with $g \in \tilde{M_1}=M_1 \cap \GL_n(q),$ so 
$$g= (\phi_{\beta})^{j_g} \cdot 
\begin{pmatrix}
g_3 & * & * \\
0  & g_2 & *\\
0 & 0 & g_1 \\
\end{pmatrix}
$$ with $g_1, g_3 \in GL_{m_1}(q),$ $g_2 \in GL_{n_2}(q)$ and $j_g \in \{0,1, \ldots, f-1\}.$ 

Since $(S \cap GL_n(q))^{\varphi}=S \cap GL_n(q),$  $$\hat{K}'=(\hat{K}^{\top})^{\phi^{j_g} g_3}.$$ Recall that $\hat{K}$ lies in an irreducible maximal solvable subgroup $T$ of $GL_{m_1}(q).$  Let $T'=(T^{\top})^{\phi^{j_g} g_3},$ so $\hat{K'}\le T'.$ 

By Theorem \ref{irred}, there exist $x_1, y_1 \in SL_{m_1}(q)$ such that $$T \cap T^{x_1} \cap T^{y_1}\le Z(GL_{m_1}(q)).$$ Let $x_1'= ((x_1^{-1})^{\top})^{\phi^{j_g} g_3}$ and let $y_1'= ((y_1^{-1})^{\top})^{\phi^{j_g} g_3}$,  so $$T' \cap T'^{x_1'} \cap T'^{x_1'} \le Z(GL_{m_1}(q)).$$  If $T$ is not as $S$ in  $(1)$ -- $(5)$ of Theorem \ref{irred}, then we assume $y_1=y_1'=I_{m_1}.$ So, by Lemma \ref{scfield}, we can assume that 
\begin{align*}
K \cap K^{x_1} \cap K^{y_1} \cap \Gamma L_{m_1}(q) \le \langle \phi_{\beta_1} \rangle Z(GL_{m_1}(q));\\
K' \cap (K')^{x_1} \cap (K')^{y_1} \cap \Gamma L_{m_1}(q) \le \langle \phi_{\beta_3} \rangle Z(GL_{m_1}(q)),
\end{align*}
 where $\beta_1$ is the basis of $U_1$ consisting of the last $m_1$ vectors of $\beta$ and $\beta_{3}=\{v_1 +W_1, \ldots, v_{m_1} + W_1\}$ is a basis of $V/W_1.$

If $T$ is as $S$ in  $(1)$ -- $(5)$ of Theorem \ref{irred}, then we claim that $T$ and $T'$ are conjugate by an element of $GL_{m_1}(q).$ Indeed, if $T$ is as $S$ in  $(1), (2) , (4),$ then both $T$ and $T'$ are normalisers of Singer cycles which are conjugate by Lemma \ref{singconj}. If $T$ is as $S$ in  $(3),$ then $T$ and $T'$ are conjugate by \cite[\S 21, Theorem 6 $(1)$]{sup}.  If $T$ is as $S$ in  $(5),$ then $T=(T^{\top})^{\phi^{j_g}},$ so $T'=T^{g_3}$ with $g_3 \in GL_{m_1}(q).$ Notice that $\Det(T)=\mathbb{F}_q^*$:  it is easy to check directly if $T$ is as $S$ in $(3)$ and $(5)$ of Theorem \ref{irred}; also recall that the determinant of a generator of a Singer cycle generates  $\mathbb{F}_q^*$. Hence, for a given $\lambda \in \mathbb{F}_q^*$, there exists $t \in GL_n(q)$ such that $\det(t)= \lambda$ and $T'=T^t.$

Let $x= \diag[x_1', x_2, x_1]$, let 
$$y=
\begin{pmatrix}
0 & 0 & t_2 y_1 \\
0 & y_2 & 0 \\
t_1 y_1'& 0 & 0
\end{pmatrix}$$
where $t_1, t_2 \in GL_{m_1}(q)$ are such that $\det(y)=1$, and 
$$T^{t_1}=T' \text{ and } (T')^{t_2}=T \text{ for } T \text{ as } S \text{ in } (1)-(5) \text{ of Theorem \ref{irred}.}$$
 So $x,y \in SL_n(q).$ 

Consider $h \in S \cap S^x \cap S^y \cap \Gamma.$ Using \eqref{GRPmnmdiag}, we obtain that elements in $S^y$ have shape 
\begin{equation}
\label{GRSy}
(\iota_{\beta} a)^l \cdot (\phi_{\beta})^j \cdot 
\begin{pmatrix}
h_1' & 0 & 0\\
 *   & h_2 & 0\\
 *   & * & h_1
\end{pmatrix}
\end{equation}
with $l \in \{0,1\}$, $j \in \{0, 1, \ldots, f-1\}$, $h_2 \in GL_{n_2}(q)$ and $h_1, h_1' \in GL_{m_1}(q).$ Therefore,  
by \eqref{GRPmnmdiag},
$$h= (\phi_{\beta})^j \cdot \diag[h_1',h_2,h_1],$$ so
\begin{align*}
h|_{_{V_2}} & =(\phi_{\beta_2})^j \cdot h_2 \le (R \cap R^{x_2} \cap R^{y_2}) \cap \Gamma L_{n_2}(q);\\  
h|_{_{U_1}} & =(\phi_{\beta_1})^j \cdot h_1 \le (K \cap K^{x_1} \cap K^{y_1}) \cap \Gamma L_{m_1}(q);\\
h|_{_{V/W_1}} & =(\phi_{\beta_3})^j \cdot h_1' \le (K' \cap (K')^{x_1'} \cap (K')^{y_1'}) \cap \Gamma L_{m_1}(q)
\end{align*} 
where $\beta_2=\{v_{m_1+1} +U_1, \ldots, v_{n-m_1}+U_1\}$ is a basis of $V_2.$ Therefore, $h_1$ and $h_1'$ are scalar matrices, and  $h_2$ is as $\varphi$ in \eqref{GRvarphdiagc1}, so $h$ has the shape claimed by the lemma.
 This concludes the proof of part $(1)$ of Lemma \ref{GRdiag}.

\bigskip

Now we start the proof of part $(2)$ of Lemma \ref{GRdiag}. So $q \in \{2,3\}$ and at least one of $m_i$ and $n_i$ is $2$ for $i \in \{1, \ldots, k\}$. Notice that $f=1$, so $\phi_{\beta}$ is trivial and $\Gamma =GL_n(q)$. Therefore, matrices in $S \cap \Gamma$ are block-upper-triangular and have shape \eqref{GRguppdiag} where 
\begin{align*}
g_i \in S_i & :=S|_{_{U_i/U_{i-1}}} \cap GL_{m_i} (q)  &&\text{ for } i\in \{1, \ldots, k\};\\
g_{k+1} \in S_k & := S|_{_{V_{n_{k+1}}}} \cap GL_{n_{k+1}} (q); &&\\
g_i \in S_i' & :=S|_{_{V_i/W_{i}}} \cap GL_{m_i} (q) && \text{ for } i\in \{i_1, \ldots, i_s\}.
 \end{align*}
Denote $GL_{m_i}(q)$ by $G_i$ for $i \in \{1, \ldots, k\}$ and $GL_{n_{k+1}}(q)$ by $G_{k+1}.$ By Lemma \ref{GRirred2}, $S_i$ and $S_i'$ are irreducible solvable subgroup of $G_i$ for $i \in \{1, \ldots, k+1\}$. By Theorem \ref{irred}, for $i \in \{1, \ldots, k+1\}$ one of the following holds:
\begin{enumerate}[label=(\roman*)]
\item $G_i=GL_2(q);$ \label{RGlistA} 
\item there exists $x_i \in G_i$ of determinant $1$ such that $S_i \cap S_i^{x_i} \le Z(G_i);$ \label{RGlistB}
\item $(G_i, S_i) \in \{(GL_4(3),GL_2(3) \wr \Sym(2)), (GL_3(2), N_{GL_3(2)}(T))\}$ where $T$ is a Singer cycle of $GL_3(2)$ and there exists $x_i, z_i \in G_i$ such that $S_i \cap S_i^{x_i} \cap S_i^{z_i} \le Z(G_i)$.\label{RGlistC}
\end{enumerate}
The same statement is true for $S_i'$ and we denote corresponding conjugating elements by $x_i'$ and $z_i'$. If $G_i=GL_2(q)$, then we take $x_i=x_i'=I_2.$

 If condition \ref{RGlistC}  holds for $S_i,$ then, by Theorem \ref{irred}, as discussed in {\bf Case~(1.2)}, $S_i'$ is conjugate to $S_i$  and for a given $\lambda \in \mathbb{F}_q^*$ there exists  $t \in G_i$ such that $\det(t)=1$ and $S_i'=S_i^t.$ 

Let us define $y_i, y_i' \in G_i$ for $i \in \{1, \ldots, k+1\}.$ 
 If $M_i$ is of type $GL_{m_i}(q) \oplus GL_{n_i-m_i}(q)$ or $i=k+1$ then let $y_i \in G_i$ be such that $S_i \cap S_i^{x_i} \cap S_i^{y_i} \le Z(G_i).$ If $M_i$ is of type $P_{m_i, n_i-m_i},$ then let $y_i, y_i' \in G_i$ be such that 
\begin{align*}
S_i \cap S_i^{x_i} \cap (S_i')^{y_i'} & \le Z(G_i)\\
S_i' \cap (S_i')^{x_i'} \cap (S_i)^{y_i} & \le Z(G_i).
\end{align*} 

If conditions \ref{RGlistA} or \ref{RGlistB} hold for $S_i$ (respectively $S_i'$), then we take $y_i$ and $y_i'$ to be the identity matrix in $G_i.$ Let 
$$y=\diag[y_{i_1}', \ldots, y_{i_s}', y_{k+1}, \ldots, y_k] \cdot a.$$
If $\det (y) \ne 1$, then $q=3$ and $\det(y)=-1.$ If so, then we pick  $i \in \{1, \ldots, k+1\}$  such that $G_i=GL_2(3)$ and change $y_i$ to be $\diag(-1,1)$ which is equivalent to multiplying a line of $y$ by $-1.$  Therefore, we can assume that $y \in SL_n(q).$   

Let $r$ be the number of $G_i$ and $G_i'$ equal to $GL_2(q)$. Our proof of $(2)$ splits into two cases: 
when $r \ge 2$
and $r=1$ respectively.

\medskip

{\bf Case (2.1).}  
  Let    $(2 \times 2)$ blocks (corresponding to the $S_i$ and the $S_i'$ lying in $GL_2(q)$) on the diagonal in matrices of $\tilde{S}=S \cap \Gamma$ occur in the rows $$(j_1, j_1+1), (j_2, j_2 +1), \ldots, (j_r, j_r+1).$$ 
Let $\tilde{x}=\diag(\sgn(\sigma), 1, \ldots, 1) \cdot {\rm perm}(\sigma) \in SL_n(q),$ where  $$\sigma=(j_1, j_1+1, j_2, j_2 +1, \ldots, j_r, j_r+1)(j_1, j_1+1).$$  
Let $x= \diag[x_{i_1}',\ldots, x_{i_s}', x_{k+1}  \ldots, x_1]\tilde{x}$. Notice that $\det x=1.$ Calculations show that if $h \in \tilde{S} \cap \tilde{S}^x \cap \tilde{S}^y$, then 
$$h=\diag[h'_{i_1}, \ldots, h'_{i_s}, h_{k+1}, \ldots, h_k]$$
where
\begin{itemize}
\item $h_i \in S_i \cap S_i^{x_i} \le Z(G_i)$ (respectively $h_i' \in S_i' \cap (S_i')^{x_i} \le Z(G_i)$) if condition \ref{RGlistB} holds for $S_i$ (respectively $S_i'$);
\item $h_i \in S_i \cap S_i^{x_i} \cap (S_i')^{y_i'} \le Z(G_i)$ (respectively $h_i' \in S_i' \cap (S_i')^{x'_i} \cap (S_i)^{y_i} \le Z(G_i)$) if condition \ref{RGlistC} holds for $S_i$ (respectively $S_i'$);
\item $h_i \in GL_2(q)$ (respectively $h_i' \in GL_2(q)$) is upper-triangular if condition \ref{RGlistA} holds
$S_i$ (respectively $S_i'$).
\end{itemize}
 Therefore, $h_i$ and $h_i'$ are either scalar or upper-triangular, so  $(2a)$   of Lemma \ref{GRdiag} holds.

\medskip

{\bf Case (2.2).} Assume that the number $r$ of $G_i$ and $G_i'$ equal to $GL_2(q)$ is $1$. Notice that if $G_i'=GL_2(q)$, then $G_i=GL_2(q)$ and $r \ge 2.$ Hence there exists unique $j \in \{1, \ldots, k+1\} \backslash \{i_1, \ldots, i_s\}$ such that $G_j$ is $GL_2(q)$.  Let $y_j=I_2$  and let $x_i, x'_i, y_i, y'_i$ for $i \in \{1, \ldots, k\} \backslash \{j\}$  be as defined before {\bf Case (2.1)}. Let $x=\diag[x_{i_1}',\ldots, x_{i_s}', x_{k+1}  \ldots, x_1]$ and $y=\diag[y_{i_1}', \ldots, y_{i_s}', y_{k+1}, \ldots, y_k] \cdot a$. It is easy to see that $(2b)$   of Lemma \ref{GRdiag} holds.  
\end{proof}

Now we prove Theorem \ref{theoremGR}.

\begin{proof}[Proof of Theorem {\rm \ref{theoremGR}}]
Let $U=U_1$, $W=W_1$ and $m=m_1.$ If  $Q \le V$ has dimension $r$, then we write 
$$Q= \left \langle 
\begin{pmatrix}
u_1 \\
 \vdots \\
u_r
\end{pmatrix}
\right \rangle,$$
where $u_1, \ldots, u_r \in V$ form a basis of $Q$.

The proof splits into two cases: when $(1)$ and $(2)$ of Lemma \ref{GRdiag} holds respectively.

\medskip

{\bf Case 1.} Assume that $(1)$ of Lemma \ref{GRdiag} holds. We study two subcases:
\begin{description}[before={\renewcommand\makelabel[1]{\bfseries ##1}}]
\item[{\bf Case (1.1)}] $m_i=1$ for $i \in \{1, \ldots, k\}$ and $n_{k+1} \in \{0,1\}$;
\item[{\bf Case (1.2)}] $m_i\ge 2$ for some $i \in \{1, \ldots, k\}$ or $n_{k+1}\ge 2.$
\end{description}

\medskip

{\bf Case (1.1).} Assume that $m_i=1$ for $i \in \{1, \ldots, k\}$ and $n_{k+1} \in \{0,1\}.$ Let $x=a$ where $a$ is as in \eqref{adefGR}, so $\det x = \pm 1$. Notice that if $\varphi \in S^x,$ then $\varphi$  has shape \eqref{GRvarphishape} with 
\begin{equation}
\label{GRginsy}
g= \begin{pmatrix}
g_{i_1}' & & & & & & \\
* & \ddots & & & & &\\
* & *&g_{i_s}' &  & & &\\
* & &* &g_{i_{k+1}} & & & \\
* &\ldots & &* & g_{k} & & \\
* &\ldots &\ldots& \ldots &*  & \ddots &\\
* &\ldots &\ldots &\ldots & &* &  g_{1}
\end{pmatrix}.
\end{equation}
So, if $\varphi \in S \cap S^x,$ then $g$ is diagonal.

 Let $y ,z \in SL_n(q)$ be as in the proof of Proposition \ref{ni1}. Let us show that if $\varphi \in S \cap S^x \cap S^y$, then $\varphi \in S \cap \Gamma.$ Assume that $\varphi \notin \Gamma,$ so $l=1$ in  \eqref{GRvarphishape}.  Since $S$ stabilises $(U,W),$ $S^y$ stabilises $(U,W)y$. Therefore, $$(U,W)y \varphi =(U,W)y$$
and $Uy=Wy\varphi$ since $\dim Uy= \dim Wy\varphi.$ With respect to $\beta$,
$$Uy= \left \langle 
\begin{pmatrix}
1, \ldots, 1 \\
\end{pmatrix}
\right \rangle.$$ On the other hand, 
$$Uy=Wy \varphi =W y (\iota_{\beta} a) \cdot (\phi_{\beta})^j \cdot g =W \iota_{\beta} y^{\iota_{\beta}}a\cdot (\phi_{\beta})^j \cdot g= W'y^{\iota_{\beta}}a\cdot (\phi_{\beta})^j \cdot g,$$
where $W'$ is spanned by $\beta \backslash (W \cap \beta).$ 

 If $U \cap W = \{0\},$ then $W'=U.$ It is now easy to see that 
$$W'y^{\iota_{\beta}}a\cdot (\phi_{\beta})^j \cdot g = 
\begin{cases} \left \langle 
\begin{pmatrix}
1,0, \ldots, 0 \\
\end{pmatrix}
\right \rangle \text{ if } M_2 \text{ is of type } GL_1(q) \oplus GL_{n_2-1}(q); \\
\left \langle 
\begin{pmatrix}
0, \ldots,1, 0 \\
\end{pmatrix}
\right \rangle \text{ if } M_2 \text{ is of type } P_{1, n_2-1}
\end{cases} $$
since $g$ is diagonal. So, $Uy \ne Wy \varphi$ which is a contradiction, Hence $\varphi \in \Gamma.$

If $U\le W$, then $W'= \left \langle 
\begin{pmatrix}
1, 0, \ldots, 0 \\
\end{pmatrix}
\right \rangle$, $W'y^{\iota_{\beta}}= \left \langle 
\begin{pmatrix}
1, 0, \ldots, -1 \\
\end{pmatrix}
\right \rangle$ and  $$W'y^{\iota_{\beta}}a\cdot (\phi_{\beta})^j \cdot g=\left \langle 
\begin{pmatrix}
\alpha_1, 0, \ldots,0, \alpha_2 \\
\end{pmatrix}
\right \rangle$$ for some $\alpha_1, \alpha_2 \in \mathbb{F}_q^*.$ So, $Uy \ne Wy \varphi$ which is a contradiction. Hence $\varphi \in \Gamma.$

Therefore, $S\cap S^x \cap S^y \cap S^z = \tilde{S}\cap \tilde{S}^x \cap \tilde{S}^y \cap \tilde{S}^z \le Z(GL_n(q))$ by Proposition \ref{ni1}. Notice that if $\det a=-1$, then we can take $x= \diag(-1, 1, \ldots, 1) \cdot a$ and the argument above still works, so we can assume $x,y,z \in SL_n(q).$

\medskip

{\bf Case (1.2).}  Let $x,y$ be as in $(1)$ of Lemma \ref{GRdiag}. Assume that  $m_i\ge 2$ for some $i \in \{1, \ldots, k\}$ or $n_{k+1}\ge 2.$ So, if $\varphi \in (S \cap S^x \cap S^y) \cap  \Gamma$, then  there exists $r \in \{1, \ldots, n\}$ such that  $\alpha_r =\alpha_{r+1}$ in \eqref{GRvarphdiagc1}. We choose $r$ to be minimal such that $\alpha_r =\alpha_{r+1}$ for all such $\varphi.$

 By \eqref{GROpleldiag}, \eqref{GRPmnmdiag} and \eqref{GRSy}, if $\varphi \in S \cap S^y$, then 
\begin{equation}
\label{GRvarphisy}
\varphi = (\iota_{\beta} a )^l \cdot (\phi_{\beta}^j) \cdot g
\end{equation}
where $l \in \{0,1\},$ $j \in \{0, 1, \ldots, f-1\},$
$$a = \begin{cases}
I_n  &  \text{ if } M_1 \text{ is of type } GL_m(q) \oplus GL_{n-m(q)};\\
a(n,m) & \text{ if } M_1 \text{ is of type } P_{m, n-m}

\end{cases}
$$
and
$$g = 
\begin{cases}
\diag[g_2, g_1] &  \text{ if } M_1 \text{ is of type } GL_m(q) \oplus GL_{n-m(q)};\\
\diag[g_1', g_2, g_1] & \text{ if } M_1 \text{ is of type } P_{m, n-m}.
\end{cases}$$
Here $g_1 \in GL_m(q)$, $g_2 \in GL_{n-m}(q)$ in the first option and  $g_1, g_1' \in GL_m(q)$, ${g_2 \in GL_{n-2m}(q)}$ in the second. Our consideration of {\bf Case (1.2)} splits into two subcases:
when  $U \ne W$
and  $U=W$ respectively.

\medskip

{\bf Case (1.2.1).} Assume that $U \ne W$. Let $\theta$ be a generator of $\mathbb{F}_q^*$ and let $z \in GL_n(q)$ be defined as follows:
\begin{equation}
\label{GRzdef1}
\begin{aligned}
(v_i)z & = v_i &&\text{ for } i \in \{1, \ldots, n-m\}; \\
(v_{n-m+1})z & =  \sum_{i=1}^{n-m} v_i -v_r + \theta v_{r} + v_{n-m+1}; && \text{ if } r \le n-m\\
(v_{n-m+1})z & =  \sum_{i=1}^{n-m-1} v_i  + \theta v_{n-m} + v_{n-m+1}; && \text{ if } r \ge n-m+1\\
(v_i)z & =  v_{n-m} + v_{i} && \text{ for } i \in \{n-m+2, \ldots, n\}.
\end{aligned} 
\end{equation}

Let $\varphi \in S \cap S^x \cap S^y \cap S^z$, so $\varphi$ has shape \eqref{GRvarphisy}.   Assume that $\varphi \notin \Gamma,$ so $l=1.$ Since $\varphi \in S^z$, it stabilises $(U,W)z,$ so $(U,W)z \varphi =(U,W)z$ and, therefore, 
$$Uz=Wz \varphi = W'z^{\iota_{\beta}}a\cdot (\phi_{\beta})^j \cdot g,$$
where $W'$ is spanned by $\beta \backslash (W \cap \beta).$  

 With respect to $\beta$,
$$Uz= \left \langle 
\left(\begin{array}{ccccccc|ccccc}
1 & \ldots & 1 & \theta &1 & \ldots & 1     &             1 & & &  \\
1 & \ldots & 1 &  1    &1 & \ldots & 1       &           & 1 & &  \\
\vdots &   &   &       &  &        &\vdots    &          & &\ddots & \\
1 & \ldots & 1 &  1    &1 & \ldots & 1         &         & & & 1 
\end{array}\right)
\right \rangle$$
where $\theta$ in the first line is either in the $r$-th or $(m-n)$-th column, and the part after the vertical line forms $I_m.$

 If $U \cap W=\{0\},$ then $W'=U$ and it is easy to see that $Wz \varphi = W'z^{\iota_{\beta}}a\cdot (\phi_{\beta})^j \cdot g= U,$ since $g$ stabilises $U$ by \eqref{GRvarphisy}. So, $Uz \ne Wz \varphi$ which is a contradiction. Hence $\varphi \in \Gamma.$

If $U \le W$, then $W'= \langle (I_m \mid 0_{ m \times (n-m)}) \rangle$, so 
$$W'z^{\iota_{\beta}}a= \left \langle 
(A \mid 0_{n \times (n-2m)} \mid I_n)
\right \rangle$$
where $A$ is $m \times m$ matrix with entries $-1$ or $- \theta$, and $-\theta$ can occur at most once.
Therefore, $W'z^{\iota_{\beta}}a \cdot g = \langle (A g_1'\mid0_{m \times (n-2m)}\mid g_1) \rangle$. Notice that $n-2m\ge 1$ since $U \ne W.$ So, $Uz \ne Wz \varphi$ which is a contradiction. Hence $\varphi \in \Gamma.$

Therefore, $\varphi = (\phi_{\beta})^j \cdot \diag(\alpha_1, \ldots, \alpha_n)$ as in $(1)$ of Lemma
\ref{GRdiag}. Since $\varphi \in S \cap S^z \cap \Gamma,$ it stabilises $U$ and $Uz.$  

Consider $((v_{n-m+1})z) \varphi.$ First, let $r\le n-m,$ so 
\begin{equation*}((v_{n-m+1})z) \varphi=
\begin{cases}
\left(\sum_{\substack{i \in \{1, \ldots, n-m\}\\ i \ne r}} \alpha_i v_i \right) + \alpha_r \theta^{p^j} v_r + \alpha_{n-m+1} v_{n-m+1};\\
\sum_{i=1}^m \delta_i (v_{n-m+i})z
\end{cases}
\end{equation*}
for some $\delta_i \in \mathbb{F}_q.$ Since $((v_{n-m+1})z) \varphi$ has no $v_i$ for $i \in \{n-m+2, \ldots, n\}$ in the decomposition with respect to $\beta$, 
$$\delta_2 = \ldots = \delta_m =0$$
and $((v_{n-m+1})z) \varphi= \delta_1 ((v_{n-m+1})z).$ Hence 
$$\delta_1 = \alpha_r \theta^{p^j-1}=\alpha_{r+1} = \alpha_i \text{ for } i \in \{1, \ldots, n-m+1\} \backslash \{r,r+1\},$$
so $j=0$ and $\varphi \in Z(GL_n(q)).$

Now let $r\ge n-m+1,$ so $m\ge 2$ and 
\begin{equation*}((v_{n-m+1})z) \varphi=
\begin{cases}
\left(\sum_{i=1}^{n-m-1} \alpha_i v_i \right) + \alpha_{n-m} \theta^{p^j} v_{n-m} + \alpha_{n-m+1} v_{n-m+1};\\
\sum_{i=1}^m \delta_i (v_{n-m+i})z
\end{cases}
\end{equation*}
for some $\delta_i \in \mathbb{F}_q.$ Since $((v_{n-m+1})z) \varphi$ has no $v_i$ for $i \in \{n-m+2, \ldots, n\}$ in the decomposition with respect to $\beta$, 
$$\delta_2 = \ldots = \delta_m =0$$
and $((v_{n-m+1})z) \varphi= \delta_1 ((v_{n-m+1})z).$ Hence 
$$\delta_1 = \alpha_{n-m} \theta^{p^j-1}=\alpha_{n-m+1} = \alpha_i \text{ for } i \in \{1, \ldots, n-m-1\}.$$
The same arguments for $((v_{n-m+2})z) \varphi$ show that $\alpha_{n-m+2}=\alpha_{n-m},$ so, since $\alpha_{n-m+1}=\alpha_{n-m+2}$ by $(1)$ of Lemma \ref{GRdiag}, we obtain $\theta^{p^j-1}=1.$ Hence $j=0$ and $\varphi \in Z(GL_n(q)).$

\medskip
{\bf Case (1.2.2).} Assume that $U=W$, so $m=n/2$ and $M$ is of type $P_{n/2,n/2}.$ In particular, $n \ge 4,$ since $n \ge 3$ by the assumption of the theorem.  Let $\theta$ be a generator of $\mathbb{F}_q^*$ and let $z \in SL_n(q)$ be defined as follows: 
\begin{equation}
\label{GRzdef2}
\begin{aligned}
(v_i)z= &v_i &&\text{ for } i \in \{1, \ldots, n-m\}; \\
(v_{m-n+1})z= & \theta v_{n-m} + v_{m-n+1}; &&\\
(v_i)z= & v_{n-m} + v_{i} && \text{ for } i \in \{n-m+2, \ldots, n\}.
\end{aligned} 
\end{equation}
  Let $\varphi \in S \cap S^x \cap S^y \cap S^z$, so $\varphi$ has shape \eqref{GRvarphisy}.  Assume that $\varphi \notin \Gamma,$ so $l=1.$ Since $\varphi \in S^z$, it stabilises $(U,W)z,$ so $(U,W)z \varphi =(U,W)z$ and, therefore, 
$$Uz=Wz \varphi = W'z^{\iota_{\beta}}a\cdot (\phi_{\beta})^j \cdot g,$$
where $W'$ is spanned by $\beta \backslash (W \cap \beta).$  

 With respect to $\beta$,
$$Uz= \left \langle 
\left(\begin{array}{cccc|cccc}
0 & \ldots & 0 &  \theta & 1 & & &  \\
0 & \ldots & 0 &  1      &   & 1 & &  \\
\vdots &  &  & \vdots &  & &\ddots & \\
0 & \ldots & 0 & 1 &  & & & 1 
\end{array}\right)
\right \rangle$$
where the part after the vertical line forms $I_m.$

Observe $W'= \langle (I_m \mid 0_{m \times m}) \rangle$, so 
$$W'z^{\iota_{\beta}}a= \left \langle 
\left(\begin{array}{cccc|cccc}
0 & \ldots & 0 &  0 & 1 & & &  \\
\vdots &  &  & \vdots &   &\ddots & & \\
0 & \ldots & 0 &  0      &  &  & 1 &   \\
- \theta & -1 & \ldots  & -1 &  & & & 1 
\end{array}\right)
\right \rangle,$$
and $W'z^{\iota_{\beta}}a \cdot g = \langle A\mid g_1 \rangle$ where 
$$A=\left(\begin{array}{cccc}
0 & \ldots & 0 &  0   \\
\vdots &  &  & \vdots \\
0 & \ldots & 0 &  0   \\
\alpha_1 & \alpha_2 & \ldots  & \alpha_{m} 
\end{array}\right)
$$ is an $m \times (n-m)$ matrix with $\alpha_i \in \mathbb{F}_q.$ So $Uz \ne Wz \varphi$ which is a contradiction. Hence $\varphi \in \Gamma.$
Therefore, $\varphi = (\phi_{\beta})^j \cdot \diag(\alpha_1, \ldots, \alpha_n)$ as in $(1)$ of Lemma
\ref{GRdiag}. In particular, $\alpha_1 = \ldots = \alpha_m$ and $\alpha_{m+1}= \ldots = \alpha_n.$ The same arguments as in the {\bf Case (1.2.1)} applied to $((v_{m+1})z) \varphi$ and $((v_{m+2})z) \varphi$ shows that $\varphi \in Z(GL_n(q)).$

\bigskip

{\bf Case 2.}  Assume that $(2)$ of Lemma \ref{GRdiag} holds.
 For $n \in \{3,4\}$ the theorem follows by computation, so  we may assume that $n \ge 5$.

We adopt notation from the proof of Lemma \ref{GRdiag}, in particular $x,y$, $S_i$, $G_i$.
Let $(2 \times 2)$ blocks (corresponding to the $S_i$ and the $S_i'$ lying in $GL_2(q)$) on the diagonal in matrices of $\tilde{S}=S \cap \Gamma$ occur in the rows $$(j_1, j_1+1), (j_2, j_2 +1), \ldots, (j_r, j_r+1).$$ Let $\Lambda=\Lambda_1 \cup \Lambda_2$ where 
$\Lambda_1=\{j_1, j_2, \ldots, j_r\}$ and $\Lambda_2=\{j_1+1, j_2+1, \ldots, j_r+1\}.$
 Let $U=U_1$, $W=W_1$ and $m=m_1.$  
 Notice, that if $\varphi \in S^y$, then it has shape \eqref{GRvarphishape} where $g$ has shape \eqref{GRginsy} with $g_i, g_i' \in GL_{m_i}$ for $i \in \{1, \ldots, k\}$ and $g_{k+1} \in GL_{n_{k+1}}(q).$ So, if $\varphi \in S \cap S^y$, then it has shape \eqref{GRvarphishape} with 
 \begin{equation}
 \label{GRgdiagcase2}
 g = \diag[g_{i_1}', \ldots, g_{i_s}', g_{k+1}, g_k, \ldots, g_1].
 \end{equation}
 
 Our consideration of {\bf Case 2} splits into two subcases: when $(2a)$ and $(2b)$ of Lemma \ref{GRdiag} holds respectively.

\medskip

{\bf Case (2.1).} Assume that $(2a)$ of Lemma \ref{GRdiag} holds. We consider two subcases: when $m \ge 2$ and $m=1$.

{\bf Case (2.1.1).} Assume that $m \ge 2.$ Let $z \in SL_n(q)$ be defined as follows
\begin{equation}
\label{GRzdefq23}
\begin{aligned}
(v_i)z & =  v_i &&\text{ for } i \in \{1, \ldots, n-m\}; \\
(v_{n-m+1})z & =  \left(\underset{i \in \{1, \ldots, n-m\} \backslash \{j_1\}}{\sum} v_i \right) + v_{n-m+1};\\
(v_{n-m+2})z & =  \left(\underset{i \in \{1, \ldots, n-m\} \backslash \Lambda_2}{\sum} v_i \right)   + v_{n-m+2};\\
(v_i)z & =  \sum_{j=1}^{n-m} v_{j} + v_{i} && \text{ for } i \in \{n-m+2, \ldots, n\}.
\end{aligned} 
\end{equation}

Let $\varphi \in S \cap S^x \cap S^y \cap S^z$, so $\varphi$ has shape \eqref{GRvarphishape} where $g$ has shape \eqref{GRgdiagcase2}.   Assume that $\varphi \notin \Gamma,$ so $l=1.$ Since $\varphi \in S^z$, it stabilises $(U,W)z,$ so $(U,W)z \varphi =(U,W)z$ and, therefore, 
$$Uz=Wz \varphi = W'z^{\iota_{\beta}}a \cdot g,$$
where $W'$ is spanned by $\beta \backslash (W \cap \beta).$  

 With respect to $\beta$,
$$Uz= \left \langle 
\left(\begin{array}{ccc|ccccc}
\lambda_1 & \ldots & \lambda_{n-m}   &           1 & & & & \\
\mu_1     & \ldots & \mu_{n-m}       &           & 1 & & & \\
  1       & \ldots &  1              &           &   &1 & &  \\
\vdots    &        &\vdots           &           &  & &\ddots & \\
1         & \ldots &  1              &           &  & & & 1 
\end{array}\right)
\right \rangle$$
where $\lambda_i, \mu_i \in \{0,1\}$ according to \eqref{GRzdefq23}, so for each $i \in \{1, \ldots, n-m\}$ at least one of $\lambda_i$ and $\mu_i$ is $1$,  and the part after the vertical line forms $I_m.$

 If $U \cap W=\{0\},$ then $W'=U$ and it is easy to see that $Wz \varphi = W'z^{\iota_{\beta}}a\cdot (\phi_{\beta})^j \cdot g= U,$ since $g$ stabilises $U$ by \eqref{GRvarphishape}. So, $Uz \ne Wz \varphi$ which is a contradiction. Hence $\varphi \in \Gamma.$

If $U \le W$, then $W'= \langle (I_m  \mid  0_{(m \times n-m)}) \rangle$, so 
$$W'z^{\iota_{\beta}}a= \left \langle 
(A \mid 0_{m \times (n-2m)} \mid I_n)
\right \rangle$$
where $A$ is $m \times m$ matrix with entries $-1$ ,$-\lambda_i$ and $- \mu_i$.
Therefore, $$W'z^{\iota_{\beta}}a \cdot g = \langle (A g_1'\mid 0_{m \times (n-2m)}\mid g_1) \rangle.$$ Notice that $n-2m\ge 1,$ since otherwise  $U = W,$ so $m=2$ and $n=4.$ Thus, $Uz \ne Wz \varphi$ which is a contradiction. Hence $\varphi \in \Gamma.$

Therefore, $\varphi = g= \diag[g_{i_1}, \ldots,g_{i_s}, g_{k+1}, g_k, \ldots,  g_1]$ as in $(a)$ of $(2)$ of Lemma
\ref{GRdiag}. Specifically, let $(v_i)\varphi =\alpha_i v_i$ for $i \in \{1, \ldots, n\} \backslash \Lambda_1$ and let $(v_i)\varphi = \alpha_i v_i + \gamma_i v_{i+1}$ for $i \in \Lambda_1$ with $\alpha_i, \gamma_i \in \mathbb{F}_q.$

 Since $\varphi \in S \cap S^z \cap \Gamma,$ it stabilises $U$ and $Uz.$  
Therefore, $((v_{n-m+2})z) \varphi$ is
\begin{equation*}
\label{vnm2zph}
\left(\underset{i \in \{1, \ldots, n-m\} \backslash \Lambda_2}{\sum} \alpha_i v_i \right) +
\left(\underset{i \in \{1, \ldots, n-m\} \cap \Lambda_2}{\sum} \gamma_{i-1} v_i \right) + \alpha_{n-m+2} v_{n-m+2},
\end{equation*}
and
$$((v_{n-m+2})z) \varphi= \sum_{i=1}^m \delta_i (v_{n-m+i})z$$
for some $\delta_i \in \mathbb{F}_q.$ Since $((v_{n-m+2})z) \varphi$ does not contain $v_i$ for $i \in \{n-m+1,n-m+3, \ldots, n\}$ in the decomposition with respect to $\beta$, 
$$\delta_1= \delta_3 = \ldots = \delta_m =0$$
and $((v_{n-m+1})z) \varphi= \delta_1 ((v_{n-m+1})z).$ Hence 
$$\gamma_{j_1} = \ldots = \gamma_{j_r}=0,$$
so $\varphi= \diag(\alpha_1, \ldots, \alpha_n)$ where 
\begin{equation}
\label{GRalphasvnm2}
\alpha_i=\alpha_{n-m+2} \text{ for } i \in \{1, \ldots, n-m\} \backslash \Lambda_2.
\end{equation} 

Consider 
\begin{equation*}((v_{n-m+1})z) \varphi=
\begin{cases}
\underset{i \in \{1, \ldots, n-m\} \backslash \{j_1\}}{\sum} \alpha_i v_i + \alpha_{n-m+1} v_{n-m+1} +\underline{\gamma_{n-m+1} v_{n-m+1}}\\
\sum_{i=1}^m \delta_i (v_{n-m+i})z
\end{cases}
\end{equation*}
for some $\delta_i \in \mathbb{F}_q.$ The underlined part is present only if $m=2.$ Since $((v_{n-m+1})z) \varphi$ contains neither $v_{j_1}$ (notice that $j_1>n-m$ since if $(2a)$ of Lemma \ref{GRdiag} holds, then $r \ge 2$) nor $v_i$ for $i \in \{n-m+3, \ldots, n\}$ in the decomposition with respect to $\beta$, 
$$\delta_2 = \ldots = \delta_m =0.$$
Here $\delta_2=0$ since $((v_{n-m+2})z) $ contains $v_{j_1}$ in the decomposition with respect to $\beta$ and $((v_{n-m+1})z)\varphi$ does not.
Thus $((v_{n-m+1})z) \varphi= \delta_1 ((v_{n-m+1})z)$ and 
$$\alpha_i = \alpha_{n-m+1} \text{ for } i \in \{1, \ldots, n-m\} \backslash \{j_1\}.$$
Combined with \eqref{GRalphasvnm2}, it implies $\varphi \in Z(GL_n(q)).$ 

\medskip

{\bf Case (2.1.2).} Assume that $m_1=1$ and let $t$ be the smallest $i \in \{2, \ldots, k\}$ such that $m_i\ge 2.$ Such $t$ exists since otherwise $m_i=1$ for $i \in \{1, \ldots, k\}$ and $n_{k+1}=2,$ so $(2b)$ of Lemma \ref{GRdiag} holds. Let $d_i=\sum_{j=1}^i m_j$ for $i \in \{1, \ldots, k\}$. Let $z \in SL_n(q)$ be defined as follows:
\begin{equation}
\label{GRzdefq23mr}
\begin{aligned}
(v_i)z & =  v_i &&\text{ for } i \in \{1, \ldots, n-d_t\}; \\
(v_{n-d_t+1})z & =  \left(\underset{i \in \{1, \ldots, n-d_t\} \backslash \{j_1\}}{\sum} v_i \right) + v_{n-d_t+1};\\
(v_{n-d_t+2})z & =  \left(\underset{i \in \{1, \ldots, n-d_t\} \backslash \Lambda_2}{\sum} v_i \right)   + v_{n-d_t+2};\\
(v_i)z & =  v_i && \text{ for } i \in \{n-d_{t}+3, \ldots, n-1\};\\
(v_n)z & =  \sum_{i=1}^n v_i.
\end{aligned} 
\end{equation}
 Let $\varphi \in S \cap S^x \cap S^y \cap S^z$, so $\varphi= (\iota_{\beta}a)^l \cdot g$ where $g=\diag[g_{i_1},\ldots, g_{i_s}, g_{k+1}, \ldots, g_1]$ with $g_i$ and $g_i'$  as in \eqref{GRguppdiag}.  
 Assume that $\varphi \notin \Gamma,$ so $l=1.$ Since $\varphi \in S^z$, it stabilises $(U,W)z,$ so $(U,W)z \varphi =(U,W)z$ and, therefore, 
$$Uz=Wz \varphi = W'z^{\iota_{\beta}}a \cdot g,$$
where $W'$ is spanned by $\beta \backslash (W \cap \beta).$  

 With respect to $\beta$,
$$Uz= \langle (1, \ldots, 1) \rangle.$$

 If $U \cap W=\{0\},$ then $W'=U$ and it is easy to see that $Wz \varphi = W'z^{\iota_{\beta}}a\cdot g= U,$ since $g$ stabilises $U$. So $Uz \ne Wz \varphi$ which is a contradiction. Hence $\varphi \in \Gamma.$

If $U \le W$, then $W'= \langle (1 , 0,\ldots, 0) \rangle$, so 
$$W'z^{\iota_{\beta}}a \cdot g= \langle 
(1 , 0, \ldots, 0, -1,-1, 0 \ldots, 0,  1)a \cdot g
 \rangle$$
where $-1$ is in the ${n-d_t+1}$ and ${n-d_t+2}$ entries. Notice that 
\begin{equation}
\label{GR212zeros}
(1 , 0, \ldots, 0, -1,-1, 0 \ldots, 0,  1)=(u_{i_1}', \ldots, u_{i_s}', u_{k+1}, \ldots, u_1)
\end{equation}
where $u_i, u_i' \in \mathbb{F}_q^{m_i}$ for $i \in \{1, \ldots, k\}$ and  $u_{k+1} \in \mathbb{F}_q^{n_{k+1}}.$ Here $u_1=u_1'=1$, $u_t=(-1,-1)$ and all other $u_i, u_i'$ are zero vectors. There is at least one such zero vector in \eqref{GR212zeros} since $n \ge 5.$ Notice that,  for a given $i \in \{1, \ldots, k+1\},$  either $a$ fixes $u_i$ and $u_i'$ or $a$ permutes them. Hence 
\begin{align*}
(u_{i_1}', \ldots, u_{i_s}', u_{k+1}, \ldots, u_1)ag & =(w_{i_1}', \ldots, w_{i_s}', w_{k+1}, \ldots, w_1)g \\
& =(w_{i_1}'g_{i_1}', \ldots, w_{i_s}'g_{i_s}', w_{k+1}g_{k+1}, \ldots, w_1g_1)
\end{align*}
where at least one of $w_i$ and $w_i'$ (and hence at least one one of $w_ig_i$ and $w_i'g_i'$) equals to a zero vector. Thus, since $Uz$ contains no non-zero vector with zero entries, $Uz \ne Wz \varphi$ which is a contradiction. Hence $\varphi =g\in \Gamma.$

The same arguments as  in {\bf Case (2.1.1)}, now applied to $(v_{n-d_t+1})z$ and $(v_{n-d_t+2})z$ instead of $(v_{n-m+1})z$ and $(v_{n-m+2})z,$ show that $g=\diag(\alpha_1, \ldots, \alpha_n)$ for some $\alpha_i \in \mathbb{F}_q.$ Since $g \in S^z \cap \Gamma,$ it stabilises $Uz$ so
$$(v_n)z g = \sum_i^n \alpha_i v_i = \alpha (v_n)z= \alpha \sum_i^n v_i$$
for some $\alpha \in \mathbb{F}_q^*.$ Therefore, $\alpha=\alpha_1 = \ldots = \alpha_n$, $g$ is scalar and $\varphi \in Z(GL_n(q)).$

\bigskip

{\bf Case (2.2).} Assume that $(2b)$ of Lemma \ref{GRdiag} holds. We consider two subcases: when $m=2$ and $m \ne 2$.

\medskip

{\bf Case (2.2.1).} Assume that $m=2,$ so if $g \in S \cap S^x \cap S^y \cap GL_n(q),$ then $g_1 \in GL_2(q)$ and $g_i$, $g_i'$ are scalar for $i \in \{2, \ldots, k+1\}.$ We may assume $U \cap W=0$ since otherwise the number of $G_i$ and $G_i'$ equal to $GL_2(q)$ is at least 2 and $(2a)$  of Lemma \ref{GRdiag} holds. Let $z \in SL_n(q)$ be defined as follows:
\begin{equation}
\label{GRzdefcase221}
\begin{aligned}
(v_i)z & =  v_i &&\text{ for } i \in \{1, \ldots, n-2\}; \\
(v_{n-1})z & =  \sum_{i=2}^{n-2} v_i + v_{n-1};\\
(v_{n})z & =  v_1+ \sum_{i=3}^{n-2} v_i   + v_{n}.\\
\end{aligned} 
\end{equation}
Let $\varphi \in S \cap S^x \cap S^y \cap S^z$, so $\varphi= (\iota_{\beta}a)^l \cdot g$ where $g=\diag[g_{i_1},\ldots, g_{i_s}, g_{k+1}, \ldots, g_1]$ where $g_i$ and $g_i'$ are as in \eqref{GRguppdiag}.  
 Assume that $\varphi \notin \Gamma,$ so $l=1.$ Since $\varphi \in S^z$, it stabilises $(U,W)z,$ so $(U,W)z \varphi =(U,W)z$ and, therefore, 
$$Uz=Wz \varphi = W'z^{\iota_{\beta}}a \cdot g,$$
where $W'$ is spanned by $\beta \backslash (W \cap \beta).$  

 With respect to $\beta$,
$$Uz= \left \langle 
\left(\begin{array}{ccccc|cc}
0  & 1& 1 & \ldots & 1   &           1 &0  \\
1  & 0& 1     & \ldots & 1       &    0       & 1 
\end{array}\right)
\right \rangle.$$
 Since $U \cap W=\{0\},$ we obtain $W'=U$ and it is easy to see that $Wz \varphi = W'z^{\iota_{\beta}}a\cdot g= U,$ since $g$ stabilises $U$. So, $Uz \ne Wz \varphi$ which is a contradiction. Hence $\varphi =g\in \Gamma.$ Therefore, $g=\diag[A,g_1]$ where $A=\diag(\alpha_1, \ldots, \alpha_{n-2})$ for some $\alpha_i \in \mathbb{F}_q^*$ and $$g_1 = \left( \begin{matrix} \delta_1 & \delta_2 \\ \delta_3 & \delta_4  \end{matrix} \right) \in GL_2(q).$$ Since $g \in S^z$, it stabilises $Uz.$
 
 Consider 
\begin{equation}
\label{GRc221vn1z}
 ((v_{n-1})z)g =
 \begin{cases}
 \sum_{i=2}^{n-2} \alpha_i v_i + \delta_1 v_{n-1} + \delta_2 v_n;\\
 \lambda_1 (v_{n-1})z + \lambda_2 (v_{n})z
 \end{cases}
 \end{equation}
 for some $\lambda_{i} \in \mathbb{F}_q.$ Since there is no $v_1$ in the first line of \eqref{GRc221vn1z}, $\lambda_2=0,$ so
 $$\alpha_2 = \ldots = \alpha_{n-2}=\delta_1 \text{ and } \delta_2=0.$$
 The same arguments for $((v_{n})z)g$ show that $\delta_3=0$ and, since $n \ge 5$,
 $$\alpha_1=\alpha_{n-2}=\delta_4.$$
 Hence $g$ is scalar and $\varphi \in Z(GL_n(q)).$
 
 \medskip
 
 {\bf Case (2.2.2).} Assume $m\ne 2.$  
  We may assume that $n_{k+1}\ne 2$. Indeed, if $n_{k+1}=2$, then  $n_{k}=3$ since $m_k\le n_k/2$ and there is only one $G_i$ equal to $GL_2(q)$ for $i \in \{1, \ldots, k+1\}.$ Therefore, $$S_k=S|_{_{V_k}}\cap GL(V_k)=GL_2(q) \times GL_1(q) \le GL(V_k)=GL_3(q)$$ and, using computation, we obtain that  there are $x_k, y_k \in SL_n(q)$ such that $S_k \cap S_k^{x_k} \cap S_k^{y_k} \le Z(GL_3(q))$, so the conclusion of  $(1)$ of Lemma \ref{GRdiag} holds and the theorem  holds by {\bf Case  1}. 
  
  Therefore, $j_1\ge 3$ where $j_1$ is as defined in the beginning of {\bf Case  2}, so the $(2 \times 2)$ block (corresponding to the $S_i$ lying in $GL_2(q)$) on the diagonal in matrices of $\tilde{S}=S \cap \Gamma$ occurs in the rows $(j_1, j_1+1)$. Recall that if $h \in  S \cap S^x \cap S^y \cap  GL_n(q)$, then $$h= \diag[\alpha_1, \ldots, \alpha_{j_1-1},A, \alpha_{j_1+2}, \ldots, \alpha_n]$$ where $\alpha_i \in \mathbb{F}_q^*$ and $A= \left( \begin{smallmatrix} \delta_1 & \delta_2 \\ \delta_3 & \delta_4  \end{smallmatrix} \right) \in GL_2(q).$ Let $z \in SL_n(q)$ be defined as follows:
\begin{equation}
\label{GRzdefcase222}
\begin{aligned}
(v_i)z & =  v_i &&\text{ for } i \in \{1, \ldots, n\} \backslash \{j_1, j_1+1, n\}; \\
(v_{j_1})z & =  v_1 + v_{j_1};\\
(v_{j_1+1})z & =  v_2 + v_{j_1+1};\\
(v_{n})z & =   \sum_{i=1}^{n-m} v_i   + v_{n}.\\
\end{aligned} 
\end{equation}
Let $\varphi \in S \cap S^x \cap S^y \cap S^z$, so $\varphi= (\iota_{\beta}a)^l \cdot g$ with $g=\diag[g_{i_1},\ldots, g_{i_s}, g_{k+1}, \ldots, g_1]$ where $g_i$ and $g_i'$ are as in \eqref{GRguppdiag}.  
 Assume that $\varphi \notin \Gamma,$ so $l=1.$ Since $\varphi \in S^z$, it stabilises $(U,W)z,$ so $(U,W)z \varphi =(U,W)z$ and, therefore, 
$$Uz=Wz \varphi = W'z^{\iota_{\beta}}a \cdot g,$$
where $W'$ is spanned by $\beta \backslash (W \cap \beta).$  

 With respect to $\beta$,
$$Uz= \left \langle 
\left(\begin{array}{ccc|ccccc}
 0        & \ldots & 0               &           1 & & &  \\
\vdots    &        &\vdots           &           & \ddots  & &  \\
0         & \ldots & 0               &            & &\ddots &  \\
1         & \ldots &  1              &           &  & & 1 
\end{array}\right)
\right \rangle$$
where the part after the vertical line forms $I_m.$

 If $U \cap W=\{0\},$ then $W'=U$ and it is easy to see that $Wz \varphi = W'z^{\iota_{\beta}}a \cdot g= U,$ since $g$ stabilises $U$ by \eqref{GRvarphisy}. So $Uz \ne Wz \varphi$ which is a contradiction. Hence $\varphi \in \Gamma.$
 
 If $U \le W,$ then 
 \begingroup
\allowdisplaybreaks
 \begin{align*}
 Wz\varphi & = W'z^{\iota_{\beta}}ag \\ & = 
 \left\langle I_m\mid 0_{m \times (n-m)} \right\rangle \, z^{\iota_{\beta}}ag \\
& =\Scale[0.95]{\left \langle \left(\begin{array}{ccccc|cccccc|cccc}
1 &  &   &        &          & 0 & \ldots 0 &-1 & 0 & 0 & \ldots 0 &       0  & 0 & \ldots &  0 \\         
  &1 &   &        &          & 0 & \ldots 0 & 0 &-1 & 0 & \ldots 0 &       0  & 0 & \ldots &  0 \\
  &  &1  &        &          & 0 & \ldots 0 & 0 &0  & 0 & \ldots 0 &       -1 & 0 & \ldots &  0 \\
  &  &   & \ddots &          & \vdots &     &   &   &   &          &   \vdots &   &        &    \\
  &  &   &        &1         & 0 & \ldots 0 & 0 &0  & 0 & \ldots 0 &       -1 & 0 & \ldots &  0 \\ 
\end{array}\right)
\right \rangle}\, a g\\
& = \Scale[0.95]{ \left \langle \left(\begin{array}{cccc|cccccc|ccccc}
 0  & 0 & \ldots &  0          & 0 & \ldots 0 &-1 & 0 & 0 & \ldots 0 &       1 &  &   &        & \\         
 0  & 0 & \ldots &  0        & 0 & \ldots 0 & 0 &-1 & 0 & \ldots 0 &         &1 &   &        &   \\
 -1 & 0 & \ldots &  0         & 0 & \ldots 0 & 0 &0  & 0 & \ldots 0 &         &  &1  &        &  \\
  &   \vdots &   &        &            & \vdots &     &   &   &   &           &  &   & \ddots &  \\
  -1 & 0 & \ldots &  0        & 0 & \ldots 0 & 0 &0  & 0 & \ldots 0 &        &  &   &        &1  \\ 
\end{array}\right)
\right \rangle}\, g
 \end{align*}
 \endgroup
 where $-1$ in the first row is in the $j_1$ entry, and $-1$ in the second row is in the $j_1+1$ entry. The result of the action of $g$ on the first two rows is 
 $$\langle 0_{2\times 1}, \ldots, 0_{2\times 1}, -g_t, 0_{2\times 1}, \ldots, 0_{2\times 1}, g_1^{(1,2)} \rangle $$ where $t \in \{1, \ldots, k+1\}$ is such that $G_t=GL_2(q)$ and $g_1^{(1,2)}$ is the matrix formed by the first two rows of $g_1.$ It is easy to see that  two such vectors cannot lie in $Uz$, so  $Uz \ne Wz \varphi$ which is a contradiction. Hence $\varphi =g\in \Gamma.$ 
 
 Therefore $g=\diag[A,g_t, B]$ where $$A=\diag(\alpha_1, \ldots, \alpha_{j_1-1}), \text{ }B=\diag(\alpha_{j_1+2}, \ldots, \alpha_{n}),$$  for some $\alpha_i \in \mathbb{F}_q^*$ and $$g_t = \left( \begin{matrix} \delta_1 & \delta_2 \\ \delta_3 & \delta_4  \end{matrix} \right) \in GL_2(q).$$ Since $g \in S^z$, it stabilises $Uz.$
 Notice that $S \cap GL_n(q)$ also stabilises $\langle v_{j_1}, \ldots, v_n \rangle,$ so $g \in S^z \cap GL_n(q)$ stabilises $\langle v_{j_1}, \ldots, v_n \rangle z$. 
 
 Consider 
\begin{equation}
\label{GR222vj1}
((v_{j_1})z)g= 
 \begin{cases}
 \alpha_1 v_1 + \delta_1 v_{j_1} + \delta_2 v_{j_1+1};\\
 \sum_{i=j_1}^n \lambda_i v_i
 \end{cases}
\end{equation}
 for some $\lambda_i \in \mathbb{F}_q.$ Since the first line of \eqref{GR222vj1} contains no terms with $v_2$ and $v_i$ for $i \ge j_1+2$, we obtain   
 $((v_{j_1})z)g= \alpha (v_{j_1})z $ for some $\alpha \in \mathbb{F}_q^*.$ Therefore, $\delta_2=0$ and $\alpha= \alpha_1 = \delta_1.$ The same arguments applied to $(v_{j_1+1})z$ show that
 $\delta_4=0$ and $\alpha_2=\delta_3.$ 
 
   The same arguments applied to $(v_{n})z$ show that
$\alpha= \alpha_1 = \ldots =\alpha_n,$ so $g$ is scalar and $\varphi \in Z(GL_n(q)).$
\end{proof}

\section{Unitary groups}

In this section $S$ is a maximal solvable subgroup of $\GU_n(q)=GU_n(q) \rtimes \langle \phi_{\beta}\rangle$ where $\beta$ is an orthonormal basis of $(V, {\bf f}).$ Our goal is to prove the following theorem.

\begin{T2}
Let $X=\GU_n(q)$, $n \ge 3$ and $(n,q)$ is not equal to $(3,2).$ If $S$ is a maximal solvable subgroup of $X$, 
 then one of the following holds:
\begin{itemize}
\item  $b_S(S \cdot SU_n(q)) \le 4,$ so $\Reg_S(S \cdot SU_n(q),5)\ge 5$;
\item   $(n,q)=(5,2)$ and $S$ is the stabiliser in $X$ of a totally isotropic subspace of dimension $1$, $b_S(S \cdot SU_n(q)) =5$ and $\Reg_S(S \cdot SU_n(q),5)\ge 5$. 
\end{itemize}
\end{T2}

Recall that $g^{\dagger}=(\overline{g}^{\top})^{-1}$ for $g \in GL_n(q^{\bf u}),$ see the discussion after Definition \ref{taudef} for details. To prove  Theorem \ref{theoremGU}, we need the following lemma.

\begin{Lem}\label{starc}
Let $(n,q, {\bf u})$ be such that $GL_n(q^{\bf u})$ is not solvable. If $S$ is an irreducible maximal solvable  subgroup of $GL_n(q^{\bf u})$, then there exist $x, y \in SL_n(q^{\bf u})$ such that $$S \cap S^x \cap (S^{\dagger})^y \le Z(GL_n(q^{\bf u})).$$
\end{Lem} 
\begin{proof}
If $b_S(S \cdot SL_n(q^{\bf u}))=2,$ then there exists $x \in SL_n(q^{\bf u})$ such that $$S \cap S^x \le Z(GL_n(q^{\bf u})),$$ so $y$ can be arbitrary. Therefore, it suffices to consider cases \eqref{irred11}--\eqref{irred15} from Theorem \ref{irred} only.
 In cases \eqref{irred11}, \eqref{irred12} and \eqref{irred14}, $S$ is the normaliser of a Singer cycle, so $S \cdot SL_n(q^{\bf u})=GL_n(q^{\bf u})$. Since all Singer cycles are conjugate in $GL_n(q^{\bf u}),$ $S^{\dagger}=S^g$ for some $g \in GL_n(q^{\bf u})$, so the statement follows by Theorem \ref{irred}. 
In  case \eqref{irred15} $S^{\dagger}=S,$ so the statement follows since $b_S(S \cdot SL_4(3))\le 3$ by Theorem \ref{irred}.
In  case \eqref{irred13} the statement is verified by computation.  
\end{proof}


\begin{Lem} \label{lemn4uni}
Theorem {\rm \ref{theoremGU}} holds for $n=3.$
\end{Lem}
\begin{proof}
If $S$  stabilises no non-zero proper subspace of $V$, then the statement follows by 
\cite[Theorem 1.1]{burness}.

\medskip

Assume that $S$ stabilises $U<V$ and $S$  stabilises no non-zero proper subspace of $U$, so  $U$ is either totally isotropic or non-degenerate.

 If $U$ is totally isotropic, then $\dim U=1$ since a maximal totally isotropic subspace of a non-degenerate unitary space of dimension $n$  has dimension $[n/2]$.  By Lemma \ref{unist}, there exists a basis $\beta=\{f,v,e\}$ such that ${\bf f}_{\beta}$ is the permutation matrix for the permutation $(1,3)$ and all elements in $S_{\beta}$ have shape $\phi^j g$ with
\begin{equation*}
g=\begin{pmatrix}
\alpha_1^{\dagger} &* &*\\
 0 & \alpha_2 &*\\
0& 0 & \alpha_1
\end{pmatrix}
\end{equation*}
where $j \in \{0,1 , \ldots, 2f-1\}$, $\alpha_i \in \mathbb{F}_{q^2}^*$ and $\alpha_2^{q+1}=1.$
 Let $\eta$ be a generator of $\mathbb{F}_{q^2}^*.$ For $q$  even let $\delta=1$, for $q$ odd let $\delta=\eta^{-(q+1)/2},$ so $\delta \cdot \delta^{\dagger}=\delta^{1-q}=-1.$
The matrix $x=\diag(\delta^{\dagger}, 1 \ldots, 1, \delta){\bf f}_{\beta}$
lies in $SU_n(q,{\bf f}_{\beta})$. It is routine to check that if $\varphi \in S_{\beta} \cap S_{\beta}^x$, then $\varphi= \phi^j g$ with $g=\diag(\alpha_1^{\dagger}, \alpha_2, \alpha_1)$.
Let $\alpha \in \mathbb{F}_{q^2}$ be such that $\alpha + \alpha^q =1.$ It exists by Lemma \ref{al}. Let $\theta \in \mathbb{F}_{q^2}$ be $\eta^{q-1}$ and let $y, z \in SU_3(q,{\bf f}_{\beta})$  be $$
\begin{pmatrix}
1 & 0& 0\\
 -1 & 1 &0\\
-\alpha& 1 & 1
\end{pmatrix} \text{ and }
\begin{pmatrix}
1 & 0& 0\\
 -\theta^{-1} & 1 &0\\
-\alpha& \theta & 1
\end{pmatrix} 
$$ respectively. If $\varphi \in S_{\beta} \cap S_{\beta}^x \cap S_{\beta}^y,$ then $\varphi$ stabilises $\langle e \rangle y= \langle e +v - \alpha f \rangle,$ so $\alpha_1=\alpha_2$ and $\varphi=\phi^j \alpha_1 I_3.$ If $\varphi \in S_{\beta} \cap S_{\beta}^x \cap S_{\beta}^y  \cap S_{\beta}^z,$ then $\varphi$ stabilises $\langle e \rangle z= \langle e +\theta v - \alpha f \rangle,$ so $\theta^{p^j}\alpha_1= \theta \alpha_1$. Thus, $\theta^{p^j-1}=1$ and $j=0$ by Lemma \ref{pj10}, so $\varphi \in Z(GU_3(q)).$

 Assume $U$ is non-degenerate, so $S$ stabilises $U^{\bot}$ and we can assume that $\dim U=1$. Let $\beta=\{f,e,v\},$ where $\{f,e\}$ is a basis of $U^{\bot}$ as in \eqref{unibasis} and $U= \langle v\rangle$.
Let $x,y,z \in SU_3(q, {\bf f}_{\beta})$ be
$$
\begin{pmatrix}
1 & 0& 0\\
 -\alpha & 1 &1\\
-1& 0 & 1
\end{pmatrix},
\begin{pmatrix}
1 & -\alpha & 1\\
 0 & 1 &0\\
0& -1 & 1
\end{pmatrix}  \text{ and }
\begin{pmatrix}
1 & 0 & 0\\
 -\alpha & 1 &\theta^{-1}\\
-\theta & 0 & 1
\end{pmatrix} 
$$ 
respectively.
If $\varphi \in S_{\beta} \cap S_{\beta}^x \cap S_{\beta}^y  \cap S_{\beta}^z,$ then $\varphi$ stabilises $\langle v \rangle$, $\langle v-f \rangle$, $\langle v-e \rangle$ and $\langle v - \theta f \rangle.$ Arguments as in the previous case show that $\varphi \in Z(GU_3(q)).$
\end{proof}

\begin{Lem}
\label{3conjM}
Let $M=S \cap GU_n(q).$ If $S$ stabilises no non-zero proper subspace of $V$, then there exist $y,z \in SU_n(q)$ such that $M \cap M^y \cap M^z \le Z(GU_n(q))$ unless $(n,q)=(4,2)$ and $M=MU_4(2)$ is as defined in  Theorem $\ref{irredGU}$. 
\end{Lem}
\begin{proof}
If $M \le GU_n(q)$ is irreducible, then such $y,z$ exist by Theorem \ref{irredGU}. Assume that $M$ is reducible. The same arguments as in the proof of Lemma \ref{GammairGL} show that $M$ is completely reducible. If $V$ is not $\mathbb{F}_{q^2}[M]$-homogeneous, then $S$ (and $M$) stabilises a decomposition of $V$ as in Lemma \ref{ashb}, and such $y,z$ exist by the proof of Theorem \ref{irredGU}. If $V$ is  $\mathbb{F}_{q^2}[M]$-homogeneous, then $M$ stabilises a decomposition as in    Lemma \ref{ashb} by \cite[(5.2) and (5.3)]{asch}, and such $y,z$ exist by the proof of Theorem \ref{irredGU}.
\end{proof}

\begin{Th}
\label{GU4sp}
Theorem {\rm \ref{theoremGU}} holds for $n \ge 4$ if $S$ stabilises no non-zero proper subspace of $V$.
\end{Th}
\begin{proof}
 The result follows by \cite[Theorem 1.1]{burness} unless $n=4$ and $S$ lies in a maximal subgroup of type $Sp_4(q)$ as in \cite[Table 1]{burness}. We now consider this outstanding case.

Let $n=4$ and $M=S \cap GU_n(q).$
If $S$ stabilises a decomposition of $V$ as in Lemma \ref{ashb}, then the statement follows by \cite[Table 2]{burness}.  Hence we can assume that if $N\le M$ is normal in $S$, then $V$ is $\mathbb{F}_{q^2}[N]$-homogeneous. In particular, every characteristic abelian  subgroup of $M$ is cyclic by \cite[Lemma 0.5]{manz}. 

Assume that $M$ is reducible, so $M$ stabilises non-zero $W<V$ such that $W$ is $\mathbb{F}_{q^2}[M]$-irreducible and $W$ is either non-degenerate or totally isotropic. If $V$ is not $\mathbb{F}_{q^2}[M]$-homogeneous, then $S$ stabilises a decomposition as in Lemma \ref{ashb} which contradicts the assumption above, so   $V$ is  $\mathbb{F}_{q^2}[M]$-homogeneous. Therefore, if $\dim W=1,$ then $M$ is a group of scalars, so $S/Z(GU_n(q))$ is cyclic and $b_S(S \cdot GU_4(q)) \le 2$ by Theorem \ref{zenab}. Hence we may assume that $\dim W=2$ and $W$ is either totally isotropic or non-degenerate.

First assume that $\dim W=2$ and $W$ is totally isotropic. By \cite[(5.2)]{asch}, 
$$V=W_1 \oplus W_2$$ where $W_i$ is a $M$-invariant submodule of $V$ isometric to $W$, so we can assume $W_1=W.$ Let $\beta$ be a basis as in \eqref{unibasis} corresponding to this decomposition of $V$. Let $M_1 \le GL_2(q^2)$ be the restriction of $M$ in $W.$ By Theorem \ref{irred}, either there exists $x_1 \in SL_2(q^2)$ such that $M_1 \cap M_1^{x_1} \le Z(GL_2(q))$ or $M_1$ is a subgroup of the normaliser of a Singer cycle in $GL_2(q^2).$ If $x_1$ as above exists, then $M \cap M^x\le Z(GU_n(q))$ where $x_{\beta}=\diag[x_1, x_1^{\dagger}],$ since $V$ is $\mathbb{F}_q[M]$-homogeneous. Therefore, $b_S(S \cdot SU_4(q)) \le 4$ by Theorem \ref{zenab}.

Let $M_1$ be  a subgroup of the normaliser of a Singer cycle in $GL_2(q^2).$ Since $M  \cong M_1,$ it has a maximal abelian normal subgroup $A$ of index at most $2$, which is also characteristic. Hence $V$ is $\mathbb{F}_{q^2}[A]$-homogeneous and the dimension of an irreducible $\mathbb{F}_{q^2}[A]$-submodule of $V$ is odd by Lemma \ref{simpcycl}, so $A$ is a group of scalars. So $M$ is cyclic modulo scalars and we obtain $b_S(S \cdot SU_4(q)) \le 4$ by applying Theorem \ref{zenab} twice.

Now let us assume that either  $\dim W=2$ and $W$ is non-degenerate or $M$ is irreducible (here we let $W=V$, so $\dim W=4$). Let $m=\dim W.$  Since every characteristic abelian subgroup of $M$ is cyclic, $M$ satisfies the conditions of \cite[Corollary 1.4]{manz}. In particular, in the notation of Lemma \ref{olaf}, the following hold: 
\begin{enumerate}[font=\normalfont]
\item $F=ET$, $Z=E \cap T$ and $T=C_F(E);$
\item a Sylow subgroup of $E$ is either cyclic of prime order or extra-special;
\item there exists $U \le  T$ of index at most $2$ with $U$ cyclic and characteristic in $M$, and
$C_T(U)=U$;
\item $EU = C_F(U)$ is characteristic in $M$.
\end{enumerate}
Since $U$ is characteristic in $M$, $V$ is $\mathbb{F}_{q^2}[U]$-homogeneous, so, by Lemma \ref{simpcycl}, $U$ is a group of scalars, $T=U$ and $M=C=C_M(U).$ Let $e$ be such that $e^2=|E/Z|.$ Let $0<L \le W$ be an $\mathbb{F}_{q^2}[EU]$-submodule. By \cite[Corollary 2.6]{manz}, $$m=e \cdot \dim L.$$
 Thus, $e \in \{1,2,4\}$, so $E$ is either cyclic or an extra-special $2$-group. By the proof of $(vii)$ and $(ix)$ of \cite[Corollary 1.10]{manz}, 
$F=C_M(E/Z)$ and $M/F$ is trivial for $e=1$ and  isomorphic to a subgroup of $Sp_e(2)$ for $e \in \{2,4\}.$

If $e=1,$ then $F=U$ is self-centralising (since the centraliser of the Fitting subgroup of a solvable group lies in the Fitting subgroup) and $W$ is  $\mathbb{F}_{q^2}[U]$-irreducible by \cite[Lemma 2.2]{manz}, which is a contradiction, since $U$ is a group of scalars. Therefore, $e \in \{2,4\}.$   

If $e=4,$ then, by the proof of Lemma \ref{c6small}, $M=M_1 \cdot Z(GU_4(q))$ and $M_1$ lies in the normaliser of a symplectic-type subgroup of $GU_4(p^t)$ for some $t \le f$. Hence  $b_M(M \cdot SU_4(q)) \le 2$ for $q>3$ by \cite[Table 2]{burness} and $b_S(S \cdot SU_4(q))\le 4$ by Theorem \ref{zenab}. For $q \le 3$ the statement is verified by computation.

Let $e=2.$ Therefore, $|M|=|U|\cdot|E/Z| \cdot |M/F|$ divides $$(q+1) \cdot e^2 \cdot |Sp_2(2)|=24(q+1).$$ So $|S|$ divides $24(q+1) \cdot 2f$ and $|S/Z(GU_4(q))|$ divides $48f.$ We claim that $$\hat{Q}((S \cdot SU_4(q)/Z(GU_4(q)),4)<1$$ where $\hat{Q}(G,c)$ is as in \eqref{ver} and $H=S/Z(GU_4(q))$. By Lemma \ref{fprAB}, if  $x_1,\ldots,x_k$ represent distinct $G$-classes such that $\sum_{i=1}^k |x_i^G \cap  H| \le A$ and $|x_i^G| \ge B$ for all $i \in \{1, \ldots, k\},$ then
$$\sum_{i=1}^m |x_i^G| \cdot \fpr (x_i)^c \le B \cdot (A/B)^c.$$ 
We take $A= 48 f \ge |H| \ge \sum_{i=1}^k |x_i^G \cap  H|.$ For elements in $PGU_4(q)$ of prime order with $s=\nu(x) \in \{1,2,3\}$ we use \eqref{5uni} as a lower bound for $|x_i^G|$. If $x \in H \backslash PGU_4(q)$ has prime order, then we use the corresponding bound for $|x^G|$ in \cite[Corollary 3.49]{fpr2}. We take $B$ to be the smallest of these bounds for $|x_i^G|.$ For $q \ge 5$, such $A$ and $B$ are sufficient to obtain $$\hat{Q}((S \cdot SU_4(q)/Z(GU_4(q)),4)<1,$$ so $b_S(S \cdot SU_4(q)) \le 4.$ For $q \le 4$ the theorem is verified by computation.
\end{proof}

\begin{Th}
\label{lem421}
Theorem {\rm \ref{theoremGU}} holds for  $n \ge 4$ if $S$ stabilises a non-zero proper subspace of $V$. 
\end{Th}
\begin{proof}
The proof proceeds in two steps. In {\bf Step 1} we obtain three conjugates of $S$ such that elements of their intersection have shape $\phi_{\beta} g$ for some basis $\beta$ of $V$ where $g \in GU_n(q, {\bf f}_{\beta})$ is diagonal or has  few non-zero entries not on the diagonal. In {\bf Step 2} we find a fourth conjugate of $S$ such that the intersection of the four is a group of scalars.

\subsection*{Step 1} 
 Fix a basis $\beta$ of the unitary space $(V, {\bf f})$ as in Lemma \ref{unist}, so ${\bf f}_{\beta}$ is as in \eqref{fst} and elements of $S$ take shape $\phi_{\beta}^{j}g$ with $g$ as in \eqref{gst} and $j \in \{0, 1, \ldots, 2f-1\}$. We consider $S$ as a subgroup of $\GU_n(q,{\bf f}_{\beta}).$ Let $M$ be $S \cap GU_n(q,{\bf f}_{\beta}).$  We obtain three conjugates of $S$ such that their intersection consists of elements $\phi_{\beta}^{j}g$ where $g$ is diagonal with respect to $\beta.$

Let $\gamma_i$ be as in Lemma \ref{unist}. 
Observe that ${\bf f}_{\beta}{\bf f}_{\beta}\overline{{\bf f}_{\beta}}^{\top}={\bf f}_{\beta},$ so ${\bf f}_{\beta} \in GU_n(q,{\bf f}_{\beta}).$  Notice that $\det({\bf f}_{\beta})=(-1)^{n_1 + \ldots + n_k}.$ 
 If $\sum_{i=1}^k n_i$ is odd, then one of the $n_r$ is odd for some $r \in \{1, \ldots, k\}.$ Let $\delta$ be as in the proof of Lemma \ref{lemn4uni}, so $\delta \delta^{\dagger}=-1.$ Notice that 
\begin{equation*} \label{hdetJn}
h=\diag[I_{n_1}, \ldots, I_{n_{r-1}}, \delta^{\dagger}I_{n_r}, I_{n_r+1}, \ldots, I_{n_r+1}, \delta I_{n_r},  I_{n_{r-1}}, \ldots, I_{n_1} ] \in GU_n(q,{\bf f}_{\beta})
\end{equation*}
 has determinant $\det({\bf f}_{\beta})$. 
  In particular, $x=h{\bf f}_{\beta} \in SU_n(q)$.
  It is easy to see that if $g \in M,$ so it has shape \eqref{gst}, then
\begin{equation*}
g^x=
\begin{pmatrix}
\gamma_{1}(g)& &\multicolumn{1}{l|}{0} & &  & & &    &0  \\
    *    & \ddots &\multicolumn{1}{l|}{} & & & & &    &\\
*        &* & \multicolumn{1}{l|}{\gamma_{k}(g)}& & & & &    & \\  \cline{1-6}
 *       &\ldots & \multicolumn{1}{l|}{*} &\gamma_{k+1}(g) & & \multicolumn{1}{l|}{0}  & &    &\\ 
  *      &\ldots &  \multicolumn{1}{l|}{*}& &\ddots &\multicolumn{1}{l|}{}       & &   &\\
   *     &\ldots &  \multicolumn{1}{l|}{*}&0 & &\multicolumn{1}{l|}{\gamma_{k+l}(g)}    &  &    & \\ \cline{4-9} 
    *    &\ldots & & & &\multicolumn{1}{l|}{*} & {{\gamma_{k}(g)}^{\dagger}}  &    &0 \\
    *    &\ldots & & & &\multicolumn{1}{l|}{*} &* & \ddots   &\\
*        &\ldots & & & &\multicolumn{1}{l|}{*} &* & *   & {{\gamma_{1}(g)}^{\dagger}} \\  
\end{pmatrix}.
\end{equation*}

Let $q>3.$ Notice that by Lemma \ref{GammairGL}, if $N$ is a solvable subgroup of $\GL_n(q)$  stabilising no  non-zero proper subspace, then $N \cap GL_n(q)$ lies in an irreducible maximal  solvable subgroup of $GL_n(q).$ Therefore, by Lemmas \ref{3conjM}  and   \ref{starc}  there exist $y_i, z_i \in SL_{n_i}(q^2)$ for $i=1, \ldots, k$ and $y_i, z_i \in SU_{n_i}(q)$ for $i =k+1, \ldots, k+l$ such that 
\begin{equation}\label{smint}
\gamma_i(M) \cap \gamma_i(M)^{y_i} \cap (\gamma_i(M)^{\dagger})^{z_i}  \le Z(GL_{n_i}(q^2)). 
\end{equation}
Notice that $\gamma_i(M)^{\dagger} = \gamma_i(M)$ for $i=k+1, \ldots, k+l$. Denote by $y$ and $z$ the block-diagonal matrices 
\begin{equation}
\label{yzdef}
\begin{split}
&\diag[y_1^{\dagger}, \ldots, y_k^{\dagger}, y_{k+1}, \ldots, y_{k+l}, y_k, \ldots, y_1]  \text{ and }\\
&\diag[z_1^{\dagger}, \ldots, z_k^{\dagger}, z_{k+1}, \ldots, z_{k+l}, z_k, \ldots, z_1]
\end{split}
\end{equation}
 respectively. It is routine to check that $y,z \in SU_n(q,{\bf f}_{\beta}).$


 Therefore, if $g \in M \cap M^{xz},$ then $g$ is the block-diagonal matrix
\begin{equation}\label{gd1}
\diag [g_1^{\dagger}, \ldots, g_k^{\dagger}, g_{k+1}, \ldots, g_{k+l}, g_k, \ldots, g_1],
\end{equation}
where $g_i \in \gamma_i(M) \cap (\gamma_i(M)^{\dagger})^{z_i}$ for $i=1, \ldots, k+l.$ 
Thus, if $g \in M \cap M^{y} \cap M^{xz}$, then $g$ has shape \eqref{gd1} where
$$g_i \in \gamma_i(M) \cap \gamma_i(M)^{y_i} \cap (\gamma_i(M)^{\dagger})^{z_i} \le Z(GL_{n_i}(q^2)) \text{ for } i=1, \ldots, k+l.$$ 
So, by Lemma \ref{scfield}, we can assume that elements in $\gamma_i(S) \cap \gamma_i(S)^{y_i} \cap (\gamma_i(S)^{\dagger})^{z_i}$ have shape $\phi^j g_i$ with $g_i \in Z(GL_{n_i}(q^2)).$ Thus, if 
$ \varphi \in S \cap S^{y} \cap S^{xz},$ then $\varphi= \phi^jg$ with $g$ as in \eqref{gd1} and $g_i \in Z(GL_{n_i}(q^2)).$
 Denote $S \cap S^{y} \cap S^{xz}$ by $\tilde{S}$ and $M \cap \tilde{S}$ by $\tilde{M}.$

\medskip

If $q \in \{2,3\}$, then it may be that $\gamma_{k+i}(M) \in \{GU_2(q), GU_3(2), MU_4(2)\}.$ Recall that $MU_n(q)$ is defined in Lemma \ref{omnom}.  In view of Theorem \ref{irredGU}, and since $GU_2(q)$ and $GU_3(2)$ are solvable, elements $y_{k+i}$ and $z_{k+i}$ as in \eqref{smint} do not exist. If there is more than one such $\gamma_{k+i}(M)$, say 
\begin{equation*}
\gamma_{k+i_1}(M), \ldots, \gamma_{k+i_{\mu}}(M),
\end{equation*}
then we join them in pairs, and there is one such group without pair if $\mu$ is odd. Let $H_1, H_2 \in \{GU_2(q), GU_3(2), MU_4(2)\}$ and let $\nu_j$ for $j \in \{1,2\}$ be the corresponding degree of $H_i,$ so $H_i \le GU_{\nu_j}(q).$ Let $$H=H_1 \times H_2=\{\diag[h_1, h_2] \mid h_j \in H_j\} \le GU_{\nu_1+ \nu_2}(q).$$  Computations show that $$b_H(H \cdot SU_{\nu_1+ \nu_2}(q)) \le 3.$$ Therefore, we can assume that there is at most one such $\gamma_{k+i}(S),$ so $\mu \le 1$. Denote the degree of such $\gamma_{k+i}(S)$ by $\nu$, so $2 \le \nu \le 4.$  Repeating the argument above for the rest of $\gamma_{i}(M)$ and $\gamma_{i}(S)$, we obtain that if $\varphi \in \tilde{S},$ then $\varphi =\phi^j g$ with $g$ as in \eqref{gd1} and either all $g_i \in Z(GL_{n_i}(q^2))$ (if $\mu=0$) or  all but one  $g_i \in Z(GL_{n_i}(q^2))$ and one $g_i$ (for $i>k$) is a 
$(\nu \times \nu)$ matrix (if $\mu=1$).   

\begin{Rem}
\label{remni1}
It may  be  that some  $\gamma_{k+i}(M)$ have degree 1, so $\gamma_{k+i}(S) \le \GU_1(q).$ 
We can treat them together. Indeed, assume that $\gamma_{k+i}(M)$ has degree 1 for $i=1, \ldots, \zeta \le l$. Define $\gamma_{k+1}{'} : S \to \GU_{\zeta}(q)$ by 
$$\gamma_{k+1}{'}(\phi^j g)= \phi^j \diag(\gamma_1(g), \ldots, \gamma_{\zeta}(g))$$ for  $g \in M$ and $j \in \{0,1, \ldots, 2f-1\}.$
 Hence the group $T=\gamma_{k+1}{'}(M)$, consisting of block-diagonal matrices, 
is an abelian subgroup of $T \cdot SU_{\zeta}(q)=GU_{\zeta}(q).$ If $q>3$, then by Theorem \ref{zenab} there exists $g \in SU_{\zeta}(q)$ such that 
$T \cap T^g\le {\bf F}(GU_{\zeta}(q))=Z(GU_{\zeta}(q)),$ where ${\bf F}(G)$ is the Fitting subgroup of a finite group $G$. So \eqref{smint} holds for $\gamma_{k+1}{'}(M)$ and  we can replace $\gamma_1, \ldots, \gamma_{\zeta}$ with 
$\gamma_{k+1}{'}$ of degree $\zeta$. 
Finally, suppose $q \le 3$. Notice that $$T \rtimes \langle \phi \rangle < ((GU_1(q))^{(1/2 +(-1)^{{\zeta}-1}/2)} \times (GU_2(q))^{[{\zeta}/2]}) \rtimes \langle \phi \rangle,$$ so if ${\zeta}>1$, then $S$ is not a maximal solvable subgroup of $GU_n(q)$ since $GU_2(q)$ is solvable.
 Therefore, we can assume that there is at most one $\gamma_{k+i}(S)$ of degree 1 in every case, so ${\zeta} \le 1.$
\end{Rem}

\bigskip

 We summarise the outcome of {\bf Step 1}. Let $\mu$ be the number of $i \in \{1, \ldots, l\}$ such that $\gamma_{k+i}(M) \in \{GU_2(2), GU_2(3), GU_3(2), MU_4(2)\}.$  We may assume $\mu \in \{0,1\}.$ In particular, $\mu=0$ if $q>3.$  There exist $x,y \in GU_n(q)$ such that if $\varphi \in \tilde{S}=S\cap S^x \cap S^y$, then    $\varphi =\phi^j g$ with $g$ as in \eqref{gd1} and either all $g_i \in Z(GL_{n_i}(q^2))$ (if $\mu=0$) or  all but one  $g_i \in Z(GL_{n_i}(q^2))$ and one $g_i$ (for $i>k$) is a 
$(\nu \times \nu)$ matrix (if $\mu=1$). Notice that $\nu$ is 2, 3, or 4 if the corresponding $\gamma_{k+i}(M)$ is $GU_2(q)$,  $GU_3(2)$ and  $MU_4(2)$ respectively.

\subsection*{Step 2} We now find a fourth conjugate of $S$ such that its intersection with $\tilde{S}$ lies in  $Z(GU_n(q)).$ Let $\varphi$ be an element of $\tilde{S}.$

Assume that $S$ is such that $\mu=0.$
First we slightly modify the basis $\beta$ from the first step. Recall that $\beta$ is such that 
$\bf f_{\beta}$ is as in \eqref{fst}. Therefore, 
\begin{equation*}
\begin{aligned}
\beta  =  & \; \{f_1^{1}, \ldots, f_{n_1}^1, \ldots, f_1^k, \ldots, f_{n_k}^k, \\ & \; x_1^1, \ldots, x_{n_{k+1}}^1, \ldots, x_1^l, \ldots, x_{n_{k+l}}^l,\\ & \; e_{1}^{k}, \ldots, e_{n_k}^k, \ldots, e_{1}^1, \ldots, e_{n_1}^1 \},
\end{aligned}
\end{equation*}
where $(e_i^j,f_i^j)=1$ and  every other pair of vectors from $\beta$ is mutually orthogonal.  Let 
\begin{equation}\label{basisW}
\begin{aligned}
U_i & =\langle x_1^i, \ldots, x_{n_{k+i}}^i \rangle, &\text{ } i&=1, \ldots, l;\\
W_i & =\langle f_1^i, \ldots, f_{n_{i}}^i,  e_{1}^i, \ldots, e_{n_{i}}^i \rangle, &\text{ } i&=1, \ldots, k. 
\end{aligned}
\end{equation}
Thus, $$V=(W_1 \bot \ldots \bot W_k) \bot (U_1 \bot \ldots \bot U_l),$$
where $W_i$, $U_i$ are $\tilde{S}$-invariant subspaces 
 and $\gamma_{k+i}(S) \le \GU(U_i)$ for $i=1, \ldots, l$.  By Lemma \ref{unibasisl}, we can choose for $U_i$ the basis 
\begin{equation}\label{basisU}
\beta_{1i}=
\begin{cases} 
\{f_1^{k+i}, \ldots, f_{m_i}^{k+i}, e_{m_i}^{k+i}, \ldots, e_1^{k+i}\}, & \text{ if $n_{k+i}=2m_i$}; \\
\{f_1^{k+i}, \ldots, f_{m_i}^{k+i}, x^{k+i}, e_{m_i}^{k+i}, \ldots, e_1^{k+i}\}, & \text{ if $n_{k+i}=2m_i+1$. } 
\end{cases}
\end{equation} 
 By the first step
$$\gamma_{k+i}(\tilde{M}) \le  Z(GU(U_i)),$$
so, by Lemmas  \ref{uniGamsdp} and \ref{scfield}, $\gamma_i(\varphi)=\phi_{\beta_{1i}}^jg_i$ with $g_i \in Z(GU(U_i)).$

Now we renumber the basis vectors of the $W_i$ from \eqref{basisW} and basis vectors of the $U_i$ from \eqref{basisU} to obtain the basis 
$$\beta_1=\{f_1, \ldots, f_m, x_1, \ldots, x_t, e_m, \ldots, e_1\},$$
where $m=\left( \sum_{i=1}^{k}n_k +\sum_{i=1}^l m_{i} \right)$ and $t$ is the number of odd $n_{k+i}$ for $i=1, \ldots, l.$ In more detail, to obtain $\beta_1$ from $\beta$, we apply the following procedure:
\begin{itemize}
\item  replace bases of $U_i$ as in \eqref{basisW} by those as in \eqref{basisU}, denote new basis by $\beta_{1/3}$;
\item rearrange vectors as follows: first write down the $f^i_j$ in the order they occur in $\beta_{1/3},$ then do the same with the $x^i$ and then write the $e^i_j$ in the order opposite to the $f^i_j$ (so if $f^i_j$ is the $t$-th entry of $\beta_{1/3},$ then $e^i_j$ is the $(n-t+1)$-th entry of $\beta_{1/3}$). Denote new basis by $\beta_{2/3}$;
\item relabel the $f$-vectors with just one index in the order they occur, do the same with the $x$-vectors and label  the  $e$-vectors such that $(f_i,e_i)=1.$
\end{itemize}
We illustrate this procedure in the following example. 
\begin{example}
Let $k=2$, $l=2$, $n_1=1,$ $n_2=2,$ $n_3=2$ and $n_4=3.$ So
\begin{equation*}
\begin{split}
U_1& = \langle x_1^1, x_2^1 \rangle= \langle f_1^3, e_1^3 \rangle\\
U_2& = \langle x_1^2, x_2^2,x_3^2 \rangle= \langle f_1^4, x^4, e_1^4 \rangle \\
W_1& = \langle f_1^1, e_1^1 \rangle \\
W_2& = \langle f_1^2, f_2^2, e_1^2, e_2^2 \rangle
\end{split}
\end{equation*}
and $$\beta=\{f_1^1, f_1^2, f_2^2, x_1^1, x_2^1, x_1^2, x_2^2, x_3^2, e_1^2,e_2^2,e_1^1\}.$$ Hence 
$$\beta_{1/3}=\{f_1^1, f_1^2, f_2^2, f_1^3, e_1^3, f_1^4, x^4, e_1^4, e_1^2,e_2^2,e_1^1\}$$
and 
$$\beta_{2/3}=\{f_1^1, f_1^2, f_2^2, f_1^3, f_1^4, x^4,  e_1^4, e_1^3, e_2^2, e_1^2, e_1^1\}.$$ The relabelling is 
\begin{equation}
\begin{array}{cccccccccccccc}
{\beta_{2/3}}&= \{& f_1^1,& f_1^2,& f_2^2,& f_1^3,& f_1^4,& x^4,&  e_1^4,& e_1^3,& e_2^2,& e_1^2,& e_1^1&\}\\
\downarrow& & \downarrow&\downarrow&\downarrow&\downarrow&\downarrow&\downarrow&\downarrow&\downarrow&\downarrow&\downarrow&\downarrow& \\
\beta_{1}&= \{& f_1,& f_2,& f_3,& f_4,& f_5,& x_1,&  e_5,& e_4,& e_3,& e_2,& e_1&\}.\\
\end{array}
\end{equation}
\end{example}

We now resume the proof of Theorem \ref{lem421}. 
 Notice that $\varphi \in \tilde{S}_{\beta_1}$ has shape $(\phi_{\beta_1})^jg$ with $g$ as in \eqref{gd1} and $g_i \in Z(GL_{n_i}(q^2))$. For simplicity we omit the subscripts and consider $S$ and $\tilde{S}$ as subgroups in $\GU_n(q, {\bf f}_{\beta_1}).$ Let $\phi^jg \in \tilde{S},$ so 
\begin{equation}\label{gdiag}
g=\diag(\alpha_1^{\dagger}, \ldots,  \alpha_m^{\dagger},\delta_1, \ldots, \delta_t, \alpha_m, \ldots, \alpha_1).
\end{equation}
If 
 \begin{equation*}
 U_i =
\begin{cases} 
\langle f_s, \ldots, f_{s+m_i}, e_{s+m_i}, \ldots, e_s\rangle, & \text{ for $n_{k+i}=2m_i$},  \\
\langle f_s, \ldots, f_{s+m_i}, x_r, e_{s+m_i}, \ldots, e_s\rangle, & \text{ for $n_{k+i}=2m_i+1$, } 
\end{cases}
\end{equation*}
 then  $\alpha_s= \ldots =\alpha_{s+m_i}=\delta_r$  and  $\alpha_s^{q+1}=1$ since $g$ is scalar on each $U_i$ by the first step. 
If 
 \begin{equation}\label{Weq}
 W_i = 
\langle f_s, \ldots, f_{s+n_i}, e_{s+n_i}, \ldots, e_s\rangle,   \\
\end{equation}
 then  $\alpha_s= \ldots =\alpha_{s+n_i}$ since $g$ is scalar on  $ \langle e_{s+n_i}, \ldots, e_s\rangle$ by the first step.

\begin{Rem}\label{x1} 
 If  $\alpha_i=\alpha_i^{\dagger}=\alpha_1$ for $i=1, \ldots, m$, then $g$ is not scalar if and only if ${\zeta}=1$ in Remark \ref{remni1}. So, if there exists $\gamma_s(S)$ of degree $1$, then we can assume, without loss of generality, that $\delta_1$ is the corresponding entry (so $\gamma_s(S)$ acts on $\langle x_1\rangle$). Therefore, if  $\alpha_i=\alpha_i^{\dagger}=\delta_1$ for $i=1, \ldots, m$, then $g=\delta_1 I_n \in Z(GU_n(q,{\bf f}_{\beta_1})).$ 
\end{Rem}

The remainder of our proof of {\bf Step 2} splits into 3 cases:
\begin{description}[before={\renewcommand\makelabel[1]{\bfseries ##1}}]
\item[{\bf Case 1.}] $\mu=0$, $k>0;$
\item[{\bf Case 2.}] $\mu=0$, $k=0;$
\item[{\bf Case 3.}] $\mu=1.$
\end{description}
Each case splits into two or three subcases depending on other parameters. In  {\bf Cases 1} and {\bf 2} we show $b_S(S \cdot SU_n(q))\le 4.$ In  {\bf Case 3} we show  $b_S(S \cdot SU_n(q))\le 4$ unless $n$ is small ($q \in \{2,3\}$ here since $\mu=1$). For small $n$ the statement of Theorem \ref{theoremGU} is verified by computation; we identify these values of $n$ in {\bf Case 3}.

\subsection*{Case 1.} Let $\mu=0$  and $k>0.$ So there is a totally singular $S$-invariant subspace 
$$V_1=\langle e_1, \ldots, e_{n_1} \rangle.$$
Recall that $n_i$ is the degree of $\gamma_i(S)$ for $i \in \{1, \ldots, k+l\}.$ Let $\alpha \in \mathbb{F}_{q^2}$ be such that  
$\alpha+\alpha^q=1,$ it exists by Lemma \ref{al}.

\medskip

The three subcases we consider correspond to the following situations:
\begin{description}[before={\renewcommand\makelabel[1]{\bfseries ##1}}]
\item[{\bf Case (1.1)}] $\dim W_i=2$ for $W_i$ in \eqref{basisW} and $i=1, \ldots, k$;
\item[{\bf Case (1.2)}] Condition of {\bf Case (1.1)} does not hold and $l=0$;
\item[{\bf Case (1.3)}] Condition of {\bf Case (1.1)} does not hold and $l>0$.
\end{description} 

\medskip

{\bf Case (1.1).} Assume that  $\dim W_i=2$ for $W_i$ in \eqref{basisW} and $i=1, \ldots, k.$
Let $\eta$ be a generator of $\mathbb{F}_{q^2}^*$ and let $\theta=\eta^{q-1}.$ We redefine $y$ from \eqref{yzdef} to $$\diag[A^{\dagger}, y_{k+1}, \ldots, y_{k+l},A]$$
where 
$$
A=
\begin{pmatrix}
1 & 0   & 0   & \ldots  & 0 \\
0 & 1      & 0   & \ldots  & 0 \\
  &        & \ddots &   &  \\
0 & \ldots & 0      & 1 & 0\\
1 & \ldots & \ldots & 1 & 1\\
\end{pmatrix}.
$$ It is easy to see that $y \in SU_n(q, {\bf f}_{\beta}).$ Let $x,z \in SU_n(q, {\bf f}_{\beta_1})$ be as in {\bf Step 1}, so $\varphi \in \tilde{S}$ has shape $\phi^j g$ with $g$ as in \eqref{gdiag}. Since $S$ stabilises $\langle e_1 \rangle$, $S^y$ stabilises $\langle e_1 \rangle y= \langle e_1 + \ldots +e_k \rangle.$ Therefore,
$$((e_1)y)\varphi=(e_1 + \ldots + e_k)\phi^j g=\alpha_1 e_1 + \ldots + \alpha_k e_k= \lambda (e_1 + \ldots + e_k)$$  
for some $\lambda \in \mathbb{F}_{q^2}^*,$ so $\alpha_1 = \ldots = \alpha_k.$

Let $k\ge 2$. We claim that there exists  $a \in SU_n(q, {\bf f}_{\beta_1})$ such that 
\begin{equation}
\begin{aligned}
(e_1)a =   & \sum_{i=3}^m e_i +\theta e_2+ e_1 + \underline{x_1 -\alpha f_1}; &  (f_1)a & =f_1;  \\
(e_{2})a = & \phantom{(} e_{2}; &  (f_{2})a & =f_{2}- \theta^{-1}f_1; &&   \\
(e_{i})a = & \phantom{(} e_{i};&  (f_{i})a & =f_{i} -  f_{1}; && i \in \{3, \ldots, m \}   \\
\underline{(x_1)a =}&\underline{  \phantom{(}x_1-f_1}, 
\end{aligned}
\end{equation}
and $a$ stabilises all other vectors from $\beta_1.$ Here the underlined part is  in the formula only if ${\zeta}=1$ and $x_1$ is as in Remark \ref{x1}. In other words, if $n_{k+i}>1$  for all $i=1, \ldots, l,$ then we  omit the underlined part.
It is routine to check that $\det(a)=1$ and $a$ is an isometry of $(V, {\bf f}),$ so $a \in SU_n(q,{\bf f}_{\beta_1}).$

We claim that $\tilde{S} \cap S^a \le Z(GU_n(q)).$ Let $\varphi= \phi^jg \in \tilde{S} \cap S^a$, where $g$ is as in \eqref{gdiag}. Observe that $S$ stabilises  $\langle e_1\rangle$, so $S^a$ stabilises $\langle e_1\rangle a.$ Therefore, 
\begin{equation}
((e_1)a)\phi^j g=
\begin{cases}
\sum_{i=3}^m\alpha_i e_i +\theta^{p^j}\alpha_2 e_2 + \alpha_1e_1 + \underline{\delta_1 x_1 -\alpha_1^{\dagger} \alpha f_1}\\
\lambda (e_1)a
\end{cases}
\end{equation}
for some $\lambda \in \mathbb{F}_{q^2}^*.$  Thus, $$\lambda=\alpha_1=\theta^{p^j-1}\alpha_2 =\underline{\alpha_1^{\dagger}} =\alpha_{3}=\ldots=\alpha_m=\underline{\delta_1}$$
and $\theta^{p^j-1}=1,$ so $j=0$ and $\varphi =g \in Z(GU_n(q, {\bf f}_{\beta})).$

\medskip

Let $k=1.$ We can assume that $n_{k+1}\ge 2$. Indeed, if  $n_{k+i}=1$ for all $i \in \{1, \ldots, l\},$ then,  by Remark \ref{remni1}, $l=1$, so $n=3$ and Theorem \ref{theoremGU} follows by Lemma \ref{lemn4uni}. Thus, $\langle e_2, f_2 \rangle \subseteq U_1$ and $\alpha_2=\alpha_2^{\dagger}.$ We claim that there exists  $a \in SU_n(q, {\bf f}_{\beta_1})$ such that 
\begin{equation}
\label{465}
\begin{aligned}
(e_1)a & =    \sum_{i=3}^m e_i +\theta e_2+ e_1 +f_2-\theta f_1 + \underline{x_1 -\alpha f_1}; \\  (f_1)a & =f_1;  \\
(e_{2})a & =  \phantom{(} e_{2}-f_1; \\  (f_{2})a & =f_{2}- \theta^q f_1; &&   \\
(e_{i})a  & = \phantom{(} e_{i}; && i \in \{3, \ldots, m \}\\  (f_{i})a & =f_{i} -  f_{1}; && i \in \{3, \ldots, m \}   \\
\underline{(x_1)a} & \underline{ \; = \phantom{(}x_1-f_1}, 
\end{aligned}
\end{equation}
and $a$ stabilises all other vectors from $\beta_1.$ Here the underlined part is in the formula only if ${\zeta}=1$ and $x_1$ is as in Remark \ref{x1}. 
It is routine to check that $\det(a)=1$ and $a$ is an isometry of $(V, {\bf f}),$ so $a \in SU_n(q,{\bf f}_{\beta_1}).$ 

We claim that $\tilde{S} \cap S^a \le Z(GU_n(q)).$ Let $\varphi= \phi^jg \in \tilde{S} \cap S^a$, where $g$ is as in \eqref{gdiag}. Observe that $S$ stabilises  $\langle e_1\rangle$, so $S^a$ stabilises $\langle e_1\rangle a.$ Therefore, 
\begin{equation}
((e_1)a)\phi^j g=
\begin{cases}
\sum_{i=3}^m\alpha_i e_i +\theta^{p^j}\alpha_2 e_2 + \alpha_1e_1 + \alpha_2 f_2 -\theta^{p^j}\alpha_1^{\dagger} f_1+ \underline{\delta_1 x_1 -\alpha_1^{\dagger} \alpha f_1}\\
\lambda (e_1)a
\end{cases}
\end{equation}
for some $\lambda \in \mathbb{F}_{q^2}^*.$  Thus, $$\lambda=\alpha_1=\alpha_2=\theta^{p^j-1}\alpha_2=\alpha_{3}=\ldots=\alpha_m=\underline{\delta_1}$$
and $\theta^{p^j-1}=1,$ so $j=0$ and $\varphi =g \in Z(GU_n(q, {\bf f}_{\beta})).$

\medskip


{\bf Case (1.2).} Assume that $l=0$ (so there is no $U_i$) and there exists $r \in \{1, \ldots, k\}$ such that $\dim W_r \ge 4$. So 
 $$W_r = \langle f_s, \ldots, f_{s+n_r}, e_{s}, \ldots, e_{s+n_r}\rangle,  \text{ for some } s \text{ and } n_r \ge 2.$$
In particular, $\alpha_s=\alpha_{s+1}.$ Let $\chi=\eta^{(q+1)/2},$ so $\chi +\chi^q=0$ and $\chi^{-q}=-\chi^{-1}.$ We claim that there exists  $a \in SU_n(q, {\bf f}_{\beta_1})$ such that 
\begin{align*}
(e_s)a= & \phantom{(} f_s + \theta f_{s+1} + \sum_{i \notin \{s,s+1\}}^m f_i + \chi e_s; & (f_s)a= & -\chi^{-1}f_s; & \\
(e_{s+1})a=  & \phantom{(} \chi e_{s+1} +\theta^q f_s; & (f_{s+1})a= & -\chi^{-1}f_{s+1}; & \\
(e_i)a= & \phantom{(} e_i+\chi^{-1} f_s; & (f_i)a= & \phantom{(} f_i &  \text{ for } i \ne s.
\end{align*}
It is routine to check that $\det(a)=1$ and $a$ is an isometry of $(V, {\bf f}),$ so $a \in SU_n(q,{\bf f}_{\beta_1}).$ 

We claim that $\tilde{S} \cap S^a \le Z(GU_n(q)).$ Let $\varphi= \phi^jg \in \tilde{S} \cap S^a$, where $g$ is as in \eqref{gdiag}. Notice that $S$ stabilises $E= \langle e_1, \ldots, e_m \rangle$, so $S^a$ stabilises  $Ea$. Therefore, 
\begin{equation}\label{sl22}
((e_s)a)\varphi=
\begin{cases}
\phantom{(} \alpha_s^{\dagger} f_s + \theta^{p^j} \alpha_{s+1}^{\dagger} f_{s+1} + \sum_{i \notin \{s,s+1\}}^m \alpha_i^{\dagger} f_i + \chi^{p^j} \alpha_s e_s\\
\eta_1 (e_1)a + \ldots + \eta_m (e_m)a.
\end{cases}
\end{equation}
Since $((e_s)a)g$ does not have terms with $e_i$ for $i \ne s$ in the first line of \eqref{sl22}, 
$((e_s)a)g=\eta_s(e_s)a,$ so 
\begin{equation}\label{sl22qe}
\eta_s=\chi^{p^j-1} \alpha_s=\alpha_s^{\dagger}=\theta^{p^j-1}\alpha_{s+1}^{\dagger} =\alpha_1^{\dagger} = \ldots= \alpha_{s-1}^{\dagger}  = \alpha_{s+1}^{\dagger} = \ldots = \alpha_m^{\dagger}
\end{equation}
and $\theta^{p^j-1}=1.$ Hence $j=0$ and $\alpha_s=\alpha_s^{\dagger}$ by \eqref{sl22qe}, so $g$ is scalar and $\tilde{S} \cap S^a\le Z(GU_n(q,{\bf f}_{\beta_1})).$


\medskip

{\bf Case (1.3).} Assume $l >0$ and there exists $i \in \{1, \ldots, k\}$ such that $\dim W_i \ge 4$. So 
 $$W_i = \langle f_s, \ldots, f_{s+n_i}, e_{s}, \ldots, e_{s+n_i}\rangle,  \text{ for some } s \text{ and } n_i \ge 2.$$
In particular, $\alpha_s=\alpha_{s+1}.$ Let $r= n_1 + \ldots + n_k.$ 
We claim that there exists  $a \in SU_n(q, {\bf f}_{\beta_1})$ such that 
\begin{align*}
(e_s)a= & (\sum_{i=r+1}^m e_i)+ e_s + \theta f_{s+1} + \sum_{i \notin \{s,s+1\}}^r f_i + \underline{(x_1-\alpha f_s)}; & (f_s)a= & f_s; & \\
\end{align*}
\begin{align*}
(e_{s+1})a=  & \phantom{(}  e_{s+1} -\theta^q f_s; & (f_{s+1})a= & f_{s+1}; & \\
(e_i)a= & \phantom{(} e_i-f_s; & (f_i)a= & \phantom{(} f_i; & & \text{ for } i \in \{1, \ldots, r\} \backslash \{s, s+1\}; \\
(e_i)a= & \phantom{(} e_i; & (f_i)a= & \phantom{(} f_i -f_s; & & \text{ for } i \in \{r+1, \ldots, m\};\\
\underline{(x_1)a=} & \underline{\phantom{(} x_1-f_s}, & & & 
\end{align*}
and $a$ stabilises all other vectors from $\beta_1.$ Here the underlined part is in the formula only if ${\zeta}=1$ and $x_1$ is as in Remark \ref{x1}. 
  It is routine to check that $\det(a)=1$ and $a$ is an isometry of $(V, {\bf f}),$ so $a \in SU_n(q,{\bf f}_{\beta_1}).$

We claim that $\tilde{S} \cap S^a \le Z(GU_n(q)).$ Let $\varphi= \phi^jg \in \tilde{S} \cap S^a$, where $g$ is as in \eqref{gdiag}. Observe that $S$ stabilises  $E=\langle e_1, \ldots, e_r \rangle$, so $S^a$ stabilises $Ea.$ Therefore, $((e_s)a)\varphi$ is
\begin{equation}\label{2pres}
(\sum_{i=r+1}^m \alpha_i e_i)+ \alpha_s e_s + \theta^{p^j} \alpha_{s+1}^{\dagger} f_{s+1} + \sum_{i \notin \{s,s+1\}}^r \alpha_i^{\dagger} f_i + \underline{(\delta_1 x_1-\alpha^{p^j} \alpha_s^{\dagger} f_s)}
\end{equation}
and
$$((e_s)a)\varphi= \eta_1 (e_1)a+ \ldots +\eta_{r}(e_{r})a.$$
Since $\eta_i(e_i)a$ for $i \in \{1, \ldots, r\} \backslash \{s\}$ has $\eta_i$ as a coefficient for $e_i$ with respect to $\beta_1$ and $((e_s)a)\phi$ has $0$ as these coefficients (see  \eqref{2pres}), $\eta_i=0$ for all $i\ne s.$  Thus, $$\eta_s= \alpha_s=\theta^{p^j-1} \alpha_{s+1}^{\dagger} =\alpha_{r+1}=\ldots=\alpha_m=\underline{\delta_1}.$$
Observe $l>0,$ so $m>r$ or ${\zeta}=1,$ so $\alpha_s^{\dagger}=\alpha_s.$ Therefore,  $\theta^{p^j-1}=1$ since $\alpha_s=\alpha_{s+1},$ so $j=0.$  Hence $\alpha_i=\alpha_s$ for $i \le r$ by \eqref{2pres} and  $g$ is scalar, so $\tilde{S} \cap S^a\le Z(GU_n(q,{\bf f}_{\beta_1})).$

\subsection*{Case 2.} Let $\mu=0$ and $k=0.$ So $S$   stabilises no non-zero singular subspace. Choose $U$
 to be one of the $U_i$ such that $\dim U = \max_{i \in \{1, \ldots, l\}} \{\dim U_i\}.$ Therefore, $V=U\bot U^{\bot}$ and $(U^{\bot})S=U^{\bot}.$ Without loss of generality, we can assume  that 
\begin{equation}
\label{U1dx}
U=U_1=\langle f_1, \ldots, f_{d}, \underline{x}, e_d, \ldots, e_1 \rangle,
\end{equation}
where $d=[\dim U /2]$ and $\{\underline{x}\}=\{x_1, \ldots, x_t\} \cap U$. If $\dim U$ is even, then  $\{\underline{x}\}$ is empty   and we read \eqref{U1dx} without $\underline{x}.$ If $\dim U$ is odd, then we assume that $\underline{x}=x_t.$ So 
$$U^{\bot}=\langle f_{d+1}, \ldots, f_{m}, {x_1, \ldots,x_s }, e_{m}, \ldots, e_{d+1} \rangle,$$
where $s=t-1$ if $\dim U$ is odd and $s=t$ otherwise. Define $x_1$ as in Remark \ref{x1}. Notice that if $\phi^j g \in \tilde{S},$ so $g$ has shape \eqref{gdiag}, 
 then $\alpha_i^{\dagger} = \alpha_i$ for $i \in \{1, \ldots, m\}$ since $g$ acts on $U_r$ containing $e_i$ and $f_i$ as a scalar.

If  $\dim U_i=1$ for all $i=1, \ldots, l$, then $M$ is abelian and, by Theorem \ref{zenab}, there exists $y \in SU_n(q)$ such that $M \cap M^y \le Z(GU_n(q)).$ Thus, $(S \cap S^y)/Z(GU_n(q))$ is an abelian subgroup of $(S \cdot SU_n(q))/Z(GU_n(q))$ and, by Theorem \ref{zenab}, there is $z \in SU_n(q)$ such that $(S \cap S^y) \cap (S \cap S^y)^z =Z(GL_n(q))$. So we can assume $\dim U \ge 2.$


\medskip

The two subcases we consider correspond to the following situations:
when $d=m$ and  $d<m$ respectively.

\medskip

{\bf Case (2.1).} Let $d=m,$ so $V=U \bot \langle x_1 \rangle.$

Assume $d\ge 2,$ so $\alpha_1=\alpha_2.$ 
We claim that there exists  $a \in SU_n(q, {\bf f}_{\beta_1})$ such that 
\begin{align*}
(e_1)a= &\phantom{(}  e_1; & (f_1)a= & f_1 +(x_1 - \alpha e_1); & \\
(e_{2})a=  & \phantom{(}  e_{2}; & (f_{2})a= & f_{2}+\theta^q(x_1 - \alpha e_2); & \\
(x_1)a= & \phantom{(} x_1 -e_1 -\theta e_2, &  & & & 
\end{align*}
and $a$ stabilises all other vectors from $\beta_1.$   It is routine to check that $\det(a)=1$ and $a$ is an isometry of $(V, {\bf f}),$ so $a \in SU_n(q,{\bf f}_{\beta_1}).$

We claim that $\tilde{S} \cap S^a \le Z(GU_n(q)).$ Let $\varphi= \phi^jg \in \tilde{S} \cap S^a$, where $g$ is as in \eqref{gdiag}. Observe that $S$ stabilises  $\langle x_1\rangle$, so $S^a$ stabilises $\langle x_1\rangle a.$ Therefore, 
\begin{equation*}
((x_1)a)\varphi= \delta_1 x_1 + \alpha_1 e_1 + \theta^{p^j} \alpha_2 e_2= \lambda ((x_1)a) 
\end{equation*}
 for some $\lambda \in \mathbb{F}_{q^2}^*.$ Hence $\lambda= \delta_1 =\alpha_1 = \theta^{p^j-1} \alpha_2,$
$g$ is scalar and $j=0$ since $\alpha_1=\alpha_2.$ So $\tilde{S} \cap S^a\le Z(GU_n(q,{\bf f}_{\beta_1})).$

Assume that $d=1,$ so $\dim U \le 3$. If $\dim U=2$, then $n=3$ and this case is considered in Lemma \ref{lemn4uni}, so we may assume $\dim U=3.$ We claim that there exists  $a \in SU_n(q, {\bf f}_{\beta_1})$ such that 
\begin{align*}
(e_1)a= &\phantom{(}  e_1+x_1 - \alpha f_1; & (f_1)a= & \alpha f_1  - e_1 - \theta \underline{x}; & \\
(\underline{x})a=  & \phantom{(} \theta^q e_{1} + \theta^q \alpha^q f_1; & (x_1)a= & e_1- \alpha f_{1}+ x_1 +\theta \underline{x}. & 
\end{align*}
   It is routine to check that $\det(a)=1$ and $a$ is an isometry of $(V, {\bf f}),$ so $a \in SU_n(q,{\bf f}_{\beta_1}).$ We claim that $\tilde{S} \cap S^a \le Z(GU_n(q)).$ Let $\varphi= \phi^jg \in \tilde{S} \cap S^a$, where $g$ is as in \eqref{gdiag}. Observe that $S$ stabilises  $\langle x_1\rangle$, so $S^a$ stabilises $\langle x_1\rangle a.$ Therefore, 
\begin{equation*}
((x_1)a)\varphi= \alpha_1 e_1 - \alpha^{p^j} \alpha_1 f_1+ \delta_1 x_1 + \alpha_1  + \theta^{p^j} \alpha_1 \underline{x} = \lambda ((x_1)a) 
\end{equation*}
 for some $\lambda \in \mathbb{F}_{q^2}^*.$ Hence $\lambda= \delta_1 =\alpha_1 = \theta^{p^j-1} \alpha_1,$
$g$ is scalar and $j=0$. So $\tilde{S} \cap S^a\le Z(GU_n(q,{\bf f}_{\beta_1})).$

\medskip

{\bf Case (2.2).} Let $d<m.$

Assume $d\ge 2,$ so $\alpha_1=\alpha_2.$ We claim that there exists  $a \in SU_n(q, {\bf f}_{\beta_1})$ such that 
\begingroup
\allowdisplaybreaks
\begin{align*}
(e_1)a & =  (\sum_{i=d+1}^m e_i)+ e_1  + \underline{(x_1-\alpha f_1)}; \\ (f_1)a & =  f_1; & \\
(e_{2})a & =    (\sum_{i=d+1}^m \theta e_i) + e_{2}; \\ (f_{2})a & =  f_{2}; & \\
(e_i)a & =  \phantom{(} e_i; & & \text{ for } i \in \{3, \ldots, d\}; \\  (f_i)a & =  f_i; & & \text{ for } i \in \{3, \ldots, d\}; \\
(e_i)a & =  \phantom{(} e_i; & & \text{ for }  i \in \{d+1, \ldots, m\};\\ (f_i)a & =  \phantom{(} f_i -f_1-\theta^q f_2; & & \text{ for }  i \in \{d+1, \ldots, m\};\\
\underline{(x_1)a} & \underline{\; =\phantom{(} x_1-f_1}, & & & 
\end{align*} 
\endgroup
and $a$ stabilises all other vectors from $\beta_1.$ It is routine to check that $\det(a)=1$ and $a$ is an isometry of $(V, {\bf f}),$ so $a \in SU_n(q,{\bf f}_{\beta_1}).$ We claim that $\tilde{S} \cap S^a \le Z(GU_n(q)).$ Let $\varphi= \phi^jg \in \tilde{S} \cap S^a$, where $g$ is as in \eqref{gdiag}. Observe that $S$ stabilises  $U$, so $S^a$ stabilises $Ua.$ Therefore, 
\begin{equation*}
((e_1)a)\varphi=
\begin{cases}
(\sum_{i=d+1}^m \alpha_i e_i)+ \alpha_1 e_1 + \underline{(\delta_1 x_1-\alpha^{p^j} \alpha_1 f_1)}\\
\eta_1 (e_1)a+ \ldots +\eta_{d}(e_{d})a + \lambda (\underline{x})a+ \mu_1(f_1)a+ \ldots \mu_d(f_d)a.
\end{cases}
\end{equation*}
 for some $\lambda, \eta_i, \mu_i \in \mathbb{F}_{q^2}^*.$ 
Since there is no $\underline{x},$ or $e_i$ and $f_i$ for $1<i\le d$ in  the first line of the formula, $$\lambda=\eta_2 =\ldots =\eta_r=\mu_2= \ldots=\mu_d=0,$$
so $\alpha_1=\alpha_{r+1}= \ldots= \alpha_{m}=\underline{\delta_1}.$

The same arguments show that $((e_2)a)\varphi=\lambda((e_2)a)$ for some $\lambda \in \mathbb{F}_{q^2}^*.$ So $$((e_2)a)\varphi=(\sum_{i=d+1}^m \theta^{p^j} \alpha_i e_i)+ \alpha_2 e_2=\lambda((e_2)a)$$
and $\lambda=\alpha_2= \theta^{p^j-1}\alpha_{d+1}.$ 
Hence $\theta^{p^j-1}\alpha_{d+1}=\alpha_{d+1}$ since $\alpha_1=\alpha_2,$ so $j=0$ and $g$ is scalar.
 Hence $\tilde{S} \cap S^a\le Z(GU_n(q,{\bf f}_{\beta_1})).$

Assume $d=1.$ Let $a \in SU_n(q, {\bf f}_{\beta_1}) $ be defined by \eqref{465}. Notice that $\langle e_2, f_2 \rangle \subseteq U_2.$ Observe that $S$ stabilises  $U$, so $S^a$ stabilises $Ua.$
Therefore,  $((e_1)a)\varphi$ is
\begin{equation*}
(\sum_{i=3}^m \alpha_i e_i)+ \theta^{p^j} \alpha_2 e_2+ \alpha_1 e_1 +\alpha_2 f_2 -\theta^{p^j} \alpha_1 f_1 + \underline{(\delta_1 x_1-\alpha^{p^j} \alpha_1 f_1)}
\end{equation*}
and
$$((e_1)a)\varphi= \eta_1 (e_1)a+  \lambda (\underline{x})a+ \mu_1(f_1)a$$
 for some $\lambda \in \mathbb{F}_{q^2}^*.$ Thus,
$$\eta_1 = \alpha_1= \alpha_2 = \theta^{p^j-1}\alpha_2= \alpha_3 = \ldots =\alpha_m = \underline{\delta_1},$$ 
 so $j=0$ and $g$ is scalar. Hence $\tilde{S} \cap S^a\le Z(GU_n(q,{\bf f}_{\beta_1})).$

\subsection*{Case 3.} Let $\mu =1,$ so $q \in \{2,3\}$. Without loss of generality, we can assume that $\gamma_{k+l}(M) \in \{GU_2(q), GU_3(2), MU_4(2)\}.$  Let $\{v_1, \ldots, v_{\nu}\}$ be an orthonormal basis of $U_l$, so $2 \le \nu \le 4$. For the remaining $W_i$ and $U_i$ we change basis as in \eqref{basisU}, so  
 \begin{equation} \label{betamu}\beta_1=\{f_1, \ldots, f_m, x_1, \ldots, x_t,v_1, \ldots, v_{\nu}, e_m, \ldots, e_1\},
\end{equation}
and $n=2m+t+ \nu.$
 Denote $n_1 + \ldots + n_k$ by $r,$ so $r\le m.$ Notice that the subspace 
$$E= \langle e_1, \ldots, e_r \rangle$$
is $S$-invariant. Let $\varphi \in \tilde{S}$, so, by {\bf Step 1}, $\varphi=\phi^j g$ with $j \in \{0,1\}$ and
\begin{equation*}
\begin{aligned}
&(e_i)g =  \alpha_i e_i & &\text{ for } i \in \{1, \ldots, m\};\\
&(f_i)g =  \alpha_i^{\dagger} f_i & &\text{ for } i \in \{1, \ldots, m\};\\
&(x_i)g =  \delta_i x_i & &\text{ for } i \in \{1, \ldots, t\}; \\
&(v_i)g =  \lambda_{i1} v_1 +\ldots + \lambda_{i \nu} v_{\nu} & & \text{ for } i \in \{1, \ldots, \nu \},
\end{aligned}
\end{equation*}
for $\alpha_i$, $\delta_i,$ $\lambda_{ji} \in \mathbb{F}_{q^2}.$ Let $\alpha \in \mathbb{F}_{q^2}^*$ be such that $\alpha + \alpha^q=1$ and $\alpha \notin \mathbb{F}_q.$ It is easy to verify existence of such $\alpha$ for $q \in \{2,3\}$ by computation.

\medskip

The three subcases we consider correspond to the following situations:
when  $m=r>0$,
  $m>r>0$,
and  $m \ge r=0$ respectively.

\medskip

{\bf Case (3.1).} Let $m=r>0$, so $l \le 2$ and $l=2$ if and only if $\dim U_1=1,$ so $\zeta=1$ and $U_1= \langle x_1 \rangle.$ If $n \ge 3 \nu+1$, then $r \ge \nu$ (recall that $2 \le \nu \le 4$). For smaller $n$, Theorem \ref{theoremGU} is verified by computation, so  we assume $r \ge \nu.$

We claim that there exists  $a \in SU_n(q, {\bf f}_{\beta_1})$ such that 
\begin{small}
\begin{equation*}
\begin{aligned}
&(e_1)a =\sum_{s=2}^r f_s +e_1 +(v_1 -\alpha f_1)+ \underline{x_1 -\alpha f_1};& & (f_1)a=f_1 \\
&(e_{i})a = e_{i} - f_1 +(v_i - \alpha f_i); & &(f_{i})a=f_{i}; & &  i\in \{2, \ldots, \nu \}  \\
&(e_{i})a =e_{i} -f_1; & &(f_{i})a=f_{i}; && i\in \{\nu+1, \ldots, r \}   \\
&\underline{(x_1)a=x_1-f_1};&& (v_i)a= v_i-f_i;&& i\in \{1, \ldots, \nu \} 
\end{aligned}
\end{equation*}
\end{small}\\
and $a$ stabilises all other vectors in $\beta_1.$ Here the underlined part is in the formula only if ${\zeta}=1$ and $x_1$ is as in Remark \ref{x1}.  It is routine to check that $\det(a)=1$ and $a$ is an isometry of $(V, {\bf f}),$ so $a \in SU_n(q,{\bf f}_{\beta_1}).$ Let $\varphi \in \tilde{S} \cap S^a$.
 Notice that $S$ stabilises $E= \langle e_1, \ldots, e_r \rangle$, so $S^a$ stabilises  $Ea$. Therefore, $((e_1)a)\varphi$ is
\begin{equation}\label{sl31}
\sum_{i=2}^r\alpha_i^{\dagger} f_i  +\alpha_1 e_1 + (\lambda_{11} v_1 +\ldots + \lambda_{1 \nu} v_{\nu} -\alpha_1^{\dagger}\alpha^{p^j} f_1)+ \underline{\delta_1 x_1 -\alpha_1^{\dagger} \alpha f_1}
\end{equation}
and
$$((e_1)a)\varphi =\eta_1 (e_1)a + \ldots +\eta_r (e_r)a.$$
Since $((e_1)a)\varphi$ does not have $e_i$ for $i \ne 1$ in  \eqref{sl31}, 
$((e_1)a)\varphi=\eta_1(e_1)a,$ so 
$$
\begin{cases}
\eta_1=\alpha_1= \alpha_2^{\dagger}= \ldots = \alpha_r^{\dagger}=\lambda_{11}=\alpha_1^{\dagger}\alpha^{p^j-1}=\underline{\delta_1};\\
\lambda_{12}= \ldots = \lambda_{1 \nu}=0.
\end{cases}
$$ Therefore, in particular, $g$ stabilises $\langle v_1 \rangle,$ a non-degenerate subspace, so $\lambda_{11}^{\dagger}=\lambda_{11}$ and $$\alpha_1^{\dagger}=\alpha_1 = \ldots = \alpha_r=\alpha_1^{\dagger}\alpha^{p^j-1}.$$ Hence $\alpha^{p^j-1}=1$ and $j=0$ since $\alpha \notin \mathbb{F}_q.$ The same arguments for $((e_i)a)\varphi$ with $i=2, \ldots , \nu$ show that $\lambda_{ii}=\alpha_1$ and $\lambda_{ij}=0$ for $j \ne i.$ Therefore, $g$ is scalar and $\tilde{S} \cap S^a \le Z(GU_n(q,{\bf f}_{\beta_1})).$

\medskip

{\bf Case (3.2).} Assume that $m>r>0$. Recall $\mu \le 1$ and Remark \ref{remni1}; thus, if $\nu=2$, then $m \ge \nu$ for $n \ge 6$; if $\nu \in \{3,4\}$, then $m \ge \nu$ for $n \ge 9$.       For smaller $n$, Theorem \ref{theoremGU} is verified by computation, so we assume $m \ge \nu$.
    
We claim that there exists  $a \in SU_n(q, {\bf f}_{\beta_1})$ such that 

\begin{equation*}
\begin{aligned}
&(e_1)a =\sum_{s=n_1+1}^m e_s +e_1 + (v_1 - \alpha f_{1}) + \underline{(x_1 -\alpha f_1)} ; \\ & (f_{1})a=f_{1}; & &  \\
&(e_i)a = e_i + (v_i - \alpha f_i); & &  i\in \{2, \ldots,  \nu \}  \\  & (f_{i})a=f_{i}; & &  i\in \{2, \ldots,  \nu \}  \\
&(e_i)a =e_i; & &  i\in \{\nu+1, \ldots, m \}  \\    & (f_{i})a=f_{i} -\delta_{(i>n_1)} f_1 ; & &  i\in \{\nu+1, \ldots, m \}  \\
&\underline{(x_1)a=x_1-f_{1}};\\ & (v_i)a= v_i-f_{i};&& i\in \{1, \ldots, \nu \}  
\end{aligned}
\end{equation*}
and $a$ stabilises all other vectors from $\beta_1.$ Here the underlined part is  in the formula only if ${\zeta}=1$ and $x_1$ is as in Remark \ref{x1}.  It is routine to check that $\det(a)=1$ and $a$ is an isometry of $(V, {\bf f}),$ so $a \in SU_n(q,{\bf f}_{\beta_1}).$
We claim that $\tilde{S} \cap S^a \le Z(GU_n(q)).$ Let $\varphi \in \tilde{S} \cap S^a$.
 Observe that $S$ stabilises  $E=\langle e_1, \ldots, e_{n_1} \rangle$, so $S^a$ stabilises $Ea.$ Therefore,  $((e_1)a)\varphi$ is
\begin{equation*}
\sum_{i=n_1+1}^m\alpha_i e_i +\alpha_1 e_1 +(\lambda_{11} v_1 +\ldots + \lambda_{1 \nu} v_{\nu} -\alpha_1^{\dagger} \alpha^{p^j} f_1)+ \underline{\delta_1 x_1 -\alpha_1^{\dagger} \alpha^{p^j-1} f_1} 
\end{equation*}
and
$$((e_1)a)\varphi=\eta_1 (e_1)a+ \ldots +\eta_r(e_r)a.$$
Since $\eta_i(e_i)a$ for $i>1$ has $\eta_i$ as a coefficient for $e_i$ with respect to $\beta_1$ and $((e_1)a)\varphi$ has $0$ as these coefficients, $\eta_i=0$ for all $i>1.$  Thus, 
\begin{equation*}
\begin{cases}
\eta_1=\alpha_1=  \alpha_{n_1+1}=\ldots=\alpha_m=\lambda_{11}=\alpha^{p^j-1}\alpha_1^{\dagger}= \underline{\delta_1}\\
\lambda_{12} = \ldots = \lambda_{1 \nu}=0.
\end{cases}
\end{equation*}
Therefore, in particular, $g$ stabilises $\langle v_1 \rangle,$ a non-degenerate subspace, so $\lambda_{11}^{\dagger}=\lambda_{11}$ and $\alpha_1^{\dagger}=\alpha_1$. Hence $\alpha^{p^j-1}=1$ and $j=0.$

The same argument for $((e_{i})a)g$ with $i=r+2, \ldots, r+ \nu$ shows that $\lambda_{ii}=\alpha_1$ for $i \in \{1, \ldots, \nu\}$ and $\lambda_{ij}=0$ for $i \ne j.$ Therefore, $g$  is scalar by Remark \ref{x1}. 

\medskip

{\bf Case (3.3).} Assume that $r=0$. Recall $\mu \le 1$ and Remark \ref{remni1}; thus, if $\nu$ is $2$, $3$ or $4$, then $m \ge \nu$ for $n \ge 6$, $10$ and $12$  respectively.     For smaller $n$, Theorem \ref{theoremGU} is verified by computation, so we assume $m \ge \nu$.    Let $U$ be one of  $\{U_1, ..., U_{l - 1}\}$ with maximum dimension.   So we can assume $$U=U_1 = \langle f_1, \ldots, f_d, \underline{x}, e_d, \ldots, e_1 \rangle$$
where $d$ and $\underline{x}$ are defined as in \eqref{U1dx}.

Assume $d=1.$
We claim that there exists  $a \in SU_n(q, {\bf f}_{\beta_1})$ such that 
\begin{align*}
&(e_1)a =\sum_{s=3}^m e_s + \alpha^{\dagger}e_2 +e_1 +f_2 -\alpha^{\dagger}f_1 + (v_1 - \alpha f_{1}) + \underline{(x_1 -\alpha f_1)} ; & & (f_{1})a=f_{1}; & &  
\end{align*}
\begin{align*}
&(e_2)a = e_2 + (v_2 - \alpha f_2); & & (f_{2})a=f_{2}- \alpha^{-1}f_1; & &    \\
&(e_i)a =e_i+ v_i - \alpha f_i + \alpha f_1;   & & (f_{i})a=f_{i} - f_1 ; & &  i\in \{3, \ldots, \nu \}  \\
&(e_i)a =e_i;   & & (f_{i})a=f_{i} - f_1 ; & &  i\in \{\nu+1, \ldots, m \}  \\
&\underline{(x_1)a=x_1-f_{1}};&& (v_i)a= v_i-f_{i};&& i\in \{1, \ldots, \nu \}  
\end{align*}
and $a$ stabilises all other vectors from $\beta_1.$ Here the underlined part is  in the formula only if ${\zeta}=1$ and $x_1$ is as in Remark \ref{x1}.  It is routine to check that $\det(a)=1$ and $a$ is an isometry of $(V, {\bf f}),$ so $a \in SU_n(q,{\bf f}_{\beta_1}).$
We claim that $\tilde{S} \cap S^a \le Z(GU_n(q)).$ Let $\varphi \in \tilde{S} \cap S^a$.


Observe that $S$ stabilises  $U$, so $S^a$ stabilises $Ua.$ Therefore, 
\begin{equation*}
((e_{1})a)\varphi=
\begin{cases}
\begin{aligned}
\sum_{s=3}^m \alpha_s e_s + (\alpha^{\dagger})^{p^j} \alpha_2 e_2 +& \alpha_1 e_1 +\alpha_2 f_2 -(\alpha^{\dagger})^{p^j} \alpha_1 f_1 + \\ +& (\sum_{i=1}^{\nu}\lambda_{1i} v_i) -\alpha_1 \alpha^{p^j} f_1)+ \underline{\delta_1 x_1 -\alpha_1 \alpha^{p^j-1} f_1}
\end{aligned} \\
\eta_1 (e_1)a+\mu_{r+1}(f_{r+1})+ \lambda \underline{x}.
\end{cases}
\end{equation*}
Thus, $$
\begin{cases}
\eta_1=\alpha_1=\alpha_2=(\alpha^{\dagger})^{p^j-1} \alpha_2 = \alpha_3= \ldots = \alpha_m=\lambda_{11}= \underline{\delta_1}\\
\lambda_{12}= \ldots = \lambda_{1 \nu}=0.
\end{cases}
$$
 Hence $\alpha^{p^j-1}=1$ and $j=0.$
The same argument for $((e_{i})a)g$ with $i=r+2, \ldots, r+ \nu$ shows that $\lambda_{ii}=\alpha_1$ for $i \in \{1, \ldots, \nu\}$ and $\lambda_{ij}=0$ for $i \ne j.$ Therefore, $g$  is scalar by Remark \ref{x1}. 

Assume $d \ge 2.$ 
We claim that there exists  $a \in SU_n(q, {\bf f}_{\beta_1})$ such that 
\begin{align*}
&(e_1)a =\sum_{s=d+1}^m e_s  +e_1  + (v_1 - \alpha f_{1}) + \underline{(x_1 -\alpha f_1)} ; & & (f_{1})a=f_{1}; \\ &  (v_1)a=v_1 -f_1 - \alpha^q f_2; \\
&(e_2)a =  e_2  +\alpha v_1 - \alpha f_1 + v_2 - 2 \cdot \alpha f_2; & & (f_{2})a=f_{2}; \\ &  (v_2)a= v_2-f_2;
\end{align*}
\begin{align*}  
&(e_i)a =e_i+ v_i - \alpha f_i + \delta_{i>d} \alpha f_1 ;   & & (f_{i})a=f_{i} - \delta_{i>d} f_1 ; & &  i\in \{3, \ldots, \nu \}  \\
&(e_i)a =e_i;   & & (f_{i})a=f_{i} - \delta_{i>d} f_1 ; & &  i\in \{\nu+1, \ldots, m \}  \\
&\underline{(x_1)a=x_1-f_{1}};&& (v_i)a= v_i-f_{i} + \delta_{i>d} f_1;&& i\in \{3, \ldots, \nu \}  
\end{align*}
and $a$ stabilises all other vectors from $\beta_1.$ Here the underlined part is in the formula only if ${\zeta}=1$ and $x_1$ is as in Remark \ref{x1}; $\delta_{i>d}$ is $1$ if $i>d$ and $0$ otherwise.  It is routine to check that $\det(a)=1$ and $a$ is an isometry of $(V, {\bf f}),$ so $a \in SU_n(q,{\bf f}_{\beta_1}).$
We claim that $\tilde{S} \cap S^a \le Z(GU_n(q)).$ Let $\varphi \in \tilde{S} \cap S^a$.

Observe that $S$ stabilises  $U$, so $S^a$ stabilises $Ua.$ Therefore, $((e_{1})a)\varphi$ is
\begin{equation*}
\sum_{s=d+1}^m \alpha_s e_s +  \alpha_1 e_1  + (\sum_{i=1}^{\nu}\lambda_{1i} v_i) -\alpha_1 \alpha^{p^j} f_1+ \underline{\delta_1 x_1 -\alpha_1 \alpha^{p^j-1} f_1} 
\end{equation*}
and
$$((e_{1})a)\varphi=\sum_i^{d} \eta_i (e_i)a+ \sum_i^{d}\mu_{i}(f_{i})+ \lambda \underline{x}.$$
Since $\eta_i(e_i)a$ for $i>1$ has $\eta_i$ as a coefficient for $e_i$ with respect to $\beta_1$ and $((e_1)a)\varphi$ has $0$ as these coefficients in the first line of the formula above, $\eta_i=0$ for all $i>1.$ The same arguments for $\mu_i$ and $\lambda$ shows that $\lambda=0$ and $\mu_i=0$ for $i>1.$ Thus, $$
\begin{cases}
\eta_1=\alpha_1 = \alpha_{d+1}= \ldots = \alpha_m=\lambda_{11}= \underline{\delta_1}\\
\lambda_{12}= \ldots = \lambda_{1 \nu}=0.
\end{cases}
$$
So $g$ stabilises $\langle v_1 \rangle$ and its orthogonal complement $\langle v_2, \ldots, v_{\nu}  \rangle$ in $\langle v_1, \ldots, v_{\nu}  \rangle$. In particular $\lambda_{i1}=0$ for $i \in \{2, \ldots, \nu\}.$    Recall that $\alpha_1=\alpha_2,$ since $d \ge 2.$ Consider 
\begin{equation*}
((e_{2})a)\varphi=
\begin{cases}
 \alpha_1 e_2  + \alpha^{p^j} \lambda_{11} - \alpha^{p^j} \alpha_1 f_1 + (\sum_{i=2}^{\nu}\lambda_{2i} v_i) - 2 \cdot \alpha_1 \alpha^{p^j} f_1 \\
\sum_i^{d} \eta_i (e_i)a+ \sum_i^{d}\mu_{i}(f_{i})+ \lambda \underline{x}.
\end{cases}
\end{equation*}
The same arguments as above show that 
$$
\begin{cases}
\alpha_1 = \alpha^{p^j-1} \lambda_{11} = \lambda_{22} \\ 
\lambda_{23}= \ldots = \lambda_{2 \nu}=0.
\end{cases}
$$
 Hence $\alpha^{p^j-1}=1$ and $j=0.$
The same argument for $((e_{i})a)g$ with $i=3, \ldots,  \nu$ shows that $\lambda_{ii}=\alpha_1$ for $i \in \{1, \ldots, \nu\}$ and $\lambda_{ij}=0$ for $i \ne j.$ Therefore, $g$  is scalar and $\tilde{S} \cap S^a \le Z(GU_n(q)).$

\medskip
Hence in all cases there exist four conjugates of $S$ in $G$ which intersect in a group of scalars, except when $$(n,q)=(5,2) \text{ with } k=l=1, n_1=1, n_2=3.$$ Here $b_S(S \cdot SU_n(q))=5$ and $\Reg_S(S \cdot GU_n(q),5)\ge 5$ are verified by computation. This already arises in {\bf Case (3.1)}. This concludes the proof of Theorem \ref{lem421}.
\end{proof}

Theorem \ref{theoremGU} now follows by  Lemma \ref{lemn4uni} and Theorems \ref{GU4sp} and  \ref{lem421}.


\section{Symplectic groups}

We prove Theorems \ref{theoremSp} and \ref{theoremSpGR} in Sections \ref{sec431} and \ref{sec432} respectively.

\subsection{Solvable subgroups contained in $\GS_n(q)$.} 
\label{sec431}
Here $S$ is a maximal solvable subgroup of $\GS_n(q)$ where $n\ge 4.$
Our goal  is to prove the following theorem.

\begin{T3}
Let $X=\GS_n(q)$ and  $n \ge 4$. If $S$ is a maximal solvable subgroup of $X$, 
 then  $b_S(S \cdot Sp_n(q)) \le 4,$ so $\Reg_S(S \cdot Sp_n(q),5)\ge 5$.
\end{T3}

 If $(n,q)=(4,2),$ then   Theorem \ref{theoremSp} is verified by computation.

\begin{Lem}
\label{3conjMSP}
Let $M= S \cap GSp_n(q).$ If $S$  stabilises no non-zero proper subspace of $V$, then there exist $y,z \in Sp_n(q)$ such that $M \cap M^y \cap M^z \le Z(GSp_n(q))$ unless $n=4$, $q \in \{2,3\}$ and $M$ is defined as $S$ in \eqref{sympirr2} of Theorem $\ref{sympirr}$. 
\end{Lem}
\begin{proof}
If $M \le Sp_n(q)$ is irreducible, then such $y,z$ exist by Theorem \ref{sympirr}. Assume that $M$ is reducible. The same arguments as in the proof of Lemma \ref{GammairGL} show that $M$ is completely reducible.  If $V$ is not $\mathbb{F}_q[M]$-homogeneous, then $S$ (and $M$) stabilises a decomposition of $V$ as in Lemma \ref{ashb}, and such $y,z$ exist by the proof of Theorem \ref{sympirr}. If $V$ is  $\mathbb{F}_q[M]$-homogeneous, then $M$ stabilises a decomposition as in    Lemma \ref{ashb} by \cite[(5.2) and (5.3)]{asch}, and such $y,z$ exist by the proof of Theorem \ref{sympirr}. 
\end{proof}

\begin{Th}
\label{SpSirr41}
Theorem {\rm \ref{theoremSp}} holds if $S$ stabilises no non-zero proper subspaces of $V$.
\end{Th}
\begin{proof}
If $f=1,$ then the theorem follows by Theorem \ref{sympirr}, so we assume $f>1.$ It follows by Theorem \ref{bernclass} unless  $S$ lies in a maximal subgroup $H$ of $S \cdot Sp_4(q)$ such that the action of $S \cdot Sp_4(q)$ on right cosets of $H$ is a standard action. Hence one of the following holds (see Definition \ref{nonstdef} and \cite[Table 1]{burness}):
\begin{itemize}
 \item[$(a)$] $q=2^f$ and $H$ is of type $O_n^{\epsilon}(q)$;
\item[$(b)$] $n=4$ and $H$ is the stabiliser of a decomposition $V= V_1 \bot V_2$ with non-degenerate $V_i$ of dimension $2$;
\item[$(c)$] $n=4$ and $H$ is the normaliser in $\GS_4(q)$ of a field extension of the field of scalar matrices.  
\end{itemize}

 First, assume that $(a)$ holds, so $q=2^f$ with $f>1$ and $H$ is a group of semisimilarities of $V$ with respect to a non-degenerate quadratic form $Q:V \to V.$ Let ${\bf f}_{Q}$ be defined by 
\begin{equation}
\label{fpolarQ}
{\bf f}_{Q}(u,v)=Q(u + v) - Q(u) - Q(v) \text{ for all } u,v \in V.
\end{equation}  By \cite[Table 4.8.A]{kleidlieb}, ${\bf f}_{Q}={\bf f}$. By \cite[Proposition 2.5.3]{kleidlieb},  there exists a basis 
$$\beta = \{f_1, \ldots, f_m, e_1, \ldots, e_m\}$$ as in Lemma \ref{sympbasisl} such that
\begin{itemize}
\item if $\epsilon=+,$ then $Q(f_i)=Q(e_i)=0$ for $i \in \{1, \ldots, m\}$;
\item if $\epsilon=-,$ then $Q(f_i)=Q(e_i)=0$ for $i \in \{1, \ldots, m-1\},$ $Q(f_m)=\mu$ and $Q(e_m)=1$.
\end{itemize}
Here $\mu \in \mathbb{F}_q^*$ is such that the polynomial $x^2+x+ \mu$ is irreducible over $\mathbb{F}_q.$ 

By Theorem \ref{sympirr}, there exist $x,y \in Sp_n(q)$ such that 
$$S \cap S^x \cap S^y \cap GSp_n(q) \le Z(GSp_n(q)).$$
Therefore, by Lemma \ref{scfield}, we may assume that if $\varphi \in S \cap S^x \cap S^y,$ then $\varphi= (\phi_{\beta})^j \cdot \lambda I_n$ for some $\lambda \in \mathbb{F}_q^*$ and $j \in \{0,1, \ldots, f-1\}.$

Let $\theta$ be a generator of $\mathbb{F}_q^*$ and let $z \in Sp_n(q)$ be defined as follows:
\begin{align*}
&(e_1)z=e_1 + \theta f_1;  && (f_1)z=f_1;\\
&(e_i)z=e_i; && (f_i)z=f_i & \text{ for } i \in \{2, \ldots, m\}.
\end{align*}

Notice that $H^z$ consists of semisimilarities of $V$ with respect to the quadratic form $Q_1$ defined by the rule $Q_1(v)=Q((v)z^{-1})$ for all $v \in V.$ Let us show that if $\varphi \in S \cap S^x \cap S^y$ is not a scalar, then it is not a semisimilarity with respect to $Q_1.$ Indeed, if $\varphi$ is a semisimilarity with respect to $Q_1,$ then 
\begin{equation}
\label{Qformsp}
Q_1((e_1 +\theta f_1)\varphi) = \delta Q_1 (e_1 + \theta f_1)^{\sigma}
\end{equation}
for some $\lambda \in \mathbb{F}_q^*$ and $\sigma \in \Aut (\mathbb{F}_q).$ 
Observe $$Q_1(e_1 +\theta f_1)=Q((e_1 +\theta f_1)z^{-1})=Q(e_1)=0.$$
On the other hand 
\begingroup
\allowdisplaybreaks
\begin{align*}
Q_1((e_1 +\theta f_1)\varphi)& =  Q_1(\lambda (e_1 +\theta^{p^j} f_1)) \\ & =  \lambda^2 Q_1(e_1 +\theta^{p^j} f_1) \\ & =   \lambda^2 Q((e_1 +\theta^{p^j} f_1)z^{-1})\\ &=  \lambda^2 Q((e_1 +\theta f_1)z^{-1} + ((\theta^{p^j}- \theta)f_1)z^{-1}) \\ & =  \lambda^2 Q(e_1  + (\theta^{p^j}- \theta)f_1)\\ & =\lambda^2 (\theta^{p^j}- \theta). 
\end{align*}
\endgroup
The last equality is obtained using \eqref{fpolarQ}. Hence \eqref{Qformsp} holds only if $j=0$ and $\varphi$ is scalar. Therefore, $S \cap S^x \cap S^y \cap S^z \le Z(GSp_n(q)).$ 

\medskip

Now assume  that $(b)$ holds, so $S$ stabilises a decomposition $V= V_1 \bot V_2$ with $V_i= \langle e_i, f_i \rangle$, where $\beta =\{e_1, f_1, e_2, f_2\}$ with $e_i$ and $f_i$ as in \eqref{sympbasis}. Let $y,z \in Sp_4(q)$ be as in {\bf Case 1b} of the proof of Theorem \ref{sympirr}. Denote $(V_i)y$ and $(V_i)z$ by $W_i$ and $U_i$ respectively for $i \in \{1,2\}.$ Let $\theta$ be a generator of $\mathbb{F}_q^*$ and let $a \in Sp_4(q, {\bf f}_{\beta})$
be $$\begin{pmatrix}
1      & \multicolumn{1}{c|}{0}& 0& 0   \\
0      & \multicolumn{1}{c|}{1}& \theta & 0 \\ \cline{1-4}
0      & \multicolumn{1}{c|}{0}& \theta & 0   \\
0      & \multicolumn{1}{c|}{1}& 0 & \theta^{-1}
\end{pmatrix}.$$
Consider $\varphi \in S \cap S^y \cap S^z \cap S^a,$ so $\varphi = \phi^j g$ with $j \in \{0,1, \ldots, f-1\}$ and $g \in GSp_4(q, {\bf f}_{\beta})$ by Lemma \ref{uniGamsdp}. By {\bf Case 1b} of the proof of Theorem \ref{sympirr}, $\varphi$ stabilises $V_i,$ $W_i$ and $U_i$ for each $i.$ Therefore, $\varphi$ stabilises 
$\langle e_1 \rangle=V_1 \cap U_1$, $\langle f_1 \rangle=V_1 \cap W_1$, $\langle e_2 \rangle=V_2 \cap W_2$.  So $\varphi$ stabilises $\langle f_1 + e_2\rangle \subseteq (V_1)z$ and $\langle f_1 + \theta e_2\rangle \subseteq (V_1)a$. Let $(e_1)g= \lambda_1 e_1,$ $(f_1)g= \lambda_2 f_1$ and $(e_2)g= \lambda_3 e_2$. Therefore,  $(f_1 + e_2)\varphi = \lambda_2 f_1+ \lambda_3 e_2$ and $\lambda_2=\lambda_3$. Also,    $(f_1 + \theta e_2)\varphi = \lambda_2 f_1+ \theta^{p^j}\lambda_3 e_2,$ so $\theta^{p^j-1}=1$ and $j=0.$ In particular, $S \cap S^y \cap S^z \cap S^a \le M \cap M^y \cap M^z$. Therefore, $\varphi$ is scalar since  $M \cap M^y \cap M^z \le Z(GSp_4(q))$ by {\bf Case 1b} of the proof of Theorem \ref{sympirr}.

\medskip

Finally, assume that $(c)$ holds, so $S$ lies in the normaliser in $\GS_4(q)$ of a field extension of the field of scalar matrices. Thus, $S \le R=GL_2(q^2)\rtimes \langle \psi \rangle$ and 
$M \le GL_2(q^2).2= GL_2(q^2)\rtimes \langle \psi^f \rangle$ where $\psi^2 = \phi$ and $\psi^f \in GSp_4(q).$ 

Assume that $M$ lies in the normaliser in  $R$ of a Singer cycle of $GL_2(q^2).$ By \eqref{2sindiagodd} and \eqref{2sindiageven}, there exists $x \in SL_2(q^2)\le Sp_4(q)$ such that $$S \cap S^x \le  GL_2(q).2 \le GSp_4(q),$$ so $S \cap S^x \le M.$ By Lemma \ref{3conjMSP}, there exist $y,z$ such that $M \cap M^y \cap M^z \le Z(GSp_4(q)),$ so 
$$(S \cap S^x) \cap S^y \cap S^z \le Z(GSp_4(q)).$$

Assume that $M$ does not lie in the normaliser in  $R$ of a Singer cycle of $GL_2(q^2)$ and let $M_1=M \cap GL_2(q).$ By Theorem \ref{irred}, there exists $x \in SL_2(q^2)$ such that $M_1 \cap M_1^x \le Z(GL_4(q))$. Hence $S \cap S^x \le Z(GL_2(q^2)) \rtimes \langle \psi \rangle.$ Let $N$ be  $S \cap S^x$.  So $|N/Z(GSp_4(q))|$ divides $(q+1)\cdot 2f$ and 
\begin{equation}
\label{Asp4}
A=|N/Z(GSp_4(q)) \cap PGSp_4(q)|
\end{equation} divides $(q+1)\cdot 2f$.

We claim that $\hat{Q}((N \cdot Sp_4(q)/Z(GSp_4(q)),2)<1$ where  $\hat{Q}(G,c)$ is as in \eqref{ver}. Denote $N/Z(GSp_4(q))$ by $H$. By Lemma \ref{fprAB}, if  $x_1,\ldots,x_k$ represent distinct $G$-classes such that $\sum_{i=1}^k |x_i^G \cap  H| \le A$ and $|x_i^G| \ge B$ for all $i \in \{1, \ldots, k\},$ then
$$\sum_{i=1}^m |x_i^G| \cdot \fpr (x_i)^c \le B \cdot (A/B)^c.$$
We take $A$ as in \eqref{Asp4} since $A \ge |H|.$  Lemma \ref{6} implies that  $\nu(g)\ge n/2 =2$ for $g \in N \cap GSp_4(q)$. For elements in $PGSp_4(q)$ of prime order with $s=\nu(x) \in \{2,3\}$  we use \eqref{5uni} as a lower bound for $|x_i^G|$. If $x \in H \backslash PGSp_4(q)$ has prime order, then we use the corresponding bound for $|x^G|$ in \cite[Corollary 3.49]{fpr2}. We take $B$ to be
the smallest of these bounds for $|x_i^G|.$ Such $A$ and $B$ are sufficient to obtain $\hat{Q}((N \cdot Sp_4(q)/Z(GSp_4(q)),2)<1$  for $q>4.$ Hence $b_S(S \cdot SU_4(q)) \le 4$. For $q=4$ the statement is verified by computation. 
\end{proof}

\begin{Th}
\label{Spqgt3}
Theorem {\rm \ref{theoremSp}} holds for $q > 3$ if $S$ stabilises a non-zero proper subspace of $V$. 
\end{Th}
\begin{proof} The proof proceeds in two steps. In {\bf Step 1} we obtain three conjugates of $S$ such that elements of their intersection have special shape. In {\bf Step 2} we find a fourth conjugate of $S$ such that the intersection of the four is a group of scalars.

\subsection*{Step 1} This is similar to the first step of the proof of Theorem \ref{theoremGU}. Fix a basis $\beta$ of  $(V, {\bf f})$ as in Lemma \ref{unist}, so ${\bf f}_{\beta}$ is as in \eqref{fst} and elements of $S$ take shape $\phi^j g$ with $g$ as in \eqref{gst} and $j \in \{0,1, \ldots, f-1\}$. We consider $S$ as a subgroup of $\GS_n(q,{\bf f}_{\beta})$ and  let $M=S \cap GSp_n(q,{\bf f}_{\beta}).$ We obtain three conjugates of $S$ such that their intersection consists of elements $\phi^j g$ where $g$ is diagonal with respect to $\beta.$

Let $\gamma_i$ be as in Lemma \ref{unist}.
Let $x$ be the matrix 
\begin{equation}\label{exSP}
 \left(
\begin{smallmatrix}
        & & & & & & &    & I_{n_1} \\
        & & & & & & &  \reflectbox{$\ddots$}  &\\
        & & & & & &I_{n_k} &    & \\
        & & &I_{n_{k+1}} & & & &    &\\ 
        & & & &\ddots & & &    &\\
        & & & & &I_{n_k+l} & &    &\\ 
        & &-I_{n_k} & & & & &    & \\
        &\reflectbox{$\ddots$} & & & & & &    &\\
-I_{n_1} & & & & & & &    & 
\end{smallmatrix} \right).
\end{equation}
Observe that $x{\bf f}_{\beta}{x}^{\top}={\bf f}_{\beta},$ so $x \in Sp_n(q,{\bf f}_{\beta}).$ It is easy to see that if $g \in M,$ so it has shape \eqref{gst}, then $g^x$ has shape \eqref{antigstSp}.

 Notice that by the proof of Theorem \ref{theorem}, if $N$ is a solvable subgroup of $\GL_n(q)$  stabilising no  non-zero proper subspace, then $N \cap GL_n(q)$ lies in an irreducible maximal  solvable subgroup of $GL_n(q).$ Therefore, by  Theorem  \ref{irred} and Lemmas \ref{starc} and \ref{3conjMSP},  there exist $y_i, z_i \in GL_{n_i}(q)$ for $i=1, \ldots, k$ and $y_i, z_i \in Sp_{n_i}(q)$ for $i \in \{k+1, \ldots, k+l\}$ such that 
\begin{equation}\label{smintSp}
\begin{split}
\gamma_i(M) \cap \gamma_i(M)^{y_i} \cap (\gamma_i(M)^{\dagger})^{z_i} & \le Z(GL_{n_i}(q)) \text{ for } i \in\{1, \ldots, k\};\\
\gamma_i(M) \cap \gamma_i(M)^{y_i} \cap \gamma_i(M)^{z_i} & \le Z(GL_{n_i}(q)) \text{ for } i \in \{k+1, \ldots, k+l \} .
\end{split}
\end{equation}
\begin{figure}[t]
\begin{equation}\label{antigstSp}
\Scale[0.95]{\begin{pmatrix}
\gamma_{1}(g)& &\multicolumn{1}{l|}{0} & &  & & &    &0  \\
    *    & \ddots &\multicolumn{1}{l|}{} & & & & &    &\\
*        &* & \multicolumn{1}{l|}{\gamma_{k}(g)}& & & & &    & \\  \cline{1-6}
 *       &\ldots & \multicolumn{1}{l|}{*} &\gamma_{k+1}(g) & & \multicolumn{1}{l|}{0}  & &    &\\ 
  *      &\ldots &  \multicolumn{1}{l|}{*}& &\ddots &\multicolumn{1}{l|}{}       & &   &\\
   *     &\ldots &  \multicolumn{1}{l|}{*}&0 & &\multicolumn{1}{l|}{\gamma_{k+l}(g)}    &  &    & \\ \cline{4-9} 
    *    &\ldots & & & &\multicolumn{1}{l|}{*} & {{\tau(g) \gamma_{k}(g)}^{\dagger}}  &    &0 \\
    *    &\ldots & & & &\multicolumn{1}{l|}{*} &* & \ddots   &\\
*        &\ldots & & & &\multicolumn{1}{l|}{*} &* & *   & {{\tau(g) \gamma_{1}(g)}^{\dagger}} \\  
\end{pmatrix}}
\end{equation}
\end{figure}
 Denote by $y$ and $z$ the block-diagonal matrices 
\begin{equation}\label{yizi}
\begin{split}
&\diag[y_1^{\dagger}, \ldots, y_k^{\dagger}, y_{k+1}, \ldots, y_{k+l}, y_k, \ldots, y_1]  \text{ and }\\
&\diag[z_1^{\dagger}, \ldots, z_k^{\dagger}, z_{k+1}, \ldots, z_{k+l}, z_k, \ldots, z_1]
\end{split}
\end{equation}
 respectively. It is routine to check that $y,z \in Sp_n(q,{\bf f}_{\beta}).$


 Therefore, if $g \in M \cap M^{xz},$ then $g$ is the block-diagonal matrix
\begin{equation}\label{gd1Sp}
\diag [\tau(g)g_1^{\dagger}, \ldots, \tau(g)g_k^{\dagger}, g_{k+1}, \ldots, g_{k+l}, g_k, \ldots, g_1],
\end{equation}
where $g_i \in \gamma_i(M) \cap (\gamma_i(M)^{\dagger})^{z_i}$ for $i\in \{1, \ldots, k+l\}.$

Thus, if $g \in M \cap M^{y} \cap M^{xz}$, then $g$ has shape \eqref{gd1Sp} where
\begin{equation*}
\begin{split}
g_i \in \gamma_i(M) \cap \gamma_i(M)^{y_i} \cap (\gamma_i(M)^{\dagger})^{z_i} & \le Z(GL_{n_i}(q)) \text{ for } i \in \{1, \ldots, k\}; \\
g_i \in \gamma_i(M) \cap \gamma_i(M)^{y_i} \cap \gamma_i(M)^{z_i} & \le Z(Sp_{n_i}(q)) \text{ for } i \in \{k+1, \ldots, k+l\}. 
\end{split}
\end{equation*} In particular, $g$ is 
\begin{equation}\label{gd1Spscal}
\Scale[0.97]{
\diag [\tau(g)\alpha_1^{\dagger} I_{n_1}, \ldots, \tau(g)\alpha_{k}^{\dagger} I_{n_k}, \alpha_{k+1} I_{n_{k+1}}, \ldots, \alpha_{k+l} I_{n_{k+l}}, \alpha_{k} I_{n_k}, \ldots, \alpha_{1} I_{n_1}]},
\end{equation}
where $\alpha_i \in \mathbb{F}_{q}$ for $i \in \{1, \ldots, k+l\}$  and $\alpha_i^{\dagger}=\alpha_i^{-1}$ for $i \in \{1, \ldots, k\}$.
By Lemma \ref{scfield}, we can assume that elements in $\gamma_i(S) \cap \gamma_i(S)^{y_i} \cap (\gamma_i(S)^{\dagger})^{z_i}$ for $i \le k$ and in $\gamma_i(S) \cap \gamma_i(S)^{y_i} \cap (\gamma_i(S))^{z_i}$ for $i>k$ have shape $\phi^j g_i$ with $g_i \in Z(GL_{n_i}(q)).$ Thus, if $\varphi \in S \cap S^y \cap S^{xz},$ then $\varphi = \phi^j_{\beta} g$ with $g$ as in \eqref{gd1Spscal}.
  Denote $S \cap S^{y} \cap S^{xz}$ by $\tilde{S}.$

\subsection*{Step 2} We now find a fourth conjugate of $S$ such that its intersection with $\tilde{S}$ lies in  $Z(GSp_n(q)).$

Recall that $\beta$ is such that 
$\bf f_{\beta}$ is as in \eqref{fst}. Therefore,
 $$\beta= \beta_{(1,1)} \cup \ldots \cup \beta_{(1,k)} \cup \beta_{k+1} \cup \ldots \cup \beta_{k+l} \cup \beta_{(2,1)} \cup \ldots \cup \beta_{(2,k)},$$
where 
\begin{equation}\label{betaij}
\begin{split}
\beta_{(1,i)} & =\{f_1^{i}, \ldots, f_{n_i}^i\} \text{ for } i \in \{1, \ldots, k\}; \\
\beta_{(2,i)} & =\{e_1^{i}, \ldots, e_{n_i}^i\} \text{ for } i \in \{1, \ldots, k\}; \\
\beta_{i} & =\{f_1^{i}, \ldots, f_{n_i/2}^i,e_1^{i}, \ldots, e_{n_i/2}^{i}\} \text{ for } i \in \{k+1, \ldots, k+l\},
\end{split}
\end{equation}
and $(f_i^j,e_i^j)=1$ for all $i,j.$ All other pairs of vectors from $\beta$ are orthogonal. For simplicity we relabel vectors $f_i^j$ in $\beta$ in the order they appear in $\beta$ using just one index, so $f_i^j$ becomes 
\begin{equation*}
\begin{aligned}
& f_{(\sum_{t=0}^{j-1} n_t +i)} & \text{ if } j \le k+1; \\
& f_{(\sum_{t=0}^{k} n_t +\sum_{t=k+1}^{j-1} (n_t/2) +i)} & \text{ if } j > k+1. 
\end{aligned}
\end{equation*}
We relabel the $e_i^j$ such that $(f_i,e_i)=1.$

If $\varphi \in \tilde{S}$, so $\varphi = \phi^j g$ with $g$ as in \eqref{gd1Spscal}, then let $\delta_i \in \mathbb{F}_q$ be such that $(e_i)g=\delta_i e_i$ for $i \in\{1, \ldots, n/2\}$ (so $\delta_i$ is some $\alpha_j$ from \eqref{gd1Spscal}).
 Let $\theta$ be a generator of $\mathbb{F}_q^*.$

The remainder of the proof splits into two cases: 
when $k \ge 1$ and
 $k=0.$
  In each   we show that $b_S(S \cdot SU_n(q))\le 4.$

\subsection*{Case 1} Let $k\ge 1$. This step splits into two  subcases. In the first  $n_i=1$ for all $i \in \{1, \ldots, k\}$; in the second  there exists $i \in \{1, \ldots, k\}$ such that $n_i \ge 2.$

\medskip

{\bf Case (1.1).} Let $n_i=1$ for all $i \in \{1, \ldots, k\}$.
We redefine $y$ in \eqref{yizi} to be
$$\diag [A^{\dagger}, y_{k+1}, \ldots, y_{k+l}, A]$$
where $$
A=
\begin{pmatrix}
1 & 0   & 0   & \ldots  & 0 \\
0 & 1      & 0   & \ldots  & 0 \\
  &        & \ddots &   &  \\
0 & \ldots & 0      & 1 & 0\\
1 & \ldots & \ldots & 1 & 1\\
\end{pmatrix}.
$$ It is easy to see that $y \in Sp_n{(q, {\bf f}_{\beta})}.$ Let $x,z \in Sp_n{(q, {\bf f}_{\beta})}$
be as in {\bf Step 1}, so $\varphi \in \tilde{S}$ has shape $\phi^j g$ with $g$ as in \eqref{gd1Spscal}. Since $S$ stabilises $\langle e_1 \rangle,$ $S^y$ stabilises $\langle e_1 \rangle y= \langle e_1 + \ldots + e_k \rangle.$ Therefore, 
$$((e_1)y)\varphi = (e_1 + \ldots + e_k)\varphi^j g = \alpha_1 e_1 + \ldots + \alpha_k e_k= \lambda(e_1 + \ldots + e_k)$$
for some $\lambda \in \mathbb{F}_q^*,$ so $\alpha_1 = \ldots = \alpha_k.$

Assume $k \ge 2.$ 
Consider $a \in GL_n(q)$ such that
\begin{equation*}
\begin{aligned}
&(e_1)a =\sum_{i=3}^{n/2} e_i +e_1 + \theta e_2 + {f_1};& & (f_1)a=f_1; \\
&(e_{2})a =  e_2; & &(f_{2})a=f_{2} - \theta f_1; & &   \\
&(e_{i})a =e_{i};  & &(f_{i})a=f_{i} -  f_{1}; && i\in \{3, \ldots, n/2 \}.   \\
\end{aligned}
\end{equation*}
 It is routine to check that $a$ is an isometry of $(V, {\bf f}),$ so we can consider $a$ as an element of $Sp_n(q, {\bf f}_{\beta}).$

We claim that $\tilde{S} \cap S^a \le Z(GSp_n(q,{\bf f}_{\beta})).$ Let $\varphi =\phi^j g \in \tilde{S} \cap S^a$.  Since $S$ stabilises the subspace $\langle e_1 \rangle$,  $S^a$ stabilises $\langle e_1 \rangle a.$  Therefore, 
\begin{equation*}
((e_1)a)\varphi=
\begin{cases}
\sum_{i=m+1}^{n/2}\delta_i e_i +\alpha_1e_1 + \theta^{p^j} \alpha_1 e_2+ \tau(g)\alpha_1^{\dagger}f_1\\
\lambda (e_1)a
\end{cases}
\end{equation*}
for some $\lambda \in \mathbb{F}_q.$
Hence 
$$\alpha_1=\theta^{p^j-1}\alpha_1=\tau(g)\alpha_1^{\dagger}=\delta_{m+1}= \ldots= \delta_{n/2}.$$
Therefore, $j=0$,  $\alpha_1= \ldots = \alpha_{k+l}$ and $\alpha_i=\tau(g)\alpha_i^{\dagger}$ for all $i,$ so $g$ is scalar  and $\tilde{S} \cap S^a \le Z(GSp_n(q,{\bf f}_{\beta})).$

Assume $k=1,$ so $l\ge 1$ (otherwise $n=2$) and $\{f_2,e_2\} \subseteq \beta_{k+1}.$ In particular, if $\varphi = \phi^j g \in \tilde{S}$, then $(f_2)g=(e_2)g=\alpha_2.$ Consider $a \in GL_n(q)$ such that
\begin{equation*}
\begin{aligned}
&(e_1)a =\sum_{i=3}^{n/2} e_i +e_1 + \theta e_2 + {f_2};& & (f_1)a=f_1; \\
&(e_{2})a =  e_2+f_1; & &(f_{2})a=f_{2} - \theta f_1; & &   \\
&(e_{i})a =e_{i};  & &(f_{i})a=f_{i} -  f_{1}; && i\in \{3, \ldots, n/2 \}.   \\
\end{aligned}
\end{equation*}
 It is routine to check that $a$ is an isometry of $(V, {\bf f}),$ so we can consider $a$ as an element of $Sp_n(q, {\bf f}_{\beta}).$

We claim that $\tilde{S} \cap S^a \le Z(GSp_n(q,{\bf f}_{\beta})).$ Let $\varphi =\phi^j g \in \tilde{S} \cap S^a$.  Since $S$ stabilises the subspace $\langle e_1 \rangle$,  $S^a$ stabilises $\langle e_1 \rangle a.$  Therefore, 
\begin{equation*}
((e_1)a)\varphi=
\begin{cases}
\sum_{i=m+1}^{n/2}\delta_i e_i +\alpha_1e_1 + \theta^{p^j} \alpha_2 e_2+ \alpha_2 f_1;\\
\lambda (e_1)a
\end{cases}
\end{equation*}
for some $\lambda \in \mathbb{F}_q.$
Hence 
$$\alpha_1=\alpha_2=\theta^{p^j-1}\alpha_2=\delta_{m+1}= \ldots= \delta_{n/2}.$$
Therefore, $j=0$,  $\alpha_1= \ldots = \alpha_{k+l}$ and $\alpha_i=\tau(g)\alpha_i^{\dagger}$ for all $i,$ so $g$ is scalar  and $\tilde{S} \cap S^a \le Z(GSp_n(q,{\bf f}_{\beta})).$

\medskip

{\bf Case (1.2).} Denote $r:= \sum_{i=1}^k n_i.$ Let $n_i \ge 2$ for some $i \le k,$ so $\delta_s= \delta_{s+1}$ for some $s<r.$ Consider $a \in GL_n(q)$ such that
\begin{equation*}
\begin{aligned}
&(e_s)a =f_s + \theta f_{s+1} + \sum_{i \notin \{s,s+1\}}^{n/2} f_i +e_s;& & (f_s)a=f_s; \\
&(e_{s+1})a =  e_{s+1}+\theta f_s; & &(f_{s+1})a=f_{s+1}; & &   \\
&(e_{i})a =e_{i} +f_s;  & &(f_{i})a=f_{i}; && \Scale[0.95]{i \in \{3, \ldots, n/2 \} \backslash \{s, s+1\}}.   \\
\end{aligned}
\end{equation*}
 It is routine to check that $a$ is an isometry of $(V, {\bf f}),$ so we can consider $a$ as an element of $Sp_n(q, {\bf f}_{\beta}).$

We claim that $\tilde{S} \cap S^a \le Z(GSp_n(q,{\bf f}_{\beta})).$ Let $\varphi =\phi^j g \in \tilde{S} \cap S^a$.  Since $S$ stabilises the subspace $W=\langle e_1 + \ldots +e_r \rangle$,  $S^a$ stabilises $W a.$  Therefore, 
\begin{equation*}
((e_s)a)\varphi=
\begin{cases}
\tau(g) \delta_s^{\dagger} f_s + \theta^{p^j}\tau(g) \delta_s^{\dagger} f_{s+1} + \sum_{i \notin \{s,s+1\}}^{n/2} \tau(g) \delta_i^{\dagger} f_i +\delta_s e_s;\\
\eta_1 (e_1)a+ \ldots +\eta_r(e_r)a
\end{cases}
\end{equation*}
for some $\eta_1, \ldots, \eta_r \in \mathbb{F}_q.$
Since $((e_1)a)\varphi$ does not have $e_i$ for $i \ne s$ in the first line of the equation above, $((e_s )a)\varphi = \eta_s (e_s )a$, so
$$\alpha_s=\tau(g)\alpha_s^{\dagger}=\theta^{p^j-1}\tau(g) \delta_s^{\dagger}=\tau(g)\delta_{i}^{\dagger}$$
for $i \in \{3, \ldots, n/2 \} \backslash \{s, s+1\}.$
Therefore, $j=0$,  $\alpha_1= \ldots = \alpha_{k+l}$ and $\alpha_i=\tau(g)\alpha_i^{\dagger}$ for all $i,$ so $g$ is scalar  and $\tilde{S} \cap S^a \le Z(GSp_n(q,{\bf f}_{\beta})).$

\subsection*{Case 2.} Let $k=0,$ so $l \ge 2.$ Denote $s:=n_1/2.$ Hence $\{f_{s+1},e_{s+1}\} \subseteq \beta_{2}.$ In particular, if $\varphi = \phi^j g \in \tilde{S}$, then $(f_{s+1})g=(e_{s+1})g=\alpha_2.$
Consider $a \in GL_n(q)$ such that
\begin{align*}
(e_1)a & =\sum_{i=s+1}^{n/2} e_i +e_1 + \theta f_{s +1} ;&  (f_1)a & =f_1; \\
(e_{s+1})a & =  e_{s+1}+\theta f_1; & (f_{s +1})a & =f_{s +1} -  f_1; 
\end{align*}
\begin{align*}
&(e_{i})a =e_{i};  & &(f_{i})a=f_{i}; && i\in \{2, \ldots, n_1/2 \}.   \\
&(e_{i})a =e_{i};  & &(f_{i})a=f_{i} -  f_{1}; && i\in \{s+2, \ldots, n/2 \}.   
\end{align*}
 It is routine to check that $a$ is an isometry of $(V, {\bf f}),$ so we can consider $a$ as an element of $Sp_n(q, {\bf f}_{\beta}).$

We claim that $\tilde{S} \cap S^a \le Z(GSp_n(q,{\bf f}_{\beta})).$ Let $\varphi =\phi^j g \in \tilde{S} \cap S^a$.  Since $S$ stabilises the subspace $W=\langle e_1 + \ldots + e_{s}, f_1 + \ldots + f_{s}  \rangle$,  $S^a$ stabilises $W a.$  Therefore, 
\begin{equation*}
((e_1)a)\varphi=
\begin{cases}
 \sum_{i=s +1}^{n/2}  \delta_i e_i + \alpha_1 e_1 + \theta^{p^j} \delta_{s+1} f_{s+1};\\
\eta_1 (e_1)a+ \ldots +\eta_{s}(e_{s})a + \mu_1 (f_1)a+ \ldots +\mu_{s}(f_{s})a
\end{cases}
\end{equation*}
for some $\eta_1, \ldots, \eta_{s}, \mu_1, \ldots, \mu_{s}  \in \mathbb{F}_q.$
Since $((e_1)a)\varphi$ does not have terms with $e_i$ for $1<i  \le {s}$ and $f_i $ for $1 \le i \le {s}$ in the first line of the equation above, $((e_1 )a)\varphi = \eta_1 (e_1 )a$, so
$$\alpha_1=\alpha_2=\theta^{p^j-1}\alpha_2=\delta_{i}$$
for $i \in \{s+1, \ldots, n/2 \}.$
Therefore, $j=0$,  $\alpha_1= \ldots = \alpha_{k+l}$ and $\alpha_i=\tau(g)\alpha_i^{\dagger}$ for all $i,$ so $g$ is scalar  and $\tilde{S} \cap S^a \le Z(GSp_n(q,{\bf f}_{\beta})).$
\end{proof}

We have now proved Theorem \ref{theoremSp} for $q > 3.$

\begin{Rem} \label{symprem23} Equation \eqref{smintSp} does not always hold for 
$q \in \{2,3\}$. In particular, it does not hold in each of the following cases: 
\begin{enumerate}[label=(\alph*)]
\item $\gamma_i(S)=GL_2(q)$ for $i \in \{1, \ldots, k\};$ \label{q23gl2}
\item $\gamma_i(S)=GSp_2(q)$ for $i \in \{k+1, \ldots, k+l\};$ \label{q23gsp2}
\item $\gamma_i(S)$ is the stabiliser in $GSp_4(q)$ of the decomposition $V=V_1 \bot V_2$ with $V_1$ and $V_2$ non-degenerate of dimension $2$  for $i \in \{k+1, \ldots, k+l\}.$ Recall that $\beta_i=\{f_1^i,f_2^i,e_1^i,e_2^i\}$ and let $V_r= \langle f_r^i, e_r^i \rangle$ for $r=1,2.$ \label{q23sp2wr2}
\end{enumerate}
\end{Rem}

The following two lemmas  are verified by computation. Let $Q$, $R$ and $T$ be $\gamma_i(S)$   from (a), (b) and (c) of Remark \ref{symprem23} respectively.

\begin{Lem}\label{smSp23}
Let $q \in \{2,3\},$ $S \le Sp_n(q)$ is a maximal solvable subgroup and $\beta$ is a basis of $V$ as in Lemma $\ref{unist}$, so matrices in $S$ have shape \eqref{gst}. Specifically, let one of the following hold:
\begin{itemize}
\item $k=0$, $l=2$, $\gamma_i(S)$ is $R$   for both $i=1,2$, so $n=4$;
\item $k=0$, $l=2$, $\gamma_i(S)$ is $T$  for both $i=1,2$, so $n=8$;  
\item $k=0$, $l=2$,  $\gamma_1(S)$ is $R$, $\gamma_2(S)$ is $T$, so $n=6$;
\item $k=1$, $l=1$,  $\gamma_1(S)$ is $Q$, $\gamma_2(S)$ is $R$, so $n=6$;
\item $k=1$, $l=1$,  $\gamma_1(S)$ is $Q$, $\gamma_2(S)$ is $T$, so $n=8$;
\item $k=2$, $l=0$,  $\gamma_i(S)$ is $Q$, for both $i=1,2$, so $n=8$.
\end{itemize} 
Then there exist $x, y \in Sp_n(q)$ such that 
$$S \cap S^x \cap S^y \le Z(Sp_n(q)).$$
\end{Lem}

\begin{Lem}\label{Spq2c23}
Let  $S \le Sp_4(q)$ be a maximal solvable subgroup and let $\beta$ be a basis of $V$ as in Lemma $\ref{unist}$, so matrices in $S$ have shape \eqref{gst}. 
\begin{enumerate}[font=\normalfont]
\item Let $k=0,$ $l=1$ and let $S=T$. If $q=3$, then there exist $x,y \in Sp_4(3)$ such that 
$$S \cap S^x \cap S^y = \left\langle 2I_4, \left( \begin{smallmatrix}
1&0&1&0\\
0&1&0&0\\
0&0&1&0\\
0&0&0&1
\end{smallmatrix} \right)
 \right\rangle.$$ For instance,  $$x=\left( \begin{smallmatrix}
0&1&2&1\\
2&2&1&0\\
0&0&2&1\\
0&0&2&0
\end{smallmatrix} \right), \;  y=\left( \begin{smallmatrix}
0&2&0&2\\
0&0&2&0\\
0&0&1&2\\
1&1&1&1
\end{smallmatrix} \right).$$
\item Let $k=1$, $l=0$ and $\gamma_1(S)=GL_2(q).$ If $q=3,$ then there exist $x,y \in Sp_4(3)$ such that 
$$S \cap S^x \cap S^y = \left\langle \left( \begin{smallmatrix}
1&0&0&0\\
1&2&0&0\\
0&0&1&1\\
0&0&0&2
\end{smallmatrix} \right), \left( \begin{smallmatrix}
2&2&0&0\\
2&1&0&0\\
0&0&2&2\\
0&0&2&1
\end{smallmatrix} \right)
 \right\rangle.$$ For instance,  $$x=\left( \begin{smallmatrix}
0&0&1&0\\
0&0&0&1\\
1&0&0&0\\
0&1&0&0
\end{smallmatrix} \right), \;  y=\left( \begin{smallmatrix}
2&1&0&0\\
1&0&2&2\\
2&0&1&2\\
1&2&1&1
\end{smallmatrix} \right).$$
\item If $q \in \{2,3\}$, then, in both  $(1)$ and $(2)$, there exist $x,y,z \in Sp_n(q)$ such that 
$$S \cap S^x \cap S^y \cap S^z =Z(Sp_n(q)).$$ 
\end{enumerate}
\end{Lem}

\begin{Th}
\label{Spq23}
Theorem {\rm\ref{theoremSp}} holds for
 $q \in \{2,3\}$ if $S$ stabilises a non-zero proper subspace of $V$.
\end{Th}
\begin{proof}
Notice that $\GS_n(q)=GSp_n(q).$ As in the case $q>3$, in {\bf Step 1}  we  obtain three conjugates of $S$ in $S \cdot Sp_n(q)$ such that their intersection consists of diagonal matrices and matrices which have few non-zero entries not on the diagonal. In {\bf Step 2} we find a fourth conjugate of $S$ such that the intersection of the four is a group of scalars.

\medskip

{\bf Step 1.} 
We commence with a technical definition.
\begin{Def}
Let $\beta=\{v_1, \ldots, v_n\}$ be a basis of a vector space $V$ over a field $\mathbb{F}$. Let $g \in GL(V)$, so $g_{\beta} \in GL_n(\mathbb{F}).$ We label the rows and columns of $g_{\beta}$ by corresponding basis vectors, so the $i$-th row (column) is labelled by $v_i.$ If $\widehat{\beta}$ is a subset of $\beta$, then the {\bf restriction} $\widehat{g_{\beta}}$  of $g_{\beta}$ to  $\widehat{\beta}$ is the matrix in $GL_{|\widehat{\beta}|}(\mathbb{F})$  obtained from $g_{\beta}$ by taking only the entries lying on the intersections of rows and columns labelled by vectors in $\widehat{\beta}.$ If $h\in GL_{|\widehat{\beta}|}(\mathbb{F}),$ then the $(h, \widehat{\beta})$-{\bf replacement} of $g_{\beta}$ is the matrix obtained from $g_{\beta}$ by replacing the entries lying on the intersections of rows and columns labelled by vectors in $\hat{\beta}$ by corresponding entries of $h.$

For example, if $n=4,$ $\beta=\{v_1,v_2,v_3,v_4\},$ $\widehat{\beta}=\{v_2,v_4\}$, 
\begin{equation*}
g=
\left( \begin{array}{c>{\columncolor{gray!20}}cc>{\columncolor{gray!20}}c}
g_{11} & g_{12} & g_{13} & g_{14}\\
\rowcolor{gray!20}
g_{21} &{\cellcolor{gray!50}} g_{22} & g_{23} &{\cellcolor{gray!50}} g_{24}\\
g_{31} & g_{32}& g_{33} & g_{34}\\
\rowcolor{gray!20}
g_{41} &{\cellcolor{gray!50}} g_{42} & g_{43} &{\cellcolor{gray!50}} g_{44}
\end{array} \right)
\text{ and } 
h=
\begin{pmatrix}
h_{11} & h_{12}\\
h_{21} & h_{22}
\end{pmatrix},
\end{equation*}
then the restriction of $g$ to $\widehat{\beta}$ and the $(h, \widehat{\beta})$-{replacement} of $g$ are
$$
\begin{pmatrix}
g_{22} & g_{24}\\
g_{42} & g_{44}
\end{pmatrix},
\text{ and } 
\left( \begin{array}{c>{\columncolor{gray!20}}cc>{\columncolor{gray!20}}c}
g_{11} & g_{12} & g_{13} & g_{14}\\
\rowcolor{gray!20}
g_{21} &{\cellcolor{gray!50}} h_{11} & g_{23} &{\cellcolor{gray!50}} h_{12}\\
g_{31} & g_{32}& g_{33} & g_{34}\\
\rowcolor{gray!20}
g_{41} &{\cellcolor{gray!50}} h_{21} & g_{43} &{\cellcolor{gray!50}} h_{22}
\end{array} \right) 
$$ respectively.
\end{Def}

We claim that we can assume that there is at most one $i \in \{1, \ldots, k+l\}$ such that $\gamma_i(S)$ is one of the groups in Remark \ref{symprem23}. Assume  $r<s$ are the only elements of $\{1, \ldots, k+l\}$ such that  $\gamma_r(S)$ and $\gamma_s(S)$ are as in Remark \ref{symprem23}. Recall that $\beta$ is as in \eqref{betaij}. If $i \in \{1, \ldots, k\}$, then denote $\beta_i= \beta_{(1,i)} \cup \beta_{(2,i)}.$ Let $\hat{\beta}$ be $\beta_r \cup \beta_s$ if $r>k$ and $\beta_{(1,r)} \cup \beta_s \cup \beta_{(2,r)}$ if $r<k.$ Consider $g \in S$ and notice that its restriction $\widehat{g}$ to $\widehat{\beta}$ lies in $GSp_{|\widehat{\beta}|}(q, {\bf f}_{\widehat{\beta}}),$ where ${\bf f}_{\widehat{\beta}}$ is the restriction of ${\bf f}_{\beta}$ to $\langle \widehat{\beta} \rangle$ with respect to the basis $\widehat{\beta}.$ If $\widehat{S}$ is the group consisting of restrictions to $\widehat{\beta}$ for all $g \in S$, then, as is easy to see, $\widehat{S}$ is one of the groups in Lemma \ref{smSp23}. For example, if $g$ is as in \eqref{gst}, $r=1$ and $s=k+l$, then $\widehat{g}$ is constructed by the dark gray blocks of the following matrix
$$
\left( \begin{array}{>{\columncolor{gray!20}}ccccc>{\columncolor{gray!20}}ccc>{\columncolor{gray!20}}c}
\rowcolor{gray!20}
{\cellcolor{gray!50}} {\tau(g)\gamma_{1}(g)^{\dagger}}&* &{ *  } & * & \ldots &{\cellcolor{gray!50}} *&* &*    &{\cellcolor{gray!50}} *  \\
        & \ddots &{} & & &\ddots & &    &\\
0        & & {\tau(g)\gamma_{k}(g)^{\dagger}}&* &* &* &\ldots &*    &* \\  
        & &  &\gamma_{k+1}(g) & & {0}  &* & \ldots   &*\\ 
        & &  & &\ddots &       & & \ddots   &\\
\rowcolor{gray!20}
{\cellcolor{gray!50}}   & &  &0 & &{\cellcolor{gray!50}} {\gamma_{k+l}(g)}    & * & \ldots   &{\cellcolor{gray!50}} *\\  
        & & & & & & \gamma_{k}(g) &  *  &* \\
        & & & & & & & \ddots   &*\\
\rowcolor{gray!20}
{\cellcolor{gray!50}} 0        & & & & &{\cellcolor{gray!50}} &0 &    &{\cellcolor{gray!50}} \gamma_{1}(g)\\ 
\end{array} \right).  
$$
Let $h \in Sp_{|\widehat{\beta}|}(q, {\bf f}_{\widehat{\beta}})$ and let $t$ be the $(h, \widehat{\beta})$-{replacement} of $I_n.$ It is routine to check that $t \in Sp_n(q, {\bf f}_{\beta})$ and the restriction of  $g^t$ to $\widehat{\beta}$ is $\widehat{g}^{{h}}.$  Let $x$ be the matrix \eqref{exSP}. Notice that if $\widehat{S}$, $\widehat{S^x}$ and $\widehat{x}$ are the restrictions of ${S}$, $S^x$ and $x$ to $\widehat{\beta}$ respectively, then $\widehat{S^x}=\widehat{S}^{\widehat{x}}.$ Therefore, by Lemma \ref{smSp23}, there exist $\hat{y}, \hat{z} \in Sp_{|\widehat{\beta}|}(q, {\bf f}_{\widehat{\beta}})$ such that 
$$\widehat{S} \cap \widehat{S}^{\widehat{x} \widehat{y}} \cap \widehat{S}^{\widehat{z}}\le Z(GSp_{|\widehat{\beta}|}(q, {\bf f}_{\widehat{\beta}})).$$
 For $i \ne r,j$ define $y_i$ and $z_i$ as in \eqref{smintSp}. Let $y$ and $z$ be the $(\widehat{y}, \widehat{\beta})$-{replacement} and $(\widehat{z}, \widehat{\beta})$-{replacement} of  matrices from \eqref{yizi} respectively. It is routine to check that $y,z \in Sp_n(q,{\bf f}_{\beta}).$ Observe now that 
$\tilde{S}= S \cap S^{xy} \cap S^{z}$ is a group of diagonal matrices.

If there is more than one such pair $(r,s),$ then the same corrections of $y$ and $z$ for each pair can be done. Therefore, we can assume that there is at most one $s \in \{1, \ldots, k+l\}$ such that $\gamma_s(S)$ is one of the groups in Remark \ref{symprem23}. If there is no such $s$, then {\bf Step 2} of the proof for $q>3$ implies the theorem, so assume that such $s$ exists.

\medskip

{\bf Step 2.} Let $s \in \{1, \ldots, k+l\}$ be such that $\gamma_s(S)$ is $Q$, $R$ or $T$ as defined after Remark \ref{symprem23}. 

\medskip Since the only diagonal matrix in $Sp_n(2)$ is $I_n$, it is enough to obtain four conjugates of $S$ in $Sp_n(2)$ such that their intersection is a group of diagonal matrices. Therefore, if $\gamma_s(S) \in \{GL_2(q), Sp_2(q) \wr \Sym(2)\},$ then by Lemma  \ref{Spq2c23} and Theorem \ref{sympirr}
 using the construction in {\bf Step 1} we obtain $x,y,z \in Sp_n(q)$ such that 
$$S \cap S^x \cap S^y \cap S^z =\{1\}.$$
Hence for $q=2$ we only need to consider the situation $\gamma_s(S)=R$. 

 We consider three distinct cases -- when  $\gamma_s(S)$ is $R$, $T$ and $Q$ respectively.

\medskip

{\bf Case 1.}
First assume $\gamma_s(S)$ is  $R$, so $s>k.$ Without loss of generality, we can assume $s=k+1.$  Let $\beta$ be as in {\bf Step 2} of the proof for $q>3$. Let $r$ be such that restriction of matrices from $S$ to vectors $\{f_r, e_r\}$ is $\gamma_s(S).$ We relabel vectors in $\beta$ as follows:
\begin{itemize}
\item $f_r$ and $e_r$ become $f$ and $e$ respectively;
\item if $i<r$, then $f_i$ and $e_i$ remain $f_i$ and $e_i$ respectively;
\item if $i>r$, then  $f_i$ and $e_i$ become $f_{i-1}$ and $e_{i-1}$ respectively.
\end{itemize}
Therefore, since $\alpha^{\dagger}=\alpha$ for $\alpha \in \mathbb{F}_q$ with $q \in \{2,3\},$ a matrix $g \in \tilde{S}$ has shape 
$$\diag[\tau(g)\alpha_1 I_1, \ldots, \tau(g) \alpha_k I_{n_{k}}, \Lambda, \alpha_{k+2} I_{n_{k+2}}, \ldots, \alpha_{k+l} I_{n_{k+l}}, \alpha_k I_{n_k}, \ldots, \alpha_1 I_{1} ],$$
where 
\begin{equation*}
\Lambda =
\begin{pmatrix}
 \lambda_1 & \lambda_2  \\
 \lambda_3 & \lambda_4  \\
\end{pmatrix}.
\end{equation*}

Let $k>0$ and $n_1 \ge 2.$ Now $W=\langle e_1, \ldots, e_{n_1} \rangle$ is an $S$-invariant subspace. Let $m=n_1$ and  let $a \in GL_n(q)$ be such that
\begingroup
\allowdisplaybreaks
\begin{align*}
&(e_1)a =\sum_{i=m+1}^{(n-2)/2} e_i +e_1 + e; \\ & (f_1)a=f_1; \\
&(e_2)a =\sum_{i=m+1}^{(n-2)/2} e_i +e_2 + f -f_1; \\ & (f_2)a=f_2; \\
&(e_{j})a = \sum_{i=m+1}^{(n-2)/2}e_i + e_j; & &  j\in \{3, \ldots, m \} \stepcounter{equation}\tag{\theequation} \label{aSp2n12} \\ &(f_{j})a=f_{j}; & &  j\in \{3, \ldots, m \}  \\
&(e_{j})a =e_{j};  && j\in \{m+1, \ldots, (n-2)/2 \}   \\ &(f_{j})a=f_{j} - \sum_{i=1}^m f_{i}; && j\in \{m+1, \ldots, (n-2)/2 \}   \\
& (e)a=e+f_2; \\ & (f)a=f-f_1.
\end{align*}
\endgroup
It is routine to check that $a \in Sp_n(q, {\bf f}_{\beta})$. Consider $g \in \tilde{S} \cap S^a.$  Since $S$ stabilises the subspace $W$,  $S^a$ stabilises $(W)a.$ As in the proof for $q>3,$ let $\delta_i \in \mathbb{F}_q$ be such that $(e_i)g=\delta_i e_i$ for $i \in\{1, \ldots, (n-1)/2\}.$  Therefore, 
\begin{equation}\label{Wisotrq23n12}
((e_1)a)g=
\begin{cases}
\sum_{i=m+1}^{(n-2)/2}\delta_i e_i +\alpha_1e_1 + \lambda_3 f +\lambda_4 e;\\
\eta_1 (e_1)a+ \ldots +\eta_m(e_m)a
\end{cases}
\end{equation}
for some $\eta_1, \ldots, \eta_m \in \mathbb{F}_q.$ Observe that $((e_1)a)g$ does not have $e_i$ for $1<i  \le m$  in the first line of \eqref{Wisotrq23n12}, so $((e_1 )a)g = \eta_1 (e_1 )a$; thus $\lambda_2=0$ and
$$\lambda_4=\alpha_1=\delta_{m+1}= \ldots= \delta_{(n-2)/2}.$$
The same argument for $((e_2)a)g$ shows that $\lambda_2=0$ and $\tau(g)\alpha_1=\lambda_1=\alpha_1.$ 
Therefore, $g=\alpha_1 I_n$, so $g \in Z(GSp_n(q, {\bf f}_{\beta})).$

Let $k=1$ and $n_1 = 1.$ So $\langle e_1\rangle$ is an $S$-invariant subspace of $V$. If $n=4,$ then  Theorem \ref{theoremSp} is verified by computation. So we can assume that $n>4.$ 
Thus, $l \ge 2$ and 
$$W= \langle f_2 \ldots, f_{(n-2)/2}, e_2, \ldots, e_{(n-2)/2}, e_1 \rangle $$
is $S$-invariant.
Let $a \in GL_n(q)$ be such that
\begin{equation}\label{aSp2n11}
\begin{aligned}
&(e_1)a =\sum_{i=2}^{(n-2)/2} e_i +e_1 + e; & & (f_1)a=f_1; \\
&(e_2)a = e_2 + f -f_1; & & (f_2)a=f_2 -f_1; \\
&(e_{j})a =e_{j};  & &(f_{j})a=f_{j} - f_1; && j\in \{2, \ldots, (n-2)/2 \}   \\
& (e)a=e+f_2; & & (f)a=f-f_1.
\end{aligned}
\end{equation}
It is routine to check that $a \in Sp_n(q, {\bf f}_{\beta})$. Consider $g \in \tilde{S} \cap S^a.$  Since $S$ stabilises the  subspaces $\langle e_1 \rangle$ and $W$,  $S^a$ stabilises $\langle (e_1)a \rangle$ and $(W)a.$ Therefore, 
\begin{equation}\label{Wisotrq23n11}
((e_1)a)g=
\sum_{i=2}^{(n-2)/2}\delta_i e_i +\alpha_1e_1 + \lambda_3 f +\lambda_4 e= \eta (e_1)a
\end{equation}
for some $\eta \in \mathbb{F}_q.$ Hence $\delta_2= \ldots = \delta_{(n-2)/2}=\lambda_4=\alpha_1$ and $\lambda_3=0.$
In the same way 
\begin{equation}\label{Wisotrq23n11f}
((e_2)a)g=
\begin{cases}
\delta_2 e_2 + \lambda_1 f +\lambda_2 e - \tau(g)\alpha_1 f_1;\\
\eta_1 (e_1)a+ \sum_{i=2}^{(n-2)/2} \eta_i (e_i)a+ \sum_{i=2}^{(n-2)/2} \xi_i (f_i)a
\end{cases}
\end{equation}
for some $\eta_i, \xi_i \in \mathbb{F}_q.$ Since $((e_2)a)g$ does not have $e_i$ for $i>2$ and $f_j$ for $j>1$ in the first line of \eqref{Wisotrq23n11f}, $((e_2)a)g= \eta_2 (e_2)a$. Therefore, 
$\tau(g)\alpha_1=\lambda_1=\alpha_1$ and $\lambda_2=0,$ so $g= \alpha_1 I_n \in Z(Sp_n(q, {\bf f}_{\beta})).$

Let $k \ge 2$ and $n_1 = 1.$ So $\langle e_1\rangle$ and $W=\langle e_1, e_2, \ldots, e_{(n_2/2+1)} \rangle $ are $S$-invariant subspaces of $V$. Let $a$ be as in \eqref{aSp2n11}. Consider $g \in \tilde{S} \cap S^a.$  Since $S$ stabilises the subspaces $\langle e_1 \rangle$ and $W$,  $S^a$ stabilises $\langle (e_1)a \rangle$ and $(W)a.$ Therefore, \eqref{Wisotrq23n11} holds, so $\delta_2= \ldots = \delta_{(n-2)/2}=\lambda_4=\alpha_1$ and $\lambda_3=0.$
In the same way 
\begin{equation}\label{Wisotrq23n11fk2}
((e_2)a)g=
\begin{cases}
\delta_2 e_2 + \lambda_1 f +\lambda_2 e - \alpha_1 f_1;\\
\eta_1 (e_1)a+ \sum_{i=2}^{(n_2/2+1)} \eta_i (e_i)a
\end{cases}
\end{equation}
for some $\eta_i \in \mathbb{F}_q.$  Since $((e_2)a)g$ does not have $e_i$ for $i>2$  in the first line of \eqref{Wisotrq23n11fk2}, $((e_2)a)g= \eta_2 (e_2)a$. Therefore, 
$\lambda_1=\alpha_1$ and $\lambda_2=0,$ so $g= \alpha_1 I_n \in Z(GSp_n(q, {\bf f}_{\beta})).$

Let $k=0,$ so $g \in S$ is a block-diagonal matrix with blocks in $\gamma_i(S) \le GSp_{n_i}(q).$ Let $W= \langle f_1, \ldots, f_{n_2/2}, e_1, \ldots, e_{n_2/2} \rangle$. If $n=4,$ then  Theorem \ref{theoremSp} follows by Lemma \ref{smSp23}, so let $n \ge 6.$ We can assume that $n_2 \ge 4.$ Indeed, if  $n_i=2$ for $i \in \{1, \ldots, l\},$ then we can consider $S_1=\diag [\gamma_2(S), \gamma_3(S)] \le \diag[GSp_2(q), GSp_2(q)]$ as a subgroup in $Sp_4(q)$. By Lemma \ref{smSp23}, $b_{S_1}(Sp_4(q))\le 3.$ We redefine $\gamma_2(g)$ to be $\diag[GSp_2(q), GSp_2(q)]$, so now $n_2=4.$   Let $m=n_2/2$ and let $a \in Sp_n(q, {\bf f}_{\beta})$ be defined by \eqref{aSp2n12}.  Arguments similar to the case $(k>0, n_1 \ge 2)$ imply $g \in Z(GSp_n(q, {\bf f}_{\beta})).$

\medskip

{\bf Case 2.} Let $q=3$ and $\gamma_s(S)$ is $T$ as defined  after Remark \ref{symprem23}.
Without loss of generality, we can assume $s=k+1.$  Let $\beta$ be as in {\bf Step 2} of the proof for $q>3$. Let $r$ be such that the restriction of matrices from $S$ to vectors $\{f_r, f_{r+1}, e_r, e_{r+1}\}$ is $\gamma_s(S).$ We relabel vectors in $\beta$ as follows:
\begin{itemize}
\item $f_r$, $f_{r+1}$, $e_r$ and  $e_{r+1}$ become $f$, $f_0$, $e$ and $e_0$ respectively;
\item if $i<r$, then $f_i$ and $e_i$ remain $f_i$ and $e_i$ respectively;
\item if $i>r$, then  $f_i$ and $e_i$ become $f_{i-2}$ and $e_{i-2}$ respectively.
\end{itemize}
Let  ${y}_s, {z}_s \in Sp_4(q)$ be such that $\gamma_s(S) \cap \gamma_s(S)^{{y}_s} \cap \gamma_s(S)^{{z}_s}$ is as in $(1)$ of Lemma \ref{Spq2c23}. For $i \ne s$ define $y_i$ and $z_i$ as in \eqref{smintSp}. Define $y$ and $x$ as in \eqref{yizi} and $\tilde{S}$ as in {\bf Step 1} of the proof for $q>3.$

 Therefore, $g \in \tilde{S}$ has shape 
$$\diag[\tau(g)\alpha_1 I_1, \ldots, \tau(g)\alpha_k I_{n_{k}}, \Lambda, \alpha_{k+2} I_{n_{k+2}}, \ldots, \alpha_{k+l} I_{n_{k+l}}, \alpha_k I_{n_k}, \ldots, \alpha_1 I_{1} ],$$
where $$
\Lambda = 
\left( \begin{smallmatrix}
\lambda_1&0&\lambda_2&0\\
0&\lambda_1&0&0\\
0&0&\lambda_1&0\\
0&0&0&\lambda_1
\end{smallmatrix} \right) \in
 \left\langle 2I_4, \left( \begin{smallmatrix}
1&0&1&0\\
0&1&0&0\\
0&0&1&0\\
0&0&0&1
\end{smallmatrix} \right)
 \right\rangle$$ and $\lambda_i \in \mathbb{F}_3$ for $i \in \{1,2\}.$
Let $W= \langle e_1, \ldots, e_{n_1} \rangle$ if $k>0$ and $$W= \langle f_1, \ldots, f_{n_{1}/2},e_1, \ldots, e_{n_1/2} \rangle$$ if $k=0.$ Let $m=n_1$ for $W$ totally singular and $m=n_1/2$ for $W$ non-degenerate. Let $a \in GL_n(q)$ be such that 
\begingroup
\allowdisplaybreaks
\begin{align*}
&(e_1)a =\sum_{i=m+1}^{(n-4)/2} e_i +e_1 + f +\underline{f_1}; \\ & (f_1)a=f_1; \\
&(e_{j})a =\sum_{i=m+1}^{(n-4)/2} e_i +e_j; && j\in \{2, \ldots, m \} \\ &(f_{j})a=f_{j};  && j\in \{2, \ldots, m \}   \\
&(e_{j})a =e_j;  && j\in \{m+1, \ldots, (n-4)/2 \} \stepcounter{equation}\tag{\theequation}  \label{aSp4qq3} \\ &(f_{j})a=f_{j} -\sum_{i=1}^{m} f_i;  && j\in \{m+1, \ldots, (n-4)/2 \}   \\
& (e)a=e+f_1; \\ & (f)a=f; \\
& (e_0)a=e_0; \\ & (f_0)a=f_0.
\end{align*}
\endgroup
Here the underlined part is present only in the case $W$ is totally isotropic. It is routine to check that $a \in Sp_n(q, {\bf f}_{\beta}).$   Since $S$ stabilises the subspace $W$,  $S^a$ stabilises $(W)a.$ 
 Therefore, 
\begin{equation}\label{Wisotrq3n14}
((e_1)a)g=
\begin{cases}
\sum_{i=m+1}^{(n-4)/2}\delta_i e_i +\alpha_1e_1 + \lambda_1 f + \lambda_2 e + \underline{\tau(g) \alpha_1 f_1};\\
\sum_{i=1}^{m} \eta_i (e_i)a+ \sum_{i=1}^{m} \xi_i (f_i)a
\end{cases}
\end{equation}
for some $\eta_i, \xi_i \in \mathbb{F}_q.$ Here all $\xi_i=0$ if $W$ is totally isotropic.
Since $((e_1)a)g$ does not have $e_i$ for $1<i  \le m$ and $f_i$ for $1 \le i \le m$ (for $W$ non-degenerate) in the first line of \eqref{Wisotrq3n14}, $((e_1 )a)g = \eta_1 (e_1 )a$, so $\lambda_2=0$ and
$$\alpha_1=\tau(g)\alpha_1=\lambda_1=\delta_{m+1}= \ldots= \delta_{(n-4)/2}.$$
Therefore, $g=\alpha_1I_n$ and  $g \in Z(GSp_n(q, {\bf f}_{\beta})).$

\medskip

{\bf Case 3.} Let $q=3$ and $\gamma_j(S)=GL_2(q),$ so $s \le k.$
If $k+l=1$, then Theorem \ref{theoremSp} follows by $(3)$ of Lemma \ref{Spq2c23}. Let $\beta$ be as in {\bf Step 2} of the proof for $q>3$. Let $r$ be such that the restriction of matrices of $S$ to vectors $\{ e_r, e_{r+1}\}$ is $\gamma_s(S).$ We relabel vectors in $\beta$ as follows:
\begin{itemize}
\item $f_r$, $f_{r+1}$, $e_r$ and  $e_{r+1}$ become $f$, $f_0$, $e$ and $e_0$ respectively;
\item if $i<r$, then $f_i$ and $e_i$ remain $f_i$ and $e_i$ respectively;
\item if $i>r$, then  $f_i$ and $e_i$ become $f_{i-2}$ and $e_{i-2}$ respectively.
\end{itemize}
Let  ${y}_s, {z}_s \in Sp_4(q)$ be such that $\gamma_s(S) \cap \gamma_s(S)^{{y}_s} \cap \gamma_j(S)^{{z}_s}$ is as in $(2)$ of Lemma \ref{Spq2c23}. For $i \ne s$ define $y_i$ and $z_i$ as in \eqref{smintSp}. Define $y$ and $x$ as in \eqref{yizi} and $\tilde{S}$ as in {\bf Step 1} of the proof for $q>3.$
Therefore, $g \in \tilde{S}$ has shape 
\begin{multline*}
\diag[\tau(g)\alpha_1 I_{n_1}, \ldots, \tau(g)\alpha_{s-1} I_{n_{s-1}}, \Lambda_1,\tau(g)\alpha_{s+1} I_{n_{s+1}}, \ldots , \tau(g)\alpha_k I_{n_{k}}, \\
 \alpha_{k+1} I_{n_{k+1}}, \ldots, \alpha_{k+l} I_{n_{k+l}}, \\
 \alpha_k I_{n_k}, \ldots, \alpha_{s+1} I_{n_{s+1}}, \Lambda_2, \alpha_{s-1} I_{n_{s-1}},  \ldots, \alpha_1 I_{n_1} ],
\end{multline*}
where $$
\diag[\Lambda_1, \Lambda_2] = 
\left( \begin{smallmatrix}
\lambda_1&\lambda_2&0&0\\
\lambda_3&\lambda_4&0&0\\
0&0&\lambda_5&\lambda_6\\
0&0&\lambda_7&\lambda_8
\end{smallmatrix} \right) \in
 \left\langle \left( \begin{smallmatrix}
1&0&0&0\\
1&2&0&0\\
0&0&1&1\\
0&0&0&2
\end{smallmatrix} \right), \left( \begin{smallmatrix}
2&2&0&0\\
2&1&0&0\\
0&0&2&2\\
0&0&2&1
\end{smallmatrix} \right)
 \right\rangle$$ and $\lambda_i \in \mathbb{F}_3$ for $i \in \{1, \ldots, 8\}.$ Notice that if $\lambda_2=\lambda_6=0$, then $\diag[\Lambda_1, \Lambda_2]$ is the scalar matrix $\alpha I_n$ with $\alpha \in \mathbb{F}_3^*,$ so $\tau(g)=\alpha^2=1.$

Assume $s>1$ and let $W=\langle e_1, \ldots, e_{m} \rangle$, where $m=n_1.$  Let $a \in GL_n(q)$ be such that 
\begin{equation*}
\begin{aligned}
&(e_1)a =\sum_{i=m+1}^{(n-4)/2} e_i +e_1 +e+ f; & & (f_1)a=f_1; \\
&(e_{j})a = e_j;  & &(f_{j})a=f_{j};  && j\in \{2, \ldots, m \}   \\
&(e_{j})a =e_j;  & &(f_{j})a=f_{j} - f_1;  && j\in \{m+1, \ldots, (n-4)/2 \}   \\
& (e)a=e+f_1; & & (f)a=f-f_1; \\
& (e_0)a=e_0; & & (f_0)a=f_0.
\end{aligned}
\end{equation*}
 It is routine to check that $a \in Sp_n(q, {\bf f}_{\beta}).$   Since $S$ stabilises  $W$,  $S^a$ stabilises $(W)a.$ 
 Therefore, 
\begin{equation}\label{Wisotrsn1}
((e_1)a)g=
\begin{cases}
\sum_{i=m+1}^{(n-4)/2}\delta_i e_i +\alpha_1e_1 + \lambda_5 e+ \lambda_6 e_0 +\lambda_1 f + \lambda_2 f_0;\\
\sum_{i=1}^{m} \eta_i (e_i)a
\end{cases}
\end{equation}
for some $\eta_i \in \mathbb{F}_q.$
Since $((e_1)a)g$ does not have $e_i$ for $1<i  \le m$  in the first line of \eqref{Wisotrsn1}, $((e_1 )a)g = \eta_1 (e_1 )a$, so $\lambda_2=\lambda_6=0$ and
$$\alpha_1=\lambda_1=\lambda_5=\delta_{m+1}= \ldots= \delta_{(n-4)/2}.$$
Therefore, $g=\alpha_1I_n$ and  $g \in Z(GSp_n(q, {\bf f}_{\beta})).$

Assume $s=1$ and let $W =\langle e,e_0 \rangle.$ Let $a \in GL_n(q)$  be such that 
\begin{equation*}
\begin{aligned}
&(e)a =e+ f; & & (f)a=f; \\
&(e_{0})a =\sum_{i=1}^{(n-4)/2} e_i + e_0;  & &(f_{0})a=f_{0};  &&   \\
&(e_{j})a =e_j;  & &(f_{j})a=f_{j} - f_0;  && j\in \{1, \ldots, (n-4)/2 \}.   \\
\end{aligned}
\end{equation*}
It is routine to check that $a \in Sp_n(q, {\bf f}_{\beta}).$   Since $S$ stabilises  $W$,  $S^a$ stabilises $(W)a.$ 
 Therefore, 
\begin{equation}\label{Wisotrsis1}
((e)a)g=
\begin{cases}
 \lambda_5 e+ \lambda_6 e_0 +\lambda_1 f + \lambda_2 f_0;\\
 \eta_1 (e)a + \eta_2 (e_0)a
\end{cases}
\end{equation}
for some $\eta_1, \eta_2\in \mathbb{F}_q.$ Since $((e)a)g$ does not have $e_i$ for $1 \le i  \le (n-4)/2$  in the first line of \eqref{Wisotrsis1}, $((e )a)g = \eta_1 (e )a$, so $\lambda_2=\lambda_6=0$ and $\diag[\Lambda_1, \Lambda_2]$ is the scalar matrix $\lambda_1 I_4.$ In the same way 
\begin{equation}\label{Wisotrsis12}
((e_0)a)g=
\begin{cases}
 \sum_{i=1}^{(n-4)/2} \delta_i e_i + \lambda_1 e_0;\\
 \eta_1 (e)a + \eta_2 (e_0)a
\end{cases}
\end{equation}
for some $\eta_1, \eta_2\in \mathbb{F}_q.$ Since $((e_0)a)g$ does not have $e$ or $f$  in the first line of \eqref{Wisotrsis12}, $((e )a)g = \eta_1 (e )a$, so 
$\lambda_1=\delta_i$ for $i \in {1, \ldots, (n-4)/2}$ and $g = \lambda_1 I_n \in Z(GSp_n(q, {\bf f}_{\beta}))$ since $\tau(g)=1.$
\end{proof}

Theorem \ref{theoremSp} now follows by Theorems \ref{SpSirr41}, \ref{Spqgt3} and \ref{Spq23}.

\subsection{Solvable subgroups not contained in $\GS_n(q)$.}
\label{sec432}

If $q=2^f$, then $Sp_4(q)$ has a graph-field automorphism $\psi$ of order $2f$; see \cite[\S 12.3]{carsim} for details. If $\beta$ is a basis of $V$ as in Lemma \ref{unist}, then we can assume that $\psi^2$ is $\phi_{\beta}$ by \cite[Proposition 12.3.3]{carsim}.

\begin{T5}
Let $q$ be even and let ${A}=\Aut(PSp_4(q)')$. If ${S} $ is a maximal solvable subgroup of  {A}, then $b_{{S}}({S} \cdot Sp_4(q)') \le 4,$ so $\Reg_S(S \cdot Sp_n(q)',5)\ge 5$.
\end{T5}
\begin{proof}
For $q=2$ the statement is verified by computation.

 Assume $q=2^f$ with $f>1.$ Let $\theta$ be a generator of $\mathbb{F}_q^*.$ By \cite[Proposition 2.4.3]{kleidlieb}, $\Delta =Sp_4(q) \times \langle \theta I_4 \rangle$ where $\Delta$ is as defined before Lemma \ref{unibasisl}. Therefore, $$\Aut(Sp_4(q)) \cong Sp_4(q) \rtimes \langle \psi \rangle,$$ and we identify these two groups. Denote $\Gamma:=Sp_4(q) \rtimes \langle \psi^2 \rangle,$ so $\Gamma = Sp_4(q) \rtimes \langle \phi \rangle.$
 
  If $S$ lies in $\Gamma$, then the statement follows by Theorem \ref{theoremSp}.
 
 Assume that $S$ does not lie in $\Gamma$, so $S$ is in a maximal subgroup $H$ of $A$ not contained in $\Gamma$. For a description of such maximal subgroups see \cite[\S 14]{asch} and \cite[Table 8.14]{maxlow}. If $H$ is a non-subspace subgroup, then the statement follows by Theorem \ref{bernclass}. If $H$ is a subspace subgroup, then $H$ is solvable by \cite[Table 8.14]{maxlow}, so $S=H$ and $b_S(S \cdot Sp_4(q))\le 3$ by \cite[Lemma 5.8]{burPS}. 
\end{proof}
